\def\coralreport{1}

\documentclass[11pt]{article}

\def\useAlgorithmic{0}

\usepackage{fullpage,graphicx,authblk}

\usepackage{amsthm}
\usepackage[authoryear]{natbib}
\usepackage{booktabs}
\usepackage[table,xcdraw]{xcolor}
\usepackage{titlesec}
\usepackage{xspace}
\usepackage{caption}
\titlespacing*{\subsection}{0pt}{2.25ex plus 1ex minus .2ex}{1ex plus .2ex}
\titlespacing*{\paragraph}{0pt}{2.25ex plus 1ex minus .2ex}{1em}
\newtheorem{theorem}{Theorem}
\newtheorem{corollary}{Corollary}
\newcounter{def}
\newtheorem{definition}[def]{Definition}

\newcommand{\mycitet}[2]{\citet{#2}}
\newcommand{\mycitep}[1]{\citep{#1}}
\newcommand{\mymanageline}[1]{#1}
\newcommand{\mymanagetable}[2]{#2}

\usepackage{amsmath,amsfonts,amssymb}

\if\useAlgorithmic1
\usepackage{algorithm}
\else
\usepackage[ruled,vlined,linesnumbered]{algorithm2e}
\fi

\usepackage[noend]{algpseudocode}

\usepackage{subcaption}
\usepackage{hyperref}
\hypersetup{
    colorlinks = true,
    citecolor={blue},
    linkcolor = {blue},
    menucolor = {black},
}

\newcommand{\argmin}{\operatorname{argmin}}
\newcommand{\argmax}{\operatorname{argmax}}
\newcommand{\proj}{\operatorname{proj}}
\newcommand{\conv}{\operatorname{conv}}
\renewcommand{\Re}{\mathbb{R}}
\newcommand{\Z}{\mathbb{Z}}
\newcommand{\Q}{\mathbb{Q}}
\newcommand{\B}{\mathbb{B}}
\newcommand{\F}{\mathcal{F}}

\newcommand{\X}{\mathcal{X}}
\newcommand{\FLP}{\mathcal{F_{\text{LP}}}}
\renewcommand{\S}{\mathcal{S}}
\newcommand{\E}{\mathcal{L}}
\newcommand{\J}{L}
\newcommand{\RR}{\mathcal{R}}
\newcommand{\RRLP}{\mathcal{R_{\text{LP}}}}
\renewcommand{\P}{\mathcal{P}}

\newcommand{\noprint}[1]{}
\newcommand{\code}[1]{\texttt{#1}}
\renewcommand{\t}[1]{\texttt{#1}}
\newcommand{\MIBS}{\texttt{MibS}}
\newcommand{\nonl}{\renewcommand{\nl}{\let\nl\oldnl}}
\newcommand{\midd}{\;\middle|\;}

\newtheorem{feascon}{Feasibility Condition}

\newtheorem{assumption}{Assumption}

\makeatletter
\let\OldStatex\Statex
\renewcommand{\Statex}[1][3]{%
  \setlength\@tempdima{\algorithmicindent}%
  \OldStatex\hskip\dimexpr#1\@tempdima\relax}
\makeatother
\algnewcommand{\An}{\textbf{and}\xspace}
\algnewcommand{\Or}{\textbf{or}\xspace}

\graphicspath{{./figures/}}
\makeatletter
\def\input@path{{figures/}}
\makeatother

\begin{document}

\title{A Branch-and-Cut Algorithm for Mixed Integer Bilevel Linear Optimization
  Problems and Its Implementation
}

\author{Sahar Tahernejad\thanks{\texttt{sat214@lehigh.edu}}}
\author{Ted K. Ralphs\thanks{\texttt{ted@lehigh.edu}}}
\affil{Department of Industrial and System Engineering, Lehigh University,
Bethlehem, PA}
\author{Scott T. DeNegre\thanks{\texttt{DeNegreS@hss.edu}}}
\affil{Hospital for Special Surgery, New York, New York}

\date{January 7, 2020}

\maketitle

\begin{abstract}
In this paper, we describe a comprehensive algorithmic framework for solving
mixed integer bilevel linear optimization problems (MIBLPs) using a generalized
branch-and-cut approach. The framework presented merges features from
existing algorithms (for both traditional mixed integer linear optimization
and MIBLPs) with new techniques to produce a flexible and robust framework
capable of solving a wide range of bilevel optimization problems. The
framework has been fully implemented in the open-source solver \MIBS{}. The
paper describes the algorithmic options offered by \MIBS{} and presents
computational results evaluating the effectiveness of the various options for
the solution of a number of classes of bilevel optimization problems from the
literature.

\if\coralreport0
\fi

\end{abstract}

\section{Introduction}
This paper describes an algorithmic framework for the solution of \emph{mixed
integer bilevel linear optimization problems} (MIBLPs) and \MIBS, its
open-source software implementation. MIBLPs comprise a difficult class of
optimization problems that arise in applications in which multiple, possibly
competing decision-makers (DMs), make a sequence of decisions over time.
For an ever-increasing array of such applications, the traditional framework
of mathematical optimization, which assumes a single DM with a single
objective function making a decision at a single point in time, is inadequate.

The motivation for the development of \MIBS{}, which was begun a decade ago,
is both to serve as an open test bed for new algorithmic ideas and to provide
a platform for solution of the wide variety of practical problems of this type
currently coming under scientific study. The modeling framework underlying
\MIBS{} is that of \emph{multilevel optimization}. Multilevel optimization
problems model applications in which decisions are made in a sequence, with
decisions made earlier in the sequence affecting the options available later
in the sequence. Under the assumption that all DMs are rational and have
complete information about their own and each other's models and input data
(there is no stochasticity), we can describe such problems formally using the
language of mathematical optimization. To do so, we consider a set of decision
variables, partitioned into subsets associated with individual DMs.
Conceptually, these sets of decision variables are ordered according to the
sequence in which decisions are to be made. We use the term \emph{level} (or
\emph{stage}, in some contexts) to denote each DM's position in the sequence.
The term is intended to conjure up a hierarchy in which decisions flow from
top to bottom, so decisions earlier in the sequence are said to be ``at a
higher level.'' The decisions made by higher-level DMs influence those of
lower-level DMs through a parametric dependence of the lower-level decision
problems on higher-level decisions. Thus, higher-level decisions must be made
taking into account the effect those decisions will have on lower-level
decisions, which in turn impact the objective value of the higher-level
solution.

The decision hierarchy of a national government provides an archetypal
example. The national government makes decisions about tax rates or subsidies
that in turn affect the decisions of state and local governments, which
finally affect the decisions of individual taxpayers. Since later decisions 
affect the degree to which the national government achieves its original
objective, the decisions at that higher level must be made in light of the
reactions of lower-level DMs to those decisions~\mycitep{bard13}.

Thanks to the steady improvement in solution methodologies for traditional
optimization problems and the availability of open-source
solvers~\mycitep{COIN-OR}, the development of practical solution methods for
MIBLPs has become more realistic. The availability of practical solution
methods has in turn driven an accompanying increase in demand for such
methodologies. The literature is now replete with applications requiring the
solution of such models
~\mycitep{salmeronetal04,bienVerma10,ZhaSnyRalXue16,gaoYou17}. 
In the remainder of this section, we introduce the
basic framework of bilevel optimization.

\subsection{Mixed Integer Bilevel Optimization \label{sec:miblp}}

In this section, we formally introduce MIBLPs, the class of multilevel
optimization problem in which we have only two levels and two associated DMs.
The decision variables of an MIBLP are partitioned into two subsets, each
conceptually controlled by one of the DMs. We generally denote by $x$ the
variables controlled by the \emph{first-level DM} or \emph{leader} and require
that the values of these variables be contained in the set $X=\Z^{r_1}_+
\times \Re^{n_1-r_1}_+$ representing integrality constraints. We denote by $y$
the variables controlled by the \emph{second-level DM} or \emph{follower} and
require the values of these variables to be in the set $Y=\Z^{r_2}_+ \times
\Re^{n_2-r_2}_+$. Throughout the paper, we refer to a sub-vector of
$x\in\Re^{n_1}$ indexed by a set $K \subseteq \{1, \dots, n_1\}$ as $x_K$ (and
similarly for sub-vectors of $y\in\Re^{n_2}$).

The general form of an MIBLP is 
\begin{equation}\label{eqn:miblp}\tag{MIBLP}
  \min_{x \in X} \left\{cx + \Xi(x)\right\},
\end{equation}
where the function $\Xi$ is a \emph{risk function} that encodes the part of
the objective value of $x$ that depends on the response to $x$ in the second
level. This function may have different forms, depending on the precise
variant of the bilevel problem being solved. In this paper, we focus on the
so-called \emph{optimistic} case~\mycitep{loridan96}, in which
\begin{equation}\label{eqn:Xi} \tag{RF-OPT}
  \Xi(x) = \min \left\{ d^1y \midd y \in \P_1(x), y \in \argmin
  \{d^2y \midd y \in \P_2(x) \cap Y\} \right\},
\end{equation}
where
\begin{equation*}
\P_1(x) = \left\{y \in \Re_+^{n_2} \midd G^1 y \geq b^1 - A^1 x\right\}
\end{equation*}
is a parametric family of polyhedra containing points satisfying the linear
constraints of the first-level problem with respect to a given
$x \in \Re^{n_1}$ and
\begin{equation*}\label{eqn:lowerProblem}
\P_2(x) = \left\{y\in \Re_+^{n_2}\midd G^2y\geq A^2 x \right\}
\end{equation*}
is a second parametric family of polyhedra containing points satisfying the
linear constraints of the second-level problem with respect to a given $x \in
\Re^{n_1}$. The input data is $A^1\in\Q^{m_1\times n_1}$, $G^1\in\Q^{m_1\times
  n_2}$, $b^1\in\Q^{m_1}$, $A^2\in\Q^{m_2\times n_1}$ and $G^2\in\Q^{m_2\times
  n_2}$. As is customary, we define $\Xi(x) = \infty$ in the case of either
$x\notin X$ or infeasibility of the problem on the right-hand side
of~\eqref{eqn:Xi}.  

Note that it is typical in the literature on bilevel optimization for the
parametric right-hand side of the second-level problem to include a fixed
component, i.e., to be of the form $b^2 - A^2x$ for some $b^2 \in \Q^{m^2}$.
The form introduced here is more general, since adding a constraint $x_1 = 1$
to the upper level problem results in a problem equivalent to the usual one.
The advantage of the form here is that it results in a risk function that is
subadditive in certain important cases (though not in general), a desirable
property for reasons that are beyond the scope of this paper. 

As we mentioned above, other variants of the risk function are possible and
the algorithm we present can be adapted to these cases (see
Section~\ref{sec:outline}). A pessimistic risk function can be obtained simply
by considering
\begin{equation}\label{eqn:XiPessimistic}\tag{RF-PES}
  \Xi(x) = \max\left\{d^1y \midd y \in \P_1(x), y\in\argmin
  \{d^2y \midd y \in \P_2(x) \cap Y\}\right\}.
\end{equation}
Note that according to this definition, the second-level solution $y$ is
required to be a member of $\P_1(x)$ (i.e., be feasible for the upper-level
constraints), though there could exist 
$y\in\argmin\{d^2y \mid y \in \P_2(x) \cap Y\}\setminus\P_1(x)$.
Allowing such second-level solutions to be chosen is also possible
and can be accommodated within the framework of \eqref{eqn:XiPessimistic} by
the addition of dummy variables at the first level. The details are beyond the
scope of this paper.  

It should be pointed out that we explicitly allow the second-level variables
to be present in the first-level constraints. This allowance is rather
non-intuitive and leads to many subtle algorithmic complications.
Nevertheless, there are applications in which it is necessary to allow this
and since \MIBS{} does allow this possibility, we felt it appropriate to
express the results in this full generality. The reader should keep in mind,
however, that many of the ideas discussed herein can be simplified in the more
usual case that $G^1 = 0$ and we endeavor to point this out in particular
cases.

Note that the formulation~\eqref{eqn:miblp} does not explicitly involve the
second-level variables. The reason for expressing the formulation in this way
is to emphasize two things. First, it emphasizes that the goal of
solving~\eqref{eqn:miblp} is to determine the optimal values of the
first-level variables only. Thus, we would ideally like to ``project out'' the
second-level variables. In contrast to the Benders method for traditional
optimization problems, however, the second-level variables are re-introduced
in solving the relaxation that we employ in our branch-and-cut algorithm (see
Section~\ref{sec:bounding}). Second, it is necessary at certain points in the
algorithm to evaluate $\Xi(x)$, and it is convenient to give it an
explicit form here to make this step clearer. 

MIBLPs also have an alternative formulation that employs the \emph{value
  function} of the second-level problem and explicitly includes the
second-level variables. This formulation is given by
\begin{equation}\label{eqn:miblp-vf}\tag{MIBLP-VF}
\min \left\{cx+d^1y\midd x\in X, y\in \P_1(x) \cap \P_2(x)\cap Y, 
d^2y\leq \phi(A^2x)\right\},
\end{equation}
where $\phi$ is the so-called \emph{value function} of the second-level
problem, which is a standard mixed integer linear optimization problem (MILP).
The value function returns the optimal value of second-level problem for
a given right-hand side, defined by 
\begin{equation}\label{eqn:phi} \tag{VF}
  \phi(\beta) = \min \left\{d^2y\midd G^2y\geq \beta, y\in Y \right\} \quad
  \forall \beta \in \Re^{m_2}. 
\end{equation}  

From this alternative formulation, it is evident that if the values of the
first-level variables are fixed in~\eqref{eqn:miblp-vf} (i.e., the constraints
imply that their values must be constant), then the problem of minimizing over
the remaining variables is an MILP. In fact, it is not difficult to observe
that only the first-level variables having non-zero coefficients in the
second-level constraints need be fixed in order for $\phi(A^2x)$ to be
rendered a constant and for~\eqref{eqn:miblp-vf} to reduce to an MILP. This
result is stated formally in Section~\ref{sec:pruning}. Because of their
central importance in what follows, we formally define the concept of
\emph{linking variables}.
\begin{definition}[Linking Variables]
Let
  \begin{equation*}
 \J = \left\{i \in \{1, \dots, n_1\}\midd A^2_i \neq 0 \right\},
 \end{equation*}
be the set of indices of first-level variables with non-zero coefficients
in the second-level problem, where $A^2_i$ denotes the $i^{th}$ column of matrix
$A^2$. We refer to such variables as \emph{linking variables}.
\end{definition}
Per our earlier notation, $x_{\J}$ is the sub-vector of $x\in\Re^{n_1}$
corresponding to the linking variables. By assuming all linking variables are
integer variables, we assure the existence of an optimal
solution~\mycitep{vicente96}.
\begin{assumption} \label{as:setJ}
$\J = \{1,\dots,k_1\}$ for $k_1 \leq r_1$.
\end{assumption}
We note here that it can in fact be assumed without loss of generality that
$k_1=r_1$ by simply moving the non-linking variables to the second level.
While this is conceptually inconsistent with the intent of the original model,
it is not difficult to see that the resulting model is \emph{mathematically}
equivalent. 

\noindent In what follows, it will be convenient to refer to the set
\begin{equation*}
  \P = \left\{(x,y) \in \Re_+^{n_1 \times n_2} \midd y \in \P_1(x) \cap \P_2(x)
  \right\}
\end{equation*} 
of all points satisfying the non-negativity and linear inequality
constraints at both levels and the set
\begin{equation*}
\S = \P \cap (X \times Y)
\end{equation*}
of points in $\P$ that also satisfy integrality restrictions.
\begin{assumption} \label{as:boundedness}
$\P$ is bounded.
\end{assumption}
This assumption is made to simplify the exposition, but is easy to relax in
practice.
Corresponding to each $x \in\Re^{n_1}$ with $x_{\J}\in\Z^{\J}$, 
we have the \emph{rational reaction set}, which is defined by
\begin{equation*}
\RR(x) = \argmin \left\{d^2y \midd y \in \P_2(x)\cap Y\right\}.
\end{equation*}
This set may be empty either because $\P_2(x)\cap Y$ is itself empty
or because there exists $r \in \Re_+^{n_2}$ such that $G^2 r \geq 0$ and $d^2r <
0$ (in which case the second-level problem is unbounded no matter what
first-level solution is chosen). The latter case can be easily detected in a
pre-processing step, so we assume w.l.o.g. that this does not occur (note
that this case cannot occur when $G^1 = 0$, since
Assumption~\ref{as:boundedness} would then imply that 
$\left\{r \in \Re_+^{n_2} \setminus 0\midd G^2 r \geq 0 \right\}=\emptyset$).
\begin{assumption}
$\left\{r \in \Re_+^{n_2} \midd G^2 r \geq 0, d^2 r < 0 \right\}=\emptyset$.
\end{assumption}
Under our assumptions, the \emph{bilevel feasible region} (with respect to the
first- and second-level variables in~\eqref{eqn:miblp-vf}) is
\begin{equation*}
\F = \left\{(x,y) \in X \times Y \midd y \in \P_1(x)\cap\RR(x) \right\}
\end{equation*}
and members of $\F$ are called \emph{bilevel feasible solutions}. Although the
ostensible goal of solving~\eqref{eqn:miblp} is to determine $x \in X$
minimizing $cx + \Xi(x)$, this is equivalent to finding a member of
$\F$ that optimizes the first-level objective function $cx + d^1y$. Observe,
however, that by this definition of feasibility, we may have $x^* \in X$ that
is optimal for~\eqref{eqn:miblp}, while for some $\hat{y}\in Y$, $(x^*,
\hat{y}) \in \F$ but $\Xi(x^*) < d^1 \hat{y}$.

Because we consider the problem to be that of determining the optimal
first-level solution, it is also useful to denote the feasible set with
respect to first-level variables only as
\begin{equation*}
\F_1 = \proj_x(\F).
\end{equation*}
For $x \in X$, we have that
\begin{equation*}
  x \in \F_1 \Leftrightarrow x \in \proj_x(\F) \Leftrightarrow
  \RR(x) \cap \P_1(x) \not= \emptyset
  \Leftrightarrow \Xi(x) < \infty
\end{equation*}
and we say that $x \in \Re^{n_1}$ is feasible if $x \in \F_1$.
We can then interpret the conditions for bilevel
feasibility of $(x, y)$ as consisting of the following two properties of the
first- and second-level parts of the solution independently.
\begin{feascon} \label{fc:first-level}
  $x \in \F_1$.
\end{feascon}
\begin{feascon} \label{fc:second-level}
  $y \in \P_1(x) \cap \RR(x)$.
\end{feascon}
We exploit this notion of feasibility
later in our algorithm. 

A related set (see Figure~\ref{fig:mooreexample}) is the set of feasible
solutions to the \emph{bilevel linear optimization problem} (BLP) that results
from discarding the integrality restrictions, defined as
\begin{equation*}
  \FLP = \left\{(x,y) \in \Re_+^{n_1\times n_2} \midd y \in \P_1(x) \cap
  \RRLP(x) \right\},
\end{equation*}
where
\begin{equation*}
\RRLP(x) = \argmin \left\{d^2y \midd y \in \P_2(x) \right\}.
\end{equation*}
Note that we do \emph{not} have in general that $\F \subseteq \FLP$, as discussed
later in Section~\ref{sec:challenges}.

\subsection{Special Cases}

There are a number of cases for which \MIBS{} has specialized methods.
One of the most important special cases is the \emph{zero sum} case in which
$d^1 = -d^2$.
The \emph{mixed integer interdiction problem} (MIPINT) is a
specific subclass of zero sum problem in which the first-level variables are
binary and are in one-to-one correspondence with the second-level variables 
($n=n_1=n_2$).
When a first-level variable is fixed to one, this prevents the associated
variable in the second-level problem from taking a non-zero value. The
interdiction problem can be formulated as~\mycitep{israeli99,DeNegre2011}
\begin{equation}\label{eqn:MIPINT}\tag{MIPINT}
  \min \left\{d^1y\midd x \in \P^{INT}_1\cap \B^{n},
  y\in\argmax\{d^1y\midd y \in \P^{INT}_2(x) \cap Y\}\right\},
\end{equation}
where
\begin{equation*}
\P^{INT}_1 = \left\{x\in \Re_+^{n}\midd A^1x\geq b^1 \right\},
\end{equation*}
\begin{equation*}
\P^{INT}_2(x) = \left\{y\in \Re_+^{n}\midd G^2y\geq b^2, y\leq
\text{diag}(u)(e-x) \right\},
\end{equation*}
$e$ represents an n-dimensional vector of ones and $u_i\in\Re$
denotes the upper bound of $y_i$ for $i = 1, ..., n$.
The special structure of MIPINTs can be employed to develop specialized
methods for these problems.

\subsection{Computational Challenges \label{sec:challenges}}

From the point of view of both theory and practice, bilevel problems are
difficult to solve. From the perspective of complexity theory, the problem is
NP-hard even for the case in which all variables are continuous ($r_1 = r_2 =
0$). The general case is in the class $\Sigma^p_2$-hard~\mycitep{J85},
which is the class of problems that can be solved in non-deterministic
polynomial time, given an NP oracle. To put it in terms that are a little less
formal, these are problems for which even the problem of checking feasibility
of a given point in $X \times Y$ is complete for NP in general. Alternatively,
simply computing $\Xi(x)$ for $x \in X$ (determining its objective function
value), is also NP-hard.

Because determining whether a given solution is feasible is an NP-complete
problem, solution of MIBLPs is inherently more difficult than solution of the
more familiar case of MILP (equivalent to the case of $\J = \emptyset$).
Search algorithms for MILPs rely on a certain amount of algorithmically guided
``luck'' to find high-quality solutions quickly. This works reasonably well in
this simpler case because checking feasibility of any given solution is
efficient. In the case of MIBLPs, we cannot rely on this property and even if
we are lucky enough to get a high-quality first-level solution, checking its
feasibility is still a difficult problem. Furthermore, some of the properties
whose exploitation we depend on in the case of MILP do not straightforwardly
generalize to the MIBLP case. For example, removing the integrality
requirement for all variables does not result in a relaxation, since relaxing
the second-level problem may make solutions that were previously feasible
infeasible. In fact, as we mentioned earlier, the feasible region $\FLP$ of
this BLP does not necessarily even contain the feasible region $\F$
of~\eqref{eqn:miblp}. Thus, even if the solution to this BLP is in 
$X \times Y$, it is not necessarily optimal
for~\eqref{eqn:miblp}. These properties can be seen in
Figure~\ref{fig:mooreexample}, which displays a well-known example originally
from~\mycitet{Moore and Bard}{moorebard90} that is well-suited for illustrating
these concepts.
\begin{figure}
\begin{center}
\begin{minipage}{.45\textwidth}
\scalebox{0.6}{\input{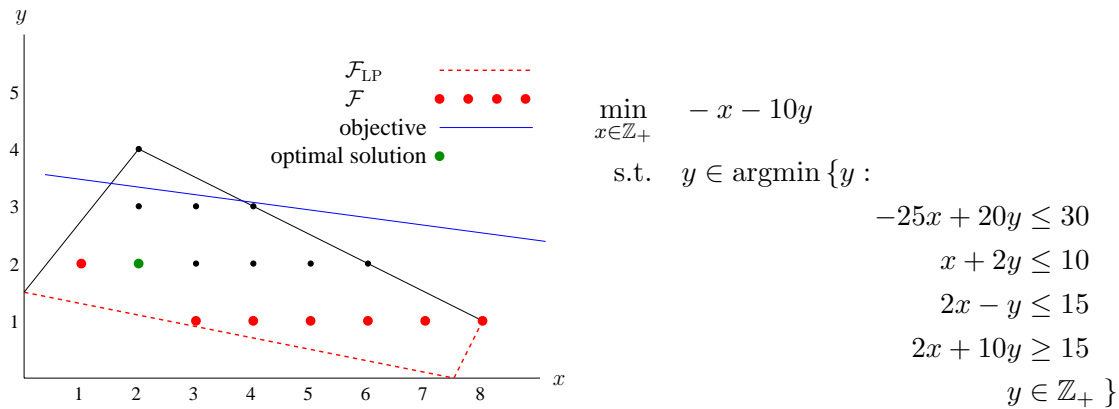}}
 \end{minipage}
\begin{minipage}{.45\textwidth}
\vskip .3in
\begin{alignat*}{3}
\min_{x\in \mathbb{Z}_+} & \quad -x - 10y\notag\\
\textrm{s.t.} &\quad y\in \operatorname{argmin}\left\{ y:\right.\notag\\
 & & - 25 x + 20 y &\leq 30\notag\\
 & & x + 2y &\leq 10\label{eqn:moorece}\\
 & &  2x - y & \leq 15\notag\\
 & & 2x + 10y & \geq 15\notag\\
 & & y &\in \mathbb{Z}_+\left.\right\}\notag
\end{alignat*}
\end{minipage}
\caption{The feasible region of IBLP \mycitep{moorebard90}.\label{fig:mooreexample}}
\end{center}
\end{figure}

\subsection{Previous Work}

Bilevel optimization has its roots in game theory and the general class of
problems considered in this paper are a type of \emph{Stackelberg game}. The
Stackelberg game was studied in a seminal work by~\mycitet{Von
  Stackelberg}{stackelberg34}. The first bilevel optimization formulations in
the form presented here were introduced and the term was coined in the 1970s
by~\mycitet{Bracken and McGill}{bracken-mcgill73}, but computational aspects
of related optimization problems have been studied since at least the 1960s
(see, e.g.,~\mycitet{}{wollmer64}). Study of algorithms for the case in which
integer variables appear is generally acknowledged to have been launched
by~\mycitet{Moore and Bard}{moorebard90}, who initiated working on general
MIBLPs and discussed the computational challenges of solving them. They
proposed a branch-and-bound algorithm for solving MIBLPs which is guaranteed
to converge if all first-level variables are integer or all second-level
variables are continuous.

Until recently, most computational work focused on special cases with
exploitable structure. \mycitet{Bard and Moore}{bard92}, \mycitet{Wen and
  Huang}{wen-huang96} and \mycitet{Fa{\'\i}sca et al.}{faisca-etal07} studied
the bilevel problems with binary variables. \mycitet{Bard and Moore}{bard92}
proposed an exact algorithm for \emph{integer bilevel optimization problems}
(IBLPs) in which all variables are binary. \mycitet{Wen and
  Huang}{wen-huang96}, on the other hand, considered MIBLPs with binary
first-level variables and continuous second-level ones and suggested a
\emph{tabu search heuristic} for generating solutions to these problems.
\mycitet{Fa{\'\i}sca et al.}{faisca-etal07} concentrated on MIBLPs in which
all discrete variables are constrained to be binary. They reformulated the
second-level problem as a multi-parametric problem with the first-level
variables as parameters.

\mycitet{DeNegre and Ralphs}{DeNRal09} and~\mycitet{DeNegre}{DeNegre2011}
generalized the existing framework of branch and cut, which is the standard
approach for solving MILPs, to the case of IBLPs by introducing the
\emph{integer no-good cut} to separate bilevel infeasible solutions from the
convex hull of bilevel feasible solutions (see Section~\ref{sec:cutting}).
They also initiated development of the open-source solver \MIBS{}. \mycitet{Xu
  and Wang}{xuwang14} focused on problems in which all first-level variables
are discrete and suggested a multi-way branch-and-bound algorithm in which
branching is done on the slack variables of the second-level problem. A
decomposition algorithm based on column-and-constraint generation was employed
by~\mycitet{Zeng and An}{zengan14} for solving general MIBLPs. Their method
finds the optimal solution if it is attainable, otherwise it finds an
$\epsilon$-optimal solution. \mycitet{Caramia and Mari}{caramiamari15}
introduced a non-linear cut for IBLPs and suggested a method of linearizing
this cut by the addition of auxiliary binary variables. Furthermore, they
introduced a branch-and-cut algorithm for IBLPs which employs the integer
no-good cut from \mycitep{DeNRal09}. Knapsack interdiction problems were
studied by~\mycitet{Caprara et al.}{capraraetal16} and they proposed an exact
algorithm for these problems. \mycitet{Hemmati and Smith}{hemmatiSmith16} formulated
competitive prioritized set covering problem as an MIBLP and proposed a
cutting plane algorithm for solving it. \mycitet{Wang and Xu}{wangxu17}
proposed the \emph{watermelon algorithm} for solving IBLPs, which removes
bilevel infeasible solutions that satisfy integrality constraints by a
non-binary branching disjunction. \mycitet{ Lozano and Smith}{lozanosmith17}
employed a single-level value function reformulation for solving the MIBLPs in
which all first-level variables are integer. \mycitet{Fischetti et
  al.}{fischettietal17a} suggested a branch-and-cut algorithm for MIBLPs
employing a class of \emph{intersection cuts} valid for MIBLPs under mild
assumptions and developed a new family of cuts for the MIBLPs with binary
first-level variables. In \mycitep{fischettietal17b}, they extended their
algorithm by suggesting new types of intersection cuts and
introduced the so-called \emph{hypercube intersection cut}, valid for MIBLPs
in which the linking variables are discrete. Finally,~\mycitet{Mitsos}{Mit10}
and~\mycitet{Kleniati and Adjiman}{KleAdj14-1,KleAdj14-2} considered the more
general case of mixed integer bilevel non-linear optimization.

\subsection{Overview of Branch and Cut \label{sec:BandC-general}}

Branch and cut (a term coined by~\mycitet{Padberg and
  Rinaldi}{padbergRinaldi87,padberg91a}) is a variant of the well-known
branch-and-bound algorithm of~\mycitet{Land and Doig}{LanDoi60}, the most
widely-used algorithm for solving many kinds of non-convex optimization
problems. For purposes of illustration, we consider here the solution of the
general optimization problem
\begin{equation} \label{eqn:generic-optimization} \tag{GO}
  \min_{x \in \X} f(x),
\end{equation}
where $\X \subseteq \Re^n$ is the \emph{feasible region}
and $f: \Re^n \rightarrow \Re$ is the \emph{objective function}.

The approach of both the branch-and-bound and the branch-and-cut algorithms is
to search the feasible region by partitioning it and then recursively solving
the resulting subproblems. Implemented naively, this results in an inefficient
complete enumeration. This potential inefficiency is avoided by utilizing
upper and lower bounds computed for each subproblem to intelligently ``prune''
the search. The recursive partitioning process can be envisioned as a process
of searching a rooted tree, each of whose nodes corresponds to a subproblem.
The root of the tree is the node corresponding to the original problem
(sometimes called the \emph{root problem}) and the nodes adjacent to it are
its \emph{children}. The parent-child relationship applies to other nodes
connected along paths in the tree in the obvious way. Nodes with no children
are called \emph{leaf nodes}. Henceforth, we denote the set of all leaf nodes
in the current search tree by $T$.

Although it is not usually described this way, we consider the algorithm as a
method for iteratively removing parts of the feasible region of the given relaxation 
that can be proven not to contain any
\emph{improving} feasible solutions (feasible solutions that are better
than the best solution found so far) until no more such regions can be found.
Although most sophisticated branch-and-bound algorithms for optimization do
implicitly discard parts of the feasible region that can be shown to contain
only suboptimal solutions, the algorithm we present does this more
\emph{explicitly}. The removal is done either by \emph{branching}
(partitioning the remaining set of improving solutions to define smaller
subproblems) or \emph{cutting} (adding inequalities satisfied by all improving
solutions).

A crucial element of both branching and cutting is the identification of
\emph{valid disjunctions}. The notion of valid disjunction defined here, which
we refer to as an \emph{improving} valid disjunction in order to distinguish it
from the more standard definition, refers to a collection of sets that contain
all \emph{improving} feasible solutions (while the standard valid disjunction
should include all feasible solutions). In the remainder of the paper, we drop
the word ``improving'' when referring to such disjunctions.
\begin{definition}[Improving Valid Disjunction] \label{def:valid-disjunction}
An \emph{(improving) valid disjunction} for~\eqref{eqn:generic-optimization}
with respect to a given $x^* \in \X$ is a disjoint collection 
\begin{equation*}
  X_1, X_2, \dots, X_k
\end{equation*}
of subsets of $\Re^n$ such that
\begin{equation*} 
  \left\{x \in \X \midd f(x) < f(x^*) \right\} \subseteq
  \bigcup_{1 \leq i \leq k} X_i .
\end{equation*}
\end{definition}
Imposing such disjunctions is our primary method of eliminating
suboptimal solutions and improving the strength of the relaxations used to
produce bounds on the optimal value. They play a crucial role in defining
methods of both branching and cutting, as we discuss in what follows. We now
briefly describe the individual components of branch and cut that we refer to
in Section~\ref{sec:BandC}.

\paragraph{Bounding.} The most important factor needed to improve the
efficiency of the algorithm is a tractable method of producing strong upper
and lower bounds on the optimal solution value of each subproblem. Typically,
lower bounds (assuming minimization) are obtained by solving a relaxation of
the original (sub)problem. In most cases, a convex relaxation is chosen in
order to ensure tractability, though with problems as difficult as MIBLPs,
even a non-convex relaxation may be tractable enough to serve the desired
purpose. Upper bounds are obtained by producing a solution feasible to the
subproblem, either by heuristic methods or by showing that the solution of the
relaxation is feasible to the original problem. We denote the lower and upper
bounds associated with node $t$ by $L^t$ and $U^t$, respectively.
Whenever we cannot obtain a feasible solution to a given subproblem $t$, we
set $U^t=\infty$. Similarly, if the relaxation at node $t$ is infeasible, we
have $L^t=\infty$.

The bounds on the individual subproblems can be aggregated to obtain global
bounds on the optimal solution value to the root problem. Upper bounds of all
leaf nodes are aggregated to obtain the global upper bound
\begin{equation*} 
U = \min_{t \in T} U^t, 
\end{equation*}
which represents the objective value of the best solution found so far (known
as the \emph{incumbent}). Note that this formula is a bit simplified for the
purposes of presentation, since it is possible for bilevel feasible solutions
that are \emph{not} feasible for the current subproblem to be produced during
the bounding step. While these do not technically update the upper bound for
the subproblem, they do contribute to improvement of the global upper bound.
Similarly, lower bounds for the leaf nodes are aggregated to form the global
lower bound
\begin{equation*}
L = \min_{t \in T} L^t.
\end{equation*}
These global upper and lower bounds are updated as the algorithm progresses
and when they become equal, the algorithm terminates.

\paragraph{Pruning.} Any node $t$ for which $L^t \geq U$ can be discarded
(\emph{pruned}), since the feasible region of such a node cannot contain a
solution with objective value lower than the current incumbent. This pruning
rule implicitly subsumes two special cases. The first is when the feasible
region of subproblem $t$ is empty in which case $L^t = \infty$. The second is
when $L^t = U^t\geq U$, which typically arises when solving the relaxation of the
subproblem associated with node $t$ produces a solution feasible for the
original problem. 

Since the global bounds are updated dynamically, this means that whether a node
can be pruned or not is not a fixed property---a node that previously could
not have been pruned may be pruned later when a better incumbent is found. We
discuss bounding techniques in more detail in Sections~\ref{sec:bounding}
and~\ref{sec:heuristics}.

\paragraph{Branching.} When bounds have been computed for a node and it cannot
be pruned, we then partition the feasible region by identifying and imposing a
valid disjunction of the form specified in
Definition~\ref{def:valid-disjunction}. The subproblems resulting from
partitioning the feasible region $\X^t$ of node $t$ are optimization problems of the form
\begin{equation*}
  \min_{x \in \X^t \cap X_i} f(x),\; 1 \leq i \leq k,
\end{equation*}
which can then be solved recursively using the same algorithm, provided that the
collection $\{X_i\}_{1\leq i \leq k}$ is chosen so as not to change the form of
the optimization problem.

Typically, the disjunction is chosen so that $\bigcup_{1 \leq i \leq k}X_i$ does
not contain the solution to the current relaxation, but we will see that this
is not always possible in the case of MIBLP. When $\X \subseteq \Z^r \times
\Re^{n-r}$, the disjunction
\begin{equation} \tag{GD} \label{eqn:GD}
  \begin{aligned}
    & X_1 = \left\{x \in \Re^{n} \midd \pi x \leq \pi_0  \right\}
    & X_2 = \left\{x \in \Re^{n} \midd \pi x \geq \pi_0 + 1\right\}
  \end{aligned}
\end{equation}
is always valid when $(\pi, \pi_0) \in \Z^{n+1}$ and $\pi_i = 0$ for $i > r$,
since we must then have $\pi x \in \Z$ for all $x \in \X$. If the solution
$\hat{x}$ to the relaxation is such that $\pi \hat{x} \not\in \Z$, then we
have $\hat{x} \not \in X_1 \cup X_2$. When $\pi$ is the $i^{\rm th}$ unit
vector and $\pi_0 = \lfloor \hat{x}_i \rfloor$ (assuming $\hat{x}_i \not \in
\Z$), then the disjunction~\eqref{eqn:GD} is called a \emph{variable
  disjunction}. In the case of a standard MILP, at least one such disjunction
must be violated whenever the solution to the relaxation is not feasible to
the original problem (assuming the relaxation is a linear optimization
problem (LP)). Thus, such disjunctions are all that is needed for a convergent
algorithm. The situation is slightly different in the case of MIBLPs.

An important question for achieving good performance is how to choose the
branching disjunctions effectively, as it is typically easy to identify many
such disjunctions. A measure often used to judge effectiveness is the
resulting increase in the lower bound observed after imposing the disjunction.
We discuss branching strategies in more detail in Section~\ref{sec:branching}.

\paragraph{Searching.} The order in which the subproblems are considered
is critically important, since, as we have already noted, the global bounds
are evolving and the order in which the subproblems are considered affects 
their evolution. Typical options are \emph{depth-first} (prioritize the
``deepest'' nodes in the tree, which tends to emphasize improvement of the
upper bound by locating better incumbents) and \emph{best-first} (prioritize
those with the smallest lower bound, which tends to emphasize improvement of
the lower bound). In practice, one usually employs hybrids that balance these
two goals.

\paragraph{Cutting.} The main way in which the branch-and-cut algorithm
differs in its basic strategy from the branch-and-bound algorithm is that it
dynamically strengthens the relaxation by adding \emph{valid inequalities}. As
with our notion of valid disjunction, we allow here for a definition slightly
different than the standard one in order to allow for inequalities that remove
feasible, non-improving solutions. We drop the term ``improving'' when
discussing such inequalities in the remainder of the paper.
  \begin{definition}[Improving Valid Inequality] \label{def:valid-inequality}
The pair $(\alpha, \beta) \in \Re^{n+1}$ is an \emph{(improving) valid 
  inequality} for~\eqref{eqn:generic-optimization} with respect to $x^* \in
\X$ if
\begin{equation*}
  \left\{x \in \X \midd f(x) < f(x^*) \right\} \subseteq
  \left\{x \in \Re^n \midd \alpha x \geq \beta \right\}.
\end{equation*}
  \end{definition}
As with branching, we typically aim to add inequalities that are violated by
the solution to the current relaxation, thus removing the solution from the
feasible region of the relaxation, along with some surrounding polyhedral
region that contains no (improving) solutions to the original problem (or
subproblem), from consideration.

Furthermore, as with branching, it is generally possible to derive a large
number of valid inequalities and a choice must be made as to which ones to add
to the current relaxation (adding too many can negatively impact the
effectiveness of the relaxation). Also as with branching, the goal of cutting
is to improve the lower bound, although selecting directly for this measure is
problematic for a number of reasons. We discuss strategies for generating
valid inequalities in more detail in Section~\ref{sec:cutting}. We refer the
reader to~\mycitep{T:achterberg:thesis} for more detailed description of the
branch-and-cut algorithm.

\subsection{Outline}

The remainder of the paper is organized as follows. In
Section~\ref{sec:BandC}, we discuss different components of  the 
basic branch-and-cut
algorithm implemented in \MIBS{}. Section~\ref{sec:framework} describes 
the overall algorithm in \MIBS{} and  the implementation of this software. 
In Section~\ref{sec:computationalResultsSection}, we discuss computational
results. Finally, in Section~\ref{sec:conclusions}, we conclude the paper with
some final thoughts.

\section{A Branch-and-cut Algorithm \label{sec:BandC}}

The algorithm implemented in \MIBS{} is based on the algorithmic framework
originally described by~\mycitet{DeNegre and Ralphs}{DeNRal09}, but with many
additional enhancements. The framework utilizes a variant of the
branch-and-cut algorithm and is similar in basic outline to state-of-the-art
algorithms currently used for solving MILPs in practice. The case of MIBLP is
different from the MILP case in a few important ways that we discuss here at a
high level before we discuss various specific components.

As we have described earlier in Section~\ref{sec:BandC-general}, one way of
viewing the general approach taken by the branch-and-cut algorithm is as a
method for iteratively removing parts of the feasible region of the given relaxation
that can be proven not to contain any improving feasible solutions.
In MILP, this is done primarily by either removing \emph{lattice-point free}
polyhedral regions (usually by imposing valid disjunctions or by
generating associated valid inequalities) or by maintaining an \emph{objective
cutoff} that removes all non-improving solutions through the usual pruning process.
In the MILP case, we do not generally need to explicitly track or remove
solutions that are feasible to the original MILP. Such a solution can only be
feasible for the relaxation at one leaf node in the search tree and is either
suboptimal for the corresponding subproblem (in which case the solution must
have been generated by an auxiliary heuristic) or optimal to that subproblem
(in which case the node will be immediately pruned). In either case, it is
unlikely the solution will arise again and there is no computational advantage
to tracking it explicitly.

In MIBLP, in contrast, we may need to track and remove discrete sets of
solutions. This is mainly to avoid the duplication of effort in evaluating the
value function $\phi$ for a specific value of the linking variables. It is
possible that the same computation will arise in more than one node in the
search tree and re-computation should be avoided. In particular, for $x^1, x^2
\in X$, we have
\begin{equation*}
x^1_L = x^2_L \Rightarrow \phi(A^2 x^1) = \phi(A^2 x^2).
\end{equation*}
Thus, tracking which sub-vectors of values for
linking variables have been seen before in a pool (called the \emph{linking
  solution pool}) can be computationally advantageous. We discuss the
mechanism for doing this in Section~\ref{sec:framework}.

\subsection{Bounding \label{sec:bounding}}

\paragraph{Lower Bound.} Perhaps the most important step in the
branch-and-bound algorithm is that of deriving a lower bound on the optimal
solution value. As we have already described, relaxing integrality
restrictions does not provide an overall relaxation. However, since $\F
\subseteq \S \subseteq \P$, relaxing the optimality condition $d^2 y \leq
\phi(A^2 x)$ in the formulation~\eqref{eqn:miblp-vf} results in two
alternative relaxations to the original problem. The first relaxation 
(known as \emph{high-point problem}~\mycitep{moorebard90}) is
\begin{equation}\label{eqn:LR}\tag{R}
\min_{(x, y) \in \S} cx+d^1y,
\end{equation}
in which only the optimality condition is relaxed, while the second is
\begin{equation}\label{eqn:LRR}\tag{LR}
\min_{(x, y) \in \P} cx+d^1y,
\end{equation}
in which the integrality constraints are also relaxed. Although both are
rather weak bounds, they can be strengthened by the addition of the valid
inequalities described in Section~\ref{sec:cutting}. Because of the enhanced
tractability of~\eqref{eqn:LRR} and because of the advantages in warm-starting
the computation during iterative computation of the bound, we
use~\eqref{eqn:LRR} as our relaxation of choice. While it
is clear that this is a relaxation of~\eqref{eqn:miblp-vf}, to see that it is
a relaxation of~\eqref{eqn:miblp}, note that we have
\begin{equation*}
d^1y^* \leq \Xi(x^*),
\end{equation*}
where $(x^*, y^*)$ is a solution to~\eqref{eqn:LRR}.

At nodes other than the root node, we can use a similar bounding strategy.
Since the branching strategy we describe shortly employs only changes to
the bounds of variables, the subproblems that arise have feasible regions of
the form
\begin{equation*}
  \F^t = \left\{(x, y) \in \F \midd l_{x}^t \leq x \leq u_{x}^t, l_{y}^t \leq
  y \leq u_{y}^t \right\}, 
\end{equation*}
where $t$ is the index of the node/subproblem, and $(l_x^t, u_x^t)$ and
$(l_y^t, u_y^t)$ represent vectors of upper and lower bounds for the first-
and second-level variables, respectively. The subproblem at node $t$ is thus
the optimization problem
\begin{equation}\label{eqn:MIBLPt}\tag{$\text{MIBLP}^t$}
\min_{(x,y)\in\F^t}\quad cx+d^1y,
\end{equation}
which one can easily show is itself an MIBLP. Thus, we can apply a similar
relaxation to obtain the bound
\begin{equation}\label{eqn:LRRt}\tag{$\text{LR}^t$}
L^t = \min_{(x,y)\in\P^t} cx+d^1y,
\end{equation}
where $\P^t = \{(x, y) \in \P \mid l_{x}^t \leq x \leq u_{x}^t, l_{y}^t \leq y
\leq u_{y}^t\}$. If it can be verified that conditions for
pruning are met (see Section~\ref{sec:pruning}), then node $t$ should be
discarded. Otherwise, the next step is to check feasibility of
\begin{equation*}
  (x^t,y^t) \in \underset{(x,y)\in\P^t}{\argmin}\;\;cx+d^1y,
\end{equation*}
the solution obtained when solving the relaxation.

\paragraph{Feasibility Check.} Recall that the upper bound for a given node
is derived by exhibiting a feasible solution. Unlike in the case of MILP,
checking whether a solution to~\eqref{eqn:LRRt} is feasible
for~\eqref{eqn:miblp-vf} may itself be a
difficult computational problem. Such check involves verifying that
$(x^t,y^t)$ satisfies the constraints that were
relaxed: integrality conditions and second-level optimality conditions.
In other words, $(x^t,y^t)$ is bilevel feasible if and only if
the following conditions are satisfied.
\begin{feascon} \label{fc:integrality}
  $x^t \in X$.
\end{feascon}
\begin{feascon} \label{fc:optimality}
  $y^t \in \RR(x^t)$.
\end{feascon}
Condition~\ref{fc:integrality} ensures that the integrality constraints for
the first-level variables are satisfied, while Condition~\ref{fc:optimality}
guarantees that $(x^t,y^t)$ satisfies the optimality constraint of
second-level problem.
Satisfaction of these two conditions (as opposed to the more general
Conditions~\ref{fc:first-level} and~\ref{fc:second-level}) is enough to ensure
$(x^t, y^t) \in \F$, given that $(x^t, y^t)$ is also a solution
to~\eqref{eqn:LRRt}.

Verifying Condition~\ref{fc:integrality} is straightforward, so this is done
first. If Condition~\ref{fc:integrality} is satisfied, then we consider
verification of Condition~\ref{fc:optimality}. This involves checking both
whether $y^t
\in Y$ and whether $d^2y^t = \phi(A^2x^t)$. Checking whether $y^t \in Y$
is inexpensive so we do this first. The latter check is more expensive and is
only required under conditions detailed later in Section~\ref{sec:framework}.
We may thus defer it until later, depending on the values of parameters
described in Section~\ref{sec:parameters}.

If we decide to undertake this latter check, we first evaluate
$\phi(A^2x^t)$ to either obtain
\begin{equation} \label{eqn:second-level-milp} \tag{SL-MILP}
\hat{y}^t \in \argmin \left\{d^2y \midd y\in\P_2(x^t) \cap Y\right\},
\end{equation}
or determine $\P_2(x^t) \cap Y = \emptyset$. The
MILP~\eqref{eqn:second-level-milp} is solved using an auxiliary MILP solver
(more details on this in Section~\ref{sec:solving-second-level}).
If~\eqref{eqn:second-level-milp} has a solution (as it must when $y^t \in Y$),
we check whether $d^2 y^t = d^2 \hat{y}^t$. If so, $(x^t, y^t)$ is bilevel
feasible and must furthermore be an optimal solution to the subproblem at node
$t$. In this case, $U^t = L^t$ and the current global upper bound $U$ can be
set to $U^t$ if $U^t < U$. If~\eqref{eqn:second-level-milp} has no solution,
we have that $x^t \not\in \F_1$ and $x^t$ can be eliminated from further
consideration either by branching or cutting. In fact, in this case, we can
eliminate not only $x^t$, but any first-level solution for which $x_{\J} = 
x_{\J}^t$. We thus add $x_{\J}^t$ to the linking solution pool in this case.

Although it is not necessary for correctness,~\eqref{eqn:second-level-milp}
may be solved even when $y^t \not\in Y$ (and even if $x_i^t \not\in \Z$ for 
some $k_1 < i \leq r_1$), since this may still lead either to the discovery of a
new bilevel feasible solution or a proof that $x^t \not\in \F_1$. In
Section~\ref{sec:framework}, we discuss when the
problem~\eqref{eqn:second-level-milp} should necessarily be solved 
and when solving it is optional. Even in the case of infeasibility of
$(x^t, y^t)$, $(x^t, \hat{y}^t)$ is bilevel feasible (though not necessarily
optimal to~\eqref{eqn:MIBLPt}) whenever $x^t\in X$ and $\hat{y}^t\in
\P_1(x^t)$. The second condition is satisfied vacuously in the usual case of $G^1 =
0$. If $(x^t, \hat{y}^t)$ is feasible, then 
we can update the global upper bound if $cx^t+d^1\hat{y}^t < U$.

Observe that solving~\eqref{eqn:second-level-milp} reveals
more information than the simple value of $\phi(A^2x^t)$.
Note that we have
\begin{equation} \label{eqn:VFBound}\tag{VFB}
  \phi(A^2x) \leq  \phi(A^2x^t) \; \forall x \in X \textrm{ such
    that } \P_2(x) \ni  \hat{y}^t,
\end{equation}
which means that $\phi(A^2x^t)$ reveals an upper bound on the second-level problem for
any other first-level solutions that admits $\hat{y}^t$ as a feasible
reaction. This upper bound can be exploited in the generation of certain valid
inequalities (see~\mycitet{}{DeNegre2011,fischettietal17a}) and is also the basis for an
algorithm by~\mycitet{Mitsos}{Mit10}. Furthermore, we have
\begin{equation*}
  \phi(A^2x) =  \phi(A^2x^t) \; \forall x \in \Re^{n_1} \textrm{
    such that } x_{\J} = x_{\J}^t,
\end{equation*}
which means we get the \emph{exact} value of the second-level problem with
respect to certain other first-level solutions ``for free''. Because (i) it may
improve the global upper bound, (ii) can reveal that $x^t
\not\in \F_1$ (and, along with other solutions sharing the same values for the
linking variables, should no longer be considered), and (iii) may
also provide bounds on the second-level value function for other first-level
solutions, \eqref{eqn:second-level-milp} may be solved in \MIBS{} even 
when $(x^t,y^t)$ does not satisfy integrality requirements
(see Section~\ref{sec:framework} for details).

\paragraph{Upper Bound.} As we have just highlighted,
the feasibility check may produce a bilevel feasible solution $(x^t,
\hat{y}^t)$, but $\RR(x^t)\cap\P_1(x^t)$ may not be a singleton and there may
exist a rational reaction $\bar{y} \in \RR(x^t)\cap \P_1(x^t)$ to $x^t$ with
$d^1\bar{y} < d^1\hat{y}^t$, which results $d^1\hat{y}^t > \Xi(x^t)$ (note
that the selected policy is the optimistic one). To compute the ``true''
objective function value $cx^t + \Xi(x^t)$ associated with $x^t$ and produce
the best possible upper bound, we may explicitly compute $\Xi(x^t)$. Once this
computation is done and the associated solution and bound are stored, we can
safely eliminate $x^t$ from the feasible region by either branching or
cutting, as we describe below, in order to force consideration of alternative
solutions. This computation is required under certain conditions and may be
optionally undertaken in others, depending on the parameter settings described
in detail in Section~\ref{sec:framework}.
$\Xi(x^t)$ is obtained by solving
\begin{equation*}\label{eqn:refinementXi}
\Xi(x^t) = \min \left\{d^1y \midd y \in \P_1(x^t) \cap \P_2(x^t)\cap Y,
d^2y\leq \phi(A^{2}x^t)\right\},
\end{equation*} 
which is an MILP since we have already computed $\phi(A^{2}x^t)$ in the
feasibility check. Solving this MILP to evaluate $\Xi(x^t)$ yields the
globally valid upper bound $cx^t + \Xi(x^t)$.

Observe that even the problem of minimizing $\Xi(x)$ over $\{x\in X \mid x_\J
= x_\J^t\}$ is still an MILP and can reveal an upper 
bound across a range of first-level solutions. This result is formalized in
Theorem~\ref{thm:computeBestUB} below.  

\subsection{Pruning \label{sec:pruning}}

As is the case with all branch-and-bound algorithms, pruning of node $t$
occurs whenever $L^t\geq U$. This again subsumes two special cases. First, if
checking Conditions~\ref{fc:integrality} and~\ref{fc:optimality} verifies that
$(x^t,y^t)$ is bilevel feasible, then we must have $L^t=U^t\geq U$. Second, when
the relaxation is infeasible, we have $L^t = \infty$ and we can again prune
the node. 

There is one additional case that is particular to MIBLPs and that is when all
linking variables are fixed. In this case, we utilize the property of MIBLPs
illustrated in the next theorem.

\begin{theorem}[\mycitep{fischettietal17a}] \label{thm:computeBestUB}
For $\gamma\in\Z^\J$, we have 
\begin{equation*}
\F \cap \left\{(x, y) \in X \times Y \midd x_{\J} = \gamma \right\} = 
\S \cap \left\{(x, y) \in X \times Y \midd d^2y \leq \phi(A^{2}x),
x_{\J} = \gamma \right\}.
\end{equation*}
\end{theorem}

\begin{proof}
From the definitions of sets $\S$ and $\F$ provided in Section~\ref{sec:miblp}, 
it follows that
\begin{equation*}
\F=\S\cap \left\{(x,y)\in X\times Y \midd d^2 y\leq \phi(A^2x) \right\}.
\end{equation*}
The result follows.\qed
\end{proof}

\begin{corollary}
  For $\gamma \in \Z^\J$, we have
\begin{equation}\label{eqn:computeBestUB}\tag{UB}
  \begin{aligned}
  \min \{cx + d^1y \mid (x, y) \in \F, x_L = \gamma\} = 
  \min \{cx+d^1y \mid (x,y) \in \S, d^2y\leq \phi(A^{2}x)&,  \\ 
  x_{\J} = \gamma \}&.
  \end{aligned}
\end{equation}
\end{corollary}
What Theorem~\ref{thm:computeBestUB} tells us is that when the values of all
linking variables are fixed at node $t$ ($l_{x_{\J}}^t=u_{x_{\J}}^t$, either
because of branching constraints or otherwise), then optimizing over $\F^t$ is
equivalent to optimizing over $\S^t$ while additionally imposing an upper
bound on the objective value of the second-level problem. In other words, we
have the following.

\begin{corollary}
Whenever $\F^t \subseteq \{(x, y) \in \F \mid x_{\J} = \gamma \}$ for some
$\gamma \in \Z^\J$, then 
\begin{equation}\label{eqn:computeBestUBForOneNode}\tag{UB$^t$}
  \begin{aligned}
  \min_{(x, y) \in \F^t} cx + d^1y  = \min \left\{cx+d^1y \midd (x,y) \in \S^t,
  d^2y\leq \phi(A^{2}x), x_{\J} = \gamma \right\},
  \end{aligned}
\end{equation}
where $\S^t = \P^t \cap (X\times Y)$.
\end{corollary}
Therefore, the optimal solution value for the subproblem at node $t$ can be
obtained by solving problem~\eqref{eqn:computeBestUBForOneNode} with
$\gamma=x_\J^t$. Furthermore, solving problem~\eqref{eqn:computeBestUB}
instead of~\eqref{eqn:computeBestUBForOneNode} provides a bilevel feasible
solution which is at least as good as the optimal solution of node $t$ since
it is a relaxation (it is explained in detail in 
Section~\ref{sec:impactOfParametersForLinkingPool} that there may be more 
than one node with $x_\J= \gamma$). Hence, node $t$ can be pruned after solving either
\eqref{eqn:computeBestUBForOneNode} or \eqref{eqn:computeBestUB}. Note that
these problems may be infeasible, in which case node $t$ is infeasible.
Although solving problem~\eqref{eqn:computeBestUB} may be more difficult than
solving problem~\eqref{eqn:computeBestUBForOneNode}, solving
problem~\eqref{eqn:computeBestUB} provides useful information about all
bilevel feasible solutions with $x_{\J}=x_{\J}^t$ (not only those feasible for 
node $t$), so it is generally recommended to solve~\eqref{eqn:computeBestUB} instead
of \eqref{eqn:computeBestUBForOneNode} (see Sections~\ref{sec:parameters}
and~\ref{sec:outline}).

Even if the values of linking variables are not fixed at node $t$
($l_{x_{\J}}^t \not= u_{x_{\J}}^t$), it may be advantageous to
solve~\eqref{eqn:computeBestUB} with $\gamma=x_{\J}^t$ whenever $x_\J^t \in
\Z^\J$ in order to improve the upper bound. In Section~\ref{sec:framework}, we
describe the parameters of \MIBS{} that control when
problem~\eqref{eqn:computeBestUB} is solved during the solution process. The
cases in which, this problem must be solved, regardless of the values of
parameters, are also discussed.

\subsection{Branching \label{sec:branching}}
As previously described, the role of the branching procedure is to remove
regions of the feasible set of the relaxation that contain no improving
bilevel feasible solutions. This is accomplished by imposing a valid
disjunction, which we typically choose such that it is violated by the
solution to the current relaxation (i.e., removes it from the feasible
region). When the branching procedure is invoked for node $t$, we have that
either
\begin{itemize}
  \item the solution $(x^t, y^t) \not\in \F$ because
    \begin{itemize}
    \item $(x^t, y^t) \not\in X \times Y$ or
    \item $d^2 y^t > \phi(A^2 x^t)$ (we
    have previously solved the corresponding~\eqref{eqn:second-level-milp});
    \end{itemize}
    or
  \item we have $(x^t, y^t) \in X \times Y$ and we are not sure of its
    feasibility status because we chose not to
    solve~\eqref{eqn:second-level-milp} (see Section~\ref{sec:framework} for
    an explanation of the conditions under which we may make this choice).
\end{itemize}
These alternatives make it clear that we may sometimes want to branch, even
though $(x^t, y^t) \in X \times Y$, necessitating a branching scheme that is
more general than the one traditionally used in MILP.

\paragraph{Branching on Linking Variables.} Due to
Theorem~\ref{thm:computeBestUB}, we know that once the values of  
linking variables are fixed completely at a given node $t$
($l_{x_{\J}}^t = u_{x_{\J}}^t$), the node can be pruned after
solving~\eqref{eqn:computeBestUB}.
One strategy for branching is therefore to think of the search as being over
the set of possible values of the linking variables and to prefer a branching
disjunction that partitions \emph{this} set of values. In such a scheme, we
only consider branching on linking variables as long as any such variables
remain unfixed. We show in Section~\ref{sec:computationalResultsSection} that
the apparent logic in such a scheme is effective empirically in some cases,
but this scheme also raises issues that do not generally arise in branching
on variables in MILP. In particular, situations may arise in which there are
no linking variables with fractional values ($x^t_\J \in \Z^\J$) and yet we
would still like to (or need to) impose a branching disjunction.

Naturally, when there \emph{does} exist $i\in \J$ such that $x_i^t\not\in \Z$,
we can impose a variable disjunction
\begin{align*} 
& X_1 = \left\{(x, y) \in \Re^{n_1 \times n_2} \midd x_i \leq \lfloor x^t_i
  \rfloor \right\}
& X_2 = \left\{(x, y) \in \Re^{n_1 \times n_2} \midd x_i \geq \lfloor x^t_i
  \rfloor + 1 \right\},
\end{align*}
as usual. When $x_\J^t \in \Z^\J$, however, it is less obvious what to do. If
we either have $x_i^t\not\in\Z$ for $k_1 < i \leq r_1$ or we have that $y^t \not\in
Y$, then there exists the option of breaking with the scheme and imposing a
standard variable disjunction on a non-linking variable (those not in set
${\J}$). Computational experience has shown that this may not always be a good
idea empirically.

When $x_\J^t \in \Z^\J$, there is no single variable
disjunction that can be imposed on linking variables that would eliminate
$x^t$ from (the projection of) the feasible region of both resulting
subproblems. Suppose we nevertheless branch on the variable disjunction
associated with some variable indexed $i \in \J$. Since $x^t_i \in \Z$ (and
assuming that $l_{x_{i}}^t\not= u_{x_{i}}^t$), we may branch, e.g., either on
the valid disjunction
\begin{align*} 
& X_1 = \left\{(x, y) \in \Re^{n_1 \times n_2} \midd x_i \leq x^t_i \right\},
& X_2 = \left\{(x, y) \in \Re^{n_1 \times n_2} \midd x_i \geq x^t_i + 1 \right\}
\end{align*}
or
\begin{align*} 
& X_1 = \left\{(x, y) \in \Re^{n_1 \times n_2} \midd x_i \leq x^t_i - 1
  \right\}, 
& X_2 = \left\{(x, y) \in \Re^{n_1 \times n_2} \midd x_i \geq x^t_i \right\},
\end{align*}
preferring the former when $x^t_i \leq u_{x_{i}}^t - 1$ (otherwise, one of the
resulting subproblem will be trivially infeasible and the other will be
equivalent to node $t$). This means $x^t$ satisfies either the first or second
term of this disjunction. Although imposing this valid disjunction seems
undesirable, since it does not remove $(x^t, y^t)$ from the union of the
feasible regions of the resulting subproblems, let us explore the logic of the
approach.

Suppose we continue to branch in this way and consider the indices of the set
of leaf nodes $T^t$ of the subtree of the search tree rooted at node $t$.
Since this set of leaf nodes represents a partition of $\S^t$, the projection
of the feasible region of exactly one of these nodes contains $x^t$. We have
that the values of all linking variables are fixed at this node, since we
would otherwise continue branching (we have already noted that such node can
be pruned after solving problem~\eqref{eqn:computeBestUB}).
The projection of the union of the feasible
regions of the remaining leaf nodes is thus equal to
\begin{equation*}
  \proj_x(\S^t) \setminus \left\{x \in X \midd x_\J = x^t_\J \right\}.
\end{equation*}
Hence, this branching rule can be seen as implicitly branching on the
disjunction
\begin{align*}
  & X_1 = \left\{(x, y) \in \Re^{n_1 \times n_2} \midd  x_\J = x^t_\J \right\},
  & X_2 = \left\{(x, y) \in \Re^{n_1 \times n_2} \midd  x_\J \not= x^t_\J
  \right\},
\end{align*}
which is obviously valid.

As a practical matter, if we implement the scheme of only branching on linking
variables using simple variable disjunctions, we inevitably create a
number of additional nodes for which $x^t$ remains in the projection of the
feasible set (one at each level of the tree down to the leaf node at which all
linking variables are fixed). This can easily be detected using the scheme for
maintaining a pool of linking solutions that have already been seen and will
not cause any significant computational issues.
For the remaining nodes, we are guaranteed
that $x^t$ is not contained in the projection (though these nodes may also be
the root of a subtree in which all linking variables are integer-valued). An
alternative would be to branch in a non-binary way using a more complicated
disjunction that would directly create the leaf nodes of the subtree rooted at
node $t$, but there seems to be no advantage to such a scheme.

In forming the list of variables that are candidates for branching, we add
only linking variables with fractional values if $x^t_{\J} \not\in \Z^{\J}$.
If $x^t_{\J} \in \Z^{\J}$, then we are forced to add integer-valued linking
variables that remain unfixed to the list. If all linking variables
are fixed, then we have several options, depending on how parameters are set
(see Section~\ref{sec:parameters}). 
\begin{itemize}
  \item We may prune the node after solving problem~\eqref{eqn:computeBestUB},
    as explained in Section~\ref{sec:pruning}.
  \item When $(x^t, y^t) \not\in X \times Y$, we may add non-linking variables 
  with fractional values to the candidate set.  
  \item $(x^t,y^t)$ may be removed by generating a cut 
  (after solving~\eqref{eqn:second-level-milp} if $(x^t,y^t)\in X\times Y$, see
    Section~\ref{sec:cutting}).
\end{itemize}  

\paragraph{Branching on Fractional Variables.} Naturally, we may also consider
a more traditional branching scheme in which we only branch on variables with
fractional values. In such a scheme, we must consider branching on all
variables (first- and second-level), since we may have no other option when
$x^t \in X$.

\paragraph{Selecting Candidates.} We have so far neglected to address the
question of \emph{which} variable to choose as the basis for our branching
disjunction when there is more than one candidate. The idea of limiting
branching only to the linking variables differs from the usual scheme employed
in algorithms for MILP in that we only consider a subset of the integer
variables for branching. Nevertheless, in both of the schemes described above,
we face a similar decision regarding which of the (possibly many) variables
available for branching should actually be chosen. It is well-known that the
choice of branching variable can make a substantial difference in practical
performance.

In branching procedures for MILP, it is typical to first choose a set of
candidates for branching (typically, all integer variables with fractional
values are considered). The selection of the ``best'' branching candidate is
then made from this set of candidates based on one of a number of schemes for
scoring/predicting the ``impact'' of choosing a particular branching variable.
The details of these scoring mechanisms can be found in any number of
references (see, e.g., \mycitet{}{AchKocMar05}). Although the method used
to define the set of branching candidates differs from the one used in MILP,
the method employed in \MIBS{} for selecting from the candidates is 
similar to the schemes that can be found in the literature. \MIBS{} offers the 
pseudocost, strong branching and reliability branching schemes.

\subsection{Cutting \label{sec:cutting}}

The generation of valid inequalities is an alternative to branching for
removing infeasible and/or non-improving solutions. We refer the reader
to~\mycitep{Marchand2002} and~\mycitep{Wolter2006} for theoretical background
on the separation problem and cutting plane methods in general. In the case of
MIBLP, the valid inequalities we aim to generate are those that separate the
solution to the current relaxation from the convex hull of the (improving)
bilevel feasible solutions. These inequalities are generated iteratively in
order to improve the relaxation, as in a standard cutting plane method.
\mycitet{DeNegre and Ralphs}{DeNRal09} showed that such a pure cutting plane
algorithm can be used to solve certain classes of MIBLPs.

In order to allow the separation of points that are in $\conv(\F)$ but have
already been identified (or can be shown not to be improving), we recall here
the notion of improving valid inequality that we introduced in
Definition~\ref{def:valid-inequality}, re-stating it in the notation
of~\eqref{eqn:miblp-vf}. A triple $(\alpha^x, \alpha^y, \beta) \in \Re^{n_1 +
  n_2 + 1}$ is said to constitute an \emph{improving valid inequality} for $\F$
if
\begin{equation*} 
  \alpha^x x + \alpha^y y \geq \beta \; \forall (x, y) \in
  \conv\left(\left\{(x, y) \in  \F \midd c x + d^1 y < U \right\}\right),
  \end{equation*}
where $U$ is the current global upper bound. As mentioned before, we shall
drop the word ``improving'' from the remaining discussion.

In Section~\ref{sec:bounding}, we stated that the difference between the
feasible region of a subproblem and that of the original problem is only in
changes to the bounds on variables. This is no longer strictly true when valid
inequalities are also being generated and added dynamically to the constraint
set of the relaxation. The feasible region $\P^t$ of the relaxation at node
$t$ from~\eqref{eqn:LRRt} then becomes
\begin{equation*}
  \P^t = \left\{(x, y) \in \P \cap \Pi^t \midd l_x^t \leq x \leq u_x^t, l_y^t
  \leq y \leq u_y^t \right\}, 
\end{equation*}
where $\Pi^t$ is a polyhedron representing the valid inequalities applied at
node $t$. The set of valid inequalities may include inequalities generated at
any of the ancestor nodes in the path to the root node of the search tree.

There are several distinct categories of valid inequality that can be
generated, depending on the feasibility conditions violated by the solution
$(x^t, y^t)$ that we are trying to separate at node $t$. Generally
speaking, all valid inequalities can be classified as
\begin{itemize}
\item Feasibility cuts: Inequalities valid for $\conv\left\{(x,y)\in \S\midd 
cx+d^1y < U\right\}$.
\item Optimality cuts: Inequalities valid for $\conv\left\{(x,y)\in \F\midd 
cx+d^1y < U\right\}$.
\item Projected optimality cuts: Inequalities valid for 
$\conv\left(\left\{(x,y)\in\Re^{n_1\times n_2} \midd x \in \F_1, cx + \Xi(x) <
  U \right\}\right)$.
\end{itemize}
The feasibility cuts include those classes generally employed in solution of
MILPs. Inequalities used to remove solutions $(x^t, y^t) \in \S$ at node $t$
may be referred to roughly as \emph{optimality cuts},
while \emph{projected optimality cuts} are those that may remove all solutions 
with specific first-level values. For example, inequalities that remove the solutions 
with the same linking component as $x^t$ once the problem~\eqref{eqn:computeBestUB} 
with $\gamma = x_L^t$ has been solved.  
Clearly, these three classes are not entirely distinct and the categorization 
is meant only to provide a rough categorization.

\MIBS{} includes separation routines for a variety of known classes of valid
inequalities in the current literature. We give a brief overview here. More
details are provided in a companion paper that analyzes the impact of various
classes of inequalities in greater detail.
\paragraph{Integer No-Good Cut.}

This class of inequalities are valid for problems in which $r_1=n_1$, $r_2=n_2$ 
and all problem data are integer (with the possible exception of the
objective function).
It was introduced by~\mycitet{DeNegre and Ralphs}{DeNRal09}
and is derived from a split disjunction obtained by taking combinations of
inequalities binding at $(x^t, y^t)$ in a fashion similar to that used in
deriving the Chv\'atal cuts valid for the feasible regions of MILPs.
The cut eliminates a single extremal solution
that is integral yet not bilevel feasible from the feasible region of the
relaxation. It is easy to derive such a cut under the given assumptions. 

\paragraph{Generalized No-good Cut.}
This class of inequalities are valid for problems in which all linking variables
are binary.
It is a
generalization of the \emph{no-good cut} introduced
by~\mycitet{DeNegre}{DeNegre2011} in the context of MIBLPs. Similar cuts have
been used in many other contexts. The idea of a ``no-good cut'' seems to have
first been suggested by~\mycitet{Balas and Jeroslow}{BalasJeroslow72}. The
purpose of this cut is to remove all solutions for which the linking variables
have certain fixed values.
If at node $t$, we have $x_\J^t \in \Z^\J$ and
either (i) we choose to solve~\eqref{eqn:second-level-milp} and it is
infeasible or (ii) we choose to solve~\eqref{eqn:computeBestUB}, then we can
store $x_\J^t$ in the linking solution pool (if applicable) and add a
generalized no-good cut with $\gamma = x_\J^t$ in order
to avoid generating the same linking solution again. This inequality can also
be added in other subproblems if/when this linking solution arises again.

\paragraph{Intersection Cut. \label{sec:sep1IC}}

\mycitet{Fischetti et al.}{fischettietal17a,fischettietal17b}
generalized the well-known concept of an intersection cut~\mycitep{balas71}, used
extensively in the MILP setting, to MIBLPs. The various classes introduced
differ from each other in the way the underlying ``bilevel-free'' convex set
used for generating the cuts is defined.
Three classes of introduced cuts are working for the problems with 
$G^2y-A^2x\in\Z^{m_2}$ for all $(x,y)\in\S$ and $d^2\in\Z^{n_2}$. These cuts 
can be employed for removing the infeasible solution $(x^t, y^t)\notin\F$ 
from the feasible region, but it is not guaranteed when $(x^t,y^t)\notin\S$.
The other class known as \emph{hypercube intersection cut} is valid for the 
MIBLPs with $x_{\J}\subseteq \Z^{\J}$ and can be added when $x_\J^t \in \Z^\J$.   
We solve~\eqref{eqn:computeBestUB} and store
the solution in the linking solution pool (if applicable) and after doing so, the
hypercube intersection cut can be added. This inequality is guaranteed to be
violated by $(x^t, y^t)$ and \emph{may} also eliminate \emph{some} other
solutions for which $x_{\J} = x_{\J}^t$, but is not guaranteed to eliminate
all such solutions.

\paragraph{Increasing Objective Cut.}
This class of inequalities are valid for problems in which linking
variables are binary and $A^2 \geq 0$. 
Proposed by~\mycitet{DeNegre}{DeNegre2011}, it is a disjunctive cut derived from
the following valid disjunction based on the value function bound~\eqref{eqn:VFBound}.
\begin{align*}
&  X_1 = \biggl\{(x, y) \in \Re^{n_1 \times n_2} \mathrel{\Big|} \sum_{i\in\J:
    x_i^t = 0} x_i = 
  0, d^2 y \leq d^2 \hat{y} \biggl\},
  & X_2 = \biggl\{(x, y) \in \Re^{n_1 \times n_2} \mathrel{\Big|} \sum_{i\in\J:
    x_i^t = 0} x_i \geq 1 \biggl\},
  \end{align*}
where $\hat{y}\in\RR(x^t)$. 
In addition to the way in which this disjunction was exploited
by~\mycitet{DeNegre}{DeNegre2011}, there are other inequalities that could
also be derived from this disjunction.
This cut can be applied to eliminate $(x^t,
y^t)\in X\times Y$ from the feasible region at node $t$ when 
$(x^t, y^t)$ is found to be
infeasible. 
Note
that by solving~\eqref{eqn:computeBestUB} following the feasibility check, we
could also apply the stronger generalized no-good cut in this situation.

\paragraph{Benders Cut.}
This class of inequalities are valid for problems in which linking
variables are binary and the second-level variables coefficients are 
not greater than 0 in the second-level constraints in which the linking 
variables do not participate. Furthermore, corresponding to each linking 
variable $x_i$, there exists $y_i$ so that $x_i = 1$ results $y_i = 0$ 
and this is the only restriction from the second-level constraints in which 
the linking variables participate.
 This cut was originally mentioned by~\mycitet{Caprara et al.}{capraraetal16} in the
context of knapsack interdiction problems, but is valid for a broader class of
problems.
This cut utilizes the value function bound~\eqref{eqn:VFBound} and 
a new such cut can be imposed anytime we find
a new feasible solution. 

\subsection{Primal Heuristics \label{sec:heuristics}}

Just as in the solution of MILPs, the primal heuristics can be employed within
a branch-and-cut algorithm for solving MIBLPs in order to enhance the
discovery of bilevel feasible solutions. Three different heuristics introduced
in~\mycitep{DeNegre2011} are implemented in \MIBS{}, as follows.

\paragraph{Improving Objective Cut Heuristic.} Let $(\bar{x},\bar{y})\in\S$
and $\hat{y}\in \RR(\bar{x})$. As we discussed in Section~\ref{sec:bounding},
$(\bar{x},\hat{y})$ is a bilevel feasible solution if it satisfies the
first-level constraints. However, $(\bar{x},\hat{y})$ is not necessarily a
good solution with respect to the first-level objective function. The idea of
this heuristic is to exploit the information from both $(\bar{x},\bar{y})$
(which is a good solution with respect to the first-level objective) and
$(\bar{x},\hat{y})$ (which is a likely bilevel feasible solution). To do so,
we first determine
\begin{equation*}
(\tilde{x}, \tilde{y}) \in \argmin \left\{ cx+d^1y \midd (x,y)\in \S, d^2y\leq
  d^2\hat{y}\right\}. 
\end{equation*}
When $(\bar{x},\hat{y}) \in \F$, such $(\tilde{x}, \tilde{y})$ must exist.
Further, we must have $(\tilde{x}, \tilde{y})\in\F$. A separate feasibility
check is thus required and this check is what is expected to finally produce
the (potentially) improved solution.

\paragraph{Second-level Priority Heuristic.}
This heuristic is similar to the improving objective cut heuristic in that it
also tries to balance bilevel feasibility with the quality of obtained
solution. The MILP solved with this heuristic, however, is
\begin{equation*} \label{eqn:priority-heur} \tag{PH}
\min \left\{ d^2y \midd (x,y)\in\S, cx+d^1y\leq U \right\},
\end{equation*}
where $U$ is the current upper bound. In this problem, the goal of the added
constraint is to improve the quality of the solution, while the objective
function of this problem increases the chance of generating a bilevel feasible
solution (however, it does not guarantee the bilevel feasibility of the
produced solution). Note that the upper bound $U$ could be replaced by any
chosen ``target'' to try to ensure improvement on the current incumbent, but
then we would not be assured feasibility of~\eqref{eqn:priority-heur}.

\paragraph{Weighted Sums Heuristic.}

This heuristic employs techniques from multi-objective optimization to
generate bilevel feasible solutions. The main idea of this heuristic is
finding a subset of \emph{efficient} solutions (those for which we cannot
improve one of the objectives without degrading the other while maintaining
feasibility~\mycitep{ehrgott04})
of the following problem by using
a weighted-sum subproblem~\mycitep{geoffrion68}.
\begin{equation*}
\underset{(x, y) \in \S}{\operatorname{argvmin}}\{cx+d^1y, d^2y\},
\end{equation*}
where the operator $\operatorname{argvmin}$ means finding a solution (or more
than one) that is efficient. For more details, see~\mycitet{}{DeNegre2011}.

\section{Software Framework \label{sec:framework}}

In this section, we describe the implementation of the \MIBS{}
software~\mycitep{MIBS}. This paper refers to version \code{1.1.2}, the latest
released version at the time of this writing.
The overall algorithm employed in \MIBS{} combines the procedures described in
Section~\ref{sec:BandC} in a fashion typical of existing branch-and-cut
algorithms for other problems. In fact, \MIBS{} is built on top of BLIS,
a parallel implementation of branch and cut for the solution of
MILPs~\mycitep{XuRalLadSal09}. In Section~\ref{sec:design}, we describe the class
structure of the C++ code that encapsulates the implementation. Furthermore, there 
are a number of important parameters,
described in Section~\ref{sec:parameters}, that determine how the
various components described in Section~\ref{sec:BandC} are coordinated.
Although there are defaults that are set automatically after analyzing the
structure of a particular instance, these parameters can and should be tuned
for use in particular applications, Understanding their role in the
algorithm is crucial to understanding the overall strategy, described in
Section~\ref{sec:outline}.
Finally, in Section~\ref{sec:solving-second-level}, we
describe important implementational issues surrounding how the second-level
problem is solved.

\subsection{Design of \MIBS{} \label{sec:design}}

We first briefly describe the overall design of the software.
\MIBS{} is an open-source implementation in C++ of the branch-and-cut
algorithm described in Section~\ref{sec:outline}. In addition to a core
library, \MIBS{} utilizes a number of other open-source packages available
from the Computational Infrastructure for Operations Research (COIN-OR)
repository~\mycitep{COIN-OR}. These packages include the following.

\begin{itemize}

\item The COIN-OR High Performance Parallel Search (CHiPPS)
    Framework~\mycitep{RalLadSal04}, which includes a hierarchy of three
    libraries.

\begin{itemize}
\item Abstract Library for Parallel Search 
(ALPS)~\mycitep{ALPS,XuRalLadSal05}: This project is utilized 
for managing the global branch and bound.

\item Branch, Constrain, and Price Software (BiCePS)~\mycitep{RalLadSal04,BICEPS}:
  This project provides base classes for objects used in \MIBS{}.

\item BiCePS Linear Integer Solver (BLIS)~\mycitep{XuRalLadSal09,BLIS}: This
  project is the parallel MILP Solver framework from which \MIBS{} is derived.

\end{itemize}
  
\item COIN-OR Linear Programming Solver (CLP)~\mycitep{Clp}:
\MIBS{} employs this software for 
solving the LPs arising in the branch-and-cut algorithm.

\item SYMPHONY~\mycitep{SYMPHONY,RalGuz05}: This 
software is used for solving the MILPs required to be 
solved in the branch-and-cut algorithm.

\item COIN-OR Cut Generation Library (CGL)~\mycitep{Cgl}:
  Both \MIBS{} and SYMPHONY utilize this library to generate valid
  inequalities valid for MILPs.

\item COIN-OR Open Solver Interface (OSI)~\mycitep{Osi}:
  This project is used for interfacing with solvers, such as SYMPHONY, Cbc 
  and CPLEX. 

\end{itemize}
\MIBS{} is comprised of a number of classes that are either
specific to \MIBS{}, or are classes derived from a class in one of the
libraries of CHiPPS. The main classes of \MIBS{} are as follows.  
\begin{itemize}

\item \code{MibSModel}: This class is derived from the \code{BlisModel} 
class. It extracts and stores the model from the input files, 
which consist of an MPS file and an auxiliary information file. We 
refer the reader to the \code{README} file of \MIBS{} for a comprehensive 
description of the input file format. 

\item \code{MibSBilevel}: This class is specific to \MIBS{} and is utilized for
  checking the bilevel feasibility of solutions of the relaxation problem.
  Furthermore, this class finds the bilevel feasible solutions, as described in
  lines~\ref{alg-bandc:step-18},~\ref{alg-bandc:step-23} 
  and~\ref{alg-bandc:step-28} of Algorithm~\ref{alg:branchandCutMain}.

\item \code{MibSTreeNode}: This class is derived from the 
class \code{BlisTreeNode} and contains the methods for processing the
branch-and-bound nodes. 

\item \code{MibSCutGenerator}: This class is specific to \MIBS{} 
  and contains the methods for generating the valid inequalities described in
  Section~\ref{sec:cutting}. 

\item \code{MibSBranchStrategyXyz}: These classes (one for each strategy
  \MIBS{} can use for selecting the final branching candidate: \t{Pseudo},
  \t{Strong}, and \t{Reliability}) are derived from the parent classes
  \code{BlisBranchStrategyXyz} and contain the implementation used for
  selecting the final branching candidate.

\item \code{MibSHeuristic}: This class is specific to \MIBS{} and 
contains the methods for generating heuristic solutions by employing the primal 
heuristics illustrated in Section~\ref{sec:heuristics}. 

\item \code{MibSSolution}: This is a class derived from 
\code{BlisSolution} class and is utilized for storing the 
bilevel feasible solutions.

\end{itemize}
Since one important feature of a practical solver is ease of use, we have
tried to make \MIBS{} as user-friendly as possible. Prior to the solution
process, \MIBS{} analyzes the problem to determine its properties, e.g., the
type of instance (interdiction or general), the type of variables present at
each level (continuous, discrete, or binary) and signs of
the coefficients in the constraint matrices. It then checks the parameters set
by the user and modifies them if it determines that those values of parameters
are not valid for this problem, informing the user of any change in the
parameters. For example, \MIBS{} turns off any cuts selected by the user that
are not valid for the instance to be solved.

\subsection{Parameters \label{sec:parameters}}
A wide range of parameters are available for controlling the algorithm.
\paragraph{Branching Strategy.} There are several parameters that control the 
branching strategy. The main one is \t{branchStrategy}, which controls
  what variables we allow as branching candidates. The options are
  \begin{itemize}
    \item \t{linking}: Branch only on linking variables, as long as any such 
    variable remains unfixed, with priority given to such variables with
    fractional values. 
    \item \t{fractional}: Branch on any integer first- or second-level variable that has
      fractional value, as is traditional in solving MILPs. 
\end{itemize}
  We also have parameters for
  controlling the strategy by which the final branching candidate
  is selected (\t{strong, pseudocost, reliability}). The default is to use
  \t{pseudocost} scoring.
  
  \paragraph{Search Strategy.}
  The search strategy used by \MIBS{} is controlled by ALPS, the underlying
  tree search framework. The default is the best-first strategy.
  
\paragraph{Cutting Strategy.} There are parameters for the types of valid
  inequalities to generate and the strategy for when to generate them (only in
  the root node, periodically, etc.). Note that we are forced to generate cuts
  whenever there are no available branching candidates and the current node
  cannot be pruned. With \t{branchStrategy} set to \t{linking}, the parameters
  for solving problems~\eqref{eqn:second-level-milp}
  and~\eqref{eqn:computeBestUB} (which are described later) can be set so that
  a pure branch and bound is possible, but this is not possible for the
  \t{fractional} case. The default settings for cuts depends on the instance.
  \MIBS{} includes an automatic analyzer that determines which classes of cuts
  are applicable and likely to be effective for a given instance. The
  frequency of generation is selected automatically by default, based on
  statistics gathered during early stages, as is standard in many solvers.
  
\paragraph{Primal Heuristics.} Types of primal heuristics to employ and the
  strategy for how often to employ them (see Section~\ref{sec:heuristics}).
  Only BLIS (generic MILP) heuristics are turned on by default.

\paragraph{Linking Solution Pool.} There is a parameter
\t{useLinkingSolutionPool} that determines whether to maintain a pool $\E$ of
linking solutions seen so far, as described earlier, in order to avoid solving
identical instances of~\eqref{eqn:second-level-milp}
and~\eqref{eqn:computeBestUB}. When the parameter is set to \t{True}, we check
$\E$ before solving either of~\eqref{eqn:second-level-milp}
or~\eqref{eqn:computeBestUB}.

Each linking solution stored in $\E$ is stored along with both a status tag
and, if appropriate, an associated solution. The status tag
is one of the following:
\begin{itemize}
\item\t{secondLevelIsInfeasible}: If the corresponding 
problem~\eqref{eqn:second-level-milp} is solved and it is infeasible.  

\item\t{secondLevelIsFeasible}: If the corresponding 
problem~\eqref{eqn:second-level-milp} is solved and it has an optimal 
solution, but problem~\eqref{eqn:computeBestUB} is not solved.

\item\t{UBIsSolved}: If the corresponding problem~\eqref{eqn:computeBestUB} 
is solved.
\end{itemize}
All linking solutions are stored in a hash table in
order to enable an efficient membership check.
  
\paragraph{Feasibility Check.} The following binary parameters determine the
strategy for when to solve the second-level
  problem~\eqref{eqn:second-level-milp} (for details on solution of the
  second-level problem, see Section~\ref{sec:solving-second-level}). 
  \begin{itemize}
  	\item \t{solveSecondLevelWhenLVarsFixed}: Whether to solve when 
          $l^t_{x_\J} = u^t_{x_\J} = x^t_{\J}$ (linking variables fixed).
	\item \t{solveSecondLevelWhenLVarsInt}: Whether to solve when 
	$x^t_{\J}\in \Z^{\J}$. 
	\item \t{solveSecondLevelWhenXVarsInt}: Whether to solve when $x^t \in
          X$. 
	\item \t{solveSecondLevelWhenXYVarsInt}: Whether to solve when $(x^t,
          y^t) \in X \times Y$. 	
  \end{itemize}
When problem~\eqref{eqn:second-level-milp} is solved, we have the following
implications, depending on the result.
\begin{itemize}
\item If~\eqref{eqn:second-level-milp} is infeasible, then all solutions $(x,
  y)$ with $x_{\J} = x_{\J}^t$ are infeasible and can be removed from the
  feasible region of the relaxation either by generation of a cut 
  or by branching. To avoid solving problems~\eqref{eqn:second-level-milp} 
  for a different solution with the same linking
  part, the tuple $[x_{\J}^t,\t{secondLevelIsInfeasible}]$ can be added to the
  linking solution pool $\E$.

\item If~\eqref{eqn:second-level-milp} has an optimal solution, we can avoid
  solving problem~\eqref{eqn:second-level-milp} for any $(x, y)$ for which
  $x_{\J} = x_{\J}^t$ arising in the future by adding the tuple $[x_{\J}^t,
    \t{secondLevelIsFeasible}]$ and the associated solution
  to~\eqref{eqn:second-level-milp} to the linking solution pool $\E$.
\end{itemize}
Regardless of parameter settings, we must always
  solve~\eqref{eqn:second-level-milp} whenever $(x^t, y^t) \in X \times Y$,
  $x^t_\J \not\in \E$ (we have not previously
  solved~\eqref{eqn:second-level-milp} for $x^t_\J$), and there are no
  branching candidates. Clearly, if we have branching candidates, then we can
  avoid solution of~\eqref{eqn:second-level-milp} by branching. Similarly, if
  $(x^t, y^t) \not\in X \times Y$, we can generate standard MILP cuts.
  Otherwise, we must have either
    \begin{itemize}
  \item \t{branchStrategy} is \t{fractional}, in which case we must
    solve~\eqref{eqn:second-level-milp} and then may either remove 
    $(x^t, y^t)$ by generating a valid inequality (if infeasible) or prune 
    the node (if feasible).  
  
  \item \t{branchStrategy} is set to \t{linking} and
    all linking variables are fixed, in which case we must also
    solve~\eqref{eqn:computeBestUB} (if~\eqref{eqn:second-level-milp} 
    is feasible) and prune the node. 
  \end{itemize}
  
    \paragraph{Computing Best UB.} The following binary parameters determine
  when the problem~\eqref{eqn:computeBestUB} should be solved in order to
  compute the best bilevel feasible solution $(x, y)$ with $x_{\J}=x_{\J}^t$
  (if such solution exists).

  \begin{itemize}
     \item \t{computeBestUBWhenLVarsFixed}: 
     Whether to solve when $l^t_{x_\J} = u^t_{x_\J} = x^t_{\J}$.
     \item \t{computeBestUBWhenLVarsInt}: 
     Whether to solve when $x_\J^t \in \Z^\J$.
     \item \t{computeBestUBWhenXVarsInt}: Whether to solve when $x^t \in X$.
  \end{itemize}
  After solving~\eqref{eqn:computeBestUB}, we know that no solution $(x, y)$
  with $x_{\J} = x_{\J}^t$ can be improving and thus, all such solutions can
  be removed either by generation of a cut or by branching. To avoid solving
  problem~\eqref{eqn:computeBestUB} for any solutions $(x, y)$ arising in the
  future for which $x_{\J}=x_{\J}^t$, the tuple $[x_{\J}^t,
    \t{secondLevelIsFeasible}]$ should be replaced with the tuple
  $[x_{\J}^t,\t{UBIsSolved}]$ and the associated solution stored in the
  linking solution pool $\E$.

Note that it would be possible to solve~\eqref{eqn:computeBestUBForOneNode}
rather than~\eqref{eqn:computeBestUB} if the goal were only to prune the node.
Solving~\eqref{eqn:computeBestUB} (which may not be much more difficult)
allows us to exploit the solution information globally through the linking
solution pool. 

Regardless of parameter setting, we must always solve
problem~\eqref{eqn:computeBestUB} (if it has not been previously solved for
$x^t_\J$) when (i) \t{branchStrategy} is \t{linking}, (ii) $(x^t,y^t)\in X\times Y$ 
and is bilevel infeasible and (iii) all linking variables are fixed. 
This does not mean 
that in this case, solving problem~\eqref{eqn:computeBestUB} 
(and further fathoming node $t$) is the only algorithmic option, 
as explained in Section~\ref{sec:branching}.

\subsection{Outline of the Algorithm \label{sec:outline}}
    
Algorithm~\ref{alg:branchandCutMain} gives the general outline of the node
processing loop in this branch-and-cut algorithm. Note that for simplicity,
the algorithm as stated assumes that the linking pool is used, though the
option of not using this pool is also provided. In almost all
cases, use of the linking pool is advantageous 
(see Section~\ref{sec:impactOfParametersForLinkingPool}). At a high level, the
procedure consists of the following steps. 
\begin{algorithm}
  \small
\caption{Node processing loop in \MIBS{}}\label{alg:branchandCutMain}
\ifthenelse{\useAlgorithmic = 1}{ %Algorithmic
\begin{algorithmic}[1]
  \Require Set $\E$
  \Ensure $L^t,U^t$,\t{finalStatus}
\State \t{branch $\leftarrow$
  False},$\,L^t\leftarrow-\infty,\,U^t\leftarrow\infty$ 
\While{\t{branch} is \t{False}}
   \State Solve~\eqref{eqn:LRRt} \label{alg-bandc:step-4}
   \State $L^t\leftarrow$ The optimal value of~\eqref{eqn:LRRt}
   \If {\eqref{eqn:LRRt} is infeasible \Or $L^t \geq U$ \Or 
     ($x_{\J}$ is fixed \An ([$x_{\J}^t,\t{secondLevelIsInfeasible}]\in\E$ 
     \Or [$x_{\J}^t,\t{UBIsSolved}]\in\E$))}
     \label{alg-bandc:step-5} 
      \State Fathom node $t$ \label{alg-bandc:step-6}
   \EndIf
   \If{\upshape$[x_{\J}^t,\cdot]\not\in\E$ 
   	\Statex[2]((\t{branchStrategy} is \t{linking} \An
        $(x^t,y^t)\in X\times Y$ \An
        $x_{\J}$ is fixed) \Or 
        \Statex[2] (\t{branchStrategy} is \t{fractional} \An
        $(x^t,y^t) \in X \times Y$) \Or
        \Statex[2] (\t{solveSecondLevelWhenXYVarsInt} \An
        $(x^t,y^t) \in X \times Y$) \Or
        \Statex[2] (\t{solveSecondLevelWhenXVarsInt} \An
        $x^t \in X$) \Or 
        \Statex[2] (\t{solveSecondLevelWhenLVarsInt} \An
        $x_{\J} \in \Z^{\J}$)\Or
        \Statex[2](\t{solveSecondLevelWhenLVarsFixed} \An
        $x_{\J}$ is fixed))} \label{alg-bandc:step-7}
      \State Solve~\eqref{eqn:second-level-milp} with $x_{\J} = x_{\J}^t$
       \If {\eqref{eqn:second-level-milp} is infeasible}
           \State $\E\leftarrow \E\cup [x_{\J}^t,\t{secondLevelIsInfeasible}]$
      \EndIf
   \EndIf
   \If{(\eqref{eqn:second-level-milp} was solved and has optimal solution) \Or 
       $[x_{\J}^t,\t{secondLevelIsFeasible}]\in\E$ \Or
     $[x_{\J}^t,\t{UBIsSolved}]\in\E$ } 
      \State $\hat{y}^t\leftarrow$ The optimal solution of
      \eqref{eqn:second-level-milp} 
       \If{$(x^t, y^t) \in \F$}
          \State $U^t\leftarrow c^1x^t + d^1 y^t, \,U \leftarrow \min\{U, U^t\}$
          \State Fathom node $t$
       \ElsIf{$(x^t, \hat{y}^t) \in \F$}
            \State $U^t\leftarrow c^1x^t + d^1\hat{y}^t,\,U\leftarrow \min\{U,
            U^t\}$ 
        \EndIf
        \If {$[x^t,\t{secondLevelIsInfeasible}]\not\in\E$ 
               \An $[x^t,\t{UBIsSolved}]\not\in\E$  \An
               \Statex[3]((\t{branchStrategy} is \t{linking} \An
               $(x^t,y^t)\in X\times Y$ \An
               $x_{\J}$ is fixed) \Or
              \label{alg-bandc:step-15}
              \Statex[3] (\t{computeBestUBWhenXVarsInt} \An
              $x^t \in X$) \Or
              \Statex[3](\t{computeBestUBWhenLVarsFixed} \An
              $x_{\J}$ is fixed)\Or
              \Statex[3] (\t{computeBestUBWhenLVarsInt}))}
              \State Solve \eqref{eqn:computeBestUB}
              \If{\eqref{eqn:computeBestUB} is feasible}
                 \State $U^t \leftarrow$ Optimal value
                 of\eqref{eqn:computeBestUB}, $\,U 
                  \leftarrow \min\{U, U^t\}$
              \EndIf  
              \State \E \leftarrow \E \cup [x_{\J}^t,\t{UBIsSolved}]$$ 
              \If{$x_{\J}$ is fixed}
                 \State Fathom node $t$ 
               \EndIf  
          \EndIf
   \EndIf
    \If {((\t{branchStrategy} is \t{linking} \An
        $x_{\J}$ is fixed) \Or \label{alg-bandc:step-25}
      \Statex[2] \t{branchStrategy} is \t{fractional}) \An
       $(x^t, y^t) \in X \times Y$}
      \State Remove $(x^t,y^t)$ by generating a cut \label{alg-bandc:step-26}
   \ElsIf {problem~\eqref{eqn:second-level-milp} was not solved \An
           $(x^t,y^t) \in X\times Y$ \An\\ 
           $[x_{\J}^t,\t{secondLevelIsInfeasible}]\not\in\E$ \An
          $[x_{\J}^t,\t{secondLevelIsFeasible}]\not\in\E$ \An 
          $[x_{\J}^t,\t{UBIsSolved}]\not\in\E$} \label{alg-bandc:step-27}
      \State \T{branch $\leftarrow$ True} \label{alg-bandc:step-28}
   \Else \label{alg-bandc:step-29}
      \State Remove $(x^t,y^t)$ by generating a cut or \t{branch $\leftarrow$
        True} \label{alg-bandc:step-30}
   \EndIf
\EndWhile
\end{algorithmic}
}
{ %Algorithm2e
\LinesNotNumbered{
  \SetKwInOut{Input}{Input}\SetKwInOut{Output}{Output}
  \Input{$[$Set $\E,U]$}
  \Output{$[$Set $\E,U,L^t,U^t]$}
  \nl\t{branch $\leftarrow$
    False},$\,L^t\leftarrow-\infty,\,U^t\leftarrow\infty$\\ 
  \nl\While{\upshape\t{branch} is \t{False}}{
     \nl Solve~\eqref{eqn:LRRt} \label{alg-bandc:step-3}\\
     \nl$L^t\leftarrow$ The optimal value of~\eqref{eqn:LRRt}\\
     \nl\If{\upshape\eqref{eqn:LRRt} is infeasible \Or $L^t \geq U$ \Or\\
     ($l_{x_{\J}}^t=u_{x_{\J}}^t$ \An 
     ([$x_{\J}^t,\t{secondLevelIsInfeasible}]\in\E$ \Or
       [$x_{\J}^t,\t{UBIsSolved}]\in\E$))}{ 
     \label{alg-bandc:step-5}
       \nl Fathom node $t$\\\label{alg-bandc:step-6}
        }
    \nl\If{\upshape$[x_{\J}^t,\cdot]\not\in\E$ \An\\
   	\quad((\t{branchStrategy} is \t{linking} \An
        $(x^t,y^t)\in X\times Y$ \An
        $l_{x_{\J}}^t=u_{x_{\J}}^t$) \Or\\ 
        \quad\:\,(\t{branchStrategy} is \t{fractional} \An
        $(x^t,y^t) \in X \times Y$) \Or\\
       \quad\:\,(\t{solveSecondLevelWhenXYVarsInt} \An
        $(x^t,y^t) \in X \times Y$) \Or\\
        \quad\:\,(\t{solveSecondLevelWhenXVarsInt} \An
        $x^t \in X$) \Or\\
       \quad\:\,(\t{solveSecondLevelWhenLVarsInt} \An
        $x_{\J} \in \Z^{\J}$)\Or\\
       \quad\:\,(\t{solveSecondLevelWhenLVarsFixed} \An
        $l_{x_{\J}}^t=u_{x_{\J}}^t$))} {\label{alg-bandc:step-7}
        \nl Solve~\eqref{eqn:second-level-milp} to find $\phi(A^2x^t)$ 
        \label{alg-bandc:step-8} \\
         \nl\If{\upshape\eqref{eqn:second-level-milp} is infeasible}{
         \label{alg-bandc:step-9}
             \nl$\E\leftarrow \E\cup [x_{\J}^t,\t{secondLevelIsInfeasible}]$
             \label{alg-bandc:step-10}\\
             \nl\If{\upshape
               $l_{x_{\J}}^t=u_{x_{\J}}^t$}{\label{alg-bandc:step-11} 
                 \nl Fathom node $t$\label{alg-bandc:step-12}
                 }
             }
          \nl\Else{ \label{alg-bandc:step-13}
             \nl$\E\leftarrow \E\cup [x_{\J}^t,\t{secondLevelIsFeasible}]$
              \label{alg-bandc:step-14}
	} 
    }
     \nl\If{\upshape$[x_{\J}^t,\t{secondLevelIsFeasible}]\in\E$ \Or
       $[x_{\J}^t,\t{UBIsSolved}]\in\E$ }{ 
     \label{alg-bandc:step-15}
       \nl$\hat{y}^t\leftarrow$ The optimal solution of
       \eqref{eqn:second-level-milp} 
       \label{alg-bandc:step-16}\\
       \nl\If{\upshape$(x^t, y^t) \in \F$}{\label{alg-bandc:step-17}
          \nl$U^t\leftarrow c x^t + d^1 y^t, \,U \leftarrow \min\{U, U^t\}$ \\
          \label{alg-bandc:step-18}
          \nl Fathom node $t$\label{alg-bandc:step-19}
          }
        \nl\If {\upshape $[x_{\J}^t,\t{UBIsSolved}]\not\in\E$ \An\\
           \quad((\t{branchStrategy} is \t{linking} \An
           $(x^t,y^t)\in X\times Y$ \An
           $l_{x_{\J}}^t=u_{x_{\J}}^t$) \Or\\
            \quad\:\,(\t{computeBestUBWhenXVarsInt} \An
              $x^t \in X$) \Or\\
             \quad\:\,(\t{computeBestUBWhenLVarsFixed} \An
              $l_{x_{\J}}^t=u_{x_{\J}}^t$) \Or\\
             \quad\:\,(\t{computeBestUBWhenLVarsInt}))}{\label{alg-bandc:step-20}
              \nl Solve \eqref{eqn:computeBestUB}\\\label{alg-bandc:step-21}
              \nl\If{\upshape\eqref{eqn:computeBestUB} is feasible}{
              \label{alg-bandc:step-22}
                 \nl
                 $\,U \leftarrow \min\{U, \text{optimal value of~\eqref{eqn:computeBestUB}}
                 \}$\\\label{alg-bandc:step-23}
              }
              \nl$\E \leftarrow \E \setminus
                 [x_{\J}^t,\t{secondLevelIsFeasible}]$,  
              $\E \leftarrow \E \cup [x_{\J}^t,\t{UBIsSolved}]$
                 \\\label{alg-bandc:step-24} 
              \nl\If{\upshape
                $l_{x_{\J}}^t=u_{x_{\J}}^t$}{\label{alg-bandc:step-25} 
                 \nl Fathom node $t$ \\\label{alg-bandc:step-26}
               }
         }
      \nl\ElseIf{\upshape$(x^t, \hat{y}^t) \in \F$}{\label{alg-bandc:step-27}
            \nl
            $U\leftarrow \min\{U, c x^t + d^1\hat{y}^t\}$
            \label{alg-bandc:step-28}
     	   }
     }
      \nl\If{\upshape\t{branchStrategy} is \t{fractional} \An $(x^t, y^t) \in
        X \times Y$}{ 
      \label{alg-bandc:step-29}
          \nl Remove $(x^t,y^t)$ by generating a cut
          \label{alg-bandc:step-30}
          }
     \nl\ElseIf{\upshape $[x_{\J}^t,\cdot]\not\in\E$ \An
             $(x^t,y^t) \in X\times Y$
             }
             {\label{alg-bandc:step-31}
          \nl\t{branch $\leftarrow$ True}\label{alg-bandc:step-32}
          }
   \nl\Else{\label{alg-bandc:step-33}
          \nl Remove $(x^t,y^t)$ by generating a cut or \t{branch $\leftarrow$
            True} \label{alg-bandc:step-34}
      }
  }
  }
}
\end{algorithm}
\begin{enumerate}
\item Solve the relaxation~\eqref{eqn:LRRt} (line \ref{alg-bandc:step-3}) and
  prune the node (lines~\ref{alg-bandc:step-5}--\ref{alg-bandc:step-6}) if
  either  
\begin{itemize}
\item \eqref{eqn:LRRt} is infeasible;
\item $L^t \geq U$; or
\item $x_{\J}$ is fixed and it has been stored in set $\E$ with 
either \t{secondLevelIsInfeasible} or \t{UBIsSolved} tags.
\end{itemize}
\item If~\eqref{eqn:second-level-milp} was not previously solved with respect
  to $x_{\J}^t$, 
  then depending on $(x^t, y^t)$ and the parameter settings,     
  we may next solve it (lines~\ref{alg-bandc:step-7}--\ref{alg-bandc:step-8}).
  \begin{itemize}
  \item \eqref{eqn:second-level-milp} is infeasible $\Rightarrow x^t
    \not\in \F_1$ 
  (line~\ref{alg-bandc:step-9}) 
    \begin{itemize}
    \item Add $[x_{\J}^t,\t{secondLevelIsInfeasible}]$ to $\E$
      (line~\ref{alg-bandc:step-10}). 
    \item If $x_{\J}$ is fixed, fathom node $t$ 
    (lines~\ref{alg-bandc:step-11}--\ref{alg-bandc:step-12}).
    \end{itemize}
  \item \eqref{eqn:second-level-milp} is feasible $\Rightarrow$ add
    $[x_{\J}^t,\t{secondLevelIsFeasible}]$ to $\E$ 
    (lines~\ref{alg-bandc:step-13}--\ref{alg-bandc:step-14}).
  \end{itemize}
  
 \item If~\eqref{eqn:second-level-milp} was solved now or previously 
 and is feasible (line~\ref{alg-bandc:step-15}), 
 we have either
   \begin{itemize}
   \item $(x^t, y^t) \in \F \Rightarrow$ update $U$ and fathom node $t$ 
   (lines~\ref{alg-bandc:step-17}--\ref{alg-bandc:step-19}).
   \item  $(x^t, y^t) \not\in \F $ 
     \begin{itemize}
           \item Update $U$ if $(x^t, \hat{y}^t) \in \F$ (in case of not
             solving~\eqref{eqn:computeBestUB})  
           (lines~\ref{alg-bandc:step-27}--\ref{alg-bandc:step-28}) 
           and eliminate $(x^t,y^t)$
           (lines~\ref{alg-bandc:step-29}--\ref{alg-bandc:step-34}).
            
     \item If~\eqref{eqn:computeBestUB} was not previously solved with respect
     to $x^t_\J$, then depending on $(x^t, y^t)$ and the parameter settings,
     we may solve it
     (lines~\ref{alg-bandc:step-20}--\ref{alg-bandc:step-21}).
       \begin{itemize}
           \item \eqref{eqn:computeBestUB} is feasible $\Rightarrow$ update $U$ 
              (lines~\ref{alg-bandc:step-22}--\ref{alg-bandc:step-23}).
            \item Remove $[x_{\J}^t,\t{secondLevelIsFeasible}]$ from set $\E$ and 
            add $[x_{\J}^t,\t{UBIsSolved}]$ 
               (line~\ref{alg-bandc:step-24}).
             \item If $x_{\J}$ is fixed, fathom node $t$ 
                (lines~\ref{alg-bandc:step-25}--\ref{alg-bandc:step-26}).
        \end{itemize}
    \end{itemize}
   \end{itemize}

\item Finally, we must either branch or remove $(x^t, y^t)$ by adding valid
  inequalities (lines~\ref{alg-bandc:step-29}--\ref{alg-bandc:step-34}). 

\begin{itemize}
    \item If there are no branching candidates, then we must remove $(x^t,
      y^t)$ by adding valid inequalities
      (lines~\ref{alg-bandc:step-29}--\ref{alg-bandc:step-30}). 
    \item If~\eqref{eqn:second-level-milp} was not solved now or previously,
      we must branch 
      if $(x^t,y^t)\in X\times Y$ 
      (lines~\ref{alg-bandc:step-31}--\ref{alg-bandc:step-32}).
    \item Otherwise, we have the choice of either adding valid inequalities or
      branching (lines~\ref{alg-bandc:step-33}--\ref{alg-bandc:step-34}).
\end{itemize}

\end{enumerate}
In the case of not using the linking solution pool, the structure of algorithm is
similar to Algorithm~\ref{alg:branchandCutMain}, but differs in the steps where
information from the linking solution pool is exploited or new information is
added to this pool. For example, the
lines~\ref{alg-bandc:step-10},~\ref{alg-bandc:step-14} and~\ref{alg-bandc:step-24} 
are eliminated and the
lines~\ref{alg-bandc:step-5},~\ref{alg-bandc:step-7},~\ref{alg-bandc:step-15} 
and~\ref{alg-bandc:step-20} are modified.

As mentioned earlier, this algorithm can be adapted for other risk functions.
For example, the pessimistic risk function~\eqref{eqn:XiPessimistic} can be
accommodated with a few modifications as follows. 
The same relaxation is used, but the feasibility check is slightly different.
At node $t$, if $(x^t, y^t)\notin X\times Y$, it is infeasible and should be
removed, as usual. Otherwise, we compute $\phi(A^2x^t)$ and check 
whether $d^2y^t = \phi(A^2x^t)$. If $d^2y^t > \phi(A^2x^t)$, $(x^t,y^t)$ is infeasible, 
as in the optimistic case, but $d^2y^t = \phi(A^2x^t)$ does not guarantee its feasibility. 
Feasibility of $(x^t, y^t)$ must be verified by also ensuring that $d^1y^t$ is
equal to the optimal value of the MILP  
\begin{equation*}
\max\left\{d^1y\midd y\in\P_1(x^t)\cap \P_2(x^t)\cap Y,
d^2y\leq\phi(A^2x^t)\right\}. 
\end{equation*}
Finding the best feasible solution with $x_\J=\gamma\in\Z^\J$ requires solving
a new bilevel problem and its details are beyond the scope of this paper. To
avoid this, the \t{fractional} branching strategy can be used instead.

\subsection{Solving the Second-level Problem \label{sec:solving-second-level}}

It is evident that in most cases, much (if not most) of the computational
effort in executing the algorithm arises from the time
required to solve~\eqref{eqn:second-level-milp} and~\eqref{eqn:computeBestUB} 
(See Table~\ref{tab:numberTimeSolveMIPs}).
A major focus of ongoing work, therefore, is the development of methodology to
reduce the time spent solving these problems.

In closely related work on solving two-stage stochastic mixed integer optimization 
problems, methodology for warm-starting the solution process of an MILP using
information derived from previously solved instances has been developed.
\mycitet{Hassanzadeh and Ralphs}{HasRal14} described a method for solving a
sequence of MILPs differing only in the right-hand side. It is shown that a
sequence of such solves can be performed within a single branch-and-bound
tree, with each solve starting where the previous one left off. Under mild
conditions, this method can be used to construct a complete description of the
value function $\phi$. This method of solving such a sequence of MILPs has
been implemented within the SYMPHONY MILP
solver~\mycitep{SYMPHONY,RalGuz06,RalGuz05}, which is one of several supported
solvers that can be used for solution of~\eqref{eqn:second-level-milp}
and~\eqref{eqn:computeBestUB} within \MIBS{}. Other supported solvers include
CPLEX~\mycitep{CPLEX} and Cbc~\mycitep{Cbc}. 
The parameter \t{feasCheckSolver} determines which MILP solver to be employed.   

\section{Computational Results \label{sec:computationalResultsSection}}
A number of experiments were conducted to evaluate the impacts of the various
algorithmic options provided by \MIBS{}. The parameters investigated in this
section are
\begin{itemize} 
\item The parameters that determine when the
  problems~\eqref{eqn:second-level-milp} and \eqref{eqn:computeBestUB} should
  be solved. 

\item The parameter that determines the branching strategy.
\item The parameter that determines whether to use the linking solution 
pool or not.
\item The parameters that determine whether to use the primal heuristics or
  not. 
\end{itemize}
Because of space constraints and to allow more in-depth analysis, testing of
the effectiveness of various classes of valid inequalities and parameters for
controlling them is not included here, but will be the subject of a future
study. Three different data sets (171 instances in total) were employed in our
experiments as follows.

\begin{itemize}

\item \t{IBLP-DEN}: This set was generated by~\mycitet{DeNegre}{DeNegre2011}, and
  contains 50 instances with 15--20 integer variables and 20 constraints, all
  at the second level.

\item \t{IBLP-FIS}: This is a selected set of 21 of the instances generated 
by~\mycitet{Fischetti et al.}{fischettietal17a}. These instances originate from
\t{MILPLIB 3.0}~\mycitep{bixbyceria98}
and all variables are binary and all constraints are at the second level.

\item \t{MIBLP-XU}: This set was introduced by~\mycitet{Xu and Wang}{xuwang14} and includes
  100 randomly-generated instances. In these problems, all first-level
  variables are integer with upper bound 10, while the second-level variables
  are continuous with probability 0.5. The number of first- and second-level
  variables are equal and $n_1$ is in the range of $10-460$ with an increments
  of 50. Furthermore, the number of first-level and second-level constraints
  are equal to $0.4n_1$. To have specific bounds on all integer variables, we
  added the very loose upper bound 1500 for all integer second-level variables
  in converting these instances to the \t{MibS} format.
\end{itemize}
Table~\ref{tab:dataSetSummary} summarizes the properties of the data sets.
Note that in the described data sets, all first-level variables are linking.

\begin{table}[h!]
\caption{The summary of data sets}
\label{tab:dataSetSummary}       
\begin{tabular}{c c c c c c c c}
\hline\noalign{\smallskip}
Data Set & \begin{tabular}[c]{@{}c@{}}First-level \\ Vars Num\end{tabular} & \begin{tabular}[c]{@{}c@{}}Second-level \\ Vars Num\end{tabular} & \begin{tabular}[c]{@{}c@{}}First-level \\ Cons Num\end{tabular} & \begin{tabular}[c]{@{}c@{}}Second-level \\ Cons Num\end{tabular} & \begin{tabular}[c]{@{}c@{}}First-level \\ Vars Type\end{tabular} & \begin{tabular}[c]{@{}c@{}}Second-level \\ Vars Type\end{tabular} & Size \\
\noalign{\smallskip}\hline\noalign{\smallskip}
IBLP-DEN                       & 5-15                                                            & 5-15                                                             & 0                                                               & 20                                                               & discrete                                                         & discrete                                                          & 50 \\[0.14cm]
IBLP-FIS                       & 4-2481                                                          & 2-2480                                                           & 0                                                               & 16-4944                                                          & binary                                                           & binary                                                            & 21   \\ 
MIBLP-XU                       & 10-460                                                          & 10-460                                                           & 4-184                                                           & 4-184                                                            & discrete                                                         & \begin{tabular}[c]{@{}c@{}}continuous, \\ discrete\end{tabular}   & 100  \\ 
\noalign{\smallskip}\hline
\end{tabular}
\end{table}
All computational results we report were generated on compute nodes running 
the Linux (Debian 8.7) operating system with dual AMD Opteron 6128 processors 
and 32 GB RAM and all experiments were run sequentially.
The time limit was 3600 seconds and the pseudocost
branching strategy was used to choose the best variable among the branching
candidates for all experiments. In all numerical experiments, the generation
of generic MILP cuts by the Cut Generation Library~\mycitep{Cgl} was disabled,
since these cuts seem unlikely to be effective in addition to the classes
specific to MIBLPs and their integration can cause other algorithmic issues
that would first need to be addressed. SYMPHONY was employed as the MILP
solver
(preprocessing and primal heuristics were turned off) in all experiments, 
unless otherwise noted.
In all experiments, the integer no-good cut was employed for solving the
instances of \t{IBLP-DEN} and \t{IBLP-FIS} sets, and the problems belonging to
the \t{MIBLP-XU} set were solved by using the hypercube intersection cut.
Furthermore, all primal heuristics of \t{MibS} were disabled in the numerical
experiments except as otherwise noted, since these have in general also not
proven to be very effective.

All instances were initially solved by all methods described in
Sections~\ref{sec:impactOfParametersForSLandUBSection}--\ref{sec:impactOfParametersForLinkingPool} 
below, but in plotting performance profiles, we chose the 125 problems that
could be solved by at least one method in 3600 seconds and whose solution time
exceeds 5 seconds for at least one method. This test set was used for all
plots and tables shown in Section~\ref{sec:computationalResultsSection}.
The details of the results employed for plotting all of these figures and tables, are shown in the appendix 
(the reported running times do not include the required time for reading the instances).
\subsection{Impact of Strategy for Solving~\eqref{eqn:second-level-milp}
  and~\eqref{eqn:computeBestUB}}\label{sec:impactOfParametersForSLandUBSection}
In order to evaluate the impacts of the parameters for solving 
problems~\eqref{eqn:second-level-milp} and~\eqref{eqn:computeBestUB}, 
we employed five different methods:

\begin{itemize}
\item \t{whenLInt-LInt}: Problems~\eqref{eqn:second-level-milp} 
and~\eqref{eqn:computeBestUB} were both solved only
when $x^t_\J \in \Z^\J$, i.e., the parameters \t{solveSecondLevelWhenLVarsInt} 
and \t{computeBestUBWhenLVarsInt} were set to \t{true}.

\item \t{whenLInt-LFixed}: Problem~\eqref{eqn:second-level-milp} was solved
  when $x^t_\J \in \Z^\J$ and problem~\eqref{eqn:computeBestUB} was solved when
  all linking variables are fixed, i.e., \mymanageline{the parameters }
  \t{solveSecondLevelWhenLVarsInt} 
  and~\t{computeBestUBWhenLVarsFixed} were set to \t{true}.

\item\t{whenLFixed-LFixed}: Problems~\eqref{eqn:second-level-milp}
  and~\eqref{eqn:computeBestUB} were both solved only when linking variables
  were fixed, i.e., \t{solveSecondLevelWhenLVarsFixed} and
  \t{computeBestUBWhenLVarsFixed} were set to \t{true}.

\item\t{whenXYInt-LFixed}: Problem~\eqref{eqn:second-level-milp} was solved
  only when $(x^t, y^t) \in X \times Y$ and problem~\eqref{eqn:computeBestUB}
  was solved only when, in addition, all linking variables were fixed, i.e., 
  \mymanageline{the parameters }\t{solveSecondLevelWhenXYVarsInt} 
  and \t{computeBestUBWhenLVarsFixed} were set to \t{true}.

\item\t{whenXYIntOrLFixed-LFixed}: Problem~\eqref{eqn:second-level-milp} was
  only solved when $(x^t, y^t) \in X \times Y$ or all linking 
 variables were fixed and problem~\eqref{eqn:computeBestUB} was solved only
 whenever all linking variables are fixed, i.e.,
 \t{solveSecondLevelWhenXYVarsInt},  
 \t{solveSecondLevelWhenLVarsFixed} and \t{computeBestUBWhenLVarsFixed} were
 set to \t{true}. 
\end{itemize}
In this first set of experiments, the \t{branchStrategy} was set to
\t{linking} and \t{useLinkingSolutionPool} was set to \t{true}. The
performance profiles shown in Figure~\ref{fig:SSUBParTimeMore5Sec} compare the
solution time of the five described methods.

Figure~\ref{fig:SSUBParTimeMore5SecNonXu} shows the results for the
\t{IBLP-DEN} and \t{IBLP-FIS} sets and Figure~\ref{fig:SSUBParTimeMore5SecXu}
shows the results for the \t{MIBLP-XU} set. These figures show the superiority
of \t{whenLFixed-LFixed} and \t{whenXYIntOrLFixed-LFixed} over the other three
methods, with roughly the same performance for each. Based on these results,
the settings for \t{whenXYIntOrLFixed-LFixed} have been chosen as the default
setting for \MIBS{} and in the remainder of these experiments unless otherwise
noted.
\begin{figure}[h!]
\begin{subfigure}{0.5\textwidth}
\includegraphics[height=2in]{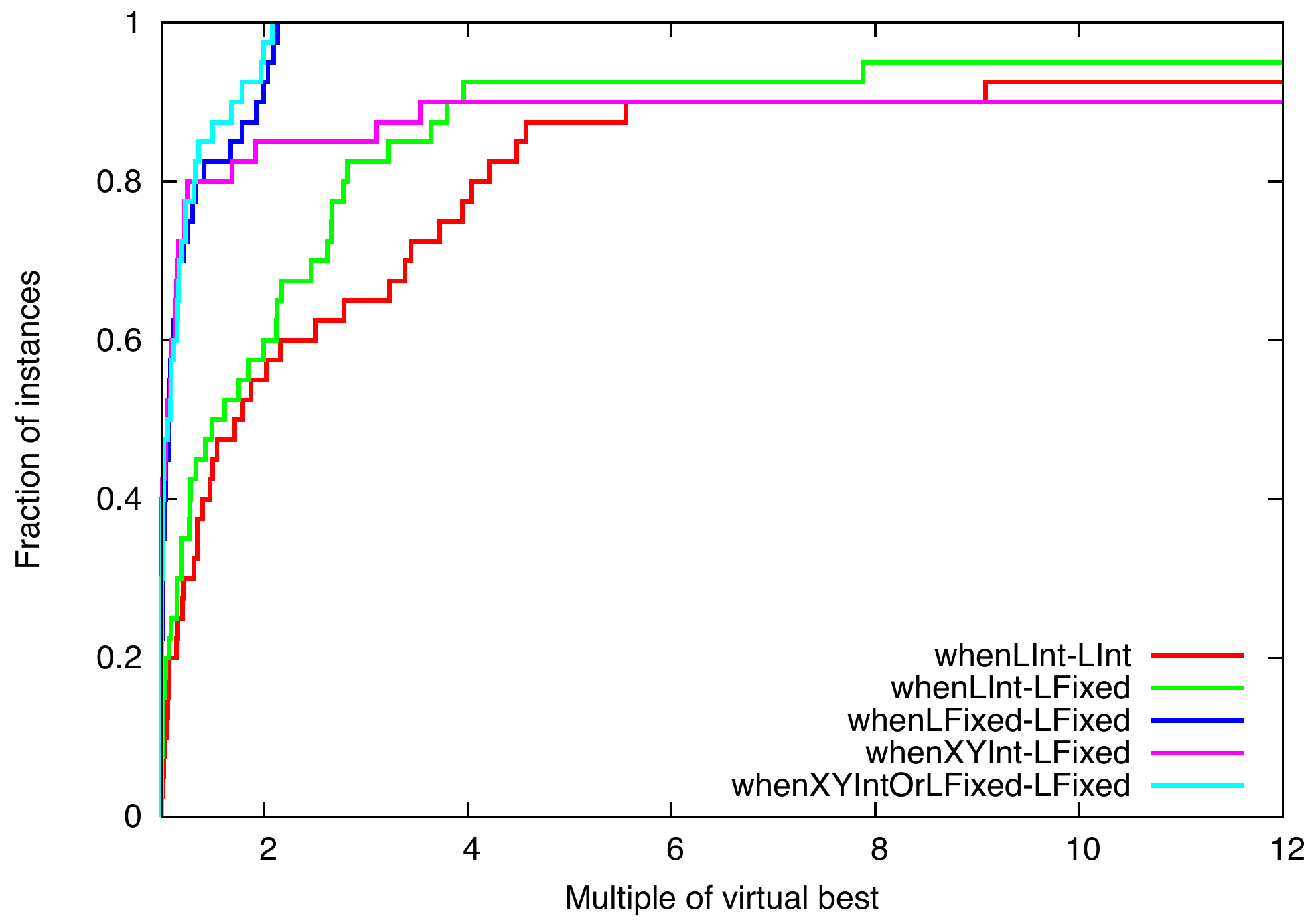}
\caption{\t{IBLP-DEN} and \t{IBLP-FIS} sets
  \label{fig:SSUBParTimeMore5SecNonXu}}
\end{subfigure}
\begin{subfigure}{0.5\textwidth}
\centering
\includegraphics[height=2in]{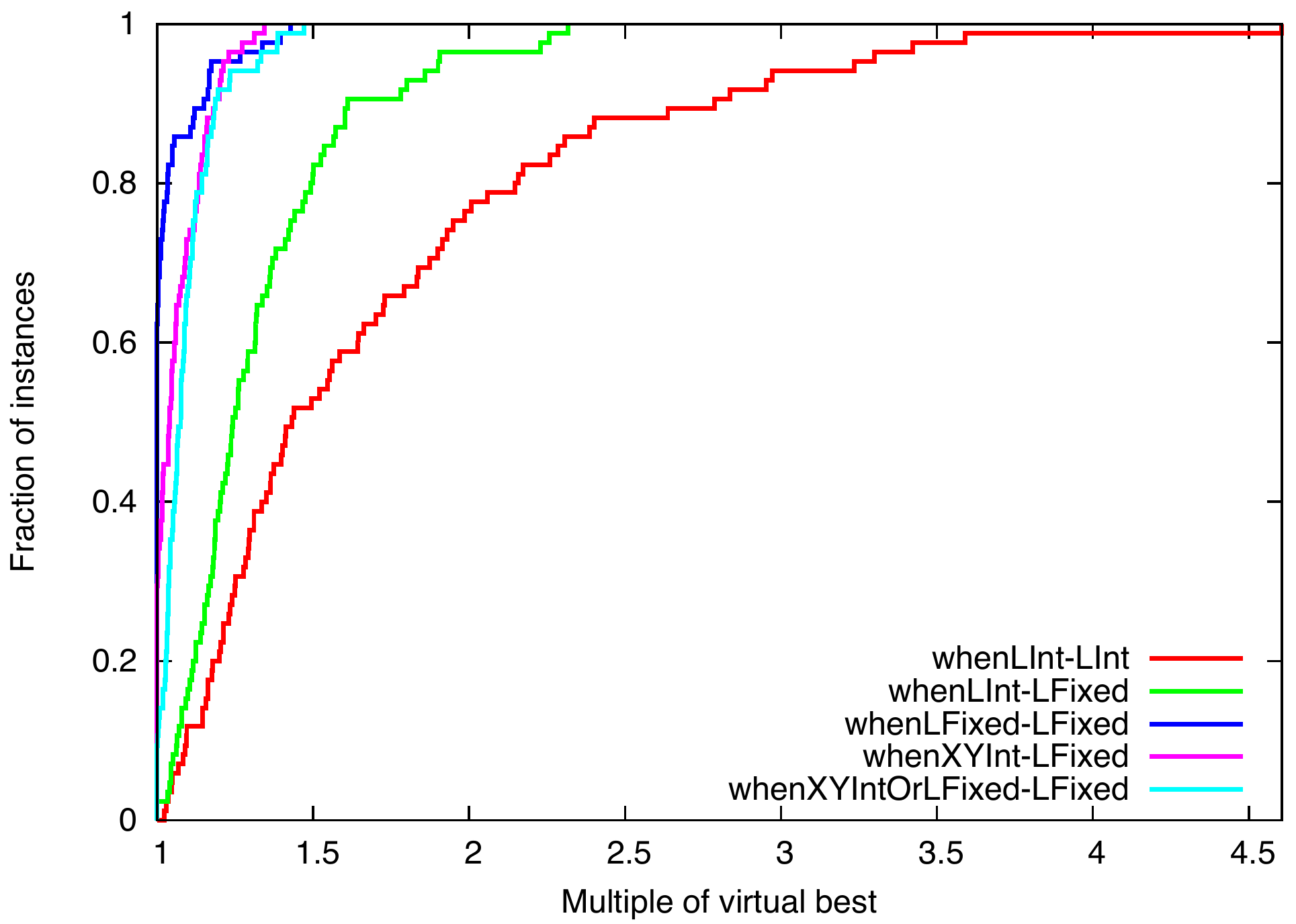}
\caption{\t{MIBLP-XU} set
  \label{fig:SSUBParTimeMore5SecXu}}
\end{subfigure}
\caption{Impact of the parameters for solving 
problems~\eqref{eqn:second-level-milp} and~\eqref{eqn:computeBestUB}.
\label{fig:SSUBParTimeMore5Sec}}
\end{figure}
\subsection{Impact of Branching Strategy}
\label{sec:impactOfParametersForBranchingStrategy}
As mentioned earlier, the \t{branchStrategy} parameter controls 
which variables are considered branching candidates and can be 
set to \t{fractional} and \t{linking}. In order to evaluate the effect of 
the parameter, we compared the performance of these two methods. The linking
solution pool was used in both cases. 

In initial testing, we observed a possible relationship between the number of
integer first- and second-level variables and the performance of branching
strategies. Hence, we further analyzed the results by dividing them into two
separate sets with $r_1\leq r_2$ (27 instances) and $r_1 > r_2$ (98
instances). Figure~\ref{fig:branchStrategyParTimeMore5Sec} shows the
performance profiles for these two sets with the solution time as the
performance measure.

\begin{figure}[h!]
\begin{subfigure}{0.5\textwidth}
\centering
\includegraphics[height=2in]{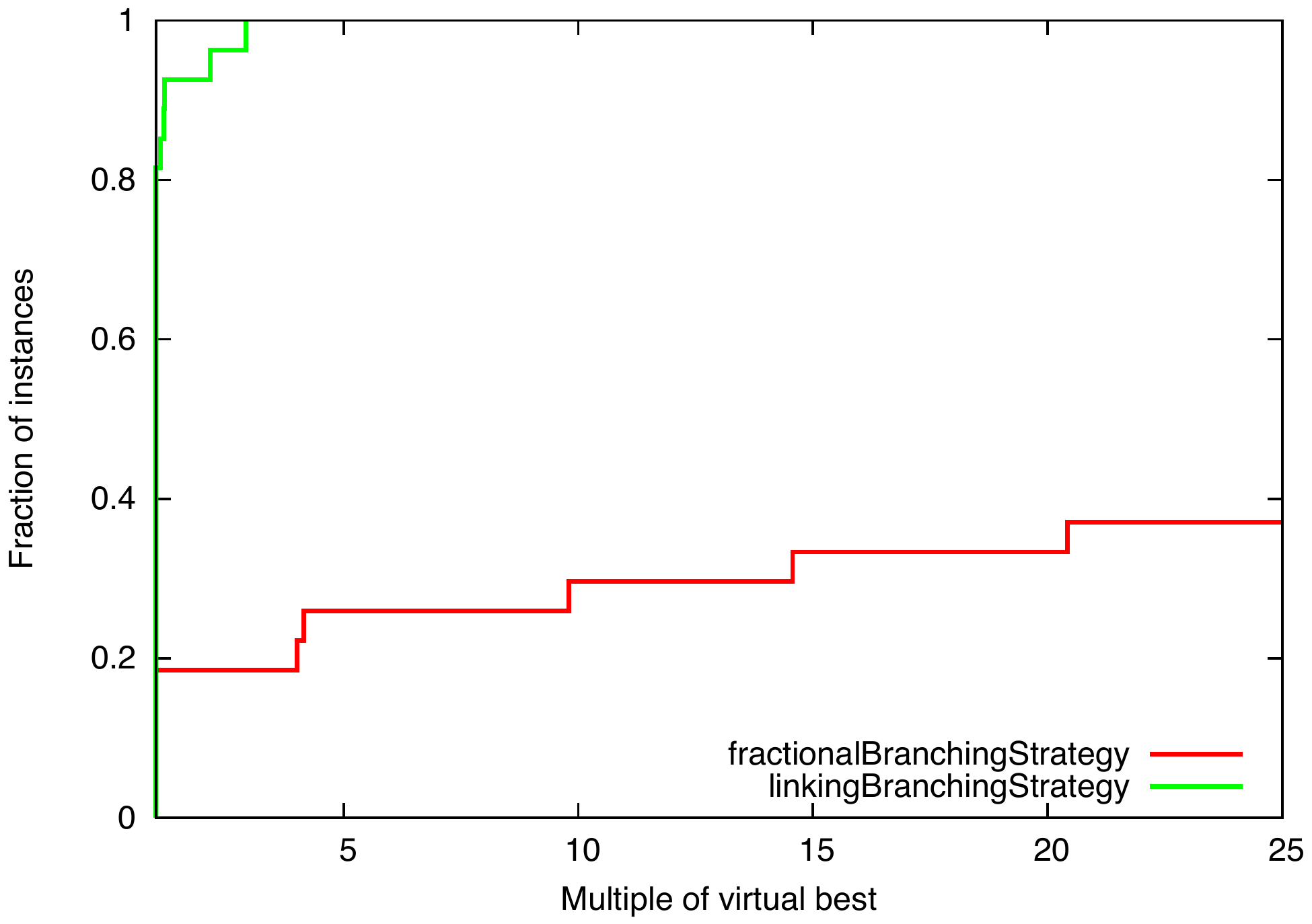} 
\caption{$r_1\leq r_2$
  \label{fig:branchStrategyParTimeMore5Secr1Leqr2}}
\end{subfigure}
\begin{subfigure}{0.5\textwidth}
\centering
\includegraphics[height=2in]{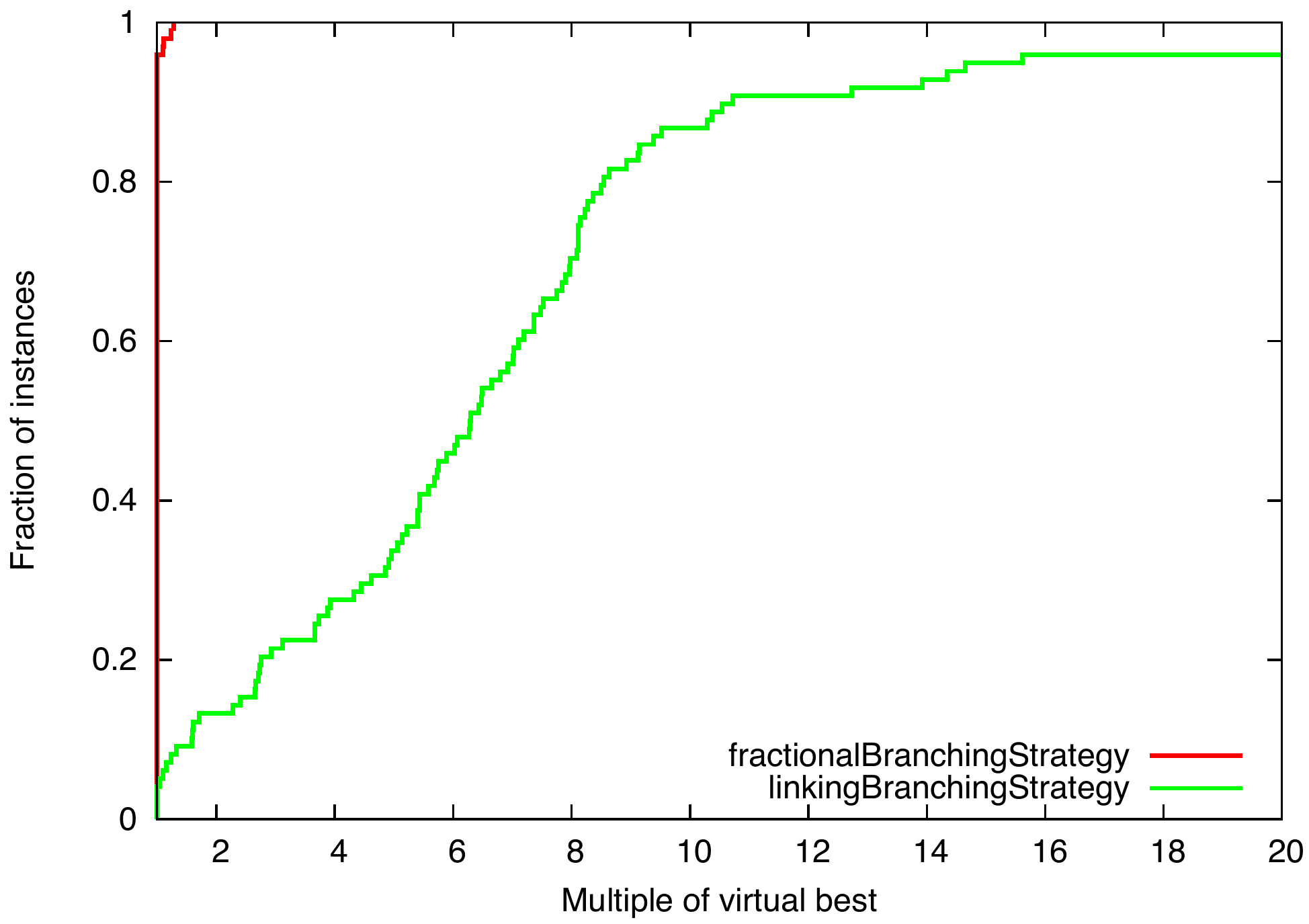}
\caption{$r_1 > r_2$
  \label{fig:branchStrategyParTimeMore5Secr1grer2}}
\end{subfigure}
\caption{Impact of the 
  \t{branchStrategy} parameter.
\label{fig:branchStrategyParTimeMore5Sec}}
\end{figure}
Figure~\ref{fig:branchStrategyParTimeMore5Secr1Leqr2} shows that \t{linking}
generally performed better when $r_1\leq r_2$, while
Figure~\ref{fig:branchStrategyParTimeMore5Secr1grer2} indicates that the
\t{fractional} branching strategy performed better when $r_1 > r_2$. Based on
these results, the default value of \t{branchStrategy} parameter in \MIBS{}
has been set to \t{linking} and \t{fractional} for the instances with $r_1\leq
r_2$ and $r_1 > r_2$, respectively. In general, it would not be easy to predict
which branching strategy will behave better for a particular given instance, 
but it is intuitive that when $r_1 << r_2$, the \t{linking}
strategy would be better.

\subsection{Impact of Linking Solution Pool}
\label{sec:impactOfParametersForLinkingPool}
A set of eight experiments were evaluated to assess the impact of the
linking solution pool. The chosen parameters for these experiments were: 

\begin{itemize}
\item\t{withoutPoolWhenXYInt-LFixed}: The linking solution pool is not used 
and  \t{whenXYInt-LFixed} strategy is used for solving 
problems~\eqref{eqn:second-level-milp} and~\eqref{eqn:computeBestUB}. 

\item \t{withPoolWhenXYInt-LFixed}: The same as 
\t{withoutPoolWhenXYInt-LFixed}, but with the use of linking solution pool. 

\item\t{withoutPoolWhenXYIntOrLFixed-LFixed}: The linking solution pool is 
not used and the strategy for solving problems~\eqref{eqn:second-level-milp} 
and~\eqref{eqn:computeBestUB} is \t{whenXYIntOrLFixed-LFixed}.

\item\t{withPoolWhenXYIntOrLFixed-LFixed}: The same as
  \t{withoutPoolWhenXYIntOrLFixed-LFixed}, but with the use of the linking
  solution pool. 
\end{itemize}
Each of the above four settings was tested with both \t{fractional} and
\t{linking} branching strategies, although the linking pool was only expected
to have a large impact when we allow branching on non-linking variables. This
is because when branching only on linking variables, linking solutions can
only arise again within the same subtree as they first arose and this can only
happen if the solution was not already removed with a valid inequality.

The performance profiles shown in
Figure~\ref{fig:linkingSolutionPooParWithTimeMore5Sec} compare the
solution time of the described methods.  
Figure~\ref{fig:linkingSolutionPooParWithFracParWithTimeMore5Sec} 
shows the methods which use \t{fractional} branching strategy, and 
Figure~\ref{fig:linkingSolutionPooParWithLinkParWithTimeMore5Sec} 
shows the methods with \t{linking} branching strategy. 

\begin{figure}[h!]
\begin{subfigure}{0.5\textwidth}
\centering
\includegraphics[height=2in]{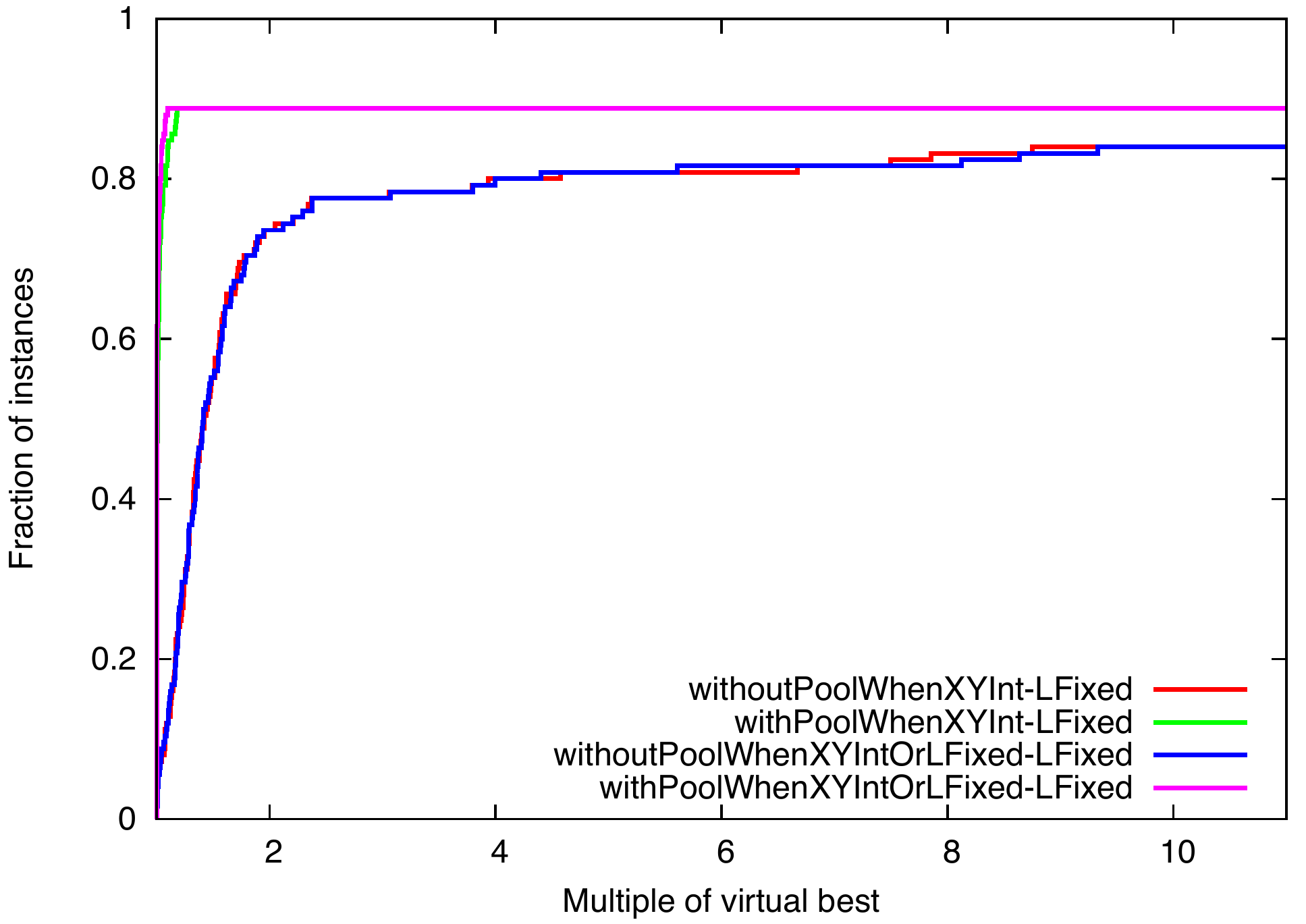}
\caption{\t{fractional} branching strategy}
\label{fig:linkingSolutionPooParWithFracParWithTimeMore5Sec}
\end{subfigure}
\begin{subfigure}{0.5\textwidth}
\centering
\includegraphics[height=2in]{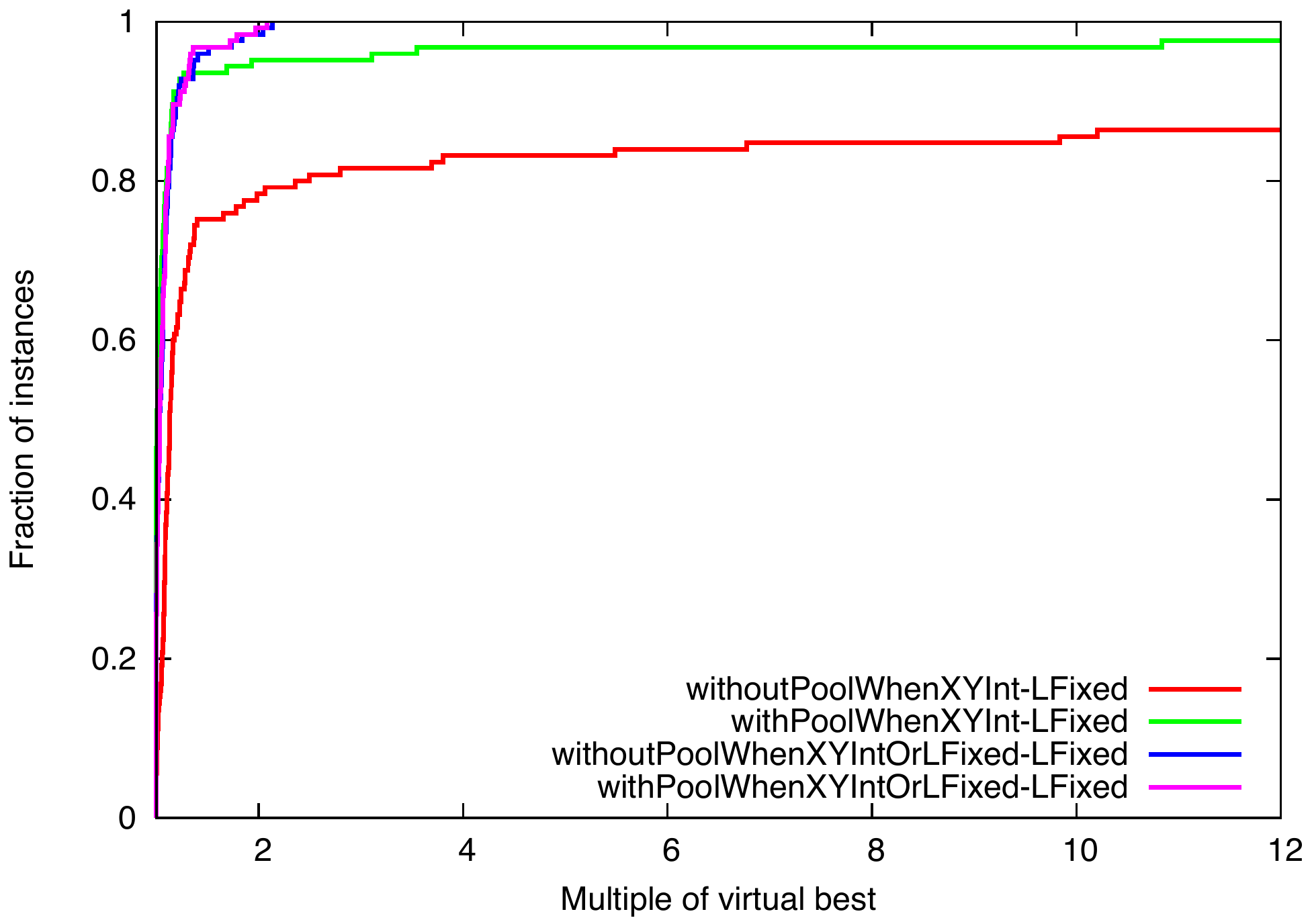}
\caption{\t{linking} branching strategy}
\label{fig:linkingSolutionPooParWithLinkParWithTimeMore5Sec} 
\end{subfigure}
\caption{Impact of the linking solution pool.}
\label{fig:linkingSolutionPooParWithTimeMore5Sec}
\end{figure}
As expected, when branching was not limited to linking variables, the 
linking solution pool was effective in decreasing the solution time. This can
be observed by comparing

\begin{itemize}
\item \t{withPoolWhenXYInt-LFixed} 
and \t{withoutPoolWhenXYInt-LFixed} for both the \t{fractional} 
and \t{linking} branching strategies; and

\item \t{withPoolWhenXYIntOrLFixed-LFixed} 
and \t{withoutPoolWhenXYIntOrLFixed-LFixed} for the \t{fractional} branching
strategy.

\end{itemize}
In these cases, branching can also be done on non-linking variables. However,
use of the solution pool for \t{withoutPoolWhenXYIntOrLFixed-LFixed} with \t{linking}
branching strategy does not improve performance because the branching was
only done on linking variables. Based on the results achieved in this section,
the linking solution pool is used by default in \MIBS{}.

Table~\ref{tab:numberTimeSolveMIPs} shows the total number of instances of
problems~\eqref{eqn:second-level-milp} and~\eqref{eqn:computeBestUB} solved
when using the methods \t{withPoolWhenXYIntOrLFixed-LFixed} and
\t{withoutPoolWhenXYIntOrLFixed-LFixed} with both \t{linking} and
\t{fractional} branching strategies, as well as the percent of total solution
time required. Comparing the columns for the \t{linking} branching strategy
verifies the results shown in
Figure~\ref{fig:linkingSolutionPooParWithLinkParWithTimeMore5Sec}. These
columns show that using the linking solution pool does not decrease the number
of~\eqref{eqn:second-level-milp} and~\eqref{eqn:computeBestUB} instances
solved, so we do not expect to improve the solution time either. Similarly,
the columns for the \t{fractional} branching strategy verify the results shown
in Figure~\ref{fig:linkingSolutionPooParWithFracParWithTimeMore5Sec} and show
that the number of instances of~\eqref{eqn:second-level-milp}
and~\eqref{eqn:computeBestUB} can be reduced significantly by using the pool
in this case, resulting in decreased solution time.

The results with both the \t{linking} and \t{fractional} strategies with the
linking pool (the first and third columns) show that fewer instances
of~\eqref{eqn:second-level-milp} and~\eqref{eqn:computeBestUB} are solved in
general (and a smaller percentage of time required) with the \t{fractional}
strategy, regardless of whether the overall solution time is smaller with the
\t{linking} strategy or not.
This does not mean, however, that using the \t{fractional} strategy
always results in the need to solve fewer problems~\eqref{eqn:second-level-milp}
and~\eqref{eqn:computeBestUB}.
It is generally only by taking advantage of the linking solution pool that
this is avoided.
Comparing \t{linking} strategy without the
pool to the \t{fractional} strategy without the pool illustrates that the
\t{fractional} strategy sometimes requires solving more
problems~\eqref{eqn:second-level-milp}. Moreover, 
Table~\ref{tab:numberTimeSolveMIPs} shows that different data sets do not 
have the same behavior from the point of required time for solving MILPs. 
\begin{table}[h!]
\caption{Analysis of total time and frequency of solving 
problems~\eqref{eqn:second-level-milp} and~\eqref{eqn:computeBestUB}}
\label{tab:numberTimeSolveMIPs}
\resizebox{\columnwidth}{!}{
\begin{tabular}{l@{\hspace{0.1cm}}cc@{\hspace{0.15cm}}c@{\hspace{0cm}}c@{\hspace{0cm}}cc@{\hspace{0.15cm}}c@{\hspace{0cm}}c@{\hspace{0cm}}cc@{\hspace{0.15cm}}c@{\hspace{0cm}}c@{\hspace{0cm}}cc@{\hspace{0.15cm}}c}
\hline
\multicolumn{1}{c}{}&
\multicolumn{3}{c}{\begin{tabular}[c]{@{}c@{}}\scriptsize{withPoolWhenXYIntOr}\\\scriptsize{LFixed-LFixed(Linking)}\end{tabular}}&
\multicolumn{1}{c}{}&
\multicolumn{3}{c}{\begin{tabular}[c]{@{}c@{}}\scriptsize{withoutPoolWhenXYIntOr}\\\scriptsize{LFixed-LFixed(Linking)}\end{tabular}}&
\multicolumn{1}{c}{}&
\multicolumn{3}{c}{\begin{tabular}[c]{@{}c@{}}\scriptsize{withPoolWhenXYIntOr}\\\scriptsize{LFixed-LFixed(Fractional)}\end{tabular}}&
\multicolumn{1}{c}{}&
\multicolumn{3}{c}{\begin{tabular}[c]{@{}c@{}}\scriptsize{withoutPoolWhenXYIntOr}\\\scriptsize{LFixed-LFixed(Fractional)}\end{tabular}}\\
\cline{2-4}
\cline{6-8}
\cline{10-12}
\cline{14-16}
Data Set& \begin{tabular}[c]{@{}c@{}}SL\\Count\end{tabular} &\begin{tabular}[c]{@{}c@{}}UB\\Count\end{tabular} &\begin{tabular}[c]{@{}c@{}}Time\\($\%$)\end{tabular} & & \begin{tabular}[c]{@{}c@{}}SL\\Count\end{tabular} &\begin{tabular}[c]{@{}c@{}}UB\\Count\end{tabular} &\begin{tabular}[c]{@{}c@{}}Time\\($\%$)\end{tabular} & & \begin{tabular}[c]{@{}c@{}}SL\\Count\end{tabular} &\begin{tabular}[c]{@{}c@{}}UB\\Count\end{tabular} &\begin{tabular}[c]{@{}c@{}}Time\\($\%$)\end{tabular} & & \begin{tabular}[c]{@{}c@{}}SL\\Count\end{tabular} &\begin{tabular}[c]{@{}c@{}}UB\\Count\end{tabular} &\begin{tabular}[c]{@{}c@{}}Time\\($\%$)\end{tabular}\\
\hline
\begin{tabular}[c]{@{}c@{}}IBLP-DEN \\ ($r_1\leq r_2$)\end{tabular} & 1,027,234 & 1,023,867 & 89 &  & 1,029,644 & 1,023,867 & 89 &  & 203,711 & 48,190 & 4  &  & 5,570,810 & 316,227 & 91 \\[0.3cm]
\begin{tabular}[c]{@{}c@{}}IBLP-DEN \\ ($r_1> r_2$)\end{tabular}  & 1,300,436 & 1,184,094 & 53 &  & 1,317,993 & 1,184,094 & 55 &  & 820,832 & 2,333  & 17 &  & 2,571,199 & 2,355   & 39 \\[0.3cm]
\begin{tabular}[c]{@{}c@{}}IBLP-FIS \\ ($r_1\leq r_2$)\end{tabular} & 699       & 695       & 84 &  & 699       & 695       & 84 &  & 201     & 29     & 0  &  & 1,336,226 & 31,058  & 43 \\[0.3cm]
\begin{tabular}[c]{@{}c@{}}IBLP-FIS \\ ($r_1> r_2$)\end{tabular}   & 91,280    & 56,014    & 13 &  & 93,013    & 56,014    & 13 &  & 19,103  & 0      & 2  &  & 27,980    & 0       & 3  \\[0.3cm]
MIBLP-XU                & 26,593    & 8,791     & 4  &  & 26,597    & 8,795     & 4  &  & 5,972   & 5,618  & 18 &  & 10,654    & 10,239  & 38\\
\hline
\end{tabular}
}
\end{table}
\subsection{Impact of Primal Heuristics}
In order to evaluate the impact of employing primal heuristics implemented in 
\t{MibS} (see Section~\ref{sec:heuristics}), we compared four different methods:
\begin{itemize}
\item \t{noHeuristics}: Parameters are set to the default values obtained from
  the results of previous sections, i.e.,
\begin{itemize}
\item the strategy for solving problems~\eqref{eqn:second-level-milp} 
and~\eqref{eqn:computeBestUB} is \t{whenXYIntOrLFixed-LFixed}.
\item \t{branchStrategy} is set to \t{linking} and \t{fractional} for the
  instances with $r_1\leq r_2$ and $r_1 > r_2$, respectively.
\item the linking solution pool is used.
\end{itemize}
\item\t{impObjectiveCut}: The same as \t{noHeuristics}, but the improving
  objective cut heuristic is turned-on.
\item\t{secondLevelPriority}: The same as \t{noHeuristics}, but the
  second-level priority heuristic is turned-on.
\item\t{weightedSums}: The same as \t{noHeuristics}, but the weighted sums
  heuristic is turned-on.
\end{itemize}
Figure~\ref{fig:heuristicsParTimeMore5Sec} shows the
performance profiles for these four methods with the solution time as the
performance measure. The frequency of using heuristics was set to 100.
\begin{figure}[h!]
\begin{center}
\includegraphics[height=2in]{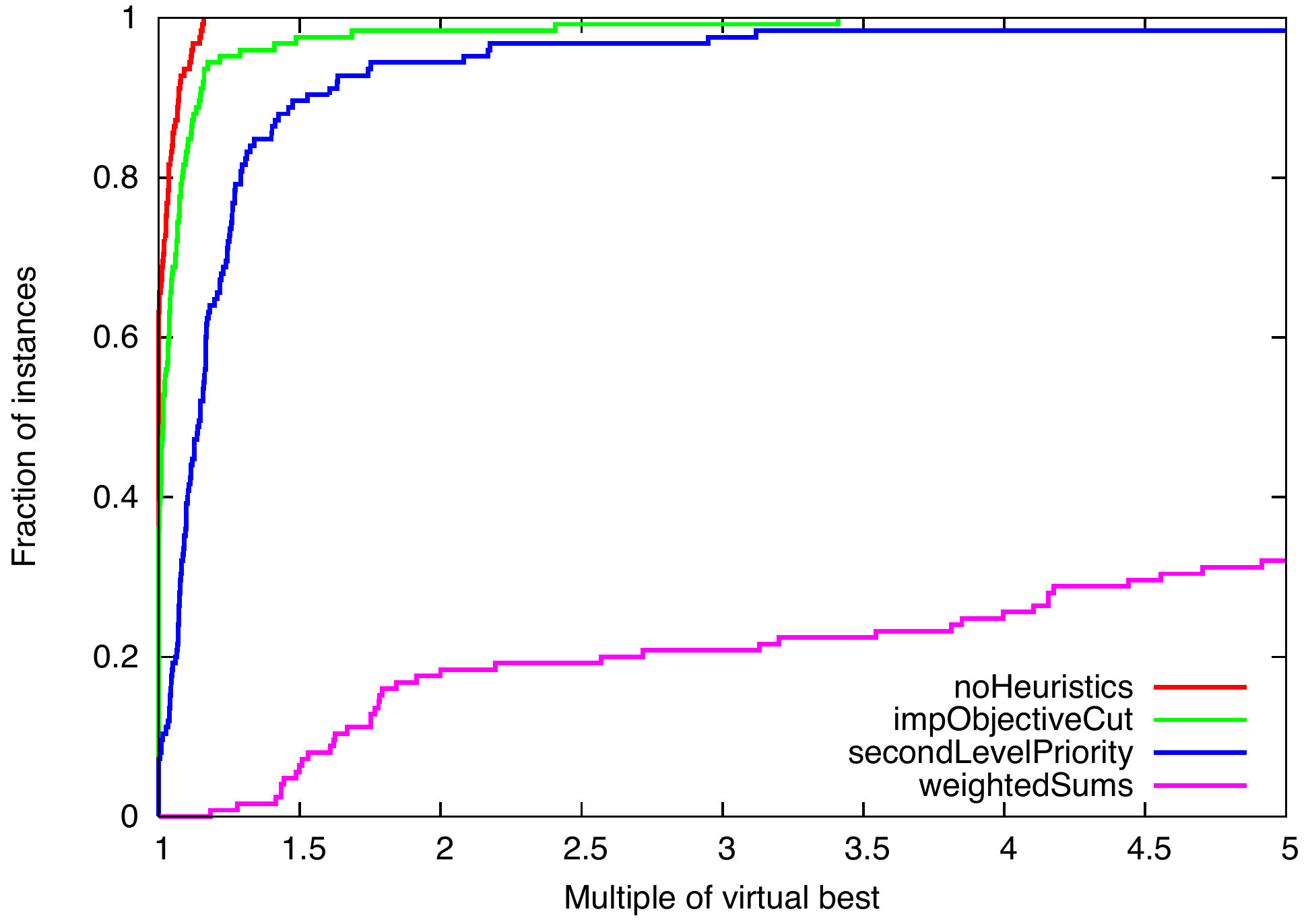}
\end{center}
\caption{Impact of the primal heuristics.
  \label{fig:heuristicsParTimeMore5Sec}}
\end{figure}

Based on Figure~\ref{fig:heuristicsParTimeMore5Sec}, the primal heuristics
implemented so far are not effective in improving the solution time. However,
it may be possible to improve their performance by parameter tuning. No
serious effort has been made so far to do this.
\subsection{Impact of MILP Solver}
SYMPHONY was employed as the MILP solver in all experiments described in
previous sections. In order to investigate the impact of the employed MILP
solver in solving MIBLPs, we repeated some of these experiments using CPLEX
(with its default setting) as the MILP solver. We observed that when the total
time for solving the required MILPs was not large and the MILPs were easy to
solve, the impact of the change in MILP solver was not considerable. In cases
where solution of the MILPs required more effort, CPLEX did reduce the time for
solving the MILPs by roughly half on average (but the exact amount of reduction
depends highly on the instance).

\section{Conclusions \label{sec:conclusions}} 
We have presented a generalized branch-and-cut algorithm, which is able to
solve a wide range of MIBLPs. The components of this algorithm have been
explained in detail and we have discussed precisely how the specific features
of MIBLPs can be utilized to solve various classes of problems that arise in
practical applications. Moreover, we have introduced \MIBS{}, which is an
open-source solver for MIBLPs. This solver has been implemented based on the
branch-and-cut algorithm described in the paper and provides a comprehensive
and flexible framework in which a variety of algorithmic options are
available. We have demonstrated the performance of \MIBS{} and shown that it
is a robust solver capable of solving generic MIBLPs of modest size. In future
papers, we will describe additional details of the methodology in \MIBS{},
including the cut generation, which we have not discussed here. Our future
plans for \MIBS{} include the implementation of additional techniques for
generating valid inequalities and the possible addition of alternative
algorithms.

\section*{Acknowledgements}

This research was made possible with support from National Science
Foundation Grants CMMI-1435453, CMMI-0728011, and ACI-0102687, as well
as Office of Naval Research Grant N000141912330.

\bibliography{paper}
\newpage

\section*{Appendices}

Tables~\ref{tab:fig2a}--\ref{tab:fig5Continue} present detailed 
results applied for plotting Figures~\ref{fig:SSUBParTimeMore5Sec}
--\ref{fig:heuristicsParTimeMore5Sec}.
\mymanagetable{
%%%%%%%%%%%%%%%%%%
%table for figure 2a
\begin{table}[h!]
\caption{Detailed results of Figure~\ref{fig:SSUBParTimeMore5SecNonXu}}
\label{tab:fig2a}
\resizebox{\columnwidth}{!}{
%\begin{tabular}{lccccccccccccccccccc}
\begin{tabular}{llllllllllllllllllll}
\hline
\multicolumn{1}{c}{}&
\multicolumn{3}{c}{whenLInt-LInt}&
\multicolumn{1}{c}{}&
\multicolumn{3}{c}{whenLInt-LFixed}&
\multicolumn{1}{c}{}&
\multicolumn{3}{c}{whenLFixed-LFixed}&
\multicolumn{1}{c}{}&
\multicolumn{3}{c}{whenXYInt-LFixed}&
\multicolumn{1}{c}{}&
\multicolumn{3}{c}{\begin{tabular}[c]{@{}c@{}}\scriptsize{whenXYIntOr}\\\scriptsize{LFixed-LFixed}\end{tabular}} \\
\cline{2-4}
\cline{6-8}
\cline{10-12}
\cline{14-16}
\cline{18-20}
Instance&   BestSol & Time(s) &Nodes  & &BestSol & Time(s) &Nodes & &BestSol & Time(s) &Nodes& &BestSol & Time(s) &Nodes& &BestSol & Time(s) &Nodes\\
\hline
miblp-20-15-50-0110-10-10 & -206.00  & 1.82               & 423     &  & -206.00  & 1.08    & 423     &  & -206.00  & 0.97    & 1414     &  & -206.00  & 1.06    & 423     &  & -206.00  & 1.70               & 423     \\
miblp-20-15-50-0110-10-2  & -398.00  & 12.25              & 2590    &  & -398.00  & 9.18    & 2450    &  & -398.00  & 11.12   & 27625    &  & -398.00  & 9.09    & 2450    &  & -398.00  & 11.52              & 2590    \\
miblp-20-15-50-0110-10-3  & -42.00   & 0.74               & 267     &  & -42.00   & 0.61    & 267     &  & -42.00   & 0.71    & 1343     &  & -42.00   & 0.62    & 267     &  & -42.00   & 0.70               & 267     \\
miblp-20-15-50-0110-10-6  & -246.00  & 12.16              & 340     &  & -246.00  & 4.47    & 340     &  & -246.00  & 2.19    & 853      &  & -246.00  & 4.56    & 340     &  & -246.00  & 8.67               & 340     \\
miblp-20-15-50-0110-10-9  & -635.00  & 26.62              & 2387    &  & -635.00  & 14.06   & 2380    &  & -635.00  & 7.86    & 6457     &  & -635.00  & 14.05   & 2380    &  & -635.00  & 25.35              & 2387    \\
miblp-20-20-50-0110-10-10 & -441.00  & 1322.97            & 134901  &  & -441.00  & 684.67  & 134585  &  & -441.00  & 526.56  & 639146   &  & -441.00  & 692.00  & 134585  &  & -441.00  & 1121.84            & 134901  \\
miblp-20-20-50-0110-10-1  & -359.00  & 264.46             & 96136   &  & -359.00  & 244.20  & 94019   &  & -359.00  & 260.31  & 423149   &  & -359.00  & 228.67  & 96141   &  & -359.00  & 248.95             & 96136   \\
miblp-20-20-50-0110-10-2  & -659.00  & 12.72              & 3261    &  & -659.00  & 3.93    & 3191    &  & -659.00  & 2.78    & 7547     &  & -659.00  & 3.69    & 3191    &  & -659.00  & 10.56              & 3261    \\
miblp-20-20-50-0110-10-3  & -618.00  & 40.06              & 21622   &  & -618.00  & 13.04   & 20780   &  & -618.00  & 10.74   & 38188    &  & -618.00  & 13.18   & 20788   &  & -618.00  & 29.85              & 21622   \\
miblp-20-20-50-0110-10-4  & -604.00  & \textgreater{}3600 & 654402  &  & -604.00  & 3127.12 & 830908  &  & -604.00  & 3405.57 & 6452671  &  & -604.00  & 3145.16 & 830808  &  & -604.00  & \textgreater{}3600 & 714877  \\
miblp-20-20-50-0110-10-7  & -683.00  & 3233.20            & 3069661 &  & -683.00  & 2046.70 & 2978073 &  & -683.00  & 3184.99 & 11502091 &  & -683.00  & 1887.37 & 3003967 &  & -683.00  & 2511.36            & 3069661 \\
miblp-20-20-50-0110-10-8  & -667.00  & 182.80             & 12868   &  & -667.00  & 87.14   & 12856   &  & -667.00  & 40.78   & 30873    &  & -667.00  & 80.45   & 12857   &  & -667.00  & 148.40             & 12868   \\
miblp-20-20-50-0110-10-9  & -256.00  & 45.39              & 35295   &  & -256.00  & 21.06   & 31244   &  & -256.00  & 23.97   & 76055    &  & -256.00  & 20.99   & 31245   &  & -256.00  & 33.92              & 35295   \\
miblp-20-20-50-0110-15-1  & -450.00  & 60.74              & 3516    &  & -450.00  & 60.84   & 3506    &  & -450.00  & 65.16   & 49137    &  & -450.00  & 59.38   & 3506    &  & -450.00  & 60.89              & 3516    \\
miblp-20-20-50-0110-15-2  & -645.00  & 73.76              & 17251   &  & -645.00  & 50.01   & 17251   &  & -645.00  & 96.13   & 346065   &  & -645.00  & 50.14   & 17251   &  & -645.00  & 63.85              & 17251   \\
miblp-20-20-50-0110-15-3  & -593.00  & 100.53             & 3083    &  & -593.00  & 70.36   & 3081    &  & -593.00  & 65.25   & 42877    &  & -593.00  & 71.15   & 3081    &  & -593.00  & 92.89              & 3083    \\
miblp-20-20-50-0110-15-4  & -441.00  & 66.29              & 1625    &  & -441.00  & 42.88   & 1625    &  & -441.00  & 36.99   & 29904    &  & -441.00  & 43.00   & 1625    &  & -441.00  & 55.23              & 1625    \\
miblp-20-20-50-0110-15-5  & -379.00  & 860.42             & 16715   &  & -379.00  & 632.14  & 16715   &  & -379.00  & 615.16  & 205025   &  & -379.00  & 651.28  & 16715   &  & -379.00  & 730.86             & 16715   \\
miblp-20-20-50-0110-15-6  & -596.00  & 23.24              & 1657    &  & -596.00  & 17.26   & 1657    &  & -596.00  & 17.87   & 29923    &  & -596.00  & 17.31   & 1657    &  & -596.00  & 22.10              & 1657    \\
miblp-20-20-50-0110-15-7  & -471.00  & 133.26             & 13405   &  & -471.00  & 110.80  & 13405   &  & -471.00  & 125.24  & 241285   &  & -471.00  & 111.02  & 13405   &  & -471.00  & 127.42             & 13405   \\
miblp-20-20-50-0110-15-8  & -370.00  & 41.05              & 21589   &  & -370.00  & 39.51   & 21589   &  & -370.00  & 138.34  & 579309   &  & -370.00  & 39.12   & 21589   &  & -370.00  & 40.56              & 21589   \\
miblp-20-20-50-0110-15-9  & -584.00  & 2.58               & 588     &  & -584.00  & 2.03    & 582     &  & -584.00  & 1.96    & 4072     &  & -584.00  & 2.00    & 582     &  & -584.00  & 2.34               & 588     \\
miblp-20-20-50-0110-5-13  & -519.00  & \textgreater{}3600 & 4221171 &  & -519.00  & 2002.06 & 4392628 &  & -519.00  & 2307.90 & 10196998 &  & -519.00  & 2464.25 & 7128138 &  & -519.00  & \textgreater{}3600 & 5830754 \\
miblp-20-20-50-0110-5-15  & -617.00  & \textgreater{}3600 & 3325012 &  & -617.00  & 1122.66 & 2749914 &  & -617.00  & 1219.33 & 6630921  &  & -617.00  & 1310.67 & 4018058 &  & -617.00  & 2953.78            & 4018081 \\
miblp-20-20-50-0110-5-16  & -833.00  & 44.15              & 18278   &  & -833.00  & 6.48    & 17680   &  & -833.00  & 4.86    & 19013    &  & -833.00  & 6.60    & 17680   &  & -833.00  & 38.30              & 18278   \\
miblp-20-20-50-0110-5-17  & -944.00  & 16.80              & 18200   &  & -944.00  & 4.62    & 17712   &  & -944.00  & 3.99    & 21541    &  & -944.00  & 4.78    & 17679   &  & -944.00  & 10.64              & 18200   \\
miblp-20-20-50-0110-5-19  & -431.00  & 104.93             & 79279   &  & -431.00  & 26.75   & 77256   &  & -431.00  & 27.52   & 147268   &  & -431.00  & 26.56   & 77256   &  & -431.00  & 65.55              & 79279   \\
miblp-20-20-50-0110-5-1   & -548.00  & 52.43              & 42629   &  & -548.00  & 14.31   & 42364   &  & -548.00  & 12.98   & 64067    &  & -548.00  & 14.50   & 42364   &  & -548.00  & 36.55              & 42629   \\
miblp-20-20-50-0110-5-20  & -438.00  & 48.38              & 60990   &  & -438.00  & 14.95   & 50531   &  & -438.00  & 15.81   & 76091    &  & -438.00  & 17.30   & 60944   &  & -438.00  & 31.77              & 60990   \\
miblp-20-20-50-0110-5-6   & -1061.00 & 202.25             & 224425  &  & -1061.00 & 62.75   & 223494  &  & -1061.00 & 58.74   & 284550   &  & -1061.00 & 63.70   & 222566  &  & -1061.00 & 127.84             & 224425  \\
lseu-0.100000             & 1120.00  & 657.06             & 1132617 &  & 1120.00  & 235.84  & 734286  &  & 1120.00  & 248.54  & 1071409  &  & 1120.00  & 270.78  & 1003976 &  & 1120.00  & 626.90             & 1132617 \\
lseu-0.900000             & 5838.00  & 13.91              & 1023    &  & 5838.00  & 14.41   & 1023    &  & 5838.00  & 1063.02 & 4718749  &  & 5838.00  & 14.10   & 1023    &  & 5838.00  & 13.83              & 1023    \\
p0033-0.500000            & 3095.00  & 12.65              & 20855   &  & 3095.00  & 10.44   & 20695   &  & 3095.00  & 6.24    & 33614    &  & 3095.00  & 10.48   & 20695   &  & 3095.00  & 11.55              & 20855   \\
p0033-0.900000            & 4679.00  & 0.06               & 27      &  & 4679.00  & 0.05    & 27      &  & 4679.00  & 0.65    & 3455     &  & 4679.00  & 0.06    & 27      &  & 4679.00  & 0.04               & 27      \\
p0201-0.900000            & 15025.00 & 7.05               & 2481    &  & 15025.00 & 6.66    & 2481    &  & 15025.00 & 20.58   & 18801    &  & 15025.00 & 6.62    & 2481    &  & 15025.00 & 6.68               & 2481    \\
stein27-0.500000          & 19.00    & 6.21               & 12115   &  & 19.00    & 6.24    & 14362   &  & 19.00    & 7.36    & 21515    &  & 19.00    & 5.88    & 12115   &  & 19.00    & 5.91               & 12115   \\
stein27-0.900000          & 24.00    & 0.01               & 15      &  & 24.00    & 0.02    & 15      &  & 24.00    & 1.25    & 4445     &  & 24.00    & 0.02    & 15      &  & 24.00    & 0.02               & 15      \\
stein45-0.100000          & 30.00    & 51.06              & 89035   &  & 30.00    & 96.92   & 61518   &  & 30.00    & 50.47   & 89035    &  & 30.00    & 50.19   & 89035   &  & 30.00    & 50.17              & 89035   \\
stein45-0.500000          & 32.00    & 554.15             & 640308  &  & 32.00    & 1088.96 & 1014908 &  & 32.00    & 635.22  & 952123   &  & 32.00    & 519.59  & 640308  &  & 32.00    & 520.73             & 640308  \\
stein45-0.900000          & 40.00    & 0.16               & 63      &  & 40.00    & 0.16    & 63      &  & 40.00    & 85.92   & 103661   &  & 40.00    & 0.14    & 63      &  & 40.00    & 0.15               & 63     \\       
\hline
\end{tabular}
}
\end{table}
%%%%%%%%%%%%%%%%%%
%table for figure 2b
\begin{table}[p]
\begin{center}
\caption{Detailed results of Figure~\ref{fig:SSUBParTimeMore5SecXu}}
\label{tab:fig2b}
\resizebox{\columnwidth}{!}{
%\begin{tabular}{lccccccccccccccccccc}
\begin{tabular}{llllllllllllllllllll}
\hline
\multicolumn{1}{c}{}&
\multicolumn{3}{c}{whenLInt-LInt}&
\multicolumn{1}{c}{}&
\multicolumn{3}{c}{whenLInt-LFixed}&
\multicolumn{1}{c}{}&
\multicolumn{3}{c}{whenLFixed-LFixed}&
\multicolumn{1}{c}{}&
\multicolumn{3}{c}{whenXYInt-LFixed}&
\multicolumn{1}{c}{}&
\multicolumn{3}{c}{\begin{tabular}[c]{@{}c@{}}\scriptsize{whenXYIntOr}\\\scriptsize{LFixed-LFixed}\end{tabular}} \\
\cline{2-4}
\cline{6-8}
\cline{10-12}
\cline{14-16}
\cline{18-20}
Instance&   BestSol & Time(s) &Nodes  & &BestSol & Time(s) &Nodes & &BestSol & Time(s) &Nodes& &BestSol & Time(s) &Nodes& &BestSol & Time(s) &Nodes\\
\hline
bmilplib-110-10 & -177.67 & 336.74  & 75952  &  & -177.67 & 150.88  & 75141  &  & -177.67 & 149.10  & 93801  &  & -177.67 & 153.82  & 75499  &  & -177.67 & 223.95  & 75952  \\
bmilplib-110-1  & -181.67 & 9.23    & 2806   &  & -181.67 & 5.35    & 3629   &  & -181.67 & 5.42    & 2878   &  & -181.67 & 5.35    & 2806   &  & -181.67 & 6.34    & 2806   \\
bmilplib-110-2  & -110.67 & 9.16    & 4183   &  & -110.67 & 6.53    & 4439   &  & -110.67 & 6.50    & 4272   &  & -110.67 & 6.48    & 4174   &  & -110.67 & 7.46    & 4183   \\
bmilplib-110-3  & -215.16 & 10.15   & 3484   &  & -215.16 & 6.07    & 3529   &  & -215.16 & 5.42    & 3556   &  & -215.16 & 5.58    & 3471   &  & -215.16 & 7.13    & 3484   \\
bmilplib-110-4  & -197.29 & 16.30   & 1917   &  & -197.29 & 3.66    & 1740   &  & -197.29 & 3.54    & 1801   &  & -197.29 & 3.69    & 1751   &  & -197.29 & 8.20    & 1917   \\
bmilplib-110-6  & -148.25 & 25.96   & 8857   &  & -148.25 & 14.57   & 8735   &  & -148.25 & 14.14   & 9392   &  & -148.25 & 15.06   & 8770   &  & -148.25 & 19.35   & 8857   \\
bmilplib-110-7  & -160.86 & 11.92   & 2221   &  & -160.86 & 4.66    & 2167   &  & -160.86 & 4.01    & 2281   &  & -160.86 & 4.94    & 2189   &  & -160.86 & 7.45    & 2221   \\
bmilplib-110-8  & -155.00 & 39.85   & 11835  &  & -155.00 & 20.66   & 11916  &  & -155.00 & 20.65   & 12498  &  & -155.00 & 20.78   & 11663  &  & -155.00 & 27.93   & 11835  \\
bmilplib-110-9  & -192.92 & 9.14    & 2155   &  & -192.92 & 3.46    & 2099   &  & -192.92 & 3.28    & 2171   &  & -192.92 & 3.66    & 2151   &  & -192.92 & 5.84    & 2155   \\
bmilplib-160-10 & -189.82 & 87.08   & 9130   &  & -189.82 & 30.70   & 8676   &  & -189.82 & 31.24   & 9344   &  & -189.82 & 33.42   & 9071   &  & -189.82 & 55.24   & 9130   \\
bmilplib-160-1  & -165.00 & 63.76   & 7908   &  & -165.00 & 25.00   & 8028   &  & -165.00 & 24.18   & 8145   &  & -165.00 & 25.70   & 7901   &  & -165.00 & 38.00   & 7908   \\
bmilplib-160-2  & -178.24 & 87.14   & 8680   &  & -178.24 & 29.51   & 8612   &  & -178.24 & 29.64   & 8710   &  & -178.24 & 30.59   & 8635   &  & -178.24 & 47.52   & 8680   \\
bmilplib-160-3  & -174.94 & 196.12  & 15164  &  & -174.94 & 62.78   & 14998  &  & -174.94 & 54.62   & 15427  &  & -174.94 & 63.37   & 15002  &  & -174.94 & 104.06  & 15164  \\
bmilplib-160-4  & -135.83 & 77.00   & 15005  &  & -135.83 & 48.60   & 14708  &  & -135.83 & 50.86   & 15354  &  & -135.83 & 51.02   & 14772  &  & -135.83 & 60.24   & 15005  \\
bmilplib-160-5  & -140.78 & 62.61   & 4981   &  & -140.78 & 19.78   & 4842   &  & -140.78 & 19.35   & 5178   &  & -140.78 & 20.52   & 4948   &  & -140.78 & 36.78   & 4981   \\
bmilplib-160-6  & -111.00 & 32.32   & 8685   &  & -111.00 & 27.50   & 10149  &  & -111.00 & 25.08   & 8864   &  & -111.00 & 24.66   & 8676   &  & -111.00 & 27.50   & 8685   \\
bmilplib-160-7  & -96.00  & 65.43   & 16589  &  & -96.00  & 51.07   & 16804  &  & -96.00  & 50.97   & 17430  &  & -96.00  & 51.20   & 16560  &  & -96.00  & 58.05   & 16589  \\
bmilplib-160-8  & -181.40 & 22.24   & 3464   &  & -181.40 & 12.47   & 3417   &  & -181.40 & 9.32    & 3540   &  & -181.40 & 10.49   & 3444   &  & -181.40 & 14.93   & 3464   \\
bmilplib-160-9  & -207.50 & 34.05   & 4715   &  & -207.50 & 16.51   & 5169   &  & -207.50 & 14.90   & 4728   &  & -207.50 & 15.85   & 4607   &  & -207.50 & 23.34   & 4715   \\
bmilplib-210-10 & -130.59 & 84.71   & 12226  &  & -130.59 & 58.89   & 12010  &  & -130.59 & 61.00   & 12497  &  & -130.59 & 61.81   & 12176  &  & -130.59 & 72.22   & 12226  \\
bmilplib-210-1  & -136.80 & 77.67   & 7429   &  & -136.80 & 37.71   & 7182   &  & -136.80 & 40.04   & 7598   &  & -136.80 & 41.96   & 7429   &  & -136.80 & 56.23   & 7429   \\
bmilplib-210-2  & -117.80 & 153.16  & 22441  &  & -117.80 & 111.87  & 22597  &  & -117.80 & 109.56  & 23096  &  & -117.80 & 113.38  & 22294  &  & -117.80 & 131.02  & 22441  \\
bmilplib-210-3  & -130.80 & 74.50   & 8897   &  & -130.80 & 44.80   & 8523   &  & -130.80 & 47.44   & 9165   &  & -130.80 & 48.19   & 8848   &  & -130.80 & 59.95   & 8897   \\
bmilplib-210-4  & -162.20 & 62.23   & 4738   &  & -162.20 & 36.19   & 5094   &  & -162.20 & 25.92   & 4806   &  & -162.20 & 26.81   & 4687   &  & -162.20 & 41.52   & 4738   \\
bmilplib-210-5  & -134.00 & 156.90  & 21193  &  & -134.00 & 106.25  & 21560  &  & -134.00 & 105.03  & 21552  &  & -134.00 & 108.22  & 21058  &  & -134.00 & 130.51  & 21193  \\
bmilplib-210-6  & -125.43 & 269.65  & 38538  &  & -125.43 & 198.77  & 38367  &  & -125.43 & 197.88  & 39941  &  & -125.43 & 201.40  & 38443  &  & -125.43 & 227.81  & 38538  \\
bmilplib-210-7  & -169.73 & 146.39  & 13960  &  & -169.73 & 76.42   & 13960  &  & -169.73 & 77.62   & 14305  &  & -169.73 & 77.84   & 13960  &  & -169.73 & 104.02  & 13960  \\
bmilplib-210-8  & -101.46 & 70.64   & 11105  &  & -101.46 & 56.97   & 11324  &  & -101.46 & 58.14   & 11347  &  & -101.46 & 61.79   & 11105  &  & -101.46 & 62.56   & 11105  \\
bmilplib-210-9  & -184.00 & 1822.79 & 143571 &  & -184.00 & 790.46  & 144603 &  & -184.00 & 859.46  & 143665 &  & -184.00 & 879.61  & 142294 &  & -184.00 & 1184.18 & 143571 \\
bmilplib-260-10 & -151.73 & 785.34  & 62063  &  & -151.73 & 503.08  & 62143  &  & -151.73 & 502.94  & 63026  &  & -151.73 & 546.55  & 61767  &  & -151.73 & 633.15  & 62063  \\
bmilplib-260-1  & -139.00 & 155.24  & 10994  &  & -139.00 & 83.01   & 11369  &  & -139.00 & 81.74   & 11242  &  & -139.00 & 85.81   & 10980  &  & -139.00 & 119.94  & 10994  \\
bmilplib-260-2  & -82.62  & 117.18  & 15300  &  & -82.62  & 156.75  & 15292  &  & -82.62  & 109.75  & 15712  &  & -82.62  & 109.75  & 15287  &  & -82.62  & 114.50  & 15300  \\
bmilplib-260-3  & -144.25 & 148.68  & 8777   &  & -144.25 & 72.67   & 9563   &  & -144.25 & 69.26   & 8859   &  & -144.25 & 71.54   & 8769   &  & -144.25 & 94.43   & 8777   \\
bmilplib-260-4  & -117.33 & 294.39  & 32964  &  & -117.33 & 245.58  & 32696  &  & -117.33 & 260.01  & 33931  &  & -117.33 & 259.48  & 32787  &  & -117.33 & 272.70  & 32964  \\
bmilplib-260-5  & -121.00 & 201.31  & 21557  &  & -121.00 & 167.30  & 22480  &  & -121.00 & 173.32  & 22121  &  & -121.00 & 166.03  & 21542  &  & -121.00 & 183.46  & 21557  \\
bmilplib-260-6  & -124.00 & 258.94  & 25420  &  & -124.00 & 191.68  & 25826  &  & -124.00 & 209.79  & 26023  &  & -124.00 & 197.04  & 25362  &  & -124.00 & 223.31  & 25420  \\
bmilplib-260-7  & -137.80 & 472.66  & 40528  &  & -137.80 & 304.36  & 40400  &  & -137.80 & 318.76  & 41180  &  & -137.80 & 312.12  & 40274  &  & -137.80 & 383.16  & 40528  \\
bmilplib-260-8  & -119.89 & 112.17  & 10094  &  & -119.89 & 73.72   & 9961   &  & -119.89 & 76.70   & 10332  &  & -119.89 & 87.01   & 10025  &  & -119.89 & 96.85   & 10094  \\
bmilplib-260-9  & -160.00 & 510.37  & 33496  &  & -160.00 & 257.08  & 33928  &  & -160.00 & 271.18  & 34236  &  & -160.00 & 273.66  & 33468  &  & -160.00 & 394.98  & 33496  \\
bmilplib-310-10 & -141.86 & 129.77  & 9904   &  & -141.86 & 110.26  & 9896   &  & -141.86 & 114.38  & 10033  &  & -141.86 & 121.07  & 9900   &  & -141.86 & 127.87  & 9904   \\
bmilplib-310-1  & -117.00 & 521.12  & 28360  &  & -117.00 & 301.34  & 28373  &  & -117.00 & 329.80  & 29055  &  & -117.00 & 329.71  & 28330  &  & -117.00 & 372.68  & 28360  \\
bmilplib-310-2  & -105.00 & 523.84  & 43448  &  & -105.00 & 457.17  & 44736  &  & -105.00 & 457.69  & 44841  &  & -105.00 & 497.66  & 43399  &  & -105.00 & 539.20  & 43448  \\
bmilplib-310-3  & -127.52 & 943.53  & 66426  &  & -127.52 & 907.46  & 72139  &  & -127.52 & 783.89  & 67441  &  & -127.52 & 777.74  & 66410  &  & -127.52 & 916.69  & 66426  \\
bmilplib-310-4  & -147.78 & 701.76  & 47718  &  & -147.78 & 510.81  & 48131  &  & -147.78 & 572.76  & 48950  &  & -147.78 & 569.43  & 47711  &  & -147.78 & 652.54  & 47718  \\
bmilplib-310-5  & -161.45 & 557.76  & 31838  &  & -161.45 & 339.05  & 31553  &  & -161.45 & 382.13  & 32454  &  & -161.45 & 366.90  & 31782  &  & -161.45 & 488.65  & 31838  \\
bmilplib-310-6  & -141.18 & 1448.71 & 102047 &  & -141.18 & 1159.42 & 110836 &  & -141.18 & 1264.60 & 103214 &  & -141.18 & 1191.51 & 101904 &  & -141.18 & 1324.84 & 102047 \\
bmilplib-310-7  & -142.00 & 1443.70 & 102030 &  & -142.00 & 1057.98 & 104876 &  & -142.00 & 1199.47 & 104166 &  & -142.00 & 1129.24 & 102000 &  & -142.00 & 1396.34 & 102030 \\
bmilplib-310-8  & -115.34 & 143.13  & 11375  &  & -115.34 & 110.74  & 11067  &  & -115.34 & 114.80  & 11293  &  & -115.34 & 127.94  & 11109  &  & -115.34 & 142.93  & 11375  \\
bmilplib-310-9  & -115.65 & 423.77  & 20552  &  & -115.65 & 231.20  & 21846  &  & -115.65 & 248.24  & 20838  &  & -115.65 & 255.27  & 20490  &  & -115.65 & 281.90  & 20552  \\
bmilplib-360-10 & -108.59 & 257.34  & 13727  &  & -108.59 & 209.20  & 13671  &  & -108.59 & 242.94  & 14064  &  & -108.59 & 234.92  & 13697  &  & -108.59 & 258.65  & 13727  \\
bmilplib-360-1  & -133.00 & 2416.81 & 75780  &  & -133.00 & 1239.50 & 77206  &  & -133.00 & 1297.32 & 77066  &  & -133.00 & 1353.62 & 75421  &  & -133.00 & 1500.96 & 75780  \\
bmilplib-360-2  & -138.44 & 1187.29 & 53004  &  & -138.44 & 829.40  & 52868  &  & -138.44 & 999.50  & 54354  &  & -138.44 & 965.52  & 52919  &  & -138.44 & 1169.80 & 53004  \\
bmilplib-360-3  & -131.00 & 834.51  & 40671  &  & -131.00 & 728.80  & 40353  &  & -131.00 & 624.78  & 41302  &  & -131.00 & 832.88  & 40487  &  & -131.00 & 731.63  & 40671  \\
bmilplib-360-4  & -119.00 & 371.93  & 18350  &  & -119.00 & 286.84  & 18870  &  & -119.00 & 338.49  & 18813  &  & -119.00 & 332.90  & 18293  &  & -119.00 & 361.79  & 18350  \\
bmilplib-360-5  & -164.26 & 593.68  & 30001  &  & -164.26 & 420.56  & 29790  &  & -164.26 & 484.66  & 30352  &  & -164.26 & 618.86  & 29947  &  & -164.26 & 600.63  & 30001  \\
bmilplib-360-6  & -110.12 & 1169.00 & 68863  &  & -110.12 & 1005.68 & 69453  &  & -110.12 & 1018.82 & 70283  &  & -110.12 & 1181.02 & 68863  &  & -110.12 & 1210.13 & 68863  \\
bmilplib-360-7  & -105.00 & 538.80  & 31092  &  & -105.00 & 457.98  & 32124  &  & -105.00 & 517.46  & 31884  &  & -105.00 & 634.17  & 30900  &  & -105.00 & 542.81  & 31092  \\
bmilplib-360-8  & -98.25  & 399.32  & 22995  &  & -98.25  & 343.65  & 22787  &  & -98.25  & 416.97  & 23337  &  & -98.25  & 362.89  & 22857  &  & -98.25  & 386.45  & 22995  \\
bmilplib-360-9  & -127.22 & 815.44  & 40383  &  & -127.22 & 622.13  & 40733  &  & -127.22 & 746.66  & 41235  &  & -127.22 & 736.16  & 40329  &  & -127.22 & 819.68  & 40383  \\
bmilplib-410-10 & -153.37 & 3527.64 & 101673 &  & -153.37 & 2725.83 & 103359 &  & -153.37 & 2879.03 & 103400 &  & -153.37 & 2729.77 & 101447 &  & -153.37 & 2976.52 & 101673 \\
bmilplib-410-1  & -103.50 & 582.60  & 20790  &  & -103.50 & 535.58  & 20612  &  & -103.50 & 589.62  & 21088  &  & -103.50 & 553.76  & 20634  &  & -103.50 & 534.12  & 20790  \\
bmilplib-410-2  & -108.59 & 803.44  & 31603  &  & -108.59 & 734.82  & 31602  &  & -108.59 & 748.06  & 32251  &  & -108.59 & 840.56  & 31603  &  & -108.59 & 882.69  & 31603  \\
bmilplib-410-3  & -96.24  & 1275.98 & 33041  &  & -96.24  & 791.92  & 33963  &  & -96.24  & 777.04  & 33473  &  & -96.24  & 871.58  & 32882  &  & -96.24  & 973.12  & 33041  \\
bmilplib-410-4  & -119.50 & 1582.86 & 38489  &  & -119.50 & 930.82  & 39270  &  & -119.50 & 1041.97 & 39371  &  & -119.50 & 1031.12 & 38489  &  & -119.50 & 1101.03 & 38489  \\
bmilplib-410-5  & -119.22 & 700.45  & 23865  &  & -119.22 & 567.85  & 23747  &  & -119.22 & 681.99  & 24163  &  & -119.22 & 678.49  & 23852  &  & -119.22 & 696.17  & 23865  \\
bmilplib-410-6  & -151.31 & 404.52  & 11130  &  & -151.31 & 261.83  & 10795  &  & -151.31 & 333.36  & 11193  &  & -151.31 & 322.88  & 11130  &  & -151.31 & 337.71  & 11130  \\
bmilplib-410-7  & -123.00 & 624.95  & 22208  &  & -123.00 & 519.43  & 22843  &  & -123.00 & 601.78  & 22628  &  & -123.00 & 560.34  & 22146  &  & -123.00 & 560.38  & 22208  \\
bmilplib-410-8  & -125.78 & 1717.64 & 62336  &  & -125.78 & 1484.95 & 61784  &  & -125.78 & 1692.35 & 63421  &  & -125.78 & 1547.88 & 62218  &  & -125.78 & 1579.12 & 62336  \\
bmilplib-410-9  & -100.77 & 505.02  & 20424  &  & -100.77 & 579.73  & 20095  &  & -100.77 & 607.44  & 20920  &  & -100.77 & 531.88  & 20400  &  & -100.77 & 494.06  & 20424  \\
bmilplib-460-10 & -102.51 & 1949.07 & 55030  &  & -102.51 & 1782.80 & 55882  &  & -102.51 & 1701.70 & 56265  &  & -102.51 & 1857.94 & 55024  &  & -102.51 & 1822.52 & 55030  \\
bmilplib-460-1  & -97.59  & 2961.39 & 86689  &  & -97.59  & 3575.71 & 93301  &  & -97.59  & 2938.02 & 87709  &  & -97.59  & 2822.97 & 86632  &  & -97.59  & 2915.62 & 86689  \\
bmilplib-460-2  & -139.00 & 739.75  & 16720  &  & -139.00 & 527.76  & 16721  &  & -139.00 & 608.40  & 16863  &  & -139.00 & 625.95  & 16623  &  & -139.00 & 694.03  & 16720  \\
bmilplib-460-3  & -86.50  & 1921.88 & 57462  &  & -86.50  & 1853.40 & 59065  &  & -86.50  & 2102.77 & 58229  &  & -86.50  & 1995.92 & 57395  &  & -86.50  & 1946.46 & 57462  \\
bmilplib-460-4  & -107.03 & 3321.16 & 95231  &  & -107.03 & 3173.50 & 94899  &  & -107.03 & 3328.91 & 97160  &  & -107.03 & 3291.28 & 95035  &  & -107.03 & 3301.27 & 95231  \\
bmilplib-460-5  & -100.50 & 1407.70 & 41312  &  & -100.50 & 1366.52 & 40903  &  & -100.50 & 1619.91 & 41796  &  & -100.50 & 1424.54 & 41109  &  & -100.50 & 1450.14 & 41312  \\
bmilplib-460-6  & -107.00 & 1623.54 & 46158  &  & -107.00 & 1537.79 & 48966  &  & -107.00 & 1483.17 & 46732  &  & -107.00 & 1635.42 & 46076  &  & -107.00 & 1598.68 & 46158  \\
bmilplib-460-7  & -83.75  & 1735.67 & 48973  &  & -83.75  & 1606.84 & 49516  &  & -83.75  & 1728.92 & 49717  &  & -83.75  & 1601.34 & 48897  &  & -83.75  & 1671.84 & 48973  \\
bmilplib-460-8  & -115.39 & 1074.21 & 27663  &  & -115.39 & 858.20  & 27275  &  & -115.39 & 911.51  & 28075  &  & -115.39 & 982.30  & 27572  &  & -115.39 & 1017.28 & 27663  \\
bmilplib-460-9  & -128.70 & 3410.18 & 86426  &  & -128.70 & 2671.50 & 86656  &  & -128.70 & 2715.98 & 87796  &  & -128.70 & 2830.59 & 85977  &  & -128.70 & 3004.49 & 86426  \\
bmilplib-60-10  & -186.21 & 14.85   & 5929   &  & -186.21 & 6.88    & 9593   &  & -186.21 & 6.91    & 6134   &  & -186.21 & 6.93    & 5922   &  & -186.21 & 9.50    & 5929   \\
bmilplib-60-1   & -153.20 & 16.56   & 4896   &  & -153.20 & 4.84    & 5366   &  & -153.20 & 6.50    & 5041   &  & -153.20 & 6.71    & 4877   &  & -153.20 & 10.92   & 4896   \\
bmilplib-60-5   & -116.40 & 15.19   & 11289  &  & -116.40 & 7.57    & 11308  &  & -116.40 & 8.65    & 13984  &  & -116.40 & 8.15    & 11202  &  & -116.40 & 11.54   & 11289  \\
bmilplib-60-6   & -187.31 & 15.38   & 7120   &  & -187.31 & 7.08    & 9241   &  & -187.31 & 7.34    & 7304   &  & -187.31 & 7.62    & 7016   &  & -187.31 & 10.44   & 7120   \\
bmilplib-60-8   & -232.12 & 8.38    & 3653   &  & -232.12 & 2.54    & 4052   &  & -232.12 & 3.33    & 3683   &  & -232.12 & 3.36    & 3570   &  & -232.12 & 5.66    & 3653   \\
bmilplib-60-9   & -136.50 & 33.85   & 27036  &  & -136.50 & 18.90   & 31312  &  & -136.50 & 20.22   & 28603  &  & -136.50 & 19.61   & 26792  &  & -136.50 & 26.84   & 27036 \\
\hline
\end{tabular}
}
\end{center}
\end{table}
%%%%%%%%%%%%%%%%%%
%table for figure 3a
\begin{table}[h!]
\caption{Detailed results of Figure~\ref{fig:branchStrategyParTimeMore5Secr1Leqr2}}
\label{tab:fig3a}
\scriptsize
\centering
%\resizebox{\columnwidth}{!}{
%\begin{tabular}{lccccccccc}
\begin{tabular}{llllllllll}
\hline
\multicolumn{1}{c}{}&
\multicolumn{1}{c}{}&
\multicolumn{1}{c}{}&
\multicolumn{3}{c}{linkingBranching}&
\multicolumn{1}{c}{}&
\multicolumn{3}{c}{fractionalBranching} \\
\cline{4-6}
\cline{8-10}
Instance&    $r_1$& $r_2$& BestSol & Time(s) &Nodes& &BestSol & Time(s) &Nodes\\
\hline
miblp-20-15-50-0110-10-10 & 5  & 10  & -206.00  & 1.06    & 423     &  & -206.00  & 4.25               & 15741    \\
miblp-20-15-50-0110-10-2  & 5  & 10  & -398.00  & 9.09    & 2450    &  & -398.00  & 947.81             & 2929420  \\
miblp-20-15-50-0110-10-3  & 5  & 10  & -42.00   & 0.62    & 267     &  & -42.00   & 12.66              & 46349    \\
miblp-20-15-50-0110-10-6  & 5  & 10  & -246.00  & 4.56    & 340     &  & -246.00  & 1.56               & 1367     \\
miblp-20-15-50-0110-10-9  & 5  & 10  & -635.00  & 14.05   & 2380    &  & -635.00  & 6.52               & 7667     \\
miblp-20-20-50-0110-10-10 & 10 & 10  & -441.00  & 692.00  & 134585  &  & -441.00  & 2867.85            & 7261247  \\
miblp-20-20-50-0110-10-1  & 10 & 10  & -359.00  & 228.67  & 96141   &  & -357.00  & \textgreater{}3600 & 9234364  \\
miblp-20-20-50-0110-10-2  & 10 & 10  & -659.00  & 3.69    & 3191    &  & -659.00  & 3.16               & 5100     \\
miblp-20-20-50-0110-10-3  & 10 & 10  & -618.00  & 13.18   & 20788   &  & -618.00  & 12.05              & 51970    \\
miblp-20-20-50-0110-10-4  & 10 & 10  & -604.00  & 3145.16 & 830808  &  & -604.00  & \textgreater{}3600 & 7988914  \\
miblp-20-20-50-0110-10-7  & 10 & 10  & -683.00  & 1887.37 & 3003967 &  & -629.00  & \textgreater{}3600 & 9709672  \\
miblp-20-20-50-0110-10-8  & 10 & 10  & -667.00  & 80.45   & 12857   &  & -667.00  & 68.08              & 75661    \\
miblp-20-20-50-0110-10-9  & 10 & 10  & -256.00  & 20.99   & 31245   &  & -256.00  & 305.78             & 757349   \\
miblp-20-20-50-0110-15-1  & 5  & 15  & -450.00  & 59.38   & 3506    &  & -317.00  & \textgreater{}3600 & 9813005  \\
miblp-20-20-50-0110-15-2  & 5  & 15  & -645.00  & 50.14   & 17251   &  & -645.00  & \textgreater{}3600 & 10839884 \\
miblp-20-20-50-0110-15-3  & 5  & 15  & -593.00  & 71.15   & 3081    &  & -593.00  & \textgreater{}3600 & 13042109 \\
miblp-20-20-50-0110-15-4  & 5  & 15  & -441.00  & 43.00   & 1625    &  & -398.00  & \textgreater{}3600 & 8692428  \\
miblp-20-20-50-0110-15-5  & 5  & 15  & -379.00  & 651.28  & 16715   &  & -320.00  & \textgreater{}3600 & 7284040  \\
miblp-20-20-50-0110-15-6  & 5  & 15  & -596.00  & 17.31   & 1657    &  & -596.00  & \textgreater{}3600 & 7851818  \\
miblp-20-20-50-0110-15-7  & 5  & 15  & -471.00  & 111.02  & 13405   &  & -471.00  & \textgreater{}3600 & 9675451  \\
miblp-20-20-50-0110-15-8  & 5  & 15  & -370.00  & 39.12   & 21589   &  & -290.00  & \textgreater{}3600 & 10350188 \\
miblp-20-20-50-0110-15-9  & 5  & 15  & -584.00  & 2.00    & 582     &  & -584.00  & 19.58              & 56459    \\
lseu-0.900000             & 9  & 80  & 5838.00  & 14.10   & 1023    &  & 5838.00  & \textgreater{}3600 & 8743754  \\
p0033-0.900000            & 4  & 29  & 4679.00  & 0.06    & 27      &  & 4679.00  & 5.56               & 28241    \\
p0201-0.900000            & 21 & 180 & 15025.00 & 6.62    & 2481    &  & 15025.00 & \textgreater{}3600 & 1310278  \\
stein27-0.900000          & 3  & 24  & 24.00    & 0.02    & 15      &  & 24.00    & 419.87             & 702055   \\
stein45-0.900000          & 5  & 40  & 40.00    & 0.14    & 63      &  & 40.00    & \textgreater{}3600 & 1499583 \\
\hline
\end{tabular}
%}
\end{table}
%%%%%%%%%%%%%%%%%%
%table for figure 3b
\begin{table}[h!]
\caption{Detailed results of Figure~\ref{fig:branchStrategyParTimeMore5Secr1grer2}}
\label{tab:fig3b}
%\resizebox{\columnwidth}{!}{
\scriptsize
\centering
%\begin{tabular}{lccccccccc}
\begin{tabular}{llllllllll}
\hline
\multicolumn{1}{c}{}&
\multicolumn{1}{c}{}&
\multicolumn{1}{c}{}&
\multicolumn{3}{c}{linkingBranching}&
\multicolumn{1}{c}{}&
\multicolumn{3}{c}{fractionalBranching} \\
\cline{4-6}
\cline{8-10}
Instance&    $r_1$& $r_2$& BestSol & Time(s) &Nodes& &BestSol & Time(s) &Nodes\\
\hline
bmilplib-110-10          & 110 & 63  & -177.67  & 153.82  & 75499   &  & -177.67  & 116.32  & 55177   \\
bmilplib-110-1           & 110 & 50  & -181.67  & 5.35    & 2806    &  & -181.67  & 0.93    & 306     \\
bmilplib-110-2           & 110 & 45  & -110.67  & 6.48    & 4174    &  & -110.67  & 1.16    & 303     \\
bmilplib-110-3           & 110 & 55  & -215.16  & 5.58    & 3471    &  & -215.16  & 0.89    & 360     \\
bmilplib-110-4           & 110 & 50  & -197.29  & 3.69    & 1751    &  & -197.29  & 1.34    & 148     \\
bmilplib-110-6           & 110 & 55  & -148.25  & 15.06   & 8770    &  & -148.25  & 3.84    & 1448    \\
bmilplib-110-7           & 110 & 61  & -160.86  & 4.94    & 2189    &  & -160.86  & 0.91    & 205     \\
bmilplib-110-8           & 110 & 54  & -155.00  & 20.78   & 11663   &  & -155.00  & 7.66    & 2274    \\
bmilplib-110-9           & 110 & 58  & -192.92  & 3.66    & 2151    &  & -192.92  & 0.39    & 146     \\
bmilplib-160-10          & 160 & 83  & -189.82  & 33.42   & 9071    &  & -189.82  & 7.73    & 728     \\
bmilplib-160-1           & 160 & 80  & -165.00  & 25.70   & 7901    &  & -165.00  & 4.76    & 881     \\
bmilplib-160-2           & 160 & 76  & -178.24  & 30.59   & 8635    &  & -178.24  & 5.95    & 507     \\
bmilplib-160-3           & 160 & 86  & -174.94  & 63.37   & 15002   &  & -174.94  & 13.71   & 1102    \\
bmilplib-160-4           & 160 & 81  & -135.83  & 51.02   & 14772   &  & -135.83  & 19.13   & 2447    \\
bmilplib-160-5           & 160 & 83  & -140.78  & 20.52   & 4948    &  & -140.78  & 7.72    & 668     \\
bmilplib-160-6           & 160 & 81  & -111.00  & 24.66   & 8676    &  & -111.00  & 5.02    & 627     \\
bmilplib-160-7           & 160 & 79  & -96.00   & 51.20   & 16560   &  & -96.00   & 18.72   & 2855    \\
bmilplib-160-8           & 160 & 85  & -181.40  & 10.49   & 3444    &  & -181.40  & 1.74    & 311     \\
bmilplib-160-9           & 160 & 77  & -207.50  & 15.85   & 4607    &  & -207.50  & 2.02    & 268     \\
bmilplib-210-10          & 210 & 99  & -130.59  & 61.81   & 12176   &  & -130.59  & 8.59    & 1100    \\
bmilplib-210-1           & 210 & 101 & -136.80  & 41.96   & 7429    &  & -136.80  & 6.67    & 550     \\
bmilplib-210-2           & 210 & 96  & -117.80  & 113.38  & 22294   &  & -117.80  & 19.78   & 2306    \\
bmilplib-210-3           & 210 & 119 & -130.80  & 48.19   & 8848    &  & -130.80  & 12.93   & 1380    \\
bmilplib-210-4           & 210 & 115 & -162.20  & 26.81   & 4687    &  & -162.20  & 3.64    & 309     \\
bmilplib-210-5           & 210 & 110 & -134.00  & 108.22  & 21058   &  & -134.00  & 20.03   & 2079    \\
bmilplib-210-6           & 210 & 115 & -125.43  & 201.40  & 38443   &  & -125.43  & 45.28   & 4875    \\
bmilplib-210-7           & 210 & 102 & -169.73  & 77.84   & 13960   &  & -169.73  & 11.46   & 1181    \\
bmilplib-210-8           & 210 & 116 & -101.46  & 61.79   & 11105   &  & -101.46  & 9.60    & 942     \\
bmilplib-210-9           & 210 & 103 & -184.00  & 879.61  & 142294  &  & -184.00  & 240.14  & 9466    \\
bmilplib-260-10          & 260 & 117 & -151.73  & 546.55  & 61767   &  & -151.73  & 73.10   & 4716    \\
bmilplib-260-1           & 260 & 135 & -139.00  & 85.81   & 10980   &  & -139.00  & 10.10   & 887     \\
bmilplib-260-2           & 260 & 126 & -82.62   & 109.75  & 15287   &  & -82.62   & 18.07   & 1607    \\
bmilplib-260-3           & 260 & 120 & -144.25  & 71.54   & 8769    &  & -144.25  & 8.70    & 518     \\
bmilplib-260-4           & 260 & 125 & -117.33  & 259.48  & 32787   &  & -117.33  & 66.89   & 4426    \\
bmilplib-260-5           & 260 & 132 & -121.00  & 166.03  & 21542   &  & -121.00  & 29.19   & 2165    \\
bmilplib-260-6           & 260 & 146 & -124.00  & 197.04  & 25362   &  & -124.00  & 40.59   & 2420    \\
bmilplib-260-7           & 260 & 129 & -137.80  & 312.12  & 40274   &  & -137.80  & 44.45   & 3200    \\
bmilplib-260-8           & 260 & 143 & -119.89  & 87.01   & 10025   &  & -119.89  & 10.92   & 1025    \\
bmilplib-260-9           & 260 & 132 & -160.00  & 273.66  & 33468   &  & -160.00  & 36.36   & 2526    \\
bmilplib-310-10          & 310 & 157 & -141.86  & 121.07  & 9900    &  & -141.86  & 5.51    & 397     \\
bmilplib-310-1           & 310 & 169 & -117.00  & 329.71  & 28330   &  & -117.00  & 44.79   & 2624    \\
bmilplib-310-2           & 310 & 154 & -105.00  & 497.66  & 43399   &  & -105.00  & 98.39   & 5372    \\
bmilplib-310-3           & 310 & 157 & -127.52  & 777.74  & 66410   &  & -127.52  & 149.07  & 7067    \\
bmilplib-310-4           & 310 & 152 & -147.78  & 569.43  & 47711   &  & -147.78  & 70.22   & 4767    \\
bmilplib-310-5           & 310 & 164 & -161.45  & 366.90  & 31782   &  & -161.45  & 34.22   & 1993    \\
bmilplib-310-6           & 310 & 148 & -141.18  & 1191.51 & 101904  &  & -141.18  & 169.96  & 8300    \\
bmilplib-310-7           & 310 & 170 & -142.00  & 1129.24 & 102000  &  & -142.00  & 139.15  & 7263    \\
bmilplib-310-8           & 310 & 154 & -115.34  & 127.94  & 11109   &  & -115.34  & 19.23   & 921     \\
bmilplib-310-9           & 310 & 150 & -115.65  & 255.27  & 20490   &  & -115.65  & 32.31   & 1590    \\
bmilplib-360-10          & 360 & 172 & -108.59  & 234.92  & 13697   &  & -108.59  & 25.75   & 1106    \\
bmilplib-360-1           & 360 & 181 & -133.00  & 1353.62 & 75421   &  & -133.00  & 158.38  & 4923    \\
bmilplib-360-2           & 360 & 179 & -138.44  & 965.52  & 52919   &  & -138.44  & 148.90  & 4493    \\
bmilplib-360-3           & 360 & 195 & -131.00  & 832.88  & 40487   &  & -131.00  & 65.41   & 2654    \\
bmilplib-360-4           & 360 & 184 & -119.00  & 332.90  & 18293   &  & -119.00  & 42.95   & 1564    \\
bmilplib-360-5           & 360 & 194 & -164.26  & 618.86  & 29947   &  & -164.26  & 44.45   & 1713    \\
bmilplib-360-6           & 360 & 172 & -110.12  & 1181.02 & 68863   &  & -110.12  & 142.81  & 5520    \\
bmilplib-360-7           & 360 & 188 & -105.00  & 634.17  & 30900   &  & -105.00  & 89.22   & 3346    \\
bmilplib-360-8           & 360 & 170 & -98.25   & 362.89  & 22857   &  & -98.25   & 44.85   & 1686    \\
bmilplib-360-9           & 360 & 184 & -127.22  & 736.16  & 40329   &  & -127.22  & 50.24   & 2642    \\
bmilplib-410-10          & 410 & 201 & -153.37  & 2729.77 & 101447  &  & -153.37  & 258.89  & 7428    \\
bmilplib-410-1           & 410 & 196 & -103.50  & 553.76  & 20634   &  & -103.50  & 87.94   & 1944    \\
bmilplib-410-2           & 410 & 189 & -108.59  & 840.56  & 31603   &  & -108.59  & 103.57  & 2887    \\
bmilplib-410-3           & 410 & 212 & -96.24   & 871.58  & 32882   &  & -96.24   & 147.99  & 3781    \\
bmilplib-410-4           & 410 & 187 & -119.50  & 1031.12 & 38489   &  & -119.50  & 100.21  & 2995    \\
bmilplib-410-5           & 410 & 209 & -119.22  & 678.49  & 23852   &  & -119.22  & 71.23   & 1520    \\
bmilplib-410-6           & 410 & 206 & -151.31  & 322.88  & 11130   &  & -151.31  & 13.54   & 533     \\
bmilplib-410-7           & 410 & 225 & -123.00  & 560.34  & 22146   &  & -123.00  & 35.87   & 1177    \\
bmilplib-410-8           & 410 & 211 & -125.78  & 1547.88 & 62218   &  & -125.78  & 169.15  & 4480    \\
bmilplib-410-9           & 410 & 216 & -100.77  & 531.88  & 20400   &  & -100.77  & 82.17   & 2071    \\
bmilplib-460-10          & 460 & 217 & -102.51  & 1857.94 & 55024   &  & -102.51  & 228.12  & 4465    \\
bmilplib-460-1           & 460 & 227 & -97.59   & 2822.97 & 86632   &  & -97.59   & 569.22  & 10803   \\
bmilplib-460-2           & 460 & 249 & -139.00  & 625.95  & 16623   &  & -139.00  & 43.65   & 964     \\
bmilplib-460-3           & 460 & 222 & -86.50   & 1995.92 & 57395   &  & -86.50   & 223.58  & 3882    \\
bmilplib-460-4           & 460 & 218 & -107.03  & 3291.28 & 95035   &  & -107.03  & 412.76  & 7856    \\
bmilplib-460-5           & 460 & 216 & -100.50  & 1424.54 & 41109   &  & -100.50  & 170.30  & 3025    \\
bmilplib-460-6           & 460 & 222 & -107.00  & 1635.42 & 46076   &  & -107.00  & 236.30  & 4143    \\
bmilplib-460-7           & 460 & 254 & -83.75   & 1601.34 & 48897   &  & -83.75   & 294.68  & 5252    \\
bmilplib-460-8           & 460 & 256 & -115.39  & 982.30  & 27572   &  & -115.39  & 94.67   & 1903    \\
bmilplib-460-9           & 460 & 224 & -128.70  & 2830.59 & 85977   &  & -128.70  & 327.68  & 6185    \\
bmilplib-60-10           & 60  & 25  & -186.21  & 6.93    & 5922    &  & -186.21  & 4.29    & 2590    \\
\hline
\end{tabular}
%}
\end{table}
%%%%%%%%%%%%%%%%%%
%table for figure 3b
\addtocounter{table}{-1}
\begin{table}[h!]
\caption{Detailed results of Figure~\ref{fig:branchStrategyParTimeMore5Secr1grer2} (continued)}
\label{tab:fig3bContinue}
%\resizebox{\columnwidth}{!}{
\scriptsize
\centering
%\begin{tabular}{lccccccccc}
\begin{tabular}{llllllllll}
\hline
\multicolumn{1}{c}{}&
\multicolumn{1}{c}{}&
\multicolumn{1}{c}{}&
\multicolumn{3}{c}{linkingBranching}&
\multicolumn{1}{c}{}&
\multicolumn{3}{c}{fractionalBranching} \\
\cline{4-6}
\cline{8-10}
Instance&    $r_1$& $r_2$& BestSol & Time(s) &Nodes& &BestSol & Time(s) &Nodes\\
\hline
bmilplib-60-1            & 60  & 29  & -153.20  & 6.71    & 4877    &  & -153.20  & 2.79    & 1094    \\
bmilplib-60-5            & 60  & 33  & -116.40  & 8.15    & 11202   &  & -116.40  & 10.08   & 9996    \\
bmilplib-60-6            & 60  & 27  & -187.31  & 7.62    & 7016    &  & -187.31  & 3.34    & 1383    \\
bmilplib-60-8            & 60  & 30  & -232.12  & 3.36    & 3570    &  & -232.12  & 1.15    & 572     \\
bmilplib-60-9            & 60  & 26  & -136.50  & 19.61   & 26792   &  & -136.50  & 12.34   & 9888    \\
miblp-20-20-50-0110-5-13 & 15  & 5   & -519.00  & 2464.25 & 7128138 &  & -519.00  & 2709.29 & 6515097 \\
miblp-20-20-50-0110-5-15 & 15  & 5   & -617.00  & 1310.67 & 4018058 &  & -617.00  & 1130.92 & 4312915 \\
miblp-20-20-50-0110-5-16 & 15  & 5   & -833.00  & 6.60    & 17680   &  & -833.00  & 1.80    & 5913    \\
miblp-20-20-50-0110-5-17 & 15  & 5   & -944.00  & 4.78    & 17679   &  & -944.00  & 1.53    & 4038    \\
miblp-20-20-50-0110-5-19 & 15  & 5   & -431.00  & 26.56   & 77256   &  & -431.00  & 25.50   & 116041  \\
miblp-20-20-50-0110-5-1  & 15  & 5   & -548.00  & 14.50   & 42364   &  & -548.00  & 8.45    & 31298   \\
miblp-20-20-50-0110-5-20 & 15  & 5   & -438.00  & 17.30   & 60944   &  & -438.00  & 10.82   & 33315   \\
miblp-20-20-50-0110-5-6  & 15  & 5   & -1061.00 & 63.70   & 222566  &  & -1061.00 & 51.74   & 213928  \\
lseu-0.100000            & 81  & 8   & 1120.00  & 270.78  & 1003976 &  & 1120.00  & 3.15    & 8603    \\
p0033-0.500000           & 17  & 16  & 3095.00  & 10.48   & 20695   &  & 3095.00  & 0.25    & 1467    \\
stein27-0.500000         & 14  & 13  & 19.00    & 5.88    & 12115   &  & 19.00    & 6.51    & 17648   \\
stein45-0.100000         & 41  & 4   & 30.00    & 50.19   & 89035   &  & 30.00    & 64.27   & 90241   \\
stein45-0.500000         & 23  & 22  & 32.00    & 519.59  & 640308  &  & 32.00    & 471.57  & 753845 \\
\hline
\end{tabular}
%}
\end{table}
%%%%%%%%%%%%%%%%%%
%table for figure 4a
\begin{table}[h!]
\caption{Detailed results of Figure~\ref{fig:linkingSolutionPooParWithFracParWithTimeMore5Sec}}
\label{tab:fig4a}
\resizebox{\columnwidth}{!}{
%\begin{tabular}{lccccccccccccccc}
\begin{tabular}{llllllllllllllll}
\hline
\multicolumn{1}{c}{}&
\multicolumn{3}{c}{\begin{tabular}[c]{@{}c@{}}\scriptsize{withoutPoolWhen}\\\scriptsize{XYInt-LFixed}\end{tabular}}&
\multicolumn{1}{c}{}&
\multicolumn{3}{c}{\begin{tabular}[c]{@{}c@{}}\scriptsize{withPoolWhen}\\\scriptsize{XYInt-LFixed}\end{tabular}}&
\multicolumn{1}{c}{}&
\multicolumn{3}{c}{\begin{tabular}[c]{@{}c@{}}\scriptsize{withoutPoolWhen}\\\scriptsize{XYIntOrLFixed-LFixed}\end{tabular}}&
\multicolumn{1}{c}{}&
\multicolumn{3}{c}{\begin{tabular}[c]{@{}c@{}}\scriptsize{withPoolWhen}\\\scriptsize{XYIntOrLFixed-LFixed}\end{tabular}} \\
\cline{2-4}
\cline{6-8}
\cline{10-12}
\cline{14-16}
Instance&   BestSol & Time(s) &Nodes  & &BestSol & Time(s) &Nodes & &BestSol & Time(s) &Nodes& &BestSol & Time(s) &Nodes\\
\hline
bmilplib-110-10           & -177.67  & 355.72             & 55177   &  & -177.67  & 117.08             & 55177    &  & -177.67  & 357.09             & 55177   &  & -177.67  & 116.32             & 55177    \\
bmilplib-110-1            & -181.67  & 1.08               & 306     &  & -181.67  & 0.94               & 306      &  & -181.67  & 1.08               & 306     &  & -181.67  & 0.93               & 306      \\
bmilplib-110-2            & -110.67  & 1.56               & 303     &  & -110.67  & 1.17               & 303      &  & -110.67  & 1.56               & 303     &  & -110.67  & 1.16               & 303      \\
bmilplib-110-3            & -215.16  & 1.18               & 360     &  & -215.16  & 0.91               & 360      &  & -215.16  & 1.14               & 360     &  & -215.16  & 0.89               & 360      \\
bmilplib-110-4            & -197.29  & 2.38               & 148     &  & -197.29  & 1.34               & 148      &  & -197.29  & 2.38               & 148     &  & -197.29  & 1.34               & 148      \\
bmilplib-110-6            & -148.25  & 8.77               & 1448    &  & -148.25  & 3.83               & 1448     &  & -148.25  & 8.78               & 1448    &  & -148.25  & 3.84               & 1448     \\
bmilplib-110-7            & -160.86  & 1.16               & 205     &  & -160.86  & 0.93               & 205      &  & -160.86  & 1.17               & 205     &  & -160.86  & 0.91               & 205      \\
bmilplib-110-8            & -155.00  & 12.40              & 2274    &  & -155.00  & 7.74               & 2274     &  & -155.00  & 12.29              & 2274    &  & -155.00  & 7.66               & 2274     \\
bmilplib-110-9            & -192.92  & 0.39               & 146     &  & -192.92  & 0.38               & 146      &  & -192.92  & 0.38               & 146     &  & -192.92  & 0.39               & 146      \\
bmilplib-160-10           & -189.82  & 17.00              & 728     &  & -189.82  & 7.70               & 728      &  & -189.82  & 16.94              & 728     &  & -189.82  & 7.73               & 728      \\
bmilplib-160-1            & -165.00  & 6.52               & 881     &  & -165.00  & 4.72               & 881      &  & -165.00  & 6.48               & 881     &  & -165.00  & 4.76               & 881      \\
bmilplib-160-2            & -178.24  & 9.39               & 507     &  & -178.24  & 5.98               & 507      &  & -178.24  & 9.42               & 507     &  & -178.24  & 5.95               & 507      \\
bmilplib-160-3            & -174.94  & 32.47              & 1102    &  & -174.94  & 13.76              & 1102     &  & -174.94  & 32.60              & 1102    &  & -174.94  & 13.71              & 1102     \\
bmilplib-160-4            & -135.83  & 35.80              & 2447    &  & -135.83  & 19.06              & 2447     &  & -135.83  & 36.00              & 2447    &  & -135.83  & 19.13              & 2447     \\
bmilplib-160-5            & -140.78  & 29.26              & 668     &  & -140.78  & 7.73               & 668      &  & -140.78  & 29.33              & 668     &  & -140.78  & 7.72               & 668      \\
bmilplib-160-6            & -111.00  & 7.26               & 627     &  & -111.00  & 5.01               & 627      &  & -111.00  & 7.30               & 627     &  & -111.00  & 5.02               & 627      \\
bmilplib-160-7            & -96.00   & 35.18              & 2855    &  & -96.00   & 18.04              & 2855     &  & -96.00   & 35.13              & 2855    &  & -96.00   & 18.72              & 2855     \\
bmilplib-160-8            & -181.40  & 2.30               & 311     &  & -181.40  & 1.67               & 311      &  & -181.40  & 2.28               & 311     &  & -181.40  & 1.74               & 311      \\
bmilplib-160-9            & -207.50  & 2.02               & 268     &  & -207.50  & 2.01               & 268      &  & -207.50  & 2.00               & 268     &  & -207.50  & 2.02               & 268      \\
bmilplib-210-10           & -130.59  & 13.92              & 1100    &  & -130.59  & 8.78               & 1100     &  & -130.59  & 14.28              & 1100    &  & -130.59  & 8.59               & 1100     \\
bmilplib-210-1            & -136.80  & 6.72               & 550     &  & -136.80  & 6.71               & 550      &  & -136.80  & 6.74               & 550     &  & -136.80  & 6.67               & 550      \\
bmilplib-210-2            & -117.80  & 40.52              & 2306    &  & -117.80  & 20.60              & 2306     &  & -117.80  & 41.96              & 2306    &  & -117.80  & 19.78              & 2306     \\
bmilplib-210-3            & -130.80  & 24.72              & 1380    &  & -130.80  & 13.09              & 1380     &  & -130.80  & 24.48              & 1380    &  & -130.80  & 12.93              & 1380     \\
bmilplib-210-4            & -162.20  & 4.82               & 309     &  & -162.20  & 3.65               & 309      &  & -162.20  & 4.80               & 309     &  & -162.20  & 3.64               & 309      \\
bmilplib-210-5            & -134.00  & 29.57              & 2079    &  & -134.00  & 20.08              & 2079     &  & -134.00  & 30.96              & 2079    &  & -134.00  & 20.03              & 2079     \\
bmilplib-210-6            & -125.43  & 105.93             & 4875    &  & -125.43  & 46.30              & 4875     &  & -125.43  & 107.54             & 4875    &  & -125.43  & 45.28              & 4875     \\
bmilplib-210-7            & -169.73  & 18.21              & 1181    &  & -169.73  & 11.59              & 1181     &  & -169.73  & 18.14              & 1181    &  & -169.73  & 11.46              & 1181     \\
bmilplib-210-8            & -101.46  & 17.89              & 942     &  & -101.46  & 9.69               & 942      &  & -101.46  & 17.97              & 942     &  & -101.46  & 9.60               & 942      \\
bmilplib-210-9            & -184.00  & 334.58             & 9466    &  & -184.00  & 241.89             & 9466     &  & -184.00  & 336.35             & 9466    &  & -184.00  & 240.14             & 9466     \\
bmilplib-260-10           & -151.73  & 106.84             & 4716    &  & -151.73  & 75.99              & 4716     &  & -151.73  & 106.74             & 4716    &  & -151.73  & 73.10              & 4716     \\
bmilplib-260-1            & -139.00  & 10.55              & 887     &  & -139.00  & 9.72               & 887      &  & -139.00  & 10.17              & 887     &  & -139.00  & 10.10              & 887      \\
bmilplib-260-2            & -82.62   & 25.82              & 1607    &  & -82.62   & 17.95              & 1607     &  & -82.62   & 24.39              & 1607    &  & -82.62   & 18.07              & 1607     \\
bmilplib-260-3            & -144.25  & 12.22              & 518     &  & -144.25  & 8.70               & 518      &  & -144.25  & 12.20              & 518     &  & -144.25  & 8.70               & 518      \\
bmilplib-260-4            & -117.33  & 82.73              & 4426    &  & -117.33  & 62.90              & 4426     &  & -117.33  & 84.74              & 4426    &  & -117.33  & 66.89              & 4426     \\
bmilplib-260-5            & -121.00  & 36.18              & 2165    &  & -121.00  & 29.80              & 2165     &  & -121.00  & 37.57              & 2165    &  & -121.00  & 29.19              & 2165     \\
bmilplib-260-6            & -124.00  & 67.59              & 2420    &  & -124.00  & 39.43              & 2420     &  & -124.00  & 69.01              & 2420    &  & -124.00  & 40.59              & 2420     \\
bmilplib-260-7            & -137.80  & 76.55              & 3200    &  & -137.80  & 46.85              & 3200     &  & -137.80  & 79.79              & 3200    &  & -137.80  & 44.45              & 3200     \\
bmilplib-260-8            & -119.89  & 16.52              & 1025    &  & -119.89  & 10.90              & 1025     &  & -119.89  & 17.44              & 1025    &  & -119.89  & 10.92              & 1025     \\
bmilplib-260-9            & -160.00  & 61.66              & 2526    &  & -160.00  & 36.58              & 2526     &  & -160.00  & 64.72              & 2526    &  & -160.00  & 36.36              & 2526     \\
bmilplib-310-10           & -141.86  & 5.42               & 397     &  & -141.86  & 5.95               & 397      &  & -141.86  & 5.88               & 397     &  & -141.86  & 5.51               & 397      \\
bmilplib-310-1            & -117.00  & 58.20              & 2624    &  & -117.00  & 43.64              & 2624     &  & -117.00  & 62.56              & 2624    &  & -117.00  & 44.79              & 2624     \\
bmilplib-310-2            & -105.00  & 109.26             & 5372    &  & -105.00  & 96.63              & 5372     &  & -105.00  & 117.75             & 5372    &  & -105.00  & 98.39              & 5372     \\
bmilplib-310-3            & -127.52  & 171.44             & 7067    &  & -127.52  & 148.52             & 7067     &  & -127.52  & 174.17             & 7067    &  & -127.52  & 149.07             & 7067     \\
bmilplib-310-4            & -147.78  & 90.78              & 4767    &  & -147.78  & 77.28              & 4767     &  & -147.78  & 95.75              & 4767    &  & -147.78  & 70.22              & 4767     \\
bmilplib-310-5            & -161.45  & 48.58              & 1993    &  & -161.45  & 34.30              & 1993     &  & -161.45  & 45.88              & 1993    &  & -161.45  & 34.22              & 1993     \\
bmilplib-310-6            & -141.18  & 197.54             & 8300    &  & -141.18  & 179.02             & 8300     &  & -141.18  & 198.16             & 8300    &  & -141.18  & 169.96             & 8300     \\
bmilplib-310-7            & -142.00  & 179.47             & 7263    &  & -142.00  & 144.99             & 7263     &  & -142.00  & 169.98             & 7263    &  & -142.00  & 139.15             & 7263     \\
bmilplib-310-8            & -115.34  & 25.48              & 921     &  & -115.34  & 20.24              & 921      &  & -115.34  & 24.68              & 921     &  & -115.34  & 19.23              & 921      \\
bmilplib-310-9            & -115.65  & 43.85              & 1590    &  & -115.65  & 37.72              & 1590     &  & -115.65  & 45.18              & 1590    &  & -115.65  & 32.31              & 1590     \\
bmilplib-360-10           & -108.59  & 27.82              & 1106    &  & -108.59  & 25.17              & 1106     &  & -108.59  & 31.69              & 1106    &  & -108.59  & 25.75              & 1106     \\
bmilplib-360-1            & -133.00  & 207.62             & 4923    &  & -133.00  & 146.36             & 4923     &  & -133.00  & 194.58             & 4923    &  & -133.00  & 158.38             & 4923     \\
bmilplib-360-2            & -138.44  & 230.49             & 4493    &  & -138.44  & 135.61             & 4493     &  & -138.44  & 225.09             & 4493    &  & -138.44  & 148.90             & 4493     \\
bmilplib-360-3            & -131.00  & 68.30              & 2654    &  & -131.00  & 63.40              & 2654     &  & -131.00  & 75.82              & 2654    &  & -131.00  & 65.41              & 2654     \\
bmilplib-360-4            & -119.00  & 54.68              & 1564    &  & -119.00  & 47.06              & 1564     &  & -119.00  & 51.33              & 1564    &  & -119.00  & 42.95              & 1564     \\
bmilplib-360-5            & -164.26  & 58.30              & 1713    &  & -164.26  & 45.64              & 1713     &  & -164.26  & 56.12              & 1713    &  & -164.26  & 44.45              & 1713     \\
bmilplib-360-6            & -110.12  & 155.42             & 5520    &  & -110.12  & 132.79             & 5520     &  & -110.12  & 146.64             & 5520    &  & -110.12  & 142.81             & 5520     \\
bmilplib-360-7            & -105.00  & 135.16             & 3346    &  & -105.00  & 93.50              & 3346     &  & -105.00  & 125.91             & 3346    &  & -105.00  & 89.22              & 3346     \\
bmilplib-360-8            & -98.25   & 66.38              & 1686    &  & -98.25   & 50.86              & 1686     &  & -98.25   & 66.34              & 1686    &  & -98.25   & 44.85              & 1686     \\
bmilplib-360-9            & -127.22  & 67.05              & 2642    &  & -127.22  & 58.36              & 2642     &  & -127.22  & 68.26              & 2642    &  & -127.22  & 50.24              & 2642     \\
bmilplib-410-10           & -153.37  & 332.48             & 7428    &  & -153.37  & 265.98             & 7428     &  & -153.37  & 333.37             & 7428    &  & -153.37  & 258.89             & 7428     \\
bmilplib-410-1            & -103.50  & 103.54             & 1944    &  & -103.50  & 85.04              & 1944     &  & -103.50  & 109.43             & 1944    &  & -103.50  & 87.94              & 1944     \\
bmilplib-410-2            & -108.59  & 128.11             & 2887    &  & -108.59  & 107.05             & 2887     &  & -108.59  & 125.37             & 2887    &  & -108.59  & 103.57             & 2887     \\
bmilplib-410-3            & -96.24   & 176.89             & 3781    &  & -96.24   & 162.97             & 3781     &  & -96.24   & 167.58             & 3781    &  & -96.24   & 147.99             & 3781     \\
bmilplib-410-4            & -119.50  & 111.28             & 2995    &  & -119.50  & 96.04              & 2995     &  & -119.50  & 115.22             & 2995    &  & -119.50  & 100.21             & 2995     \\
bmilplib-410-5            & -119.22  & 72.48              & 1520    &  & -119.22  & 67.90              & 1520     &  & -119.22  & 73.09              & 1520    &  & -119.22  & 71.23              & 1520     \\
bmilplib-410-6            & -151.31  & 14.12              & 533     &  & -151.31  & 14.66              & 533      &  & -151.31  & 13.95              & 533     &  & -151.31  & 13.54              & 533      \\
bmilplib-410-7            & -123.00  & 43.96              & 1177    &  & -123.00  & 42.46              & 1177     &  & -123.00  & 42.43              & 1177    &  & -123.00  & 35.87              & 1177     \\
bmilplib-410-8            & -125.78  & 292.85             & 4480    &  & -125.78  & 168.87             & 4480     &  & -125.78  & 284.91             & 4480    &  & -125.78  & 169.15             & 4480     \\
bmilplib-410-9            & -100.77  & 95.70              & 2071    &  & -100.77  & 81.09              & 2071     &  & -100.77  & 89.94              & 2071    &  & -100.77  & 82.17              & 2071     \\
bmilplib-460-10           & -102.51  & 284.10             & 4465    &  & -102.51  & 232.16             & 4465     &  & -102.51  & 272.39             & 4465    &  & -102.51  & 228.12             & 4465     \\
bmilplib-460-1            & -97.59   & 638.06             & 10803   &  & -97.59   & 574.36             & 10803    &  & -97.59   & 636.80             & 10803   &  & -97.59   & 569.22             & 10803    \\
bmilplib-460-2            & -139.00  & 65.03              & 964     &  & -139.00  & 51.23              & 964      &  & -139.00  & 68.32              & 964     &  & -139.00  & 43.65              & 964      \\
bmilplib-460-3            & -86.50   & 226.69             & 3882    &  & -86.50   & 223.11             & 3882     &  & -86.50   & 229.43             & 3882    &  & -86.50   & 223.58             & 3882     \\
bmilplib-460-4            & -107.03  & 640.32             & 7856    &  & -107.03  & 413.15             & 7856     &  & -107.03  & 607.29             & 7856    &  & -107.03  & 412.76             & 7856     \\
bmilplib-460-5            & -100.50  & 229.94             & 3025    &  & -100.50  & 171.50             & 3025     &  & -100.50  & 196.94             & 3025    &  & -100.50  & 170.30             & 3025     \\
bmilplib-460-6            & -107.00  & 251.71             & 4143    &  & -107.00  & 219.39             & 4143     &  & -107.00  & 221.43             & 4143    &  & -107.00  & 236.30             & 4143     \\
bmilplib-460-7            & -83.75   & 334.13             & 5252    &  & -83.75   & 295.91             & 5252     &  & -83.75   & 331.02             & 5252    &  & -83.75   & 294.68             & 5252     \\
bmilplib-460-8            & -115.39  & 117.85             & 1903    &  & -115.39  & 102.20             & 1903     &  & -115.39  & 134.03             & 1903    &  & -115.39  & 94.67              & 1903     \\
bmilplib-460-9            & -128.70  & 409.66             & 6185    &  & -128.70  & 346.00             & 6185     &  & -128.70  & 388.74             & 6185    &  & -128.70  & 327.68             & 6185     \\
bmilplib-60-10            & -186.21  & 4.82               & 2590    &  & -186.21  & 4.12               & 2590     &  & -186.21  & 4.90               & 2590    &  & -186.21  & 4.29               & 2590     \\
bmilplib-60-1             & -153.20  & 3.59               & 1094    &  & -153.20  & 2.83               & 1094     &  & -153.20  & 3.82               & 1094    &  & -153.20  & 2.79               & 1094     \\
bmilplib-60-5             & -116.40  & 16.18              & 9996    &  & -116.40  & 10.33              & 9996     &  & -116.40  & 15.85              & 9996    &  & -116.40  & 10.08              & 9996     \\
bmilplib-60-6             & -187.31  & 3.58               & 1383    &  & -187.31  & 3.43               & 1383     &  & -187.31  & 3.65               & 1383    &  & -187.31  & 3.34               & 1383     \\
bmilplib-60-8             & -232.12  & 1.79               & 572     &  & -232.12  & 1.16               & 572      &  & -232.12  & 1.78               & 572     &  & -232.12  & 1.15               & 572      \\
bmilplib-60-9             & -136.50  & 14.85              & 9888    &  & -136.50  & 12.36              & 9888     &  & -136.50  & 15.14              & 9888    &  & -136.50  & 12.34              & 9888     \\
miblp-20-15-50-0110-10-10 & -206.00  & 36.93              & 21203   &  & -206.00  & 4.22               & 15741    &  & -206.00  & 39.37              & 15741   &  & -206.00  & 4.25               & 15741    \\
miblp-20-15-50-0110-10-2  & -354.00  & \textgreater{}3600 & 642894  &  & -398.00  & 970.06             & 2929420  &  & -354.00  & \textgreater{}3600 & 596423  &  & -398.00  & 947.81             & 2929420  \\
miblp-20-15-50-0110-10-3  & -42.00   & 188.57             & 48461   &  & -42.00   & 12.65              & 46349    &  & -42.00   & 203.74             & 46349   &  & -42.00   & 12.66              & 46349    \\
miblp-20-15-50-0110-10-6  & -246.00  & 12.02              & 1385    &  & -246.00  & 1.53               & 1367     &  & -246.00  & 13.22              & 1367    &  & -246.00  & 1.56               & 1367     \\
miblp-20-15-50-0110-10-9  & -635.00  & 25.65              & 7667    &  & -635.00  & 6.62               & 7667     &  & -635.00  & 26.03              & 7667    &  & -635.00  & 6.52               & 7667     \\
miblp-20-20-50-0110-10-10 & -405.00  & \textgreater{}3600 & 304054  &  & -441.00  & 2908.94            & 7261247  &  & -405.00  & \textgreater{}3600 & 297842  &  & -441.00  & 2867.85            & 7261247  \\
miblp-20-20-50-0110-10-1  & -315.00  & \textgreater{}3600 & 281846  &  & -357.00  & \textgreater{}3600 & 9176214  &  & -315.00  & \textgreater{}3600 & 268812  &  & -357.00  & \textgreater{}3600 & 9234364  \\
miblp-20-20-50-0110-10-2  & -659.00  & 14.46              & 5106    &  & -659.00  & 3.21               & 5100     &  & -659.00  & 13.90              & 5100    &  & -659.00  & 3.16               & 5100     \\
miblp-20-20-50-0110-10-3  & -618.00  & 16.74              & 52259   &  & -618.00  & 12.21              & 51970    &  & -618.00  & 17.08              & 51970   &  & -618.00  & 12.05              & 51970    \\
\hline
\end{tabular}
}
\end{table}
%%%%%%%%%%%%%%%%%%
%table for figure 4a (continued)
\addtocounter{table}{-1}
\begin{table}[h!]
\caption{Detailed results of Figure~\ref{fig:linkingSolutionPooParWithFracParWithTimeMore5Sec} (continued)}
\label{tab:fig4aContinue}
\resizebox{\columnwidth}{!}{
%\begin{tabular}{lccccccccccccccc}
\begin{tabular}{llllllllllllllll}
\hline
\multicolumn{1}{c}{}&
\multicolumn{3}{c}{\begin{tabular}[c]{@{}c@{}}\scriptsize{withoutPoolWhen}\\\scriptsize{XYInt-LFixed}\end{tabular}}&
\multicolumn{1}{c}{}&
\multicolumn{3}{c}{\begin{tabular}[c]{@{}c@{}}\scriptsize{withPoolWhen}\\\scriptsize{XYInt-LFixed}\end{tabular}}&
\multicolumn{1}{c}{}&
\multicolumn{3}{c}{\begin{tabular}[c]{@{}c@{}}\scriptsize{withoutPoolWhen}\\\scriptsize{XYIntOrLFixed-LFixed}\end{tabular}}&
\multicolumn{1}{c}{}&
\multicolumn{3}{c}{\begin{tabular}[c]{@{}c@{}}\scriptsize{withPoolWhen}\\\scriptsize{XYIntOrLFixed-LFixed}\end{tabular}} \\
\cline{2-4}
\cline{6-8}
\cline{10-12}
\cline{14-16}
Instance&   BestSol & Time(s) &Nodes  & &BestSol & Time(s) &Nodes & &BestSol & Time(s) &Nodes& &BestSol & Time(s) &Nodes\\
\hline
miblp-20-20-50-0110-10-4  & -597.00  & \textgreater{}3600 & 509252  &  & -604.00  & \textgreater{}3600 & 8018349  &  & -597.00  & \textgreater{}3600 & 502089  &  & -604.00  & \textgreater{}3600 & 7988914  \\
miblp-20-20-50-0110-10-7  & -627.00  & \textgreater{}3600 & 4935849 &  & -629.00  & \textgreater{}3600 & 9735490  &  & -627.00  & \textgreater{}3600 & 4698965 &  & -629.00  & \textgreater{}3600 & 9709672  \\
miblp-20-20-50-0110-10-8  & -667.00  & 998.04             & 80603   &  & -667.00  & 68.15              & 75661    &  & -667.00  & 1085.58            & 75661   &  & -667.00  & 68.08              & 75661    \\
miblp-20-20-50-0110-10-9  & -256.00  & 2292.26            & 852328  &  & -256.00  & 306.60             & 757349   &  & -256.00  & 2483.24            & 757349  &  & -256.00  & 305.78             & 757349   \\
miblp-20-20-50-0110-15-1  & -246.00  & \textgreater{}3600 & 250531  &  & -317.00  & \textgreater{}3600 & 9807301  &  & -246.00  & \textgreater{}3600 & 238329  &  & -317.00  & \textgreater{}3600 & 9813005  \\
miblp-20-20-50-0110-15-2  & -645.00  & \textgreater{}3600 & 1416126 &  & -645.00  & \textgreater{}3600 & 10740865 &  & -645.00  & \textgreater{}3600 & 1395442 &  & -645.00  & \textgreater{}3600 & 10839884 \\
miblp-20-20-50-0110-15-3  & -476.00  & \textgreater{}3600 & 189624  &  & -593.00  & \textgreater{}3600 & 13332780 &  & -476.00  & \textgreater{}3600 & 183565  &  & -593.00  & \textgreater{}3600 & 13042109 \\
miblp-20-20-50-0110-15-4  & -310.00  & \textgreater{}3600 & 125392  &  & -398.00  & \textgreater{}3600 & 8670375  &  & -310.00  & \textgreater{}3600 & 125871  &  & -398.00  & \textgreater{}3600 & 8692428  \\
miblp-20-20-50-0110-15-5  & -60.00   & \textgreater{}3600 & 48876   &  & -320.00  & \textgreater{}3600 & 7275287  &  & -60.00   & \textgreater{}3600 & 47748   &  & -320.00  & \textgreater{}3600 & 7284040  \\
miblp-20-20-50-0110-15-6  & -596.00  & \textgreater{}3600 & 227767  &  & -596.00  & \textgreater{}3600 & 7842912  &  & -596.00  & \textgreater{}3600 & 197263  &  & -596.00  & \textgreater{}3600 & 7851818  \\
miblp-20-20-50-0110-15-7  & -471.00  & \textgreater{}3600 & 393906  &  & -471.00  & \textgreater{}3600 & 9577404  &  & -471.00  & \textgreater{}3600 & 362335  &  & -471.00  & \textgreater{}3600 & 9675451  \\
miblp-20-20-50-0110-15-8  & -290.00  & \textgreater{}3600 & 1772115 &  & -290.00  & \textgreater{}3600 & 10439854 &  & -290.00  & \textgreater{}3600 & 1614730 &  & -290.00  & \textgreater{}3600 & 10350188 \\
miblp-20-20-50-0110-15-9  & -584.00  & 269.87             & 56931   &  & -584.00  & 19.38              & 56459    &  & -584.00  & 271.42             & 56459   &  & -584.00  & 19.58              & 56459    \\
miblp-20-20-50-0110-5-13  & -507.00  & \textgreater{}3600 & 6006058 &  & -519.00  & 2723.02            & 6515097  &  & -507.00  & \textgreater{}3600 & 6189742 &  & -519.00  & 2709.29            & 6515097  \\
miblp-20-20-50-0110-5-15  & -617.00  & 1754.55            & 4312915 &  & -617.00  & 1113.82            & 4312915  &  & -617.00  & 1844.46            & 4312915 &  & -617.00  & 1130.92            & 4312915  \\
miblp-20-20-50-0110-5-16  & -833.00  & 2.06               & 5913    &  & -833.00  & 1.76               & 5913     &  & -833.00  & 2.06               & 5913    &  & -833.00  & 1.80               & 5913     \\
miblp-20-20-50-0110-5-17  & -944.00  & 1.61               & 4050    &  & -944.00  & 1.50               & 4038     &  & -944.00  & 1.66               & 4038    &  & -944.00  & 1.53               & 4038     \\
miblp-20-20-50-0110-5-19  & -431.00  & 38.74              & 116041  &  & -431.00  & 25.54              & 116041   &  & -431.00  & 38.55              & 116041  &  & -431.00  & 25.50              & 116041   \\
miblp-20-20-50-0110-5-1   & -548.00  & 13.07              & 31432   &  & -548.00  & 8.63               & 31298    &  & -548.00  & 13.53              & 31298   &  & -548.00  & 8.45               & 31298    \\
miblp-20-20-50-0110-5-20  & -438.00  & 13.53              & 33315   &  & -438.00  & 10.80              & 33315    &  & -438.00  & 13.51              & 33315   &  & -438.00  & 10.82              & 33315    \\
miblp-20-20-50-0110-5-6   & -1061.00 & 66.75              & 214009  &  & -1061.00 & 51.65              & 213928   &  & -1061.00 & 67.85              & 213928  &  & -1061.00 & 51.74              & 213928   \\
lseu-0.100000             & 1120.00  & 3.17               & 8603    &  & 1120.00  & 3.12               & 8603     &  & 1120.00  & 3.16               & 8603    &  & 1120.00  & 3.15               & 8603     \\
lseu-0.900000             & 5838.00  & \textgreater{}3600 & 2463002 &  & 5838.00  & \textgreater{}3600 & 8762978  &  & 5838.00  & \textgreater{}3600 & 2460978 &  & 5838.00  & \textgreater{}3600 & 8743754  \\
p0033-0.500000            & 3095.00  & 0.28               & 1467    &  & 3095.00  & 0.27               & 1467     &  & 3095.00  & 0.26               & 1467    &  & 3095.00  & 0.25               & 1467     \\
p0033-0.900000            & 4679.00  & 37.10              & 33451   &  & 4679.00  & 5.58               & 28241    &  & 4679.00  & 31.19              & 28241   &  & 4679.00  & 5.56               & 28241    \\
p0201-0.900000            & 15025.00 & \textgreater{}3600 & 690347  &  & 15025.00 & \textgreater{}3600 & 1309676  &  & 15025.00 & \textgreater{}3600 & 689225  &  & 15025.00 & \textgreater{}3600 & 1310278  \\
stein27-0.500000          & 19.00    & 6.60               & 17648   &  & 19.00    & 6.43               & 17648    &  & 19.00    & 6.84               & 17648   &  & 19.00    & 6.51               & 17648    \\
stein27-0.900000          & 24.00    & 644.53             & 713061  &  & 24.00    & 413.77             & 702055   &  & 24.00    & 638.54             & 702055  &  & 24.00    & 419.87             & 702055   \\
stein45-0.100000          & 30.00    & 64.50              & 90241   &  & 30.00    & 63.84              & 90241    &  & 30.00    & 64.46              & 90241   &  & 30.00    & 64.27              & 90241    \\
stein45-0.500000          & 32.00    & 465.23             & 753845  &  & 32.00    & 468.17             & 753845   &  & 32.00    & 472.32             & 753845  &  & 32.00    & 471.57             & 753845   \\
stein45-0.900000          & 40.00    & \textgreater{}3600 & 1379438 &  & 40.00    & \textgreater{}3600 & 1504193  &  & 40.00    & \textgreater{}3600 & 1376878 &  & 40.00    & \textgreater{}3600 & 1499583 \\
\hline
\end{tabular}
}
\end{table}
%%%%%%%%%%%%%%%%%%
%table for figure 4b
\begin{table}[h!]
\caption{Detailed results of Figure~\ref{fig:linkingSolutionPooParWithLinkParWithTimeMore5Sec}}
\label{tab:fig4b}
\resizebox{\columnwidth}{!}{
%\begin{tabular}{lccccccccccccccc}
\begin{tabular}{llllllllllllllll}
\hline
\multicolumn{1}{c}{}&
\multicolumn{3}{c}{\begin{tabular}[c]{@{}c@{}}\scriptsize{withoutPoolWhen}\\\scriptsize{XYInt-LFixed}\end{tabular}}&
\multicolumn{1}{c}{}&
\multicolumn{3}{c}{\begin{tabular}[c]{@{}c@{}}\scriptsize{withPoolWhen}\\\scriptsize{XYInt-LFixed}\end{tabular}}&
\multicolumn{1}{c}{}&
\multicolumn{3}{c}{\begin{tabular}[c]{@{}c@{}}\scriptsize{withoutPoolWhen}\\\scriptsize{XYIntOrLFixed-LFixed}\end{tabular}}&
\multicolumn{1}{c}{}&
\multicolumn{3}{c}{\begin{tabular}[c]{@{}c@{}}\scriptsize{withPoolWhen}\\\scriptsize{XYIntOrLFixed-LFixed}\end{tabular}} \\
\cline{2-4}
\cline{6-8}
\cline{10-12}
\cline{14-16}
Instance&   BestSol & Time(s) &Nodes  & &BestSol & Time(s) &Nodes & &BestSol & Time(s) &Nodes& &BestSol & Time(s) &Nodes\\
\hline
bmilplib-110-10           & -177.67  & 416.71             & 121469   &  & -177.67  & 149.10  & 93801    &  & -177.67  & 155.24  & 75499   &  & -177.67  & 153.82  & 75499   \\
bmilplib-110-1            & -181.67  & 5.52               & 2892     &  & -181.67  & 5.42    & 2878     &  & -181.67  & 5.32    & 2806    &  & -181.67  & 5.35    & 2806    \\
bmilplib-110-2            & -110.67  & 7.04               & 4279     &  & -110.67  & 6.50    & 4272     &  & -110.67  & 6.49    & 4174    &  & -110.67  & 6.48    & 4174    \\
bmilplib-110-3            & -215.16  & 6.11               & 3599     &  & -215.16  & 5.42    & 3556     &  & -215.16  & 5.55    & 3471    &  & -215.16  & 5.58    & 3471    \\
bmilplib-110-4            & -197.29  & 4.64               & 1827     &  & -197.29  & 3.54    & 1801     &  & -197.29  & 3.70    & 1751    &  & -197.29  & 3.69    & 1751    \\
bmilplib-110-6            & -148.25  & 19.75              & 9578     &  & -148.25  & 14.14   & 9392     &  & -148.25  & 14.88   & 8770    &  & -148.25  & 15.06   & 8770    \\
bmilplib-110-7            & -160.86  & 4.30               & 2291     &  & -160.86  & 4.01    & 2281     &  & -160.86  & 4.86    & 2189    &  & -160.86  & 4.94    & 2189    \\
bmilplib-110-8            & -155.00  & 25.37              & 12659    &  & -155.00  & 20.65   & 12498    &  & -155.00  & 20.78   & 11663   &  & -155.00  & 20.78   & 11663   \\
bmilplib-110-9            & -192.92  & 3.33               & 2177     &  & -192.92  & 3.28    & 2171     &  & -192.92  & 3.63    & 2151    &  & -192.92  & 3.66    & 2151    \\
bmilplib-160-10           & -189.82  & 40.10              & 9425     &  & -189.82  & 31.24   & 9344     &  & -189.82  & 33.76   & 9071    &  & -189.82  & 33.42   & 9071    \\
bmilplib-160-1            & -165.00  & 26.18              & 8157     &  & -165.00  & 24.18   & 8145     &  & -165.00  & 25.67   & 7901    &  & -165.00  & 25.70   & 7901    \\
bmilplib-160-2            & -178.24  & 33.86              & 8719     &  & -178.24  & 29.64   & 8710     &  & -178.24  & 30.82   & 8635    &  & -178.24  & 30.59   & 8635    \\
bmilplib-160-3            & -174.94  & 74.85              & 15456    &  & -174.94  & 54.62   & 15427    &  & -174.94  & 65.13   & 15002   &  & -174.94  & 63.37   & 15002   \\
bmilplib-160-4            & -135.83  & 69.61              & 15513    &  & -135.83  & 50.86   & 15354    &  & -135.83  & 51.56   & 14772   &  & -135.83  & 51.02   & 14772   \\
bmilplib-160-5            & -140.78  & 45.58              & 5362     &  & -140.78  & 19.35   & 5178     &  & -140.78  & 20.58   & 4948    &  & -140.78  & 20.52   & 4948    \\
bmilplib-160-6            & -111.00  & 27.27              & 8878     &  & -111.00  & 25.08   & 8864     &  & -111.00  & 24.82   & 8676    &  & -111.00  & 24.66   & 8676    \\
bmilplib-160-7            & -96.00   & 67.91              & 17668    &  & -96.00   & 50.97   & 17430    &  & -96.00   & 52.58   & 16560   &  & -96.00   & 51.20   & 16560   \\
bmilplib-160-8            & -181.40  & 9.50               & 3564     &  & -181.40  & 9.32    & 3540     &  & -181.40  & 10.45   & 3444    &  & -181.40  & 10.49   & 3444    \\
bmilplib-160-9            & -207.50  & 15.01              & 4754     &  & -207.50  & 14.90   & 4728     &  & -207.50  & 15.87   & 4607    &  & -207.50  & 15.85   & 4607    \\
bmilplib-210-10           & -130.59  & 66.32              & 12596    &  & -130.59  & 61.00   & 12497    &  & -130.59  & 61.65   & 12176   &  & -130.59  & 61.81   & 12176   \\
bmilplib-210-1            & -136.80  & 43.30              & 7608     &  & -136.80  & 40.04   & 7598     &  & -136.80  & 41.01   & 7429    &  & -136.80  & 41.96   & 7429    \\
bmilplib-210-2            & -117.80  & 132.27             & 23482    &  & -117.80  & 109.56  & 23096    &  & -117.80  & 118.80  & 22294   &  & -117.80  & 113.38  & 22294   \\
bmilplib-210-3            & -130.80  & 60.91              & 9404     &  & -130.80  & 47.44   & 9165     &  & -130.80  & 48.09   & 8848    &  & -130.80  & 48.19   & 8848    \\
bmilplib-210-4            & -162.20  & 27.29              & 4812     &  & -162.20  & 25.92   & 4806     &  & -162.20  & 26.63   & 4687    &  & -162.20  & 26.81   & 4687    \\
bmilplib-210-5            & -134.00  & 117.77             & 21646    &  & -134.00  & 105.03  & 21552    &  & -134.00  & 106.92  & 21058   &  & -134.00  & 108.22  & 21058   \\
bmilplib-210-6            & -125.43  & 261.84             & 40736    &  & -125.43  & 197.88  & 39941    &  & -125.43  & 213.31  & 38443   &  & -125.43  & 201.40  & 38443   \\
bmilplib-210-7            & -169.73  & 82.57              & 14351    &  & -169.73  & 77.62   & 14305    &  & -169.73  & 77.43   & 13960   &  & -169.73  & 77.84   & 13960   \\
bmilplib-210-8            & -101.46  & 67.82              & 11406    &  & -101.46  & 58.14   & 11347    &  & -101.46  & 56.85   & 11105   &  & -101.46  & 61.79   & 11105   \\
bmilplib-210-9            & -184.00  & 912.62             & 143788   &  & -184.00  & 859.46  & 143665   &  & -184.00  & 849.38  & 142294  &  & -184.00  & 879.61  & 142294  \\
bmilplib-260-10           & -151.73  & 559.37             & 63243    &  & -151.73  & 502.94  & 63026    &  & -151.73  & 514.35  & 61767   &  & -151.73  & 546.55  & 61767   \\
bmilplib-260-1            & -139.00  & 82.50              & 11269    &  & -139.00  & 81.74   & 11242    &  & -139.00  & 81.82   & 10980   &  & -139.00  & 85.81   & 10980   \\
bmilplib-260-2            & -82.62   & 125.39             & 15896    &  & -82.62   & 109.75  & 15712    &  & -82.62   & 107.89  & 15287   &  & -82.62   & 109.75  & 15287   \\
bmilplib-260-3            & -144.25  & 74.50              & 8865     &  & -144.25  & 69.26   & 8859     &  & -144.25  & 74.50   & 8769    &  & -144.25  & 71.54   & 8769    \\
bmilplib-260-4            & -117.33  & 284.35             & 34237    &  & -117.33  & 260.01  & 33931    &  & -117.33  & 272.56  & 32787   &  & -117.33  & 259.48  & 32787   \\
bmilplib-260-5            & -121.00  & 181.75             & 22295    &  & -121.00  & 173.32  & 22121    &  & -121.00  & 167.03  & 21542   &  & -121.00  & 166.03  & 21542   \\
bmilplib-260-6            & -124.00  & 240.41             & 26166    &  & -124.00  & 209.79  & 26023    &  & -124.00  & 195.61  & 25362   &  & -124.00  & 197.04  & 25362   \\
bmilplib-260-7            & -137.80  & 357.63             & 41519    &  & -137.80  & 318.76  & 41180    &  & -137.80  & 342.97  & 40274   &  & -137.80  & 312.12  & 40274   \\
bmilplib-260-8            & -119.89  & 87.09              & 10445    &  & -119.89  & 76.70   & 10332    &  & -119.89  & 75.72   & 10025   &  & -119.89  & 87.01   & 10025   \\
bmilplib-260-9            & -160.00  & 318.76             & 34545    &  & -160.00  & 271.18  & 34236    &  & -160.00  & 273.16  & 33468   &  & -160.00  & 273.66  & 33468   \\
bmilplib-310-10           & -141.86  & 119.62             & 10032    &  & -141.86  & 114.38  & 10033    &  & -141.86  & 110.92  & 9900    &  & -141.86  & 121.07  & 9900    \\
bmilplib-310-1            & -117.00  & 363.13             & 29350    &  & -117.00  & 329.80  & 29055    &  & -117.00  & 346.52  & 28330   &  & -117.00  & 329.71  & 28330   \\
bmilplib-310-2            & -105.00  & 519.72             & 45235    &  & -105.00  & 457.69  & 44841    &  & -105.00  & 519.36  & 43399   &  & -105.00  & 497.66  & 43399   \\
bmilplib-310-3            & -127.52  & 781.32             & 67649    &  & -127.52  & 783.89  & 67441    &  & -127.52  & 803.04  & 66410   &  & -127.52  & 777.74  & 66410   \\
bmilplib-310-4            & -147.78  & 584.37             & 49620    &  & -147.78  & 572.76  & 48950    &  & -147.78  & 564.98  & 47711   &  & -147.78  & 569.43  & 47711   \\
bmilplib-310-5            & -161.45  & 392.22             & 32631    &  & -161.45  & 382.13  & 32454    &  & -161.45  & 396.10  & 31782   &  & -161.45  & 366.90  & 31782   \\
bmilplib-310-6            & -141.18  & 1191.16            & 103385   &  & -141.18  & 1264.60 & 103214   &  & -141.18  & 1229.32 & 101904  &  & -141.18  & 1191.51 & 101904  \\
bmilplib-310-7            & -142.00  & 1132.52            & 104378   &  & -142.00  & 1199.47 & 104166   &  & -142.00  & 1142.79 & 102000  &  & -142.00  & 1129.24 & 102000  \\
bmilplib-310-8            & -115.34  & 129.19             & 11336    &  & -115.34  & 114.80  & 11293    &  & -115.34  & 137.89  & 11109   &  & -115.34  & 127.94  & 11109   \\
bmilplib-310-9            & -115.65  & 262.22             & 20893    &  & -115.65  & 248.24  & 20838    &  & -115.65  & 254.18  & 20490   &  & -115.65  & 255.27  & 20490   \\
bmilplib-360-10           & -108.59  & 254.46             & 14127    &  & -108.59  & 242.94  & 14064    &  & -108.59  & 261.28  & 13697   &  & -108.59  & 234.92  & 13697   \\
bmilplib-360-1            & -133.00  & 1390.85            & 77371    &  & -133.00  & 1297.32 & 77066    &  & -133.00  & 1380.33 & 75421   &  & -133.00  & 1353.62 & 75421   \\
bmilplib-360-2            & -138.44  & 980.62             & 54874    &  & -138.44  & 999.50  & 54354    &  & -138.44  & 953.84  & 52919   &  & -138.44  & 965.52  & 52919   \\
bmilplib-360-3            & -131.00  & 689.69             & 41459    &  & -131.00  & 624.78  & 41302    &  & -131.00  & 643.52  & 40487   &  & -131.00  & 832.88  & 40487   \\
bmilplib-360-4            & -119.00  & 334.15             & 18870    &  & -119.00  & 338.49  & 18813    &  & -119.00  & 331.26  & 18293   &  & -119.00  & 332.90  & 18293   \\
bmilplib-360-5            & -164.26  & 527.22             & 30465    &  & -164.26  & 484.66  & 30352    &  & -164.26  & 526.38  & 29947   &  & -164.26  & 618.86  & 29947   \\
bmilplib-360-6            & -110.12  & 1147.18            & 70479    &  & -110.12  & 1018.82 & 70283    &  & -110.12  & 1131.67 & 68863   &  & -110.12  & 1181.02 & 68863   \\
bmilplib-360-7            & -105.00  & 554.89             & 32224    &  & -105.00  & 517.46  & 31884    &  & -105.00  & 479.93  & 30900   &  & -105.00  & 634.17  & 30900   \\
bmilplib-360-8            & -98.25   & 438.62             & 23402    &  & -98.25   & 416.97  & 23337    &  & -98.25   & 423.75  & 22857   &  & -98.25   & 362.89  & 22857   \\
bmilplib-360-9            & -127.22  & 708.99             & 41321    &  & -127.22  & 746.66  & 41235    &  & -127.22  & 656.26  & 40329   &  & -127.22  & 736.16  & 40329   \\
bmilplib-410-10           & -153.37  & 2756.08            & 103891   &  & -153.37  & 2879.03 & 103400   &  & -153.37  & 2624.30 & 101447  &  & -153.37  & 2729.77 & 101447  \\
bmilplib-410-1            & -103.50  & 614.03             & 21120    &  & -103.50  & 589.62  & 21088    &  & -103.50  & 590.37  & 20634   &  & -103.50  & 553.76  & 20634   \\
bmilplib-410-2            & -108.59  & 846.60             & 32377    &  & -108.59  & 748.06  & 32251    &  & -108.59  & 746.96  & 31603   &  & -108.59  & 840.56  & 31603   \\
bmilplib-410-3            & -96.24   & 877.76             & 33590    &  & -96.24   & 777.04  & 33473    &  & -96.24   & 874.84  & 32882   &  & -96.24   & 871.58  & 32882   \\
bmilplib-410-4            & -119.50  & 957.38             & 39535    &  & -119.50  & 1041.97 & 39371    &  & -119.50  & 944.66  & 38489   &  & -119.50  & 1031.12 & 38489   \\
bmilplib-410-5            & -119.22  & 677.99             & 24188    &  & -119.22  & 681.99  & 24163    &  & -119.22  & 583.64  & 23852   &  & -119.22  & 678.49  & 23852   \\
bmilplib-410-6            & -151.31  & 291.60             & 11213    &  & -151.31  & 333.36  & 11193    &  & -151.31  & 300.03  & 11130   &  & -151.31  & 322.88  & 11130   \\
bmilplib-410-7            & -123.00  & 514.64             & 22684    &  & -123.00  & 601.78  & 22628    &  & -123.00  & 588.83  & 22146   &  & -123.00  & 560.34  & 22146   \\
bmilplib-410-8            & -125.78  & 1679.79            & 64013    &  & -125.78  & 1692.35 & 63421    &  & -125.78  & 1462.52 & 62218   &  & -125.78  & 1547.88 & 62218   \\
bmilplib-410-9            & -100.77  & 561.87             & 20968    &  & -100.77  & 607.44  & 20920    &  & -100.77  & 539.37  & 20400   &  & -100.77  & 531.88  & 20400   \\
bmilplib-460-10           & -102.51  & 1837.01            & 56744    &  & -102.51  & 1701.70 & 56265    &  & -102.51  & 1964.44 & 55024   &  & -102.51  & 1857.94 & 55024   \\
bmilplib-460-1            & -97.59   & 2829.00            & 87769    &  & -97.59   & 2938.02 & 87709    &  & -97.59   & 2843.15 & 86632   &  & -97.59   & 2822.97 & 86632   \\
bmilplib-460-2            & -139.00  & 652.85             & 16943    &  & -139.00  & 608.40  & 16863    &  & -139.00  & 651.97  & 16623   &  & -139.00  & 625.95  & 16623   \\
bmilplib-460-3            & -86.50   & 1976.55            & 58270    &  & -86.50   & 2102.77 & 58229    &  & -86.50   & 1941.05 & 57395   &  & -86.50   & 1995.92 & 57395   \\
bmilplib-460-4            & -107.03  & 3577.74            & 97665    &  & -107.03  & 3328.91 & 97160    &  & -107.03  & 3160.13 & 95035   &  & -107.03  & 3291.28 & 95035   \\
bmilplib-460-5            & -100.50  & 1519.85            & 41863    &  & -100.50  & 1619.91 & 41796    &  & -100.50  & 1517.44 & 41109   &  & -100.50  & 1424.54 & 41109   \\
bmilplib-460-6            & -107.00  & 1603.92            & 46833    &  & -107.00  & 1483.17 & 46732    &  & -107.00  & 1528.18 & 46076   &  & -107.00  & 1635.42 & 46076   \\
bmilplib-460-7            & -83.75   & 1813.99            & 49915    &  & -83.75   & 1728.92 & 49717    &  & -83.75   & 1826.46 & 48897   &  & -83.75   & 1601.34 & 48897   \\
bmilplib-460-8            & -115.39  & 983.08             & 28147    &  & -115.39  & 911.51  & 28075    &  & -115.39  & 901.08  & 27572   &  & -115.39  & 982.30  & 27572   \\
bmilplib-460-9            & -128.70  & 3006.84            & 88287    &  & -128.70  & 2715.98 & 87796    &  & -128.70  & 3043.95 & 85977   &  & -128.70  & 2830.59 & 85977   \\
bmilplib-60-10            & -186.21  & 7.00               & 6136     &  & -186.21  & 6.91    & 6134     &  & -186.21  & 6.70    & 5922    &  & -186.21  & 6.93    & 5922    \\
bmilplib-60-1             & -153.20  & 7.31               & 5060     &  & -153.20  & 6.50    & 5041     &  & -153.20  & 6.48    & 4877    &  & -153.20  & 6.71    & 4877    \\
bmilplib-60-5             & -116.40  & 13.21              & 15368    &  & -116.40  & 8.65    & 13984    &  & -116.40  & 7.98    & 11202   &  & -116.40  & 8.15    & 11202   \\
bmilplib-60-6             & -187.31  & 7.57               & 7310     &  & -187.31  & 7.34    & 7304     &  & -187.31  & 7.22    & 7016    &  & -187.31  & 7.62    & 7016    \\
bmilplib-60-8             & -232.12  & 3.66               & 3700     &  & -232.12  & 3.33    & 3683     &  & -232.12  & 3.16    & 3570    &  & -232.12  & 3.36    & 3570    \\
bmilplib-60-9             & -136.50  & 22.60              & 28947    &  & -136.50  & 20.22   & 28603    &  & -136.50  & 19.79   & 26792   &  & -136.50  & 19.61   & 26792   \\
miblp-20-15-50-0110-10-10 & -206.00  & 5.32               & 3150     &  & -206.00  & 0.97    & 1414     &  & -206.00  & 1.03    & 423     &  & -206.00  & 1.06    & 423     \\
miblp-20-15-50-0110-10-2  & -398.00  & 427.10             & 229000   &  & -398.00  & 11.12   & 27625    &  & -398.00  & 9.11    & 2450    &  & -398.00  & 9.09    & 2450    \\
miblp-20-15-50-0110-10-3  & -42.00   & 6.10               & 3151     &  & -42.00   & 0.71    & 1343     &  & -42.00   & 0.62    & 267     &  & -42.00   & 0.62    & 267     \\
miblp-20-15-50-0110-10-6  & -246.00  & 5.47               & 1183     &  & -246.00  & 2.19    & 853      &  & -246.00  & 4.47    & 340     &  & -246.00  & 4.56    & 340     \\
miblp-20-15-50-0110-10-9  & -635.00  & 16.24              & 7425     &  & -635.00  & 7.86    & 6457     &  & -635.00  & 14.46   & 2380    &  & -635.00  & 14.05   & 2380    \\
miblp-20-20-50-0110-10-10 & -441.00  & 3568.40            & 1159333  &  & -441.00  & 526.56  & 639146   &  & -441.00  & 717.56  & 134585  &  & -441.00  & 692.00  & 134585  \\
miblp-20-20-50-0110-10-1  & -353.00  & \textgreater{}3600 & 686972   &  & -359.00  & 260.31  & 423149   &  & -359.00  & 231.63  & 96141   &  & -359.00  & 228.67  & 96141   \\
miblp-20-20-50-0110-10-2  & -659.00  & 5.16               & 7729     &  & -659.00  & 2.78    & 7547     &  & -659.00  & 3.80    & 3191    &  & -659.00  & 3.69    & 3191    \\
miblp-20-20-50-0110-10-3  & -618.00  & 14.11              & 42906    &  & -618.00  & 10.74   & 38188    &  & -618.00  & 13.14   & 20788   &  & -618.00  & 13.18   & 20788   \\
\hline
\end{tabular}
}
\end{table}
%%%%%%%%%%%%%%%%%%
%table for figure 4b (continued)
\addtocounter{table}{-1}
\begin{table}[h!]
\caption{Detailed results of Figure~\ref{fig:linkingSolutionPooParWithLinkParWithTimeMore5Sec} (continued)}
\label{tab:fig4bContinue}
\resizebox{\columnwidth}{!}{
%\begin{tabular}{lccccccccccccccc}
\begin{tabular}{llllllllllllllll}
\hline
\multicolumn{1}{c}{}&
\multicolumn{3}{c}{\begin{tabular}[c]{@{}c@{}}\scriptsize{withoutPoolWhen}\\\scriptsize{XYInt-LFixed}\end{tabular}}&
\multicolumn{1}{c}{}&
\multicolumn{3}{c}{\begin{tabular}[c]{@{}c@{}}\scriptsize{withPoolWhen}\\\scriptsize{XYInt-LFixed}\end{tabular}}&
\multicolumn{1}{c}{}&
\multicolumn{3}{c}{\begin{tabular}[c]{@{}c@{}}\scriptsize{withoutPoolWhen}\\\scriptsize{XYIntOrLFixed-LFixed}\end{tabular}}&
\multicolumn{1}{c}{}&
\multicolumn{3}{c}{\begin{tabular}[c]{@{}c@{}}\scriptsize{withPoolWhen}\\\scriptsize{XYIntOrLFixed-LFixed}\end{tabular}} \\
\cline{2-4}
\cline{6-8}
\cline{10-12}
\cline{14-16}
Instance&   BestSol & Time(s) &Nodes  & &BestSol & Time(s) &Nodes & &BestSol & Time(s) &Nodes& &BestSol & Time(s) &Nodes\\
\hline
miblp-20-20-50-0110-10-4  & -604.00  & \textgreater{}3600 & 1351975  &  & -604.00  & 3405.57 & 6452671  &  & -604.00  & 3089.99 & 830808  &  & -604.00  & 3145.16 & 830808  \\
miblp-20-20-50-0110-10-7  & -650.00  & \textgreater{}3600 & 7509639  &  & -683.00  & 3184.99 & 11502091 &  & -683.00  & 2009.54 & 3003967 &  & -683.00  & 1887.37 & 3003967 \\
miblp-20-20-50-0110-10-8  & -667.00  & 155.00             & 39744    &  & -667.00  & 40.78   & 30873    &  & -667.00  & 87.06   & 12857   &  & -667.00  & 80.45   & 12857   \\
miblp-20-20-50-0110-10-9  & -256.00  & 77.53              & 127374   &  & -256.00  & 23.97   & 76055    &  & -256.00  & 21.78   & 31245   &  & -256.00  & 20.99   & 31245   \\
miblp-20-20-50-0110-15-1  & -289.00  & \textgreater{}3600 & 474713   &  & -450.00  & 65.16   & 49137    &  & -450.00  & 65.21   & 3506    &  & -450.00  & 59.38   & 3506    \\
miblp-20-20-50-0110-15-2  & -645.00  & \textgreater{}3600 & 1897076  &  & -645.00  & 96.13   & 346065   &  & -645.00  & 49.73   & 17251   &  & -645.00  & 50.14   & 17251   \\
miblp-20-20-50-0110-15-3  & -593.00  & \textgreater{}3600 & 231893   &  & -593.00  & 65.25   & 42877    &  & -593.00  & 72.10   & 3081    &  & -593.00  & 71.15   & 3081    \\
miblp-20-20-50-0110-15-4  & -323.00  & \textgreater{}3600 & 205392   &  & -441.00  & 36.99   & 29904    &  & -441.00  & 43.47   & 1625    &  & -441.00  & 43.00   & 1625    \\
miblp-20-20-50-0110-15-5  & -75.00   & \textgreater{}3600 & 103173   &  & -379.00  & 615.16  & 205025   &  & -379.00  & 644.67  & 16715   &  & -379.00  & 651.28  & 16715   \\
miblp-20-20-50-0110-15-6  & -596.00  & \textgreater{}3600 & 317849   &  & -596.00  & 17.87   & 29923    &  & -596.00  & 18.98   & 1657    &  & -596.00  & 17.31   & 1657    \\
miblp-20-20-50-0110-15-7  & -471.00  & \textgreater{}3600 & 1084548  &  & -471.00  & 125.24  & 241285   &  & -471.00  & 126.10  & 13405   &  & -471.00  & 111.02  & 13405   \\
miblp-20-20-50-0110-15-8  & -290.00  & \textgreater{}3600 & 4898557  &  & -370.00  & 138.34  & 579309   &  & -370.00  & 39.01   & 21589   &  & -370.00  & 39.12   & 21589   \\
miblp-20-20-50-0110-15-9  & -584.00  & 20.00              & 12193    &  & -584.00  & 1.96    & 4072     &  & -584.00  & 2.07    & 582     &  & -584.00  & 2.00    & 582     \\
miblp-20-20-50-0110-5-13  & -519.00  & 2582.26            & 10718411 &  & -519.00  & 2307.90 & 10196998 &  & -519.00  & 2516.55 & 7128138 &  & -519.00  & 2464.25 & 7128138 \\
miblp-20-20-50-0110-5-15  & -617.00  & 2170.42            & 8462491  &  & -617.00  & 1219.33 & 6630921  &  & -617.00  & 1297.16 & 4018058 &  & -617.00  & 1310.67 & 4018058 \\
miblp-20-20-50-0110-5-16  & -833.00  & 6.04               & 20129    &  & -833.00  & 4.86    & 19013    &  & -833.00  & 7.36    & 17680   &  & -833.00  & 6.60    & 17680   \\
miblp-20-20-50-0110-5-17  & -944.00  & 3.71               & 20049    &  & -944.00  & 3.99    & 21541    &  & -944.00  & 5.21    & 17679   &  & -944.00  & 4.78    & 17679   \\
miblp-20-20-50-0110-5-19  & -431.00  & 52.75              & 197157   &  & -431.00  & 27.52   & 147268   &  & -431.00  & 32.98   & 77256   &  & -431.00  & 26.56   & 77256   \\
miblp-20-20-50-0110-5-1   & -548.00  & 16.54              & 69127    &  & -548.00  & 12.98   & 64067    &  & -548.00  & 17.63   & 42364   &  & -548.00  & 14.50   & 42364   \\
miblp-20-20-50-0110-5-20  & -438.00  & 17.24              & 80526    &  & -438.00  & 15.81   & 76091    &  & -438.00  & 18.83   & 60944   &  & -438.00  & 17.30   & 60944   \\
miblp-20-20-50-0110-5-6   & -1061.00 & 66.85              & 301416   &  & -1061.00 & 58.74   & 284550   &  & -1061.00 & 63.82   & 222566  &  & -1061.00 & 63.70   & 222566  \\
lseu-0.100000             & 1120.00  & 248.50             & 1071409  &  & 1120.00  & 248.54  & 1071409  &  & 1120.00  & 262.55  & 1003976 &  & 1120.00  & 270.78  & 1003976 \\
lseu-0.900000             & 5838.00  & \textgreater{}3600 & 1631763  &  & 5838.00  & 1063.02 & 4718749  &  & 5838.00  & 14.12   & 1023    &  & 5838.00  & 14.10   & 1023    \\
p0033-0.500000            & 3095.00  & 6.08               & 33614    &  & 3095.00  & 6.24    & 33614    &  & 3095.00  & 10.55   & 20695   &  & 3095.00  & 10.48   & 20695   \\
p0033-0.900000            & 4679.00  & 14.66              & 23699    &  & 4679.00  & 0.65    & 3455     &  & 4679.00  & 0.06    & 27      &  & 4679.00  & 0.06    & 27      \\
p0201-0.900000            & 15025.00 & \textgreater{}3600 & 1037538  &  & 15025.00 & 20.58   & 18801    &  & 15025.00 & 6.66    & 2481    &  & 15025.00 & 6.62    & 2481    \\
stein27-0.500000          & 19.00    & 7.94               & 22537    &  & 19.00    & 7.36    & 21515    &  & 19.00    & 5.82    & 12115   &  & 19.00    & 5.88    & 12115   \\
stein27-0.900000          & 24.00    & 29.12              & 36927    &  & 24.00    & 1.25    & 4445     &  & 24.00    & 0.02    & 15      &  & 24.00    & 0.02    & 15      \\
stein45-0.100000          & 30.00    & 49.86              & 89035    &  & 30.00    & 50.47   & 89035    &  & 30.00    & 50.33   & 89035   &  & 30.00    & 50.19   & 89035   \\
stein45-0.500000          & 32.00    & 640.82             & 963098   &  & 32.00    & 635.22  & 952123   &  & 32.00    & 516.60  & 640308  &  & 32.00    & 519.59  & 640308  \\
stein45-0.900000          & 40.00    & \textgreater{}3600 & 2651427  &  & 40.00    & 85.92   & 103661   &  & 40.00    & 0.16    & 63      &  & 40.00    & 0.14    & 63     \\
\hline
\end{tabular}
}
\end{table}
%%%%%%%%%%%%%%%%%%
%table for figure 5
\begin{table}[h!]
\caption{Detailed results of Figure~\ref{fig:heuristicsParTimeMore5Sec}}
\label{tab:fig5}
\centering
\resizebox{\columnwidth}{!}{
%\begin{tabular}{lccccccccccccccc}
\begin{tabular}{llllllllllllllll}
\hline
\multicolumn{1}{c}{}&
\multicolumn{3}{c}{noHeuristics}&
\multicolumn{1}{c}{}&
\multicolumn{3}{c}{impObjectiveCut}&
\multicolumn{1}{c}{}&
\multicolumn{3}{c}{secondLevelPriority}&
\multicolumn{1}{c}{}&
\multicolumn{3}{c}{weightedSums} \\
\cline{2-4}
\cline{6-8}
\cline{10-12}
\cline{14-16}
Instance&   BestSol & Time(s) &Nodes  & &BestSol & Time(s) &Nodes & &BestSol & Time(s) &Nodes& &BestSol & Time(s) &Nodes\\
\hline
bmilplib-110-10           & -177.67  & 116.32  & 55177   &  & -177.67  & 117.32  & 55177   &  & -177.67  & 146.32  & 55177   &  & -177.67  & 1024.90            & 55177   \\
bmilplib-110-1            & -181.67  & 0.93    & 306     &  & -181.67  & 0.96    & 306     &  & -181.67  & 2.90    & 306     &  & -181.67  & 67.04              & 306     \\
bmilplib-110-2            & -110.67  & 1.16    & 303     &  & -110.67  & 1.26    & 303     &  & -110.67  & 1.64    & 303     &  & -110.67  & 25.89              & 303     \\
bmilplib-110-3            & -215.16  & 0.89    & 360     &  & -215.16  & 1.50    & 361     &  & -215.16  & 1.93    & 360     &  & -215.16  & 30.52              & 360     \\
bmilplib-110-4            & -197.29  & 1.34    & 148     &  & -197.29  & 1.54    & 148     &  & -197.29  & 2.79    & 148     &  & -197.29  & 68.13              & 148     \\
bmilplib-110-6            & -148.25  & 3.84    & 1448    &  & -148.25  & 4.24    & 1448    &  & -148.25  & 4.84    & 1446    &  & -148.25  & 59.02              & 1446    \\
bmilplib-110-7            & -160.86  & 0.91    & 205     &  & -160.86  & 2.19    & 205     &  & -160.86  & 1.98    & 205     &  & -160.86  & 52.88              & 207     \\
bmilplib-110-8            & -155.00  & 7.66    & 2274    &  & -155.00  & 7.68    & 2274    &  & -155.00  & 10.06   & 2274    &  & -155.00  & 73.62              & 2274    \\
bmilplib-110-9            & -192.92  & 0.39    & 146     &  & -192.92  & 1.33    & 146     &  & -192.92  & 1.15    & 146     &  & -192.92  & 16.48              & 154     \\
bmilplib-160-10           & -189.82  & 7.73    & 728     &  & -189.82  & 7.89    & 728     &  & -189.82  & 11.28   & 728     &  & -189.82  & 65.95              & 728     \\
bmilplib-160-1            & -165.00  & 4.76    & 881     &  & -165.00  & 4.76    & 881     &  & -165.00  & 5.82    & 893     &  & -165.00  & 51.48              & 893     \\
bmilplib-160-2            & -178.24  & 5.95    & 507     &  & -178.24  & 6.64    & 507     &  & -178.24  & 6.96    & 507     &  & -178.24  & 53.77              & 507     \\
bmilplib-160-3            & -174.94  & 13.71   & 1102    &  & -174.94  & 13.57   & 1102    &  & -174.94  & 15.78   & 1102    &  & -174.94  & 70.46              & 1102    \\
bmilplib-160-4            & -135.83  & 19.13   & 2447    &  & -135.83  & 18.91   & 2463    &  & -135.83  & 24.00   & 2447    &  & -135.83  & 116.48             & 2422    \\
bmilplib-160-5            & -140.78  & 7.72    & 668     &  & -140.78  & 8.94    & 668     &  & -140.78  & 10.34   & 668     &  & -140.78  & 56.74              & 668     \\
bmilplib-160-6            & -111.00  & 5.02    & 627     &  & -111.00  & 5.03    & 627     &  & -111.00  & 5.40    & 627     &  & -111.00  & 19.14              & 627     \\
bmilplib-160-7            & -96.00   & 18.72   & 2855    &  & -96.00   & 17.83   & 2855    &  & -96.00   & 19.22   & 2855    &  & -96.00   & 74.44              & 2855    \\
bmilplib-160-8            & -181.40  & 1.74    & 311     &  & -181.40  & 2.12    & 311     &  & -181.40  & 2.44    & 311     &  & -181.40  & 28.91              & 311     \\
bmilplib-160-9            & -207.50  & 2.02    & 268     &  & -207.50  & 2.85    & 268     &  & -207.50  & 2.98    & 268     &  & -207.50  & 40.33              & 268     \\
bmilplib-210-10           & -130.59  & 8.59    & 1100    &  & -130.59  & 8.66    & 1100    &  & -130.59  & 10.02   & 1100    &  & -130.59  & 73.44              & 1100    \\
bmilplib-210-1            & -136.80  & 6.67    & 550     &  & -136.80  & 6.68    & 550     &  & -136.80  & 7.59    & 550     &  & -136.80  & 80.52              & 550     \\
bmilplib-210-2            & -117.80  & 19.78   & 2306    &  & -117.80  & 20.49   & 2306    &  & -117.80  & 22.29   & 2298    &  & -117.80  & 149.96             & 2308    \\
bmilplib-210-3            & -130.80  & 12.93   & 1380    &  & -130.80  & 13.04   & 1386    &  & -130.80  & 16.16   & 1380    &  & -130.80  & 218.18             & 1380    \\
bmilplib-210-4            & -162.20  & 3.64    & 309     &  & -162.20  & 5.41    & 309     &  & -162.20  & 4.63    & 309     &  & -162.20  & 41.06              & 309     \\
bmilplib-210-5            & -134.00  & 20.03   & 2079    &  & -134.00  & 20.03   & 2079    &  & -134.00  & 22.31   & 2079    &  & -134.00  & 161.23             & 2079    \\
bmilplib-210-6            & -125.43  & 45.28   & 4875    &  & -125.43  & 46.46   & 4875    &  & -125.43  & 52.83   & 4878    &  & -125.43  & 230.84             & 4875    \\
bmilplib-210-7            & -169.73  & 11.46   & 1181    &  & -169.73  & 12.35   & 1153    &  & -169.73  & 14.19   & 1181    &  & -169.73  & 148.58             & 1152    \\
bmilplib-210-8            & -101.46  & 9.60    & 942     &  & -101.46  & 12.36   & 976     &  & -101.46  & 12.71   & 942     &  & -101.46  & 255.24             & 942     \\
bmilplib-210-9            & -184.00  & 240.14  & 9466    &  & -184.00  & 242.65  & 9466    &  & -184.00  & 267.90  & 9466    &  & -184.00  & 1639.04            & 9466    \\
bmilplib-260-10           & -151.73  & 73.10   & 4716    &  & -151.73  & 75.91   & 4716    &  & -151.73  & 119.40  & 4716    &  & -151.73  & 783.50             & 4716    \\
bmilplib-260-1            & -139.00  & 10.10   & 887     &  & -139.00  & 11.58   & 887     &  & -139.00  & 11.79   & 887     &  & -139.00  & 135.30             & 887     \\
bmilplib-260-2            & -82.62   & 18.07   & 1607    &  & -82.62   & 18.48   & 1607    &  & -82.62   & 29.56   & 1607    &  & -82.62   & 195.56             & 1607    \\
bmilplib-260-3            & -144.25  & 8.70    & 518     &  & -144.25  & 9.34    & 518     &  & -144.25  & 9.60    & 518     &  & -144.25  & 105.89             & 518     \\
bmilplib-260-4            & -117.33  & 66.89   & 4426    &  & -117.33  & 62.45   & 4426    &  & -117.33  & 80.66   & 4426    &  & -117.33  & 671.57             & 4426    \\
bmilplib-260-5            & -121.00  & 29.19   & 2165    &  & -121.00  & 27.30   & 2165    &  & -121.00  & 30.22   & 2165    &  & -121.00  & 212.90             & 2165    \\
bmilplib-260-6            & -124.00  & 40.59   & 2420    &  & -124.00  & 39.48   & 2420    &  & -124.00  & 42.70   & 2420    &  & -124.00  & 426.31             & 2420    \\
bmilplib-260-7            & -137.80  & 44.45   & 3200    &  & -137.80  & 45.80   & 3200    &  & -137.80  & 48.73   & 3200    &  & -137.80  & 267.10             & 3200    \\
bmilplib-260-8            & -119.89  & 10.92   & 1025    &  & -119.89  & 10.46   & 1025    &  & -119.89  & 11.77   & 1025    &  & -119.89  & 102.82             & 1025    \\
bmilplib-260-9            & -160.00  & 36.36   & 2526    &  & -160.00  & 38.08   & 2526    &  & -160.00  & 39.54   & 2526    &  & -160.00  & 231.78             & 2526    \\
bmilplib-310-10           & -141.86  & 5.51    & 397     &  & -141.86  & 6.24    & 397     &  & -141.86  & 6.86    & 397     &  & -141.86  & 95.24              & 397     \\
bmilplib-310-1            & -117.00  & 44.79   & 2624    &  & -117.00  & 41.07   & 2624    &  & -117.00  & 44.15   & 2624    &  & -117.00  & 381.97             & 2624    \\
bmilplib-310-2            & -105.00  & 98.39   & 5372    &  & -105.00  & 94.00   & 5372    &  & -105.00  & 97.58   & 5372    &  & -105.00  & 385.78             & 5372    \\
bmilplib-310-3            & -127.52  & 149.07  & 7067    &  & -127.52  & 141.25  & 7067    &  & -127.52  & 151.24  & 7067    &  & -127.52  & 694.04             & 7067    \\
bmilplib-310-4            & -147.78  & 70.22   & 4767    &  & -147.78  & 73.75   & 4908    &  & -147.78  & 74.72   & 4863    &  & -147.78  & 510.45             & 4892    \\
bmilplib-310-5            & -161.45  & 34.22   & 1993    &  & -161.45  & 34.42   & 1993    &  & -161.45  & 35.32   & 1993    &  & -161.45  & 205.57             & 1993    \\
bmilplib-310-6            & -141.18  & 169.96  & 8300    &  & -141.18  & 172.88  & 8300    &  & -141.18  & 214.66  & 8300    &  & -141.18  & \textgreater{}3600 & 4459    \\
bmilplib-310-7            & -142.00  & 139.15  & 7263    &  & -142.00  & 140.50  & 7263    &  & -142.00  & 164.46  & 7263    &  & -142.00  & 1416.36            & 7263    \\
bmilplib-310-8            & -115.34  & 19.23   & 921     &  & -115.34  & 19.53   & 952     &  & -115.34  & 23.06   & 921     &  & -115.34  & 425.10             & 921     \\
bmilplib-310-9            & -115.65  & 32.31   & 1590    &  & -115.65  & 32.39   & 1590    &  & -115.65  & 34.56   & 1590    &  & -115.65  & 151.96             & 1590    \\
bmilplib-360-10           & -108.59  & 25.75   & 1106    &  & -108.59  & 23.23   & 1106    &  & -108.59  & 26.41   & 1102    &  & -108.59  & 156.13             & 1102    \\
bmilplib-360-1            & -133.00  & 158.38  & 4923    &  & -133.00  & 148.68  & 4920    &  & -133.00  & 188.86  & 4923    &  & -120.00  & \textgreater{}3600 & 1611    \\
bmilplib-360-2            & -138.44  & 148.90  & 4493    &  & -138.44  & 138.96  & 4493    &  & -138.44  & 160.91  & 4493    &  & -138.44  & 1961.79            & 4493    \\
bmilplib-360-3            & -131.00  & 65.41   & 2654    &  & -131.00  & 56.59   & 2654    &  & -131.00  & 66.60   & 2654    &  & -131.00  & 609.13             & 2654    \\
bmilplib-360-4            & -119.00  & 42.95   & 1564    &  & -119.00  & 44.65   & 1564    &  & -119.00  & 61.20   & 1564    &  & -119.00  & 825.58             & 1564    \\
bmilplib-360-5            & -164.26  & 44.45   & 1713    &  & -164.26  & 45.39   & 1713    &  & -164.26  & 41.95   & 1713    &  & -164.26  & 221.40             & 1713    \\
bmilplib-360-6            & -110.12  & 142.81  & 5520    &  & -110.12  & 127.35  & 5520    &  & -110.12  & 148.17  & 5520    &  & -110.12  & 2890.94            & 5556    \\
bmilplib-360-7            & -105.00  & 89.22   & 3346    &  & -105.00  & 88.94   & 3346    &  & -105.00  & 86.08   & 3346    &  & -105.00  & 357.79             & 3346    \\
bmilplib-360-8            & -98.25   & 44.85   & 1686    &  & -98.25   & 48.77   & 1686    &  & -98.25   & 72.01   & 1686    &  & -98.25   & 2774.87            & 1686    \\
bmilplib-360-9            & -127.22  & 50.24   & 2642    &  & -127.22  & 52.33   & 2642    &  & -127.22  & 58.90   & 2642    &  & -127.22  & 1191.63            & 2642    \\
bmilplib-410-10           & -153.37  & 258.89  & 7428    &  & -153.37  & 254.18  & 7428    &  & -153.37  & 278.87  & 7428    &  & -153.37  & 2497.78            & 7428    \\
bmilplib-410-1            & -103.50  & 87.94   & 1944    &  & -103.50  & 76.78   & 1944    &  & -103.50  & 88.08   & 1944    &  & -103.50  & 2256.70            & 1944    \\
bmilplib-410-2            & -108.59  & 103.57  & 2887    &  & -108.59  & 92.86   & 2887    &  & -108.59  & 97.35   & 2887    &  & -108.59  & 947.06             & 2887    \\
bmilplib-410-3            & -96.24   & 147.99  & 3781    &  & -96.24   & 157.30  & 3781    &  & -96.24   & 257.83  & 3781    &  & -79.49   & \textgreater{}3600 & 878     \\
bmilplib-410-4            & -119.50  & 100.21  & 2995    &  & -119.50  & 100.04  & 2995    &  & -119.50  & 100.85  & 2995    &  & -119.50  & 1042.52            & 2995    \\
bmilplib-410-5            & -119.22  & 71.23   & 1520    &  & -119.22  & 61.46   & 1527    &  & -119.22  & 107.73  & 1520    &  & -119.22  & 426.60             & 1527    \\
bmilplib-410-6            & -151.31  & 13.54   & 533     &  & -151.31  & 13.74   & 533     &  & -151.31  & 17.48   & 533     &  & -151.31  & 727.04             & 533     \\
bmilplib-410-7            & -123.00  & 35.87   & 1177    &  & -123.00  & 35.79   & 1175    &  & -123.00  & 41.41   & 1177    &  & -123.00  & 406.53             & 1177    \\
bmilplib-410-8            & -125.78  & 169.15  & 4480    &  & -125.78  & 171.56  & 4480    &  & -125.78  & 190.33  & 4480    &  & -125.78  & 1408.43            & 4488    \\
bmilplib-410-9            & -100.77  & 82.17   & 2071    &  & -100.77  & 79.34   & 2071    &  & -100.77  & 85.38   & 2071    &  & -100.77  & 478.13             & 2071    \\
bmilplib-460-10           & -102.51  & 228.12  & 4465    &  & -102.51  & 204.04  & 4465    &  & -102.51  & 238.28  & 4465    &  & -102.51  & 1557.08            & 4465    \\
bmilplib-460-1            & -97.59   & 569.22  & 10803   &  & -97.59   & 610.72  & 10803   &  & -97.59   & 666.13  & 10803   &  & -93.40   & \textgreater{}3600 & 1240    \\
bmilplib-460-2            & -139.00  & 43.65   & 964     &  & -139.00  & 46.56   & 964     &  & -139.00  & 45.27   & 964     &  & -139.00  & 557.54             & 964     \\
bmilplib-460-3            & -86.50   & 223.58  & 3882    &  & -86.50   & 219.35  & 3882    &  & -86.50   & 251.89  & 3882    &  & -82.83   & \textgreater{}3600 & 2966    \\
bmilplib-460-4            & -107.03  & 412.76  & 7856    &  & -107.03  & 425.15  & 7856    &  & -107.03  & 399.79  & 7794    &  & -102.61  & \textgreater{}3600 & 4784    \\
bmilplib-460-5            & -100.50  & 170.30  & 3025    &  & -100.50  & 189.63  & 3025    &  & -100.50  & 177.04  & 3025    &  & -100.50  & 1788.18            & 3025    \\
bmilplib-460-6            & -107.00  & 236.30  & 4143    &  & -107.00  & 205.48  & 4143    &  & -107.00  & 266.17  & 4143    &  & -107.00  & 2457.96            & 4143    \\
bmilplib-460-7            & -83.75   & 294.68  & 5252    &  & -83.75   & 282.83  & 5252    &  & -83.75   & 341.16  & 5252    &  & -83.75   & \textgreater{}3600 & 1538    \\
bmilplib-460-8            & -115.39  & 94.67   & 1903    &  & -115.39  & 92.26   & 1908    &  & -115.39  & 140.92  & 1903    &  & -103.50  & \textgreater{}3600 & 611     \\
bmilplib-460-9            & -128.70  & 327.68  & 6185    &  & -128.70  & 316.77  & 6227    &  & -128.70  & 318.04  & 6185    &  & -128.70  & 3166.24            & 6185    \\
bmilplib-60-10            & -186.21  & 4.29    & 2590    &  & -186.21  & 4.32    & 2590    &  & -186.21  & 4.68    & 2590    &  & -186.21  & 22.24              & 2590    \\
bmilplib-60-1             & -153.20  & 2.79    & 1094    &  & -153.20  & 3.01    & 1094    &  & -153.20  & 3.47    & 1094    &  & -153.20  & 18.65              & 1094    \\
bmilplib-60-5             & -116.40  & 10.08   & 9996    &  & -116.40  & 10.41   & 9996    &  & -116.40  & 12.63   & 9996    &  & -116.40  & 35.73              & 9996    \\
bmilplib-60-6             & -187.31  & 3.34    & 1383    &  & -187.31  & 3.68    & 1383    &  & -187.31  & 4.07    & 1383    &  & -187.31  & 22.35              & 1383    \\
bmilplib-60-8             & -232.12  & 1.15    & 572     &  & -232.12  & 1.22    & 572     &  & -232.12  & 1.61    & 572     &  & -232.12  & 9.46               & 572     \\
bmilplib-60-9             & -136.50  & 12.34   & 9888    &  & -136.50  & 12.45   & 9888    &  & -136.50  & 13.56   & 9888    &  & -136.50  & 47.50              & 9888    \\
miblp-20-15-50-0110-10-10 & -206.00  & 1.06    & 423     &  & -206.00  & 1.07    & 423     &  & -206.00  & 1.01    & 423     &  & -206.00  & 1.81               & 423     \\
miblp-20-15-50-0110-10-2  & -398.00  & 9.09    & 2450    &  & -398.00  & 9.43    & 2450    &  & -398.00  & 8.85    & 2450    &  & -398.00  & 15.52              & 2450    \\
miblp-20-15-50-0110-10-3  & -42.00   & 0.62    & 267     &  & -42.00   & 0.72    & 267     &  & -42.00   & 0.68    & 267     &  & -42.00   & 1.36               & 267     \\
miblp-20-15-50-0110-10-6  & -246.00  & 4.56    & 340     &  & -246.00  & 4.43    & 340     &  & -246.00  & 4.24    & 340     &  & -246.00  & 6.40               & 340     \\
miblp-20-15-50-0110-10-9  & -635.00  & 14.05   & 2380    &  & -635.00  & 14.97   & 2382    &  & -635.00  & 15.48   & 2380    &  & -635.00  & 19.90              & 2380    \\
miblp-20-20-50-0110-10-10 & -441.00  & 692.00  & 134585  &  & -441.00  & 758.37  & 134585  &  & -441.00  & 743.82  & 134585  &  & -441.00  & 1113.34            & 134585  \\
miblp-20-20-50-0110-10-1  & -359.00  & 228.67  & 96141   &  & -359.00  & 232.30  & 96141   &  & -359.00  & 242.52  & 96141   &  & -359.00  & 407.03             & 96141   \\
miblp-20-20-50-0110-10-2  & -659.00  & 3.69    & 3191    &  & -659.00  & 4.15    & 3191    &  & -659.00  & 3.94    & 3191    &  & -659.00  & 9.48               & 3191    \\
miblp-20-20-50-0110-10-3  & -618.00  & 13.18   & 20788   &  & -618.00  & 13.68   & 20788   &  & -618.00  & 14.77   & 20788   &  & -618.00  & 54.78              & 20715   \\
miblp-20-20-50-0110-10-4  & -604.00  & 3145.16 & 830808  &  & -604.00  & 3286.17 & 830808  &  & -604.00  & 3223.61 & 830808  &  & -604.00  & \textgreater{}3600 & 666302  \\
miblp-20-20-50-0110-10-7  & -683.00  & 1887.37 & 3003967 &  & -683.00  & 1960.64 & 3014456 &  & -683.00  & 2049.54 & 3004274 &  & -683.00  & \textgreater{}3600 & 2542388 \\
\hline
\end{tabular}
}
\end{table}
%%%%%%%%%%%%%%%%%%
%table for figure 5 (continued)
\addtocounter{table}{-1}
\begin{table}[h!]
\caption{Detailed results of Figure~\ref{fig:heuristicsParTimeMore5Sec} (continued)}
\label{tab:fig5Continue}
\centering
\resizebox{\columnwidth}{!}{
%\begin{tabular}{lccccccccccccccc}
\begin{tabular}{llllllllllllllll}
\hline
\multicolumn{1}{c}{}&
\multicolumn{3}{c}{noHeuristics}&
\multicolumn{1}{c}{}&
\multicolumn{3}{c}{impObjectiveCut}&
\multicolumn{1}{c}{}&
\multicolumn{3}{c}{secondLevelPriority}&
\multicolumn{1}{c}{}&
\multicolumn{3}{c}{weightedSums} \\
\cline{2-4}
\cline{6-8}
\cline{10-12}
\cline{14-16}
Instance&   BestSol & Time(s) &Nodes  & &BestSol & Time(s) &Nodes & &BestSol & Time(s) &Nodes& &BestSol & Time(s) &Nodes\\
\hline
miblp-20-20-50-0110-10-8  & -667.00  & 80.45   & 12857   &  & -667.00  & 94.47   & 12857   &  & -667.00  & 87.72   & 12858   &  & -667.00  & 134.18             & 12859   \\
miblp-20-20-50-0110-10-9  & -256.00  & 20.99   & 31245   &  & -256.00  & 21.84   & 31245   &  & -256.00  & 22.38   & 31245   &  & -256.00  & 38.71              & 31245   \\
miblp-20-20-50-0110-15-1  & -450.00  & 59.38   & 3506    &  & -450.00  & 59.36   & 3506    &  & -450.00  & 62.04   & 3506    &  & -450.00  & 85.24              & 3506    \\
miblp-20-20-50-0110-15-2  & -645.00  & 50.14   & 17251   &  & -645.00  & 52.11   & 17251   &  & -645.00  & 61.00   & 17251   &  & -645.00  & 222.62             & 17251   \\
miblp-20-20-50-0110-15-3  & -593.00  & 71.15   & 3081    &  & -593.00  & 66.66   & 3081    &  & -593.00  & 67.67   & 3081    &  & -593.00  & 95.58              & 3083    \\
miblp-20-20-50-0110-15-4  & -441.00  & 43.00   & 1625    &  & -441.00  & 43.74   & 1625    &  & -441.00  & 39.84   & 1625    &  & -441.00  & 57.56              & 1625    \\
miblp-20-20-50-0110-15-5  & -379.00  & 651.28  & 16715   &  & -379.00  & 634.58  & 16715   &  & -379.00  & 639.80  & 16715   &  & -379.00  & 750.81             & 16715   \\
miblp-20-20-50-0110-15-6  & -596.00  & 17.31   & 1657    &  & -596.00  & 17.40   & 1657    &  & -596.00  & 18.50   & 1657    &  & -596.00  & 25.94              & 1657    \\
miblp-20-20-50-0110-15-7  & -471.00  & 111.02  & 13405   &  & -471.00  & 109.14  & 13405   &  & -471.00  & 113.83  & 13405   &  & -471.00  & 191.29             & 13405   \\
miblp-20-20-50-0110-15-8  & -370.00  & 39.12   & 21589   &  & -370.00  & 39.12   & 21589   &  & -370.00  & 44.93   & 21589   &  & -370.00  & 122.50             & 21589   \\
miblp-20-20-50-0110-15-9  & -584.00  & 2.00    & 582     &  & -584.00  & 2.24    & 582     &  & -584.00  & 1.93    & 582     &  & -584.00  & 3.44               & 584     \\
miblp-20-20-50-0110-5-13  & -519.00  & 2709.29 & 6515097 &  & -519.00  & 2756.40 & 6514987 &  & -519.00  & 2828.20 & 6521367 &  & -519.00  & 3465.54            & 6515355 \\
miblp-20-20-50-0110-5-15  & -617.00  & 1130.92 & 4312915 &  & -617.00  & 1123.27 & 4312915 &  & -617.00  & 1202.16 & 4312915 &  & -617.00  & 1826.00            & 4309751 \\
miblp-20-20-50-0110-5-16  & -833.00  & 1.80    & 5913    &  & -833.00  & 1.76    & 5913    &  & -833.00  & 2.22    & 5925    &  & -833.00  & 9.74               & 5925    \\
miblp-20-20-50-0110-5-17  & -944.00  & 1.53    & 4038    &  & -944.00  & 1.48    & 4037    &  & -944.00  & 1.84    & 4038    &  & -944.00  & 8.31               & 4212    \\
miblp-20-20-50-0110-5-19  & -431.00  & 25.50   & 116041  &  & -431.00  & 26.02   & 116041  &  & -431.00  & 27.33   & 116065  &  & -431.00  & 45.05              & 116285  \\
miblp-20-20-50-0110-5-1   & -548.00  & 8.45    & 31298   &  & -548.00  & 9.08    & 31298   &  & -548.00  & 9.13    & 31298   &  & -548.00  & 22.96              & 31298   \\
miblp-20-20-50-0110-5-20  & -438.00  & 10.82   & 33315   &  & -438.00  & 10.92   & 33315   &  & -438.00  & 11.25   & 33315   &  & -438.00  & 17.54              & 33130   \\
miblp-20-20-50-0110-5-6   & -1061.00 & 51.74   & 213928  &  & -1061.00 & 52.57   & 213928  &  & -1061.00 & 56.75   & 214449  &  & -1061.00 & 235.63             & 234647  \\
lseu-0.100000             & 1120.00  & 3.15    & 8603    &  & 1120.00  & 3.11    & 8603    &  & 1120.00  & 3.82    & 13352   &  & 1120.00  & 45.32              & 22411   \\
lseu-0.900000             & 5838.00  & 14.10   & 1023    &  & 5838.00  & 14.11   & 1023    &  & 5838.00  & 251.51  & 1023    &  & 5838.00  & 21.57              & 1023    \\
p0033-0.500000            & 3095.00  & 0.25    & 1467    &  & 3095.00  & 0.28    & 1467    &  & 3095.00  & 0.29    & 1471    &  & 3095.00  & 0.80               & 2239    \\
p0033-0.900000            & 4679.00  & 0.06    & 27      &  & 4679.00  & 0.06    & 27      &  & 4679.00  & 0.07    & 27      &  & 4679.00  & 0.12               & 27      \\
p0201-0.900000            & 15025.00 & 6.62    & 2481    &  & 15025.00 & 7.09    & 2481    &  & 15025.00 & 359.66  & 2481    &  & 15025.00 & 34.52              & 2481    \\
stein27-0.500000          & 19.00    & 6.51    & 17648   &  & 19.00    & 6.44    & 17648   &  & 19.00    & 6.76    & 16969   &  & 19.00    & 9.57               & 16969   \\
stein27-0.900000          & 24.00    & 0.02    & 15      &  & 24.00    & 0.02    & 15      &  & 24.00    & 0.02    & 15      &  & 24.00    & 1.62               & 15      \\
stein45-0.100000          & 30.00    & 64.27   & 90241   &  & 30.00    & 64.12   & 90241   &  & 30.00    & 83.80   & 90002   &  & 30.00    & 256.23             & 58729   \\
stein45-0.500000          & 32.00    & 471.57  & 753845  &  & 32.00    & 459.73  & 753845  &  & 32.00    & 511.05  & 862249  &  & 32.00    & 880.47             & 565867  \\
stein45-0.900000          & 40.00    & 0.14    & 63      &  & 40.00    & 0.16    & 63      &  & 40.00    & 0.16    & 63      &  & 40.00    & 72.23              & 63 \\  
\hline
\end{tabular}
}
\end{table}
}
{
%coralReport
%%%%%%%%%%%%%%%%%%
%table for figure 2a
\begin{table}[h!]
\caption{Detailed results of Figure~\ref{fig:SSUBParTimeMore5SecNonXu}}
\label{tab:fig2a}
\resizebox{\columnwidth}{!}{
%\begin{tabular}{lccccccccccccccccccc}
\begin{tabular}{llllllllllllllllllll}
\hline
\multicolumn{1}{c}{}&
\multicolumn{3}{c}{whenLInt-LInt}&
\multicolumn{1}{c}{}&
\multicolumn{3}{c}{whenLInt-LFixed}&
\multicolumn{1}{c}{}&
\multicolumn{3}{c}{whenLFixed-LFixed}&
\multicolumn{1}{c}{}&
\multicolumn{3}{c}{whenXYInt-LFixed}&
\multicolumn{1}{c}{}&
\multicolumn{3}{c}{\begin{tabular}[c]{@{}c@{}}\scriptsize{whenXYIntOr}\\\scriptsize{LFixed-LFixed}\end{tabular}} \\
\cline{2-4}
\cline{6-8}
\cline{10-12}
\cline{14-16}
\cline{18-20}
Instance&   BestSol & Time(s) &Nodes  & &BestSol & Time(s) &Nodes & &BestSol & Time(s) &Nodes& &BestSol & Time(s) &Nodes& &BestSol & Time(s) &Nodes\\
\hline
miblp-20-15-50-0110-10-10 & -206.00  & 1.82               & 423     &  & -206.00  & 1.08    & 423     &  & -206.00  & 0.97    & 1414     &  & -206.00  & 1.06    & 423     &  & -206.00  & 1.70               & 423     \\
miblp-20-15-50-0110-10-2  & -398.00  & 12.25              & 2590    &  & -398.00  & 9.18    & 2450    &  & -398.00  & 11.12   & 27625    &  & -398.00  & 9.09    & 2450    &  & -398.00  & 11.52              & 2590    \\
miblp-20-15-50-0110-10-3  & -42.00   & 0.74               & 267     &  & -42.00   & 0.61    & 267     &  & -42.00   & 0.71    & 1343     &  & -42.00   & 0.62    & 267     &  & -42.00   & 0.70               & 267     \\
miblp-20-15-50-0110-10-6  & -246.00  & 12.16              & 340     &  & -246.00  & 4.47    & 340     &  & -246.00  & 2.19    & 853      &  & -246.00  & 4.56    & 340     &  & -246.00  & 8.67               & 340     \\
miblp-20-15-50-0110-10-9  & -635.00  & 26.62              & 2387    &  & -635.00  & 14.06   & 2380    &  & -635.00  & 7.86    & 6457     &  & -635.00  & 14.05   & 2380    &  & -635.00  & 25.35              & 2387    \\
miblp-20-20-50-0110-10-10 & -441.00  & 1322.97            & 134901  &  & -441.00  & 684.67  & 134585  &  & -441.00  & 526.56  & 639146   &  & -441.00  & 692.00  & 134585  &  & -441.00  & 1121.84            & 134901  \\
miblp-20-20-50-0110-10-1  & -359.00  & 264.46             & 96136   &  & -359.00  & 244.20  & 94019   &  & -359.00  & 260.31  & 423149   &  & -359.00  & 228.67  & 96141   &  & -359.00  & 248.95             & 96136   \\
miblp-20-20-50-0110-10-2  & -659.00  & 12.72              & 3261    &  & -659.00  & 3.93    & 3191    &  & -659.00  & 2.78    & 7547     &  & -659.00  & 3.69    & 3191    &  & -659.00  & 10.56              & 3261    \\
miblp-20-20-50-0110-10-3  & -618.00  & 40.06              & 21622   &  & -618.00  & 13.04   & 20780   &  & -618.00  & 10.74   & 38188    &  & -618.00  & 13.18   & 20788   &  & -618.00  & 29.85              & 21622   \\
miblp-20-20-50-0110-10-4  & -604.00  & \textgreater{}3600 & 654402  &  & -604.00  & 3127.12 & 830908  &  & -604.00  & 3405.57 & 6452671  &  & -604.00  & 3145.16 & 830808  &  & -604.00  & \textgreater{}3600 & 714877  \\
miblp-20-20-50-0110-10-7  & -683.00  & 3233.20            & 3069661 &  & -683.00  & 2046.70 & 2978073 &  & -683.00  & 3184.99 & 11502091 &  & -683.00  & 1887.37 & 3003967 &  & -683.00  & 2511.36            & 3069661 \\
miblp-20-20-50-0110-10-8  & -667.00  & 182.80             & 12868   &  & -667.00  & 87.14   & 12856   &  & -667.00  & 40.78   & 30873    &  & -667.00  & 80.45   & 12857   &  & -667.00  & 148.40             & 12868   \\
miblp-20-20-50-0110-10-9  & -256.00  & 45.39              & 35295   &  & -256.00  & 21.06   & 31244   &  & -256.00  & 23.97   & 76055    &  & -256.00  & 20.99   & 31245   &  & -256.00  & 33.92              & 35295   \\
miblp-20-20-50-0110-15-1  & -450.00  & 60.74              & 3516    &  & -450.00  & 60.84   & 3506    &  & -450.00  & 65.16   & 49137    &  & -450.00  & 59.38   & 3506    &  & -450.00  & 60.89              & 3516    \\
miblp-20-20-50-0110-15-2  & -645.00  & 73.76              & 17251   &  & -645.00  & 50.01   & 17251   &  & -645.00  & 96.13   & 346065   &  & -645.00  & 50.14   & 17251   &  & -645.00  & 63.85              & 17251   \\
miblp-20-20-50-0110-15-3  & -593.00  & 100.53             & 3083    &  & -593.00  & 70.36   & 3081    &  & -593.00  & 65.25   & 42877    &  & -593.00  & 71.15   & 3081    &  & -593.00  & 92.89              & 3083    \\
miblp-20-20-50-0110-15-4  & -441.00  & 66.29              & 1625    &  & -441.00  & 42.88   & 1625    &  & -441.00  & 36.99   & 29904    &  & -441.00  & 43.00   & 1625    &  & -441.00  & 55.23              & 1625    \\
miblp-20-20-50-0110-15-5  & -379.00  & 860.42             & 16715   &  & -379.00  & 632.14  & 16715   &  & -379.00  & 615.16  & 205025   &  & -379.00  & 651.28  & 16715   &  & -379.00  & 730.86             & 16715   \\
miblp-20-20-50-0110-15-6  & -596.00  & 23.24              & 1657    &  & -596.00  & 17.26   & 1657    &  & -596.00  & 17.87   & 29923    &  & -596.00  & 17.31   & 1657    &  & -596.00  & 22.10              & 1657    \\
miblp-20-20-50-0110-15-7  & -471.00  & 133.26             & 13405   &  & -471.00  & 110.80  & 13405   &  & -471.00  & 125.24  & 241285   &  & -471.00  & 111.02  & 13405   &  & -471.00  & 127.42             & 13405   \\
miblp-20-20-50-0110-15-8  & -370.00  & 41.05              & 21589   &  & -370.00  & 39.51   & 21589   &  & -370.00  & 138.34  & 579309   &  & -370.00  & 39.12   & 21589   &  & -370.00  & 40.56              & 21589   \\
miblp-20-20-50-0110-15-9  & -584.00  & 2.58               & 588     &  & -584.00  & 2.03    & 582     &  & -584.00  & 1.96    & 4072     &  & -584.00  & 2.00    & 582     &  & -584.00  & 2.34               & 588     \\
miblp-20-20-50-0110-5-13  & -519.00  & \textgreater{}3600 & 4221171 &  & -519.00  & 2002.06 & 4392628 &  & -519.00  & 2307.90 & 10196998 &  & -519.00  & 2464.25 & 7128138 &  & -519.00  & \textgreater{}3600 & 5830754 \\
miblp-20-20-50-0110-5-15  & -617.00  & \textgreater{}3600 & 3325012 &  & -617.00  & 1122.66 & 2749914 &  & -617.00  & 1219.33 & 6630921  &  & -617.00  & 1310.67 & 4018058 &  & -617.00  & 2953.78            & 4018081 \\
miblp-20-20-50-0110-5-16  & -833.00  & 44.15              & 18278   &  & -833.00  & 6.48    & 17680   &  & -833.00  & 4.86    & 19013    &  & -833.00  & 6.60    & 17680   &  & -833.00  & 38.30              & 18278   \\
miblp-20-20-50-0110-5-17  & -944.00  & 16.80              & 18200   &  & -944.00  & 4.62    & 17712   &  & -944.00  & 3.99    & 21541    &  & -944.00  & 4.78    & 17679   &  & -944.00  & 10.64              & 18200   \\
miblp-20-20-50-0110-5-19  & -431.00  & 104.93             & 79279   &  & -431.00  & 26.75   & 77256   &  & -431.00  & 27.52   & 147268   &  & -431.00  & 26.56   & 77256   &  & -431.00  & 65.55              & 79279   \\
miblp-20-20-50-0110-5-1   & -548.00  & 52.43              & 42629   &  & -548.00  & 14.31   & 42364   &  & -548.00  & 12.98   & 64067    &  & -548.00  & 14.50   & 42364   &  & -548.00  & 36.55              & 42629   \\
miblp-20-20-50-0110-5-20  & -438.00  & 48.38              & 60990   &  & -438.00  & 14.95   & 50531   &  & -438.00  & 15.81   & 76091    &  & -438.00  & 17.30   & 60944   &  & -438.00  & 31.77              & 60990   \\
miblp-20-20-50-0110-5-6   & -1061.00 & 202.25             & 224425  &  & -1061.00 & 62.75   & 223494  &  & -1061.00 & 58.74   & 284550   &  & -1061.00 & 63.70   & 222566  &  & -1061.00 & 127.84             & 224425  \\
lseu-0.100000             & 1120.00  & 657.06             & 1132617 &  & 1120.00  & 235.84  & 734286  &  & 1120.00  & 248.54  & 1071409  &  & 1120.00  & 270.78  & 1003976 &  & 1120.00  & 626.90             & 1132617 \\
lseu-0.900000             & 5838.00  & 13.91              & 1023    &  & 5838.00  & 14.41   & 1023    &  & 5838.00  & 1063.02 & 4718749  &  & 5838.00  & 14.10   & 1023    &  & 5838.00  & 13.83              & 1023    \\
p0033-0.500000            & 3095.00  & 12.65              & 20855   &  & 3095.00  & 10.44   & 20695   &  & 3095.00  & 6.24    & 33614    &  & 3095.00  & 10.48   & 20695   &  & 3095.00  & 11.55              & 20855   \\
p0033-0.900000            & 4679.00  & 0.06               & 27      &  & 4679.00  & 0.05    & 27      &  & 4679.00  & 0.65    & 3455     &  & 4679.00  & 0.06    & 27      &  & 4679.00  & 0.04               & 27      \\
p0201-0.900000            & 15025.00 & 7.05               & 2481    &  & 15025.00 & 6.66    & 2481    &  & 15025.00 & 20.58   & 18801    &  & 15025.00 & 6.62    & 2481    &  & 15025.00 & 6.68               & 2481    \\
stein27-0.500000          & 19.00    & 6.21               & 12115   &  & 19.00    & 6.24    & 14362   &  & 19.00    & 7.36    & 21515    &  & 19.00    & 5.88    & 12115   &  & 19.00    & 5.91               & 12115   \\
stein27-0.900000          & 24.00    & 0.01               & 15      &  & 24.00    & 0.02    & 15      &  & 24.00    & 1.25    & 4445     &  & 24.00    & 0.02    & 15      &  & 24.00    & 0.02               & 15      \\
stein45-0.100000          & 30.00    & 51.06              & 89035   &  & 30.00    & 96.92   & 61518   &  & 30.00    & 50.47   & 89035    &  & 30.00    & 50.19   & 89035   &  & 30.00    & 50.17              & 89035   \\
stein45-0.500000          & 32.00    & 554.15             & 640308  &  & 32.00    & 1088.96 & 1014908 &  & 32.00    & 635.22  & 952123   &  & 32.00    & 519.59  & 640308  &  & 32.00    & 520.73             & 640308  \\
stein45-0.900000          & 40.00    & 0.16               & 63      &  & 40.00    & 0.16    & 63      &  & 40.00    & 85.92   & 103661   &  & 40.00    & 0.14    & 63      &  & 40.00    & 0.15               & 63     \\     
\hline
\end{tabular}
}
\end{table}
%%%%%%%%%%%%%%%%%%
%table for figure 2b
\begin{table}[p]
\begin{center}
\caption{Detailed results of Figure~\ref{fig:SSUBParTimeMore5SecXu}}
\label{tab:fig2b}
\resizebox{\columnwidth}{!}{
%\begin{tabular}{lccccccccccccccccccc}
\begin{tabular}{llllllllllllllllllll}
\hline
\multicolumn{1}{c}{}&
\multicolumn{3}{c}{whenLInt-LInt}&
\multicolumn{1}{c}{}&
\multicolumn{3}{c}{whenLInt-LFixed}&
\multicolumn{1}{c}{}&
\multicolumn{3}{c}{whenLFixed-LFixed}&
\multicolumn{1}{c}{}&
\multicolumn{3}{c}{whenXYInt-LFixed}&
\multicolumn{1}{c}{}&
\multicolumn{3}{c}{\begin{tabular}[c]{@{}c@{}}\scriptsize{whenXYIntOr}\\\scriptsize{LFixed-LFixed}\end{tabular}} \\
\cline{2-4}
\cline{6-8}
\cline{10-12}
\cline{14-16}
\cline{18-20}
Instance&   BestSol & Time(s) &Nodes  & &BestSol & Time(s) &Nodes & &BestSol & Time(s) &Nodes& &BestSol & Time(s) &Nodes& &BestSol & Time(s) &Nodes\\
\hline
bmilplib-110-10 & -177.67 & 336.74  & 75952  &  & -177.67 & 150.88  & 75141  &  & -177.67 & 149.10  & 93801  &  & -177.67 & 153.82  & 75499  &  & -177.67 & 223.95  & 75952  \\
bmilplib-110-1  & -181.67 & 9.23    & 2806   &  & -181.67 & 5.35    & 3629   &  & -181.67 & 5.42    & 2878   &  & -181.67 & 5.35    & 2806   &  & -181.67 & 6.34    & 2806   \\
bmilplib-110-2  & -110.67 & 9.16    & 4183   &  & -110.67 & 6.53    & 4439   &  & -110.67 & 6.50    & 4272   &  & -110.67 & 6.48    & 4174   &  & -110.67 & 7.46    & 4183   \\
bmilplib-110-3  & -215.16 & 10.15   & 3484   &  & -215.16 & 6.07    & 3529   &  & -215.16 & 5.42    & 3556   &  & -215.16 & 5.58    & 3471   &  & -215.16 & 7.13    & 3484   \\
bmilplib-110-4  & -197.29 & 16.30   & 1917   &  & -197.29 & 3.66    & 1740   &  & -197.29 & 3.54    & 1801   &  & -197.29 & 3.69    & 1751   &  & -197.29 & 8.20    & 1917   \\
bmilplib-110-6  & -148.25 & 25.96   & 8857   &  & -148.25 & 14.57   & 8735   &  & -148.25 & 14.14   & 9392   &  & -148.25 & 15.06   & 8770   &  & -148.25 & 19.35   & 8857   \\
bmilplib-110-7  & -160.86 & 11.92   & 2221   &  & -160.86 & 4.66    & 2167   &  & -160.86 & 4.01    & 2281   &  & -160.86 & 4.94    & 2189   &  & -160.86 & 7.45    & 2221   \\
bmilplib-110-8  & -155.00 & 39.85   & 11835  &  & -155.00 & 20.66   & 11916  &  & -155.00 & 20.65   & 12498  &  & -155.00 & 20.78   & 11663  &  & -155.00 & 27.93   & 11835  \\
bmilplib-110-9  & -192.92 & 9.14    & 2155   &  & -192.92 & 3.46    & 2099   &  & -192.92 & 3.28    & 2171   &  & -192.92 & 3.66    & 2151   &  & -192.92 & 5.84    & 2155   \\
bmilplib-160-10 & -189.82 & 87.08   & 9130   &  & -189.82 & 30.70   & 8676   &  & -189.82 & 31.24   & 9344   &  & -189.82 & 33.42   & 9071   &  & -189.82 & 55.24   & 9130   \\
bmilplib-160-1  & -165.00 & 63.76   & 7908   &  & -165.00 & 25.00   & 8028   &  & -165.00 & 24.18   & 8145   &  & -165.00 & 25.70   & 7901   &  & -165.00 & 38.00   & 7908   \\
bmilplib-160-2  & -178.24 & 87.14   & 8680   &  & -178.24 & 29.51   & 8612   &  & -178.24 & 29.64   & 8710   &  & -178.24 & 30.59   & 8635   &  & -178.24 & 47.52   & 8680   \\
bmilplib-160-3  & -174.94 & 196.12  & 15164  &  & -174.94 & 62.78   & 14998  &  & -174.94 & 54.62   & 15427  &  & -174.94 & 63.37   & 15002  &  & -174.94 & 104.06  & 15164  \\
bmilplib-160-4  & -135.83 & 77.00   & 15005  &  & -135.83 & 48.60   & 14708  &  & -135.83 & 50.86   & 15354  &  & -135.83 & 51.02   & 14772  &  & -135.83 & 60.24   & 15005  \\
bmilplib-160-5  & -140.78 & 62.61   & 4981   &  & -140.78 & 19.78   & 4842   &  & -140.78 & 19.35   & 5178   &  & -140.78 & 20.52   & 4948   &  & -140.78 & 36.78   & 4981   \\
bmilplib-160-6  & -111.00 & 32.32   & 8685   &  & -111.00 & 27.50   & 10149  &  & -111.00 & 25.08   & 8864   &  & -111.00 & 24.66   & 8676   &  & -111.00 & 27.50   & 8685   \\
bmilplib-160-7  & -96.00  & 65.43   & 16589  &  & -96.00  & 51.07   & 16804  &  & -96.00  & 50.97   & 17430  &  & -96.00  & 51.20   & 16560  &  & -96.00  & 58.05   & 16589  \\
bmilplib-160-8  & -181.40 & 22.24   & 3464   &  & -181.40 & 12.47   & 3417   &  & -181.40 & 9.32    & 3540   &  & -181.40 & 10.49   & 3444   &  & -181.40 & 14.93   & 3464   \\
bmilplib-160-9  & -207.50 & 34.05   & 4715   &  & -207.50 & 16.51   & 5169   &  & -207.50 & 14.90   & 4728   &  & -207.50 & 15.85   & 4607   &  & -207.50 & 23.34   & 4715   \\
bmilplib-210-10 & -130.59 & 84.71   & 12226  &  & -130.59 & 58.89   & 12010  &  & -130.59 & 61.00   & 12497  &  & -130.59 & 61.81   & 12176  &  & -130.59 & 72.22   & 12226  \\
bmilplib-210-1  & -136.80 & 77.67   & 7429   &  & -136.80 & 37.71   & 7182   &  & -136.80 & 40.04   & 7598   &  & -136.80 & 41.96   & 7429   &  & -136.80 & 56.23   & 7429   \\
bmilplib-210-2  & -117.80 & 153.16  & 22441  &  & -117.80 & 111.87  & 22597  &  & -117.80 & 109.56  & 23096  &  & -117.80 & 113.38  & 22294  &  & -117.80 & 131.02  & 22441  \\
bmilplib-210-3  & -130.80 & 74.50   & 8897   &  & -130.80 & 44.80   & 8523   &  & -130.80 & 47.44   & 9165   &  & -130.80 & 48.19   & 8848   &  & -130.80 & 59.95   & 8897   \\
bmilplib-210-4  & -162.20 & 62.23   & 4738   &  & -162.20 & 36.19   & 5094   &  & -162.20 & 25.92   & 4806   &  & -162.20 & 26.81   & 4687   &  & -162.20 & 41.52   & 4738   \\
bmilplib-210-5  & -134.00 & 156.90  & 21193  &  & -134.00 & 106.25  & 21560  &  & -134.00 & 105.03  & 21552  &  & -134.00 & 108.22  & 21058  &  & -134.00 & 130.51  & 21193  \\
bmilplib-210-6  & -125.43 & 269.65  & 38538  &  & -125.43 & 198.77  & 38367  &  & -125.43 & 197.88  & 39941  &  & -125.43 & 201.40  & 38443  &  & -125.43 & 227.81  & 38538  \\
bmilplib-210-7  & -169.73 & 146.39  & 13960  &  & -169.73 & 76.42   & 13960  &  & -169.73 & 77.62   & 14305  &  & -169.73 & 77.84   & 13960  &  & -169.73 & 104.02  & 13960  \\
bmilplib-210-8  & -101.46 & 70.64   & 11105  &  & -101.46 & 56.97   & 11324  &  & -101.46 & 58.14   & 11347  &  & -101.46 & 61.79   & 11105  &  & -101.46 & 62.56   & 11105  \\
bmilplib-210-9  & -184.00 & 1822.79 & 143571 &  & -184.00 & 790.46  & 144603 &  & -184.00 & 859.46  & 143665 &  & -184.00 & 879.61  & 142294 &  & -184.00 & 1184.18 & 143571 \\
bmilplib-260-10 & -151.73 & 785.34  & 62063  &  & -151.73 & 503.08  & 62143  &  & -151.73 & 502.94  & 63026  &  & -151.73 & 546.55  & 61767  &  & -151.73 & 633.15  & 62063  \\
bmilplib-260-1  & -139.00 & 155.24  & 10994  &  & -139.00 & 83.01   & 11369  &  & -139.00 & 81.74   & 11242  &  & -139.00 & 85.81   & 10980  &  & -139.00 & 119.94  & 10994  \\
bmilplib-260-2  & -82.62  & 117.18  & 15300  &  & -82.62  & 156.75  & 15292  &  & -82.62  & 109.75  & 15712  &  & -82.62  & 109.75  & 15287  &  & -82.62  & 114.50  & 15300  \\
bmilplib-260-3  & -144.25 & 148.68  & 8777   &  & -144.25 & 72.67   & 9563   &  & -144.25 & 69.26   & 8859   &  & -144.25 & 71.54   & 8769   &  & -144.25 & 94.43   & 8777   \\
bmilplib-260-4  & -117.33 & 294.39  & 32964  &  & -117.33 & 245.58  & 32696  &  & -117.33 & 260.01  & 33931  &  & -117.33 & 259.48  & 32787  &  & -117.33 & 272.70  & 32964  \\
bmilplib-260-5  & -121.00 & 201.31  & 21557  &  & -121.00 & 167.30  & 22480  &  & -121.00 & 173.32  & 22121  &  & -121.00 & 166.03  & 21542  &  & -121.00 & 183.46  & 21557  \\
bmilplib-260-6  & -124.00 & 258.94  & 25420  &  & -124.00 & 191.68  & 25826  &  & -124.00 & 209.79  & 26023  &  & -124.00 & 197.04  & 25362  &  & -124.00 & 223.31  & 25420  \\
bmilplib-260-7  & -137.80 & 472.66  & 40528  &  & -137.80 & 304.36  & 40400  &  & -137.80 & 318.76  & 41180  &  & -137.80 & 312.12  & 40274  &  & -137.80 & 383.16  & 40528  \\
bmilplib-260-8  & -119.89 & 112.17  & 10094  &  & -119.89 & 73.72   & 9961   &  & -119.89 & 76.70   & 10332  &  & -119.89 & 87.01   & 10025  &  & -119.89 & 96.85   & 10094  \\
bmilplib-260-9  & -160.00 & 510.37  & 33496  &  & -160.00 & 257.08  & 33928  &  & -160.00 & 271.18  & 34236  &  & -160.00 & 273.66  & 33468  &  & -160.00 & 394.98  & 33496  \\
bmilplib-310-10 & -141.86 & 129.77  & 9904   &  & -141.86 & 110.26  & 9896   &  & -141.86 & 114.38  & 10033  &  & -141.86 & 121.07  & 9900   &  & -141.86 & 127.87  & 9904   \\
bmilplib-310-1  & -117.00 & 521.12  & 28360  &  & -117.00 & 301.34  & 28373  &  & -117.00 & 329.80  & 29055  &  & -117.00 & 329.71  & 28330  &  & -117.00 & 372.68  & 28360  \\
bmilplib-310-2  & -105.00 & 523.84  & 43448  &  & -105.00 & 457.17  & 44736  &  & -105.00 & 457.69  & 44841  &  & -105.00 & 497.66  & 43399  &  & -105.00 & 539.20  & 43448  \\
bmilplib-310-3  & -127.52 & 943.53  & 66426  &  & -127.52 & 907.46  & 72139  &  & -127.52 & 783.89  & 67441  &  & -127.52 & 777.74  & 66410  &  & -127.52 & 916.69  & 66426  \\
bmilplib-310-4  & -147.78 & 701.76  & 47718  &  & -147.78 & 510.81  & 48131  &  & -147.78 & 572.76  & 48950  &  & -147.78 & 569.43  & 47711  &  & -147.78 & 652.54  & 47718  \\
bmilplib-310-5  & -161.45 & 557.76  & 31838  &  & -161.45 & 339.05  & 31553  &  & -161.45 & 382.13  & 32454  &  & -161.45 & 366.90  & 31782  &  & -161.45 & 488.65  & 31838  \\
bmilplib-310-6  & -141.18 & 1448.71 & 102047 &  & -141.18 & 1159.42 & 110836 &  & -141.18 & 1264.60 & 103214 &  & -141.18 & 1191.51 & 101904 &  & -141.18 & 1324.84 & 102047 \\
bmilplib-310-7  & -142.00 & 1443.70 & 102030 &  & -142.00 & 1057.98 & 104876 &  & -142.00 & 1199.47 & 104166 &  & -142.00 & 1129.24 & 102000 &  & -142.00 & 1396.34 & 102030 \\
bmilplib-310-8  & -115.34 & 143.13  & 11375  &  & -115.34 & 110.74  & 11067  &  & -115.34 & 114.80  & 11293  &  & -115.34 & 127.94  & 11109  &  & -115.34 & 142.93  & 11375  \\
bmilplib-310-9  & -115.65 & 423.77  & 20552  &  & -115.65 & 231.20  & 21846  &  & -115.65 & 248.24  & 20838  &  & -115.65 & 255.27  & 20490  &  & -115.65 & 281.90  & 20552  \\
bmilplib-360-10 & -108.59 & 257.34  & 13727  &  & -108.59 & 209.20  & 13671  &  & -108.59 & 242.94  & 14064  &  & -108.59 & 234.92  & 13697  &  & -108.59 & 258.65  & 13727  \\
bmilplib-360-1  & -133.00 & 2416.81 & 75780  &  & -133.00 & 1239.50 & 77206  &  & -133.00 & 1297.32 & 77066  &  & -133.00 & 1353.62 & 75421  &  & -133.00 & 1500.96 & 75780  \\
bmilplib-360-2  & -138.44 & 1187.29 & 53004  &  & -138.44 & 829.40  & 52868  &  & -138.44 & 999.50  & 54354  &  & -138.44 & 965.52  & 52919  &  & -138.44 & 1169.80 & 53004  \\
bmilplib-360-3  & -131.00 & 834.51  & 40671  &  & -131.00 & 728.80  & 40353  &  & -131.00 & 624.78  & 41302  &  & -131.00 & 832.88  & 40487  &  & -131.00 & 731.63  & 40671  \\
bmilplib-360-4  & -119.00 & 371.93  & 18350  &  & -119.00 & 286.84  & 18870  &  & -119.00 & 338.49  & 18813  &  & -119.00 & 332.90  & 18293  &  & -119.00 & 361.79  & 18350  \\
bmilplib-360-5  & -164.26 & 593.68  & 30001  &  & -164.26 & 420.56  & 29790  &  & -164.26 & 484.66  & 30352  &  & -164.26 & 618.86  & 29947  &  & -164.26 & 600.63  & 30001  \\
bmilplib-360-6  & -110.12 & 1169.00 & 68863  &  & -110.12 & 1005.68 & 69453  &  & -110.12 & 1018.82 & 70283  &  & -110.12 & 1181.02 & 68863  &  & -110.12 & 1210.13 & 68863  \\
bmilplib-360-7  & -105.00 & 538.80  & 31092  &  & -105.00 & 457.98  & 32124  &  & -105.00 & 517.46  & 31884  &  & -105.00 & 634.17  & 30900  &  & -105.00 & 542.81  & 31092  \\
bmilplib-360-8  & -98.25  & 399.32  & 22995  &  & -98.25  & 343.65  & 22787  &  & -98.25  & 416.97  & 23337  &  & -98.25  & 362.89  & 22857  &  & -98.25  & 386.45  & 22995  \\
bmilplib-360-9  & -127.22 & 815.44  & 40383  &  & -127.22 & 622.13  & 40733  &  & -127.22 & 746.66  & 41235  &  & -127.22 & 736.16  & 40329  &  & -127.22 & 819.68  & 40383  \\
bmilplib-410-10 & -153.37 & 3527.64 & 101673 &  & -153.37 & 2725.83 & 103359 &  & -153.37 & 2879.03 & 103400 &  & -153.37 & 2729.77 & 101447 &  & -153.37 & 2976.52 & 101673 \\
bmilplib-410-1  & -103.50 & 582.60  & 20790  &  & -103.50 & 535.58  & 20612  &  & -103.50 & 589.62  & 21088  &  & -103.50 & 553.76  & 20634  &  & -103.50 & 534.12  & 20790  \\
bmilplib-410-2  & -108.59 & 803.44  & 31603  &  & -108.59 & 734.82  & 31602  &  & -108.59 & 748.06  & 32251  &  & -108.59 & 840.56  & 31603  &  & -108.59 & 882.69  & 31603  \\
bmilplib-410-3  & -96.24  & 1275.98 & 33041  &  & -96.24  & 791.92  & 33963  &  & -96.24  & 777.04  & 33473  &  & -96.24  & 871.58  & 32882  &  & -96.24  & 973.12  & 33041  \\
bmilplib-410-4  & -119.50 & 1582.86 & 38489  &  & -119.50 & 930.82  & 39270  &  & -119.50 & 1041.97 & 39371  &  & -119.50 & 1031.12 & 38489  &  & -119.50 & 1101.03 & 38489  \\
bmilplib-410-5  & -119.22 & 700.45  & 23865  &  & -119.22 & 567.85  & 23747  &  & -119.22 & 681.99  & 24163  &  & -119.22 & 678.49  & 23852  &  & -119.22 & 696.17  & 23865  \\
bmilplib-410-6  & -151.31 & 404.52  & 11130  &  & -151.31 & 261.83  & 10795  &  & -151.31 & 333.36  & 11193  &  & -151.31 & 322.88  & 11130  &  & -151.31 & 337.71  & 11130  \\
bmilplib-410-7  & -123.00 & 624.95  & 22208  &  & -123.00 & 519.43  & 22843  &  & -123.00 & 601.78  & 22628  &  & -123.00 & 560.34  & 22146  &  & -123.00 & 560.38  & 22208  \\
bmilplib-410-8  & -125.78 & 1717.64 & 62336  &  & -125.78 & 1484.95 & 61784  &  & -125.78 & 1692.35 & 63421  &  & -125.78 & 1547.88 & 62218  &  & -125.78 & 1579.12 & 62336  \\
bmilplib-410-9  & -100.77 & 505.02  & 20424  &  & -100.77 & 579.73  & 20095  &  & -100.77 & 607.44  & 20920  &  & -100.77 & 531.88  & 20400  &  & -100.77 & 494.06  & 20424  \\
bmilplib-460-10 & -102.51 & 1949.07 & 55030  &  & -102.51 & 1782.80 & 55882  &  & -102.51 & 1701.70 & 56265  &  & -102.51 & 1857.94 & 55024  &  & -102.51 & 1822.52 & 55030  \\
bmilplib-460-1  & -97.59  & 2961.39 & 86689  &  & -97.59  & 3575.71 & 93301  &  & -97.59  & 2938.02 & 87709  &  & -97.59  & 2822.97 & 86632  &  & -97.59  & 2915.62 & 86689  \\
bmilplib-460-2  & -139.00 & 739.75  & 16720  &  & -139.00 & 527.76  & 16721  &  & -139.00 & 608.40  & 16863  &  & -139.00 & 625.95  & 16623  &  & -139.00 & 694.03  & 16720  \\
bmilplib-460-3  & -86.50  & 1921.88 & 57462  &  & -86.50  & 1853.40 & 59065  &  & -86.50  & 2102.77 & 58229  &  & -86.50  & 1995.92 & 57395  &  & -86.50  & 1946.46 & 57462  \\
bmilplib-460-4  & -107.03 & 3321.16 & 95231  &  & -107.03 & 3173.50 & 94899  &  & -107.03 & 3328.91 & 97160  &  & -107.03 & 3291.28 & 95035  &  & -107.03 & 3301.27 & 95231  \\
bmilplib-460-5  & -100.50 & 1407.70 & 41312  &  & -100.50 & 1366.52 & 40903  &  & -100.50 & 1619.91 & 41796  &  & -100.50 & 1424.54 & 41109  &  & -100.50 & 1450.14 & 41312  \\
bmilplib-460-6  & -107.00 & 1623.54 & 46158  &  & -107.00 & 1537.79 & 48966  &  & -107.00 & 1483.17 & 46732  &  & -107.00 & 1635.42 & 46076  &  & -107.00 & 1598.68 & 46158  \\
\hline
\end{tabular}
}
\end{center}
\end{table}
%%%%%%%%%%%%%%%%%%
%table for figure 2b
\begin{table}[p]
\vskip -0.9in
\begin{center}
\caption{Detailed results of Figure~\ref{fig:SSUBParTimeMore5SecXu} (continued)}
\label{tab:fig2bContinue}
\resizebox{\columnwidth}{!}{
%\begin{tabular}{lccccccccccccccccccc}
\begin{tabular}{llllllllllllllllllll}
\hline
\multicolumn{1}{c}{}&
\multicolumn{3}{c}{whenLInt-LInt}&
\multicolumn{1}{c}{}&
\multicolumn{3}{c}{whenLInt-LFixed}&
\multicolumn{1}{c}{}&
\multicolumn{3}{c}{whenLFixed-LFixed}&
\multicolumn{1}{c}{}&
\multicolumn{3}{c}{whenXYInt-LFixed}&
\multicolumn{1}{c}{}&
\multicolumn{3}{c}{\begin{tabular}[c]{@{}c@{}}\scriptsize{whenXYIntOr}\\\scriptsize{LFixed-LFixed}\end{tabular}} \\
\cline{2-4}
\cline{6-8}
\cline{10-12}
\cline{14-16}
\cline{18-20}
Instance&   BestSol & Time(s) &Nodes  & &BestSol & Time(s) &Nodes & &BestSol & Time(s) &Nodes& &BestSol & Time(s) &Nodes& &BestSol & Time(s) &Nodes\\
\hline
bmilplib-460-7  & -83.75  & 1735.67 & 48973  &  & -83.75  & 1606.84 & 49516  &  & -83.75  & 1728.92 & 49717  &  & -83.75  & 1601.34 & 48897  &  & -83.75  & 1671.84 & 48973  \\
bmilplib-460-8  & -115.39 & 1074.21 & 27663  &  & -115.39 & 858.20  & 27275  &  & -115.39 & 911.51  & 28075  &  & -115.39 & 982.30  & 27572  &  & -115.39 & 1017.28 & 27663  \\
bmilplib-460-9  & -128.70 & 3410.18 & 86426  &  & -128.70 & 2671.50 & 86656  &  & -128.70 & 2715.98 & 87796  &  & -128.70 & 2830.59 & 85977  &  & -128.70 & 3004.49 & 86426  \\
bmilplib-60-10  & -186.21 & 14.85   & 5929   &  & -186.21 & 6.88    & 9593   &  & -186.21 & 6.91    & 6134   &  & -186.21 & 6.93    & 5922   &  & -186.21 & 9.50    & 5929   \\
bmilplib-60-1   & -153.20 & 16.56   & 4896   &  & -153.20 & 4.84    & 5366   &  & -153.20 & 6.50    & 5041   &  & -153.20 & 6.71    & 4877   &  & -153.20 & 10.92   & 4896   \\
bmilplib-60-5   & -116.40 & 15.19   & 11289  &  & -116.40 & 7.57    & 11308  &  & -116.40 & 8.65    & 13984  &  & -116.40 & 8.15    & 11202  &  & -116.40 & 11.54   & 11289  \\
bmilplib-60-6   & -187.31 & 15.38   & 7120   &  & -187.31 & 7.08    & 9241   &  & -187.31 & 7.34    & 7304   &  & -187.31 & 7.62    & 7016   &  & -187.31 & 10.44   & 7120   \\
bmilplib-60-8   & -232.12 & 8.38    & 3653   &  & -232.12 & 2.54    & 4052   &  & -232.12 & 3.33    & 3683   &  & -232.12 & 3.36    & 3570   &  & -232.12 & 5.66    & 3653   \\
bmilplib-60-9   & -136.50 & 33.85   & 27036  &  & -136.50 & 18.90   & 31312  &  & -136.50 & 20.22   & 28603  &  & -136.50 & 19.61   & 26792  &  & -136.50 & 26.84   & 27036 \\ 
\hline
\end{tabular}
}
\end{center}
\end{table}
%%%%%%%%%%%%%%%%%%
%table for figure 3a
\begin{table}[h!]
\caption{Detailed results of Figure~\ref{fig:branchStrategyParTimeMore5Secr1Leqr2}}
\label{tab:fig3a}
\scriptsize
\centering
%\resizebox{\columnwidth}{!}{
%\begin{tabular}{lccccccccc}
\begin{tabular}{llllllllll}
\hline
\multicolumn{1}{c}{}&
\multicolumn{1}{c}{}&
\multicolumn{1}{c}{}&
\multicolumn{3}{c}{linkingBranching}&
\multicolumn{1}{c}{}&
\multicolumn{3}{c}{fractionalBranching} \\
\cline{4-6}
\cline{8-10}
Instance&    $r_1$& $r_2$& BestSol & Time(s) &Nodes& &BestSol & Time(s) &Nodes\\
\hline
miblp-20-15-50-0110-10-10 & 5  & 10  & -206.00  & 1.06    & 423     &  & -206.00  & 4.25               & 15741    \\
miblp-20-15-50-0110-10-2  & 5  & 10  & -398.00  & 9.09    & 2450    &  & -398.00  & 947.81             & 2929420  \\
miblp-20-15-50-0110-10-3  & 5  & 10  & -42.00   & 0.62    & 267     &  & -42.00   & 12.66              & 46349    \\
miblp-20-15-50-0110-10-6  & 5  & 10  & -246.00  & 4.56    & 340     &  & -246.00  & 1.56               & 1367     \\
miblp-20-15-50-0110-10-9  & 5  & 10  & -635.00  & 14.05   & 2380    &  & -635.00  & 6.52               & 7667     \\
miblp-20-20-50-0110-10-10 & 10 & 10  & -441.00  & 692.00  & 134585  &  & -441.00  & 2867.85            & 7261247  \\
miblp-20-20-50-0110-10-1  & 10 & 10  & -359.00  & 228.67  & 96141   &  & -357.00  & \textgreater{}3600 & 9234364  \\
miblp-20-20-50-0110-10-2  & 10 & 10  & -659.00  & 3.69    & 3191    &  & -659.00  & 3.16               & 5100     \\
miblp-20-20-50-0110-10-3  & 10 & 10  & -618.00  & 13.18   & 20788   &  & -618.00  & 12.05              & 51970    \\
miblp-20-20-50-0110-10-4  & 10 & 10  & -604.00  & 3145.16 & 830808  &  & -604.00  & \textgreater{}3600 & 7988914  \\
miblp-20-20-50-0110-10-7  & 10 & 10  & -683.00  & 1887.37 & 3003967 &  & -629.00  & \textgreater{}3600 & 9709672  \\
miblp-20-20-50-0110-10-8  & 10 & 10  & -667.00  & 80.45   & 12857   &  & -667.00  & 68.08              & 75661    \\
miblp-20-20-50-0110-10-9  & 10 & 10  & -256.00  & 20.99   & 31245   &  & -256.00  & 305.78             & 757349   \\
miblp-20-20-50-0110-15-1  & 5  & 15  & -450.00  & 59.38   & 3506    &  & -317.00  & \textgreater{}3600 & 9813005  \\
miblp-20-20-50-0110-15-2  & 5  & 15  & -645.00  & 50.14   & 17251   &  & -645.00  & \textgreater{}3600 & 10839884 \\
miblp-20-20-50-0110-15-3  & 5  & 15  & -593.00  & 71.15   & 3081    &  & -593.00  & \textgreater{}3600 & 13042109 \\
miblp-20-20-50-0110-15-4  & 5  & 15  & -441.00  & 43.00   & 1625    &  & -398.00  & \textgreater{}3600 & 8692428  \\
miblp-20-20-50-0110-15-5  & 5  & 15  & -379.00  & 651.28  & 16715   &  & -320.00  & \textgreater{}3600 & 7284040  \\
miblp-20-20-50-0110-15-6  & 5  & 15  & -596.00  & 17.31   & 1657    &  & -596.00  & \textgreater{}3600 & 7851818  \\
miblp-20-20-50-0110-15-7  & 5  & 15  & -471.00  & 111.02  & 13405   &  & -471.00  & \textgreater{}3600 & 9675451  \\
miblp-20-20-50-0110-15-8  & 5  & 15  & -370.00  & 39.12   & 21589   &  & -290.00  & \textgreater{}3600 & 10350188 \\
miblp-20-20-50-0110-15-9  & 5  & 15  & -584.00  & 2.00    & 582     &  & -584.00  & 19.58              & 56459    \\
lseu-0.900000             & 9  & 80  & 5838.00  & 14.10   & 1023    &  & 5838.00  & \textgreater{}3600 & 8743754  \\
p0033-0.900000            & 4  & 29  & 4679.00  & 0.06    & 27      &  & 4679.00  & 5.56               & 28241    \\
p0201-0.900000            & 21 & 180 & 15025.00 & 6.62    & 2481    &  & 15025.00 & \textgreater{}3600 & 1310278  \\
stein27-0.900000          & 3  & 24  & 24.00    & 0.02    & 15      &  & 24.00    & 419.87             & 702055   \\
stein45-0.900000          & 5  & 40  & 40.00    & 0.14    & 63      &  & 40.00    & \textgreater{}3600 & 1499583 \\
\hline
\end{tabular}
%}
\end{table}
%%%%%%%%%%%%%%%%%%
%table for figure 3b
\begin{table}[h!]
\caption{Detailed results of Figure~\ref{fig:branchStrategyParTimeMore5Secr1grer2}}
\label{tab:fig3b}
%\resizebox{\columnwidth}{!}{
\scriptsize
\centering
%\begin{tabular}{lccccccccc}
\begin{tabular}{llllllllll}
\hline
\multicolumn{1}{c}{}&
\multicolumn{1}{c}{}&
\multicolumn{1}{c}{}&
\multicolumn{3}{c}{linkingBranching}&
\multicolumn{1}{c}{}&
\multicolumn{3}{c}{fractionalBranching} \\
\cline{4-6}
\cline{8-10}
Instance&    $r_1$& $r_2$& BestSol & Time(s) &Nodes& &BestSol & Time(s) &Nodes\\
\hline
bmilplib-110-10          & 110 & 63  & -177.67  & 153.82  & 75499   &  & -177.67  & 116.32  & 55177   \\
bmilplib-110-1           & 110 & 50  & -181.67  & 5.35    & 2806    &  & -181.67  & 0.93    & 306     \\
bmilplib-110-2           & 110 & 45  & -110.67  & 6.48    & 4174    &  & -110.67  & 1.16    & 303     \\
bmilplib-110-3           & 110 & 55  & -215.16  & 5.58    & 3471    &  & -215.16  & 0.89    & 360     \\
bmilplib-110-4           & 110 & 50  & -197.29  & 3.69    & 1751    &  & -197.29  & 1.34    & 148     \\
bmilplib-110-6           & 110 & 55  & -148.25  & 15.06   & 8770    &  & -148.25  & 3.84    & 1448    \\
bmilplib-110-7           & 110 & 61  & -160.86  & 4.94    & 2189    &  & -160.86  & 0.91    & 205     \\
bmilplib-110-8           & 110 & 54  & -155.00  & 20.78   & 11663   &  & -155.00  & 7.66    & 2274    \\
bmilplib-110-9           & 110 & 58  & -192.92  & 3.66    & 2151    &  & -192.92  & 0.39    & 146     \\
bmilplib-160-10          & 160 & 83  & -189.82  & 33.42   & 9071    &  & -189.82  & 7.73    & 728     \\
bmilplib-160-1           & 160 & 80  & -165.00  & 25.70   & 7901    &  & -165.00  & 4.76    & 881     \\
bmilplib-160-2           & 160 & 76  & -178.24  & 30.59   & 8635    &  & -178.24  & 5.95    & 507     \\
bmilplib-160-3           & 160 & 86  & -174.94  & 63.37   & 15002   &  & -174.94  & 13.71   & 1102    \\
bmilplib-160-4           & 160 & 81  & -135.83  & 51.02   & 14772   &  & -135.83  & 19.13   & 2447    \\
bmilplib-160-5           & 160 & 83  & -140.78  & 20.52   & 4948    &  & -140.78  & 7.72    & 668     \\
bmilplib-160-6           & 160 & 81  & -111.00  & 24.66   & 8676    &  & -111.00  & 5.02    & 627     \\
bmilplib-160-7           & 160 & 79  & -96.00   & 51.20   & 16560   &  & -96.00   & 18.72   & 2855    \\
bmilplib-160-8           & 160 & 85  & -181.40  & 10.49   & 3444    &  & -181.40  & 1.74    & 311     \\
bmilplib-160-9           & 160 & 77  & -207.50  & 15.85   & 4607    &  & -207.50  & 2.02    & 268     \\
bmilplib-210-10          & 210 & 99  & -130.59  & 61.81   & 12176   &  & -130.59  & 8.59    & 1100    \\
bmilplib-210-1           & 210 & 101 & -136.80  & 41.96   & 7429    &  & -136.80  & 6.67    & 550     \\
bmilplib-210-2           & 210 & 96  & -117.80  & 113.38  & 22294   &  & -117.80  & 19.78   & 2306    \\
bmilplib-210-3           & 210 & 119 & -130.80  & 48.19   & 8848    &  & -130.80  & 12.93   & 1380    \\
bmilplib-210-4           & 210 & 115 & -162.20  & 26.81   & 4687    &  & -162.20  & 3.64    & 309     \\
bmilplib-210-5           & 210 & 110 & -134.00  & 108.22  & 21058   &  & -134.00  & 20.03   & 2079    \\
bmilplib-210-6           & 210 & 115 & -125.43  & 201.40  & 38443   &  & -125.43  & 45.28   & 4875    \\
bmilplib-210-7           & 210 & 102 & -169.73  & 77.84   & 13960   &  & -169.73  & 11.46   & 1181    \\
bmilplib-210-8           & 210 & 116 & -101.46  & 61.79   & 11105   &  & -101.46  & 9.60    & 942     \\
bmilplib-210-9           & 210 & 103 & -184.00  & 879.61  & 142294  &  & -184.00  & 240.14  & 9466    \\
bmilplib-260-10          & 260 & 117 & -151.73  & 546.55  & 61767   &  & -151.73  & 73.10   & 4716    \\
bmilplib-260-1           & 260 & 135 & -139.00  & 85.81   & 10980   &  & -139.00  & 10.10   & 887     \\
bmilplib-260-2           & 260 & 126 & -82.62   & 109.75  & 15287   &  & -82.62   & 18.07   & 1607    \\
bmilplib-260-3           & 260 & 120 & -144.25  & 71.54   & 8769    &  & -144.25  & 8.70    & 518     \\
bmilplib-260-4           & 260 & 125 & -117.33  & 259.48  & 32787   &  & -117.33  & 66.89   & 4426    \\
bmilplib-260-5           & 260 & 132 & -121.00  & 166.03  & 21542   &  & -121.00  & 29.19   & 2165    \\
bmilplib-260-6           & 260 & 146 & -124.00  & 197.04  & 25362   &  & -124.00  & 40.59   & 2420    \\
bmilplib-260-7           & 260 & 129 & -137.80  & 312.12  & 40274   &  & -137.80  & 44.45   & 3200    \\
bmilplib-260-8           & 260 & 143 & -119.89  & 87.01   & 10025   &  & -119.89  & 10.92   & 1025    \\
bmilplib-260-9           & 260 & 132 & -160.00  & 273.66  & 33468   &  & -160.00  & 36.36   & 2526    \\
bmilplib-310-10          & 310 & 157 & -141.86  & 121.07  & 9900    &  & -141.86  & 5.51    & 397     \\
bmilplib-310-1           & 310 & 169 & -117.00  & 329.71  & 28330   &  & -117.00  & 44.79   & 2624    \\
bmilplib-310-2           & 310 & 154 & -105.00  & 497.66  & 43399   &  & -105.00  & 98.39   & 5372    \\
bmilplib-310-3           & 310 & 157 & -127.52  & 777.74  & 66410   &  & -127.52  & 149.07  & 7067    \\
bmilplib-310-4           & 310 & 152 & -147.78  & 569.43  & 47711   &  & -147.78  & 70.22   & 4767    \\
bmilplib-310-5           & 310 & 164 & -161.45  & 366.90  & 31782   &  & -161.45  & 34.22   & 1993    \\
bmilplib-310-6           & 310 & 148 & -141.18  & 1191.51 & 101904  &  & -141.18  & 169.96  & 8300    \\
bmilplib-310-7           & 310 & 170 & -142.00  & 1129.24 & 102000  &  & -142.00  & 139.15  & 7263    \\
bmilplib-310-8           & 310 & 154 & -115.34  & 127.94  & 11109   &  & -115.34  & 19.23   & 921     \\
bmilplib-310-9           & 310 & 150 & -115.65  & 255.27  & 20490   &  & -115.65  & 32.31   & 1590    \\
bmilplib-360-10          & 360 & 172 & -108.59  & 234.92  & 13697   &  & -108.59  & 25.75   & 1106    \\
bmilplib-360-1           & 360 & 181 & -133.00  & 1353.62 & 75421   &  & -133.00  & 158.38  & 4923    \\
bmilplib-360-2           & 360 & 179 & -138.44  & 965.52  & 52919   &  & -138.44  & 148.90  & 4493    \\
bmilplib-360-3           & 360 & 195 & -131.00  & 832.88  & 40487   &  & -131.00  & 65.41   & 2654    \\
bmilplib-360-4           & 360 & 184 & -119.00  & 332.90  & 18293   &  & -119.00  & 42.95   & 1564    \\
bmilplib-360-5           & 360 & 194 & -164.26  & 618.86  & 29947   &  & -164.26  & 44.45   & 1713    \\
bmilplib-360-6           & 360 & 172 & -110.12  & 1181.02 & 68863   &  & -110.12  & 142.81  & 5520    \\
bmilplib-360-7           & 360 & 188 & -105.00  & 634.17  & 30900   &  & -105.00  & 89.22   & 3346    \\
bmilplib-360-8           & 360 & 170 & -98.25   & 362.89  & 22857   &  & -98.25   & 44.85   & 1686    \\
bmilplib-360-9           & 360 & 184 & -127.22  & 736.16  & 40329   &  & -127.22  & 50.24   & 2642    \\
bmilplib-410-10          & 410 & 201 & -153.37  & 2729.77 & 101447  &  & -153.37  & 258.89  & 7428    \\
bmilplib-410-1           & 410 & 196 & -103.50  & 553.76  & 20634   &  & -103.50  & 87.94   & 1944    \\
\hline
\end{tabular}
%}
\end{table}
%%%%%%%%%%%%%%%%%%
%table for figure 3b
\addtocounter{table}{-1}
\begin{table}[h!]
\vskip -3.1in
\caption{Detailed results of Figure~\ref{fig:branchStrategyParTimeMore5Secr1grer2} (continued)}
\label{tab:fig3bContinue}
%\resizebox{\columnwidth}{!}{
\scriptsize
\centering
%\begin{tabular}{lccccccccc}
\begin{tabular}{llllllllll}
\hline
\multicolumn{1}{c}{}&
\multicolumn{1}{c}{}&
\multicolumn{1}{c}{}&
\multicolumn{3}{c}{linkingBranching}&
\multicolumn{1}{c}{}&
\multicolumn{3}{c}{fractionalBranching} \\
\cline{4-6}
\cline{8-10}
Instance&    $r_1$& $r_2$& BestSol & Time(s) &Nodes& &BestSol & Time(s) &Nodes\\
\hline
bmilplib-410-2           & 410 & 189 & -108.59  & 840.56  & 31603   &  & -108.59  & 103.57  & 2887    \\
bmilplib-410-3           & 410 & 212 & -96.24   & 871.58  & 32882   &  & -96.24   & 147.99  & 3781    \\
bmilplib-410-4           & 410 & 187 & -119.50  & 1031.12 & 38489   &  & -119.50  & 100.21  & 2995    \\
bmilplib-410-5           & 410 & 209 & -119.22  & 678.49  & 23852   &  & -119.22  & 71.23   & 1520    \\
bmilplib-410-6           & 410 & 206 & -151.31  & 322.88  & 11130   &  & -151.31  & 13.54   & 533     \\
bmilplib-410-7           & 410 & 225 & -123.00  & 560.34  & 22146   &  & -123.00  & 35.87   & 1177    \\
bmilplib-410-8           & 410 & 211 & -125.78  & 1547.88 & 62218   &  & -125.78  & 169.15  & 4480    \\
bmilplib-410-9           & 410 & 216 & -100.77  & 531.88  & 20400   &  & -100.77  & 82.17   & 2071    \\
bmilplib-460-10          & 460 & 217 & -102.51  & 1857.94 & 55024   &  & -102.51  & 228.12  & 4465    \\
bmilplib-460-1           & 460 & 227 & -97.59   & 2822.97 & 86632   &  & -97.59   & 569.22  & 10803   \\
bmilplib-460-2           & 460 & 249 & -139.00  & 625.95  & 16623   &  & -139.00  & 43.65   & 964     \\
bmilplib-460-3           & 460 & 222 & -86.50   & 1995.92 & 57395   &  & -86.50   & 223.58  & 3882    \\
bmilplib-460-4           & 460 & 218 & -107.03  & 3291.28 & 95035   &  & -107.03  & 412.76  & 7856    \\
bmilplib-460-5           & 460 & 216 & -100.50  & 1424.54 & 41109   &  & -100.50  & 170.30  & 3025    \\
bmilplib-460-6           & 460 & 222 & -107.00  & 1635.42 & 46076   &  & -107.00  & 236.30  & 4143    \\
bmilplib-460-7           & 460 & 254 & -83.75   & 1601.34 & 48897   &  & -83.75   & 294.68  & 5252    \\
bmilplib-460-8           & 460 & 256 & -115.39  & 982.30  & 27572   &  & -115.39  & 94.67   & 1903    \\
bmilplib-460-9           & 460 & 224 & -128.70  & 2830.59 & 85977   &  & -128.70  & 327.68  & 6185    \\
bmilplib-60-10           & 60  & 25  & -186.21  & 6.93    & 5922    &  & -186.21  & 4.29    & 2590    \\
bmilplib-60-1            & 60  & 29  & -153.20  & 6.71    & 4877    &  & -153.20  & 2.79    & 1094    \\
bmilplib-60-5            & 60  & 33  & -116.40  & 8.15    & 11202   &  & -116.40  & 10.08   & 9996    \\
bmilplib-60-6            & 60  & 27  & -187.31  & 7.62    & 7016    &  & -187.31  & 3.34    & 1383    \\
bmilplib-60-8            & 60  & 30  & -232.12  & 3.36    & 3570    &  & -232.12  & 1.15    & 572     \\
bmilplib-60-9            & 60  & 26  & -136.50  & 19.61   & 26792   &  & -136.50  & 12.34   & 9888    \\
miblp-20-20-50-0110-5-13 & 15  & 5   & -519.00  & 2464.25 & 7128138 &  & -519.00  & 2709.29 & 6515097 \\
miblp-20-20-50-0110-5-15 & 15  & 5   & -617.00  & 1310.67 & 4018058 &  & -617.00  & 1130.92 & 4312915 \\
miblp-20-20-50-0110-5-16 & 15  & 5   & -833.00  & 6.60    & 17680   &  & -833.00  & 1.80    & 5913    \\
miblp-20-20-50-0110-5-17 & 15  & 5   & -944.00  & 4.78    & 17679   &  & -944.00  & 1.53    & 4038    \\
miblp-20-20-50-0110-5-19 & 15  & 5   & -431.00  & 26.56   & 77256   &  & -431.00  & 25.50   & 116041  \\
miblp-20-20-50-0110-5-1  & 15  & 5   & -548.00  & 14.50   & 42364   &  & -548.00  & 8.45    & 31298   \\
miblp-20-20-50-0110-5-20 & 15  & 5   & -438.00  & 17.30   & 60944   &  & -438.00  & 10.82   & 33315   \\
miblp-20-20-50-0110-5-6  & 15  & 5   & -1061.00 & 63.70   & 222566  &  & -1061.00 & 51.74   & 213928  \\
lseu-0.100000            & 81  & 8   & 1120.00  & 270.78  & 1003976 &  & 1120.00  & 3.15    & 8603    \\
p0033-0.500000           & 17  & 16  & 3095.00  & 10.48   & 20695   &  & 3095.00  & 0.25    & 1467    \\
stein27-0.500000         & 14  & 13  & 19.00    & 5.88    & 12115   &  & 19.00    & 6.51    & 17648   \\
stein45-0.100000         & 41  & 4   & 30.00    & 50.19   & 89035   &  & 30.00    & 64.27   & 90241   \\
stein45-0.500000         & 23  & 22  & 32.00    & 519.59  & 640308  &  & 32.00    & 471.57  & 753845 \\
\hline
\end{tabular}
%}
\end{table}
%%%%%%%%%%%%%%%%%%
%table for figure 4a
\begin{table}[h!]
\caption{Detailed results of Figure~\ref{fig:linkingSolutionPooParWithFracParWithTimeMore5Sec}}
\label{tab:fig4a}
\resizebox{\columnwidth}{!}{
%\begin{tabular}{lccccccccccccccc}
\begin{tabular}{llllllllllllllll}
\hline
\multicolumn{1}{c}{}&
\multicolumn{3}{c}{\begin{tabular}[c]{@{}c@{}}\scriptsize{withoutPoolWhen}\\\scriptsize{XYInt-LFixed}\end{tabular}}&
\multicolumn{1}{c}{}&
\multicolumn{3}{c}{\begin{tabular}[c]{@{}c@{}}\scriptsize{withPoolWhen}\\\scriptsize{XYInt-LFixed}\end{tabular}}&
\multicolumn{1}{c}{}&
\multicolumn{3}{c}{\begin{tabular}[c]{@{}c@{}}\scriptsize{withoutPoolWhen}\\\scriptsize{XYIntOrLFixed-LFixed}\end{tabular}}&
\multicolumn{1}{c}{}&
\multicolumn{3}{c}{\begin{tabular}[c]{@{}c@{}}\scriptsize{withPoolWhen}\\\scriptsize{XYIntOrLFixed-LFixed}\end{tabular}} \\
\cline{2-4}
\cline{6-8}
\cline{10-12}
\cline{14-16}
Instance&   BestSol & Time(s) &Nodes  & &BestSol & Time(s) &Nodes & &BestSol & Time(s) &Nodes& &BestSol & Time(s) &Nodes\\
\hline
bmilplib-110-10           & -177.67  & 355.72             & 55177   &  & -177.67  & 117.08             & 55177    &  & -177.67  & 357.09             & 55177   &  & -177.67  & 116.32             & 55177    \\
bmilplib-110-1            & -181.67  & 1.08               & 306     &  & -181.67  & 0.94               & 306      &  & -181.67  & 1.08               & 306     &  & -181.67  & 0.93               & 306      \\
bmilplib-110-2            & -110.67  & 1.56               & 303     &  & -110.67  & 1.17               & 303      &  & -110.67  & 1.56               & 303     &  & -110.67  & 1.16               & 303      \\
bmilplib-110-3            & -215.16  & 1.18               & 360     &  & -215.16  & 0.91               & 360      &  & -215.16  & 1.14               & 360     &  & -215.16  & 0.89               & 360      \\
bmilplib-110-4            & -197.29  & 2.38               & 148     &  & -197.29  & 1.34               & 148      &  & -197.29  & 2.38               & 148     &  & -197.29  & 1.34               & 148      \\
bmilplib-110-6            & -148.25  & 8.77               & 1448    &  & -148.25  & 3.83               & 1448     &  & -148.25  & 8.78               & 1448    &  & -148.25  & 3.84               & 1448     \\
bmilplib-110-7            & -160.86  & 1.16               & 205     &  & -160.86  & 0.93               & 205      &  & -160.86  & 1.17               & 205     &  & -160.86  & 0.91               & 205      \\
bmilplib-110-8            & -155.00  & 12.40              & 2274    &  & -155.00  & 7.74               & 2274     &  & -155.00  & 12.29              & 2274    &  & -155.00  & 7.66               & 2274     \\
bmilplib-110-9            & -192.92  & 0.39               & 146     &  & -192.92  & 0.38               & 146      &  & -192.92  & 0.38               & 146     &  & -192.92  & 0.39               & 146      \\
bmilplib-160-10           & -189.82  & 17.00              & 728     &  & -189.82  & 7.70               & 728      &  & -189.82  & 16.94              & 728     &  & -189.82  & 7.73               & 728      \\
bmilplib-160-1            & -165.00  & 6.52               & 881     &  & -165.00  & 4.72               & 881      &  & -165.00  & 6.48               & 881     &  & -165.00  & 4.76               & 881      \\
bmilplib-160-2            & -178.24  & 9.39               & 507     &  & -178.24  & 5.98               & 507      &  & -178.24  & 9.42               & 507     &  & -178.24  & 5.95               & 507      \\
bmilplib-160-3            & -174.94  & 32.47              & 1102    &  & -174.94  & 13.76              & 1102     &  & -174.94  & 32.60              & 1102    &  & -174.94  & 13.71              & 1102     \\
bmilplib-160-4            & -135.83  & 35.80              & 2447    &  & -135.83  & 19.06              & 2447     &  & -135.83  & 36.00              & 2447    &  & -135.83  & 19.13              & 2447     \\
bmilplib-160-5            & -140.78  & 29.26              & 668     &  & -140.78  & 7.73               & 668      &  & -140.78  & 29.33              & 668     &  & -140.78  & 7.72               & 668      \\
bmilplib-160-6            & -111.00  & 7.26               & 627     &  & -111.00  & 5.01               & 627      &  & -111.00  & 7.30               & 627     &  & -111.00  & 5.02               & 627      \\
bmilplib-160-7            & -96.00   & 35.18              & 2855    &  & -96.00   & 18.04              & 2855     &  & -96.00   & 35.13              & 2855    &  & -96.00   & 18.72              & 2855     \\
bmilplib-160-8            & -181.40  & 2.30               & 311     &  & -181.40  & 1.67               & 311      &  & -181.40  & 2.28               & 311     &  & -181.40  & 1.74               & 311      \\
bmilplib-160-9            & -207.50  & 2.02               & 268     &  & -207.50  & 2.01               & 268      &  & -207.50  & 2.00               & 268     &  & -207.50  & 2.02               & 268      \\
bmilplib-210-10           & -130.59  & 13.92              & 1100    &  & -130.59  & 8.78               & 1100     &  & -130.59  & 14.28              & 1100    &  & -130.59  & 8.59               & 1100     \\
bmilplib-210-1            & -136.80  & 6.72               & 550     &  & -136.80  & 6.71               & 550      &  & -136.80  & 6.74               & 550     &  & -136.80  & 6.67               & 550      \\
bmilplib-210-2            & -117.80  & 40.52              & 2306    &  & -117.80  & 20.60              & 2306     &  & -117.80  & 41.96              & 2306    &  & -117.80  & 19.78              & 2306     \\
bmilplib-210-3            & -130.80  & 24.72              & 1380    &  & -130.80  & 13.09              & 1380     &  & -130.80  & 24.48              & 1380    &  & -130.80  & 12.93              & 1380     \\
bmilplib-210-4            & -162.20  & 4.82               & 309     &  & -162.20  & 3.65               & 309      &  & -162.20  & 4.80               & 309     &  & -162.20  & 3.64               & 309      \\
bmilplib-210-5            & -134.00  & 29.57              & 2079    &  & -134.00  & 20.08              & 2079     &  & -134.00  & 30.96              & 2079    &  & -134.00  & 20.03              & 2079     \\
bmilplib-210-6            & -125.43  & 105.93             & 4875    &  & -125.43  & 46.30              & 4875     &  & -125.43  & 107.54             & 4875    &  & -125.43  & 45.28              & 4875     \\
bmilplib-210-7            & -169.73  & 18.21              & 1181    &  & -169.73  & 11.59              & 1181     &  & -169.73  & 18.14              & 1181    &  & -169.73  & 11.46              & 1181     \\
bmilplib-210-8            & -101.46  & 17.89              & 942     &  & -101.46  & 9.69               & 942      &  & -101.46  & 17.97              & 942     &  & -101.46  & 9.60               & 942      \\
bmilplib-210-9            & -184.00  & 334.58             & 9466    &  & -184.00  & 241.89             & 9466     &  & -184.00  & 336.35             & 9466    &  & -184.00  & 240.14             & 9466     \\
bmilplib-260-10           & -151.73  & 106.84             & 4716    &  & -151.73  & 75.99              & 4716     &  & -151.73  & 106.74             & 4716    &  & -151.73  & 73.10              & 4716     \\
bmilplib-260-1            & -139.00  & 10.55              & 887     &  & -139.00  & 9.72               & 887      &  & -139.00  & 10.17              & 887     &  & -139.00  & 10.10              & 887      \\
bmilplib-260-2            & -82.62   & 25.82              & 1607    &  & -82.62   & 17.95              & 1607     &  & -82.62   & 24.39              & 1607    &  & -82.62   & 18.07              & 1607     \\
bmilplib-260-3            & -144.25  & 12.22              & 518     &  & -144.25  & 8.70               & 518      &  & -144.25  & 12.20              & 518     &  & -144.25  & 8.70               & 518      \\
bmilplib-260-4            & -117.33  & 82.73              & 4426    &  & -117.33  & 62.90              & 4426     &  & -117.33  & 84.74              & 4426    &  & -117.33  & 66.89              & 4426     \\
bmilplib-260-5            & -121.00  & 36.18              & 2165    &  & -121.00  & 29.80              & 2165     &  & -121.00  & 37.57              & 2165    &  & -121.00  & 29.19              & 2165     \\
bmilplib-260-6            & -124.00  & 67.59              & 2420    &  & -124.00  & 39.43              & 2420     &  & -124.00  & 69.01              & 2420    &  & -124.00  & 40.59              & 2420     \\
bmilplib-260-7            & -137.80  & 76.55              & 3200    &  & -137.80  & 46.85              & 3200     &  & -137.80  & 79.79              & 3200    &  & -137.80  & 44.45              & 3200     \\
bmilplib-260-8            & -119.89  & 16.52              & 1025    &  & -119.89  & 10.90              & 1025     &  & -119.89  & 17.44              & 1025    &  & -119.89  & 10.92              & 1025     \\
bmilplib-260-9            & -160.00  & 61.66              & 2526    &  & -160.00  & 36.58              & 2526     &  & -160.00  & 64.72              & 2526    &  & -160.00  & 36.36              & 2526     \\
bmilplib-310-10           & -141.86  & 5.42               & 397     &  & -141.86  & 5.95               & 397      &  & -141.86  & 5.88               & 397     &  & -141.86  & 5.51               & 397      \\
bmilplib-310-1            & -117.00  & 58.20              & 2624    &  & -117.00  & 43.64              & 2624     &  & -117.00  & 62.56              & 2624    &  & -117.00  & 44.79              & 2624     \\
bmilplib-310-2            & -105.00  & 109.26             & 5372    &  & -105.00  & 96.63              & 5372     &  & -105.00  & 117.75             & 5372    &  & -105.00  & 98.39              & 5372     \\
bmilplib-310-3            & -127.52  & 171.44             & 7067    &  & -127.52  & 148.52             & 7067     &  & -127.52  & 174.17             & 7067    &  & -127.52  & 149.07             & 7067     \\
bmilplib-310-4            & -147.78  & 90.78              & 4767    &  & -147.78  & 77.28              & 4767     &  & -147.78  & 95.75              & 4767    &  & -147.78  & 70.22              & 4767     \\
bmilplib-310-5            & -161.45  & 48.58              & 1993    &  & -161.45  & 34.30              & 1993     &  & -161.45  & 45.88              & 1993    &  & -161.45  & 34.22              & 1993     \\
bmilplib-310-6            & -141.18  & 197.54             & 8300    &  & -141.18  & 179.02             & 8300     &  & -141.18  & 198.16             & 8300    &  & -141.18  & 169.96             & 8300     \\
bmilplib-310-7            & -142.00  & 179.47             & 7263    &  & -142.00  & 144.99             & 7263     &  & -142.00  & 169.98             & 7263    &  & -142.00  & 139.15             & 7263     \\
bmilplib-310-8            & -115.34  & 25.48              & 921     &  & -115.34  & 20.24              & 921      &  & -115.34  & 24.68              & 921     &  & -115.34  & 19.23              & 921      \\
bmilplib-310-9            & -115.65  & 43.85              & 1590    &  & -115.65  & 37.72              & 1590     &  & -115.65  & 45.18              & 1590    &  & -115.65  & 32.31              & 1590     \\
bmilplib-360-10           & -108.59  & 27.82              & 1106    &  & -108.59  & 25.17              & 1106     &  & -108.59  & 31.69              & 1106    &  & -108.59  & 25.75              & 1106     \\
bmilplib-360-1            & -133.00  & 207.62             & 4923    &  & -133.00  & 146.36             & 4923     &  & -133.00  & 194.58             & 4923    &  & -133.00  & 158.38             & 4923     \\
bmilplib-360-2            & -138.44  & 230.49             & 4493    &  & -138.44  & 135.61             & 4493     &  & -138.44  & 225.09             & 4493    &  & -138.44  & 148.90             & 4493     \\
bmilplib-360-3            & -131.00  & 68.30              & 2654    &  & -131.00  & 63.40              & 2654     &  & -131.00  & 75.82              & 2654    &  & -131.00  & 65.41              & 2654     \\
bmilplib-360-4            & -119.00  & 54.68              & 1564    &  & -119.00  & 47.06              & 1564     &  & -119.00  & 51.33              & 1564    &  & -119.00  & 42.95              & 1564     \\
bmilplib-360-5            & -164.26  & 58.30              & 1713    &  & -164.26  & 45.64              & 1713     &  & -164.26  & 56.12              & 1713    &  & -164.26  & 44.45              & 1713     \\
bmilplib-360-6            & -110.12  & 155.42             & 5520    &  & -110.12  & 132.79             & 5520     &  & -110.12  & 146.64             & 5520    &  & -110.12  & 142.81             & 5520     \\
bmilplib-360-7            & -105.00  & 135.16             & 3346    &  & -105.00  & 93.50              & 3346     &  & -105.00  & 125.91             & 3346    &  & -105.00  & 89.22              & 3346     \\
bmilplib-360-8            & -98.25   & 66.38              & 1686    &  & -98.25   & 50.86              & 1686     &  & -98.25   & 66.34              & 1686    &  & -98.25   & 44.85              & 1686     \\
bmilplib-360-9            & -127.22  & 67.05              & 2642    &  & -127.22  & 58.36              & 2642     &  & -127.22  & 68.26              & 2642    &  & -127.22  & 50.24              & 2642     \\
bmilplib-410-10           & -153.37  & 332.48             & 7428    &  & -153.37  & 265.98             & 7428     &  & -153.37  & 333.37             & 7428    &  & -153.37  & 258.89             & 7428     \\
bmilplib-410-1            & -103.50  & 103.54             & 1944    &  & -103.50  & 85.04              & 1944     &  & -103.50  & 109.43             & 1944    &  & -103.50  & 87.94              & 1944     \\
\hline
\end{tabular}
}
\end{table}
%%%%%%%%%%%%%%%%%%
%table for figure 4a (continued)
\addtocounter{table}{-1}
\begin{table}[h!]
\vskip -0.7in
\caption{Detailed results of Figure~\ref{fig:linkingSolutionPooParWithFracParWithTimeMore5Sec} (continued)}
\label{tab:fig4aContinue}
\resizebox{\columnwidth}{!}{
%\begin{tabular}{lccccccccccccccc}
\begin{tabular}{llllllllllllllll}
\hline
\multicolumn{1}{c}{}&
\multicolumn{3}{c}{\begin{tabular}[c]{@{}c@{}}\scriptsize{withoutPoolWhen}\\\scriptsize{XYInt-LFixed}\end{tabular}}&
\multicolumn{1}{c}{}&
\multicolumn{3}{c}{\begin{tabular}[c]{@{}c@{}}\scriptsize{withPoolWhen}\\\scriptsize{XYInt-LFixed}\end{tabular}}&
\multicolumn{1}{c}{}&
\multicolumn{3}{c}{\begin{tabular}[c]{@{}c@{}}\scriptsize{withoutPoolWhen}\\\scriptsize{XYIntOrLFixed-LFixed}\end{tabular}}&
\multicolumn{1}{c}{}&
\multicolumn{3}{c}{\begin{tabular}[c]{@{}c@{}}\scriptsize{withPoolWhen}\\\scriptsize{XYIntOrLFixed-LFixed}\end{tabular}} \\
\cline{2-4}
\cline{6-8}
\cline{10-12}
\cline{14-16}
Instance&   BestSol & Time(s) &Nodes  & &BestSol & Time(s) &Nodes & &BestSol & Time(s) &Nodes& &BestSol & Time(s) &Nodes\\
\hline
bmilplib-410-2            & -108.59  & 128.11             & 2887    &  & -108.59  & 107.05             & 2887     &  & -108.59  & 125.37             & 2887    &  & -108.59  & 103.57             & 2887     \\
bmilplib-410-3            & -96.24   & 176.89             & 3781    &  & -96.24   & 162.97             & 3781     &  & -96.24   & 167.58             & 3781    &  & -96.24   & 147.99             & 3781     \\
bmilplib-410-4            & -119.50  & 111.28             & 2995    &  & -119.50  & 96.04              & 2995     &  & -119.50  & 115.22             & 2995    &  & -119.50  & 100.21             & 2995     \\
bmilplib-410-5            & -119.22  & 72.48              & 1520    &  & -119.22  & 67.90              & 1520     &  & -119.22  & 73.09              & 1520    &  & -119.22  & 71.23              & 1520     \\
bmilplib-410-6            & -151.31  & 14.12              & 533     &  & -151.31  & 14.66              & 533      &  & -151.31  & 13.95              & 533     &  & -151.31  & 13.54              & 533      \\
bmilplib-410-7            & -123.00  & 43.96              & 1177    &  & -123.00  & 42.46              & 1177     &  & -123.00  & 42.43              & 1177    &  & -123.00  & 35.87              & 1177     \\
bmilplib-410-8            & -125.78  & 292.85             & 4480    &  & -125.78  & 168.87             & 4480     &  & -125.78  & 284.91             & 4480    &  & -125.78  & 169.15             & 4480     \\
bmilplib-410-9            & -100.77  & 95.70              & 2071    &  & -100.77  & 81.09              & 2071     &  & -100.77  & 89.94              & 2071    &  & -100.77  & 82.17              & 2071     \\
bmilplib-460-10           & -102.51  & 284.10             & 4465    &  & -102.51  & 232.16             & 4465     &  & -102.51  & 272.39             & 4465    &  & -102.51  & 228.12             & 4465     \\
bmilplib-460-1            & -97.59   & 638.06             & 10803   &  & -97.59   & 574.36             & 10803    &  & -97.59   & 636.80             & 10803   &  & -97.59   & 569.22             & 10803    \\
bmilplib-460-2            & -139.00  & 65.03              & 964     &  & -139.00  & 51.23              & 964      &  & -139.00  & 68.32              & 964     &  & -139.00  & 43.65              & 964      \\
bmilplib-460-3            & -86.50   & 226.69             & 3882    &  & -86.50   & 223.11             & 3882     &  & -86.50   & 229.43             & 3882    &  & -86.50   & 223.58             & 3882     \\
bmilplib-460-4            & -107.03  & 640.32             & 7856    &  & -107.03  & 413.15             & 7856     &  & -107.03  & 607.29             & 7856    &  & -107.03  & 412.76             & 7856     \\
bmilplib-460-5            & -100.50  & 229.94             & 3025    &  & -100.50  & 171.50             & 3025     &  & -100.50  & 196.94             & 3025    &  & -100.50  & 170.30             & 3025     \\
bmilplib-460-6            & -107.00  & 251.71             & 4143    &  & -107.00  & 219.39             & 4143     &  & -107.00  & 221.43             & 4143    &  & -107.00  & 236.30             & 4143     \\
bmilplib-460-7            & -83.75   & 334.13             & 5252    &  & -83.75   & 295.91             & 5252     &  & -83.75   & 331.02             & 5252    &  & -83.75   & 294.68             & 5252     \\
bmilplib-460-8            & -115.39  & 117.85             & 1903    &  & -115.39  & 102.20             & 1903     &  & -115.39  & 134.03             & 1903    &  & -115.39  & 94.67              & 1903     \\
bmilplib-460-9            & -128.70  & 409.66             & 6185    &  & -128.70  & 346.00             & 6185     &  & -128.70  & 388.74             & 6185    &  & -128.70  & 327.68             & 6185     \\
bmilplib-60-10            & -186.21  & 4.82               & 2590    &  & -186.21  & 4.12               & 2590     &  & -186.21  & 4.90               & 2590    &  & -186.21  & 4.29               & 2590     \\
bmilplib-60-1             & -153.20  & 3.59               & 1094    &  & -153.20  & 2.83               & 1094     &  & -153.20  & 3.82               & 1094    &  & -153.20  & 2.79               & 1094     \\
bmilplib-60-5             & -116.40  & 16.18              & 9996    &  & -116.40  & 10.33              & 9996     &  & -116.40  & 15.85              & 9996    &  & -116.40  & 10.08              & 9996     \\
bmilplib-60-6             & -187.31  & 3.58               & 1383    &  & -187.31  & 3.43               & 1383     &  & -187.31  & 3.65               & 1383    &  & -187.31  & 3.34               & 1383     \\
bmilplib-60-8             & -232.12  & 1.79               & 572     &  & -232.12  & 1.16               & 572      &  & -232.12  & 1.78               & 572     &  & -232.12  & 1.15               & 572      \\
bmilplib-60-9             & -136.50  & 14.85              & 9888    &  & -136.50  & 12.36              & 9888     &  & -136.50  & 15.14              & 9888    &  & -136.50  & 12.34              & 9888     \\
miblp-20-15-50-0110-10-10 & -206.00  & 36.93              & 21203   &  & -206.00  & 4.22               & 15741    &  & -206.00  & 39.37              & 15741   &  & -206.00  & 4.25               & 15741    \\
miblp-20-15-50-0110-10-2  & -354.00  & \textgreater{}3600 & 642894  &  & -398.00  & 970.06             & 2929420  &  & -354.00  & \textgreater{}3600 & 596423  &  & -398.00  & 947.81             & 2929420  \\
miblp-20-15-50-0110-10-3  & -42.00   & 188.57             & 48461   &  & -42.00   & 12.65              & 46349    &  & -42.00   & 203.74             & 46349   &  & -42.00   & 12.66              & 46349    \\
miblp-20-15-50-0110-10-6  & -246.00  & 12.02              & 1385    &  & -246.00  & 1.53               & 1367     &  & -246.00  & 13.22              & 1367    &  & -246.00  & 1.56               & 1367     \\
miblp-20-15-50-0110-10-9  & -635.00  & 25.65              & 7667    &  & -635.00  & 6.62               & 7667     &  & -635.00  & 26.03              & 7667    &  & -635.00  & 6.52               & 7667     \\
miblp-20-20-50-0110-10-10 & -405.00  & \textgreater{}3600 & 304054  &  & -441.00  & 2908.94            & 7261247  &  & -405.00  & \textgreater{}3600 & 297842  &  & -441.00  & 2867.85            & 7261247  \\
miblp-20-20-50-0110-10-1  & -315.00  & \textgreater{}3600 & 281846  &  & -357.00  & \textgreater{}3600 & 9176214  &  & -315.00  & \textgreater{}3600 & 268812  &  & -357.00  & \textgreater{}3600 & 9234364  \\
miblp-20-20-50-0110-10-2  & -659.00  & 14.46              & 5106    &  & -659.00  & 3.21               & 5100     &  & -659.00  & 13.90              & 5100    &  & -659.00  & 3.16               & 5100     \\
miblp-20-20-50-0110-10-3  & -618.00  & 16.74              & 52259   &  & -618.00  & 12.21              & 51970    &  & -618.00  & 17.08              & 51970   &  & -618.00  & 12.05              & 51970    \\
miblp-20-20-50-0110-10-4  & -597.00  & \textgreater{}3600 & 509252  &  & -604.00  & \textgreater{}3600 & 8018349  &  & -597.00  & \textgreater{}3600 & 502089  &  & -604.00  & \textgreater{}3600 & 7988914  \\
miblp-20-20-50-0110-10-7  & -627.00  & \textgreater{}3600 & 4935849 &  & -629.00  & \textgreater{}3600 & 9735490  &  & -627.00  & \textgreater{}3600 & 4698965 &  & -629.00  & \textgreater{}3600 & 9709672  \\
miblp-20-20-50-0110-10-8  & -667.00  & 998.04             & 80603   &  & -667.00  & 68.15              & 75661    &  & -667.00  & 1085.58            & 75661   &  & -667.00  & 68.08              & 75661    \\
miblp-20-20-50-0110-10-9  & -256.00  & 2292.26            & 852328  &  & -256.00  & 306.60             & 757349   &  & -256.00  & 2483.24            & 757349  &  & -256.00  & 305.78             & 757349   \\
miblp-20-20-50-0110-15-1  & -246.00  & \textgreater{}3600 & 250531  &  & -317.00  & \textgreater{}3600 & 9807301  &  & -246.00  & \textgreater{}3600 & 238329  &  & -317.00  & \textgreater{}3600 & 9813005  \\
miblp-20-20-50-0110-15-2  & -645.00  & \textgreater{}3600 & 1416126 &  & -645.00  & \textgreater{}3600 & 10740865 &  & -645.00  & \textgreater{}3600 & 1395442 &  & -645.00  & \textgreater{}3600 & 10839884 \\
miblp-20-20-50-0110-15-3  & -476.00  & \textgreater{}3600 & 189624  &  & -593.00  & \textgreater{}3600 & 13332780 &  & -476.00  & \textgreater{}3600 & 183565  &  & -593.00  & \textgreater{}3600 & 13042109 \\
miblp-20-20-50-0110-15-4  & -310.00  & \textgreater{}3600 & 125392  &  & -398.00  & \textgreater{}3600 & 8670375  &  & -310.00  & \textgreater{}3600 & 125871  &  & -398.00  & \textgreater{}3600 & 8692428  \\
miblp-20-20-50-0110-15-5  & -60.00   & \textgreater{}3600 & 48876   &  & -320.00  & \textgreater{}3600 & 7275287  &  & -60.00   & \textgreater{}3600 & 47748   &  & -320.00  & \textgreater{}3600 & 7284040  \\
miblp-20-20-50-0110-15-6  & -596.00  & \textgreater{}3600 & 227767  &  & -596.00  & \textgreater{}3600 & 7842912  &  & -596.00  & \textgreater{}3600 & 197263  &  & -596.00  & \textgreater{}3600 & 7851818  \\
miblp-20-20-50-0110-15-7  & -471.00  & \textgreater{}3600 & 393906  &  & -471.00  & \textgreater{}3600 & 9577404  &  & -471.00  & \textgreater{}3600 & 362335  &  & -471.00  & \textgreater{}3600 & 9675451  \\
miblp-20-20-50-0110-15-8  & -290.00  & \textgreater{}3600 & 1772115 &  & -290.00  & \textgreater{}3600 & 10439854 &  & -290.00  & \textgreater{}3600 & 1614730 &  & -290.00  & \textgreater{}3600 & 10350188 \\
miblp-20-20-50-0110-15-9  & -584.00  & 269.87             & 56931   &  & -584.00  & 19.38              & 56459    &  & -584.00  & 271.42             & 56459   &  & -584.00  & 19.58              & 56459    \\
miblp-20-20-50-0110-5-13  & -507.00  & \textgreater{}3600 & 6006058 &  & -519.00  & 2723.02            & 6515097  &  & -507.00  & \textgreater{}3600 & 6189742 &  & -519.00  & 2709.29            & 6515097  \\
miblp-20-20-50-0110-5-15  & -617.00  & 1754.55            & 4312915 &  & -617.00  & 1113.82            & 4312915  &  & -617.00  & 1844.46            & 4312915 &  & -617.00  & 1130.92            & 4312915  \\
miblp-20-20-50-0110-5-16  & -833.00  & 2.06               & 5913    &  & -833.00  & 1.76               & 5913     &  & -833.00  & 2.06               & 5913    &  & -833.00  & 1.80               & 5913     \\
miblp-20-20-50-0110-5-17  & -944.00  & 1.61               & 4050    &  & -944.00  & 1.50               & 4038     &  & -944.00  & 1.66               & 4038    &  & -944.00  & 1.53               & 4038     \\
miblp-20-20-50-0110-5-19  & -431.00  & 38.74              & 116041  &  & -431.00  & 25.54              & 116041   &  & -431.00  & 38.55              & 116041  &  & -431.00  & 25.50              & 116041   \\
miblp-20-20-50-0110-5-1   & -548.00  & 13.07              & 31432   &  & -548.00  & 8.63               & 31298    &  & -548.00  & 13.53              & 31298   &  & -548.00  & 8.45               & 31298    \\
miblp-20-20-50-0110-5-20  & -438.00  & 13.53              & 33315   &  & -438.00  & 10.80              & 33315    &  & -438.00  & 13.51              & 33315   &  & -438.00  & 10.82              & 33315    \\
miblp-20-20-50-0110-5-6   & -1061.00 & 66.75              & 214009  &  & -1061.00 & 51.65              & 213928   &  & -1061.00 & 67.85              & 213928  &  & -1061.00 & 51.74              & 213928   \\
lseu-0.100000             & 1120.00  & 3.17               & 8603    &  & 1120.00  & 3.12               & 8603     &  & 1120.00  & 3.16               & 8603    &  & 1120.00  & 3.15               & 8603     \\
lseu-0.900000             & 5838.00  & \textgreater{}3600 & 2463002 &  & 5838.00  & \textgreater{}3600 & 8762978  &  & 5838.00  & \textgreater{}3600 & 2460978 &  & 5838.00  & \textgreater{}3600 & 8743754  \\
p0033-0.500000            & 3095.00  & 0.28               & 1467    &  & 3095.00  & 0.27               & 1467     &  & 3095.00  & 0.26               & 1467    &  & 3095.00  & 0.25               & 1467     \\
p0033-0.900000            & 4679.00  & 37.10              & 33451   &  & 4679.00  & 5.58               & 28241    &  & 4679.00  & 31.19              & 28241   &  & 4679.00  & 5.56               & 28241    \\
p0201-0.900000            & 15025.00 & \textgreater{}3600 & 690347  &  & 15025.00 & \textgreater{}3600 & 1309676  &  & 15025.00 & \textgreater{}3600 & 689225  &  & 15025.00 & \textgreater{}3600 & 1310278  \\
stein27-0.500000          & 19.00    & 6.60               & 17648   &  & 19.00    & 6.43               & 17648    &  & 19.00    & 6.84               & 17648   &  & 19.00    & 6.51               & 17648    \\
stein27-0.900000          & 24.00    & 644.53             & 713061  &  & 24.00    & 413.77             & 702055   &  & 24.00    & 638.54             & 702055  &  & 24.00    & 419.87             & 702055   \\
stein45-0.100000          & 30.00    & 64.50              & 90241   &  & 30.00    & 63.84              & 90241    &  & 30.00    & 64.46              & 90241   &  & 30.00    & 64.27              & 90241    \\
stein45-0.500000          & 32.00    & 465.23             & 753845  &  & 32.00    & 468.17             & 753845   &  & 32.00    & 472.32             & 753845  &  & 32.00    & 471.57             & 753845   \\
stein45-0.900000          & 40.00    & \textgreater{}3600 & 1379438 &  & 40.00    & \textgreater{}3600 & 1504193  &  & 40.00    & \textgreater{}3600 & 1376878 &  & 40.00    & \textgreater{}3600 & 1499583 \\
\hline
\end{tabular}
}
\end{table}
%%%%%%%%%%%%%%%%%%
%table for figure 4b
\begin{table}[h!]
\caption{Detailed results of Figure~\ref{fig:linkingSolutionPooParWithLinkParWithTimeMore5Sec}}
\label{tab:fig4b}
\resizebox{\columnwidth}{!}{
%\begin{tabular}{lccccccccccccccc}
\begin{tabular}{llllllllllllllll}
\hline
\multicolumn{1}{c}{}&
\multicolumn{3}{c}{\begin{tabular}[c]{@{}c@{}}\scriptsize{withoutPoolWhen}\\\scriptsize{XYInt-LFixed}\end{tabular}}&
\multicolumn{1}{c}{}&
\multicolumn{3}{c}{\begin{tabular}[c]{@{}c@{}}\scriptsize{withPoolWhen}\\\scriptsize{XYInt-LFixed}\end{tabular}}&
\multicolumn{1}{c}{}&
\multicolumn{3}{c}{\begin{tabular}[c]{@{}c@{}}\scriptsize{withoutPoolWhen}\\\scriptsize{XYIntOrLFixed-LFixed}\end{tabular}}&
\multicolumn{1}{c}{}&
\multicolumn{3}{c}{\begin{tabular}[c]{@{}c@{}}\scriptsize{withPoolWhen}\\\scriptsize{XYIntOrLFixed-LFixed}\end{tabular}} \\
\cline{2-4}
\cline{6-8}
\cline{10-12}
\cline{14-16}
Instance&   BestSol & Time(s) &Nodes  & &BestSol & Time(s) &Nodes & &BestSol & Time(s) &Nodes& &BestSol & Time(s) &Nodes\\
\hline
bmilplib-110-10           & -177.67  & 416.71             & 121469   &  & -177.67  & 149.10  & 93801    &  & -177.67  & 155.24  & 75499   &  & -177.67  & 153.82  & 75499   \\
bmilplib-110-1            & -181.67  & 5.52               & 2892     &  & -181.67  & 5.42    & 2878     &  & -181.67  & 5.32    & 2806    &  & -181.67  & 5.35    & 2806    \\
bmilplib-110-2            & -110.67  & 7.04               & 4279     &  & -110.67  & 6.50    & 4272     &  & -110.67  & 6.49    & 4174    &  & -110.67  & 6.48    & 4174    \\
bmilplib-110-3            & -215.16  & 6.11               & 3599     &  & -215.16  & 5.42    & 3556     &  & -215.16  & 5.55    & 3471    &  & -215.16  & 5.58    & 3471    \\
bmilplib-110-4            & -197.29  & 4.64               & 1827     &  & -197.29  & 3.54    & 1801     &  & -197.29  & 3.70    & 1751    &  & -197.29  & 3.69    & 1751    \\
bmilplib-110-6            & -148.25  & 19.75              & 9578     &  & -148.25  & 14.14   & 9392     &  & -148.25  & 14.88   & 8770    &  & -148.25  & 15.06   & 8770    \\
bmilplib-110-7            & -160.86  & 4.30               & 2291     &  & -160.86  & 4.01    & 2281     &  & -160.86  & 4.86    & 2189    &  & -160.86  & 4.94    & 2189    \\
bmilplib-110-8            & -155.00  & 25.37              & 12659    &  & -155.00  & 20.65   & 12498    &  & -155.00  & 20.78   & 11663   &  & -155.00  & 20.78   & 11663   \\
bmilplib-110-9            & -192.92  & 3.33               & 2177     &  & -192.92  & 3.28    & 2171     &  & -192.92  & 3.63    & 2151    &  & -192.92  & 3.66    & 2151    \\
bmilplib-160-10           & -189.82  & 40.10              & 9425     &  & -189.82  & 31.24   & 9344     &  & -189.82  & 33.76   & 9071    &  & -189.82  & 33.42   & 9071    \\
bmilplib-160-1            & -165.00  & 26.18              & 8157     &  & -165.00  & 24.18   & 8145     &  & -165.00  & 25.67   & 7901    &  & -165.00  & 25.70   & 7901    \\
bmilplib-160-2            & -178.24  & 33.86              & 8719     &  & -178.24  & 29.64   & 8710     &  & -178.24  & 30.82   & 8635    &  & -178.24  & 30.59   & 8635    \\
bmilplib-160-3            & -174.94  & 74.85              & 15456    &  & -174.94  & 54.62   & 15427    &  & -174.94  & 65.13   & 15002   &  & -174.94  & 63.37   & 15002   \\
bmilplib-160-4            & -135.83  & 69.61              & 15513    &  & -135.83  & 50.86   & 15354    &  & -135.83  & 51.56   & 14772   &  & -135.83  & 51.02   & 14772   \\
bmilplib-160-5            & -140.78  & 45.58              & 5362     &  & -140.78  & 19.35   & 5178     &  & -140.78  & 20.58   & 4948    &  & -140.78  & 20.52   & 4948    \\
bmilplib-160-6            & -111.00  & 27.27              & 8878     &  & -111.00  & 25.08   & 8864     &  & -111.00  & 24.82   & 8676    &  & -111.00  & 24.66   & 8676    \\
bmilplib-160-7            & -96.00   & 67.91              & 17668    &  & -96.00   & 50.97   & 17430    &  & -96.00   & 52.58   & 16560   &  & -96.00   & 51.20   & 16560   \\
bmilplib-160-8            & -181.40  & 9.50               & 3564     &  & -181.40  & 9.32    & 3540     &  & -181.40  & 10.45   & 3444    &  & -181.40  & 10.49   & 3444    \\
bmilplib-160-9            & -207.50  & 15.01              & 4754     &  & -207.50  & 14.90   & 4728     &  & -207.50  & 15.87   & 4607    &  & -207.50  & 15.85   & 4607    \\
bmilplib-210-10           & -130.59  & 66.32              & 12596    &  & -130.59  & 61.00   & 12497    &  & -130.59  & 61.65   & 12176   &  & -130.59  & 61.81   & 12176   \\
bmilplib-210-1            & -136.80  & 43.30              & 7608     &  & -136.80  & 40.04   & 7598     &  & -136.80  & 41.01   & 7429    &  & -136.80  & 41.96   & 7429    \\
bmilplib-210-2            & -117.80  & 132.27             & 23482    &  & -117.80  & 109.56  & 23096    &  & -117.80  & 118.80  & 22294   &  & -117.80  & 113.38  & 22294   \\
bmilplib-210-3            & -130.80  & 60.91              & 9404     &  & -130.80  & 47.44   & 9165     &  & -130.80  & 48.09   & 8848    &  & -130.80  & 48.19   & 8848    \\
bmilplib-210-4            & -162.20  & 27.29              & 4812     &  & -162.20  & 25.92   & 4806     &  & -162.20  & 26.63   & 4687    &  & -162.20  & 26.81   & 4687    \\
bmilplib-210-5            & -134.00  & 117.77             & 21646    &  & -134.00  & 105.03  & 21552    &  & -134.00  & 106.92  & 21058   &  & -134.00  & 108.22  & 21058   \\
bmilplib-210-6            & -125.43  & 261.84             & 40736    &  & -125.43  & 197.88  & 39941    &  & -125.43  & 213.31  & 38443   &  & -125.43  & 201.40  & 38443   \\
bmilplib-210-7            & -169.73  & 82.57              & 14351    &  & -169.73  & 77.62   & 14305    &  & -169.73  & 77.43   & 13960   &  & -169.73  & 77.84   & 13960   \\
bmilplib-210-8            & -101.46  & 67.82              & 11406    &  & -101.46  & 58.14   & 11347    &  & -101.46  & 56.85   & 11105   &  & -101.46  & 61.79   & 11105   \\
bmilplib-210-9            & -184.00  & 912.62             & 143788   &  & -184.00  & 859.46  & 143665   &  & -184.00  & 849.38  & 142294  &  & -184.00  & 879.61  & 142294  \\
bmilplib-260-10           & -151.73  & 559.37             & 63243    &  & -151.73  & 502.94  & 63026    &  & -151.73  & 514.35  & 61767   &  & -151.73  & 546.55  & 61767   \\
bmilplib-260-1            & -139.00  & 82.50              & 11269    &  & -139.00  & 81.74   & 11242    &  & -139.00  & 81.82   & 10980   &  & -139.00  & 85.81   & 10980   \\
bmilplib-260-2            & -82.62   & 125.39             & 15896    &  & -82.62   & 109.75  & 15712    &  & -82.62   & 107.89  & 15287   &  & -82.62   & 109.75  & 15287   \\
bmilplib-260-3            & -144.25  & 74.50              & 8865     &  & -144.25  & 69.26   & 8859     &  & -144.25  & 74.50   & 8769    &  & -144.25  & 71.54   & 8769    \\
bmilplib-260-4            & -117.33  & 284.35             & 34237    &  & -117.33  & 260.01  & 33931    &  & -117.33  & 272.56  & 32787   &  & -117.33  & 259.48  & 32787   \\
bmilplib-260-5            & -121.00  & 181.75             & 22295    &  & -121.00  & 173.32  & 22121    &  & -121.00  & 167.03  & 21542   &  & -121.00  & 166.03  & 21542   \\
bmilplib-260-6            & -124.00  & 240.41             & 26166    &  & -124.00  & 209.79  & 26023    &  & -124.00  & 195.61  & 25362   &  & -124.00  & 197.04  & 25362   \\
bmilplib-260-7            & -137.80  & 357.63             & 41519    &  & -137.80  & 318.76  & 41180    &  & -137.80  & 342.97  & 40274   &  & -137.80  & 312.12  & 40274   \\
bmilplib-260-8            & -119.89  & 87.09              & 10445    &  & -119.89  & 76.70   & 10332    &  & -119.89  & 75.72   & 10025   &  & -119.89  & 87.01   & 10025   \\
bmilplib-260-9            & -160.00  & 318.76             & 34545    &  & -160.00  & 271.18  & 34236    &  & -160.00  & 273.16  & 33468   &  & -160.00  & 273.66  & 33468   \\
bmilplib-310-10           & -141.86  & 119.62             & 10032    &  & -141.86  & 114.38  & 10033    &  & -141.86  & 110.92  & 9900    &  & -141.86  & 121.07  & 9900    \\
bmilplib-310-1            & -117.00  & 363.13             & 29350    &  & -117.00  & 329.80  & 29055    &  & -117.00  & 346.52  & 28330   &  & -117.00  & 329.71  & 28330   \\
bmilplib-310-2            & -105.00  & 519.72             & 45235    &  & -105.00  & 457.69  & 44841    &  & -105.00  & 519.36  & 43399   &  & -105.00  & 497.66  & 43399   \\
bmilplib-310-3            & -127.52  & 781.32             & 67649    &  & -127.52  & 783.89  & 67441    &  & -127.52  & 803.04  & 66410   &  & -127.52  & 777.74  & 66410   \\
bmilplib-310-4            & -147.78  & 584.37             & 49620    &  & -147.78  & 572.76  & 48950    &  & -147.78  & 564.98  & 47711   &  & -147.78  & 569.43  & 47711   \\
bmilplib-310-5            & -161.45  & 392.22             & 32631    &  & -161.45  & 382.13  & 32454    &  & -161.45  & 396.10  & 31782   &  & -161.45  & 366.90  & 31782   \\
bmilplib-310-6            & -141.18  & 1191.16            & 103385   &  & -141.18  & 1264.60 & 103214   &  & -141.18  & 1229.32 & 101904  &  & -141.18  & 1191.51 & 101904  \\
bmilplib-310-7            & -142.00  & 1132.52            & 104378   &  & -142.00  & 1199.47 & 104166   &  & -142.00  & 1142.79 & 102000  &  & -142.00  & 1129.24 & 102000  \\
bmilplib-310-8            & -115.34  & 129.19             & 11336    &  & -115.34  & 114.80  & 11293    &  & -115.34  & 137.89  & 11109   &  & -115.34  & 127.94  & 11109   \\
bmilplib-310-9            & -115.65  & 262.22             & 20893    &  & -115.65  & 248.24  & 20838    &  & -115.65  & 254.18  & 20490   &  & -115.65  & 255.27  & 20490   \\
bmilplib-360-10           & -108.59  & 254.46             & 14127    &  & -108.59  & 242.94  & 14064    &  & -108.59  & 261.28  & 13697   &  & -108.59  & 234.92  & 13697   \\
bmilplib-360-1            & -133.00  & 1390.85            & 77371    &  & -133.00  & 1297.32 & 77066    &  & -133.00  & 1380.33 & 75421   &  & -133.00  & 1353.62 & 75421   \\
bmilplib-360-2            & -138.44  & 980.62             & 54874    &  & -138.44  & 999.50  & 54354    &  & -138.44  & 953.84  & 52919   &  & -138.44  & 965.52  & 52919   \\
bmilplib-360-3            & -131.00  & 689.69             & 41459    &  & -131.00  & 624.78  & 41302    &  & -131.00  & 643.52  & 40487   &  & -131.00  & 832.88  & 40487   \\
bmilplib-360-4            & -119.00  & 334.15             & 18870    &  & -119.00  & 338.49  & 18813    &  & -119.00  & 331.26  & 18293   &  & -119.00  & 332.90  & 18293   \\
bmilplib-360-5            & -164.26  & 527.22             & 30465    &  & -164.26  & 484.66  & 30352    &  & -164.26  & 526.38  & 29947   &  & -164.26  & 618.86  & 29947   \\
bmilplib-360-6            & -110.12  & 1147.18            & 70479    &  & -110.12  & 1018.82 & 70283    &  & -110.12  & 1131.67 & 68863   &  & -110.12  & 1181.02 & 68863   \\
bmilplib-360-7            & -105.00  & 554.89             & 32224    &  & -105.00  & 517.46  & 31884    &  & -105.00  & 479.93  & 30900   &  & -105.00  & 634.17  & 30900   \\
bmilplib-360-8            & -98.25   & 438.62             & 23402    &  & -98.25   & 416.97  & 23337    &  & -98.25   & 423.75  & 22857   &  & -98.25   & 362.89  & 22857   \\
bmilplib-360-9            & -127.22  & 708.99             & 41321    &  & -127.22  & 746.66  & 41235    &  & -127.22  & 656.26  & 40329   &  & -127.22  & 736.16  & 40329   \\
bmilplib-410-10           & -153.37  & 2756.08            & 103891   &  & -153.37  & 2879.03 & 103400   &  & -153.37  & 2624.30 & 101447  &  & -153.37  & 2729.77 & 101447  \\
bmilplib-410-1            & -103.50  & 614.03             & 21120    &  & -103.50  & 589.62  & 21088    &  & -103.50  & 590.37  & 20634   &  & -103.50  & 553.76  & 20634   \\
bmilplib-410-2            & -108.59  & 846.60             & 32377    &  & -108.59  & 748.06  & 32251    &  & -108.59  & 746.96  & 31603   &  & -108.59  & 840.56  & 31603   \\
\hline
\end{tabular}
}
\end{table}
%%%%%%%%%%%%%%%%%%
%table for figure 4b (continued)
\addtocounter{table}{-1}
\begin{table}[h!]
\vskip -0.8in
\caption{Detailed results of Figure~\ref{fig:linkingSolutionPooParWithLinkParWithTimeMore5Sec} (continued)}
\label{tab:fig4bContinue}
\resizebox{\columnwidth}{!}{
%\begin{tabular}{lccccccccccccccc}
\begin{tabular}{llllllllllllllll}
\hline
\multicolumn{1}{c}{}&
\multicolumn{3}{c}{\begin{tabular}[c]{@{}c@{}}\scriptsize{withoutPoolWhen}\\\scriptsize{XYInt-LFixed}\end{tabular}}&
\multicolumn{1}{c}{}&
\multicolumn{3}{c}{\begin{tabular}[c]{@{}c@{}}\scriptsize{withPoolWhen}\\\scriptsize{XYInt-LFixed}\end{tabular}}&
\multicolumn{1}{c}{}&
\multicolumn{3}{c}{\begin{tabular}[c]{@{}c@{}}\scriptsize{withoutPoolWhen}\\\scriptsize{XYIntOrLFixed-LFixed}\end{tabular}}&
\multicolumn{1}{c}{}&
\multicolumn{3}{c}{\begin{tabular}[c]{@{}c@{}}\scriptsize{withPoolWhen}\\\scriptsize{XYIntOrLFixed-LFixed}\end{tabular}} \\
\cline{2-4}
\cline{6-8}
\cline{10-12}
\cline{14-16}
Instance&   BestSol & Time(s) &Nodes  & &BestSol & Time(s) &Nodes & &BestSol & Time(s) &Nodes& &BestSol & Time(s) &Nodes\\
\hline
bmilplib-410-3            & -96.24   & 877.76             & 33590    &  & -96.24   & 777.04  & 33473    &  & -96.24   & 874.84  & 32882   &  & -96.24   & 871.58  & 32882   \\
bmilplib-410-4            & -119.50  & 957.38             & 39535    &  & -119.50  & 1041.97 & 39371    &  & -119.50  & 944.66  & 38489   &  & -119.50  & 1031.12 & 38489   \\
bmilplib-410-5            & -119.22  & 677.99             & 24188    &  & -119.22  & 681.99  & 24163    &  & -119.22  & 583.64  & 23852   &  & -119.22  & 678.49  & 23852   \\
bmilplib-410-6            & -151.31  & 291.60             & 11213    &  & -151.31  & 333.36  & 11193    &  & -151.31  & 300.03  & 11130   &  & -151.31  & 322.88  & 11130   \\
bmilplib-410-7            & -123.00  & 514.64             & 22684    &  & -123.00  & 601.78  & 22628    &  & -123.00  & 588.83  & 22146   &  & -123.00  & 560.34  & 22146   \\
bmilplib-410-8            & -125.78  & 1679.79            & 64013    &  & -125.78  & 1692.35 & 63421    &  & -125.78  & 1462.52 & 62218   &  & -125.78  & 1547.88 & 62218   \\
bmilplib-410-9            & -100.77  & 561.87             & 20968    &  & -100.77  & 607.44  & 20920    &  & -100.77  & 539.37  & 20400   &  & -100.77  & 531.88  & 20400   \\
bmilplib-460-10           & -102.51  & 1837.01            & 56744    &  & -102.51  & 1701.70 & 56265    &  & -102.51  & 1964.44 & 55024   &  & -102.51  & 1857.94 & 55024   \\
bmilplib-460-1            & -97.59   & 2829.00            & 87769    &  & -97.59   & 2938.02 & 87709    &  & -97.59   & 2843.15 & 86632   &  & -97.59   & 2822.97 & 86632   \\
bmilplib-460-2            & -139.00  & 652.85             & 16943    &  & -139.00  & 608.40  & 16863    &  & -139.00  & 651.97  & 16623   &  & -139.00  & 625.95  & 16623   \\
bmilplib-460-3            & -86.50   & 1976.55            & 58270    &  & -86.50   & 2102.77 & 58229    &  & -86.50   & 1941.05 & 57395   &  & -86.50   & 1995.92 & 57395   \\
bmilplib-460-4            & -107.03  & 3577.74            & 97665    &  & -107.03  & 3328.91 & 97160    &  & -107.03  & 3160.13 & 95035   &  & -107.03  & 3291.28 & 95035   \\
bmilplib-460-5            & -100.50  & 1519.85            & 41863    &  & -100.50  & 1619.91 & 41796    &  & -100.50  & 1517.44 & 41109   &  & -100.50  & 1424.54 & 41109   \\
bmilplib-460-6            & -107.00  & 1603.92            & 46833    &  & -107.00  & 1483.17 & 46732    &  & -107.00  & 1528.18 & 46076   &  & -107.00  & 1635.42 & 46076   \\
bmilplib-460-7            & -83.75   & 1813.99            & 49915    &  & -83.75   & 1728.92 & 49717    &  & -83.75   & 1826.46 & 48897   &  & -83.75   & 1601.34 & 48897   \\
bmilplib-460-8            & -115.39  & 983.08             & 28147    &  & -115.39  & 911.51  & 28075    &  & -115.39  & 901.08  & 27572   &  & -115.39  & 982.30  & 27572   \\
bmilplib-460-9            & -128.70  & 3006.84            & 88287    &  & -128.70  & 2715.98 & 87796    &  & -128.70  & 3043.95 & 85977   &  & -128.70  & 2830.59 & 85977   \\
bmilplib-60-10            & -186.21  & 7.00               & 6136     &  & -186.21  & 6.91    & 6134     &  & -186.21  & 6.70    & 5922    &  & -186.21  & 6.93    & 5922    \\
bmilplib-60-1             & -153.20  & 7.31               & 5060     &  & -153.20  & 6.50    & 5041     &  & -153.20  & 6.48    & 4877    &  & -153.20  & 6.71    & 4877    \\
bmilplib-60-5             & -116.40  & 13.21              & 15368    &  & -116.40  & 8.65    & 13984    &  & -116.40  & 7.98    & 11202   &  & -116.40  & 8.15    & 11202   \\
bmilplib-60-6             & -187.31  & 7.57               & 7310     &  & -187.31  & 7.34    & 7304     &  & -187.31  & 7.22    & 7016    &  & -187.31  & 7.62    & 7016    \\
bmilplib-60-8             & -232.12  & 3.66               & 3700     &  & -232.12  & 3.33    & 3683     &  & -232.12  & 3.16    & 3570    &  & -232.12  & 3.36    & 3570    \\
bmilplib-60-9             & -136.50  & 22.60              & 28947    &  & -136.50  & 20.22   & 28603    &  & -136.50  & 19.79   & 26792   &  & -136.50  & 19.61   & 26792   \\
miblp-20-15-50-0110-10-10 & -206.00  & 5.32               & 3150     &  & -206.00  & 0.97    & 1414     &  & -206.00  & 1.03    & 423     &  & -206.00  & 1.06    & 423     \\
miblp-20-15-50-0110-10-2  & -398.00  & 427.10             & 229000   &  & -398.00  & 11.12   & 27625    &  & -398.00  & 9.11    & 2450    &  & -398.00  & 9.09    & 2450    \\
miblp-20-15-50-0110-10-3  & -42.00   & 6.10               & 3151     &  & -42.00   & 0.71    & 1343     &  & -42.00   & 0.62    & 267     &  & -42.00   & 0.62    & 267     \\
miblp-20-15-50-0110-10-6  & -246.00  & 5.47               & 1183     &  & -246.00  & 2.19    & 853      &  & -246.00  & 4.47    & 340     &  & -246.00  & 4.56    & 340     \\
miblp-20-15-50-0110-10-9  & -635.00  & 16.24              & 7425     &  & -635.00  & 7.86    & 6457     &  & -635.00  & 14.46   & 2380    &  & -635.00  & 14.05   & 2380    \\
miblp-20-20-50-0110-10-10 & -441.00  & 3568.40            & 1159333  &  & -441.00  & 526.56  & 639146   &  & -441.00  & 717.56  & 134585  &  & -441.00  & 692.00  & 134585  \\
miblp-20-20-50-0110-10-1  & -353.00  & \textgreater{}3600 & 686972   &  & -359.00  & 260.31  & 423149   &  & -359.00  & 231.63  & 96141   &  & -359.00  & 228.67  & 96141   \\
miblp-20-20-50-0110-10-2  & -659.00  & 5.16               & 7729     &  & -659.00  & 2.78    & 7547     &  & -659.00  & 3.80    & 3191    &  & -659.00  & 3.69    & 3191    \\
miblp-20-20-50-0110-10-3  & -618.00  & 14.11              & 42906    &  & -618.00  & 10.74   & 38188    &  & -618.00  & 13.14   & 20788   &  & -618.00  & 13.18   & 20788   \\
miblp-20-20-50-0110-10-4  & -604.00  & \textgreater{}3600 & 1351975  &  & -604.00  & 3405.57 & 6452671  &  & -604.00  & 3089.99 & 830808  &  & -604.00  & 3145.16 & 830808  \\
miblp-20-20-50-0110-10-7  & -650.00  & \textgreater{}3600 & 7509639  &  & -683.00  & 3184.99 & 11502091 &  & -683.00  & 2009.54 & 3003967 &  & -683.00  & 1887.37 & 3003967 \\
miblp-20-20-50-0110-10-8  & -667.00  & 155.00             & 39744    &  & -667.00  & 40.78   & 30873    &  & -667.00  & 87.06   & 12857   &  & -667.00  & 80.45   & 12857   \\
miblp-20-20-50-0110-10-9  & -256.00  & 77.53              & 127374   &  & -256.00  & 23.97   & 76055    &  & -256.00  & 21.78   & 31245   &  & -256.00  & 20.99   & 31245   \\
miblp-20-20-50-0110-15-1  & -289.00  & \textgreater{}3600 & 474713   &  & -450.00  & 65.16   & 49137    &  & -450.00  & 65.21   & 3506    &  & -450.00  & 59.38   & 3506    \\
miblp-20-20-50-0110-15-2  & -645.00  & \textgreater{}3600 & 1897076  &  & -645.00  & 96.13   & 346065   &  & -645.00  & 49.73   & 17251   &  & -645.00  & 50.14   & 17251   \\
miblp-20-20-50-0110-15-3  & -593.00  & \textgreater{}3600 & 231893   &  & -593.00  & 65.25   & 42877    &  & -593.00  & 72.10   & 3081    &  & -593.00  & 71.15   & 3081    \\
miblp-20-20-50-0110-15-4  & -323.00  & \textgreater{}3600 & 205392   &  & -441.00  & 36.99   & 29904    &  & -441.00  & 43.47   & 1625    &  & -441.00  & 43.00   & 1625    \\
miblp-20-20-50-0110-15-5  & -75.00   & \textgreater{}3600 & 103173   &  & -379.00  & 615.16  & 205025   &  & -379.00  & 644.67  & 16715   &  & -379.00  & 651.28  & 16715   \\
miblp-20-20-50-0110-15-6  & -596.00  & \textgreater{}3600 & 317849   &  & -596.00  & 17.87   & 29923    &  & -596.00  & 18.98   & 1657    &  & -596.00  & 17.31   & 1657    \\
miblp-20-20-50-0110-15-7  & -471.00  & \textgreater{}3600 & 1084548  &  & -471.00  & 125.24  & 241285   &  & -471.00  & 126.10  & 13405   &  & -471.00  & 111.02  & 13405   \\
miblp-20-20-50-0110-15-8  & -290.00  & \textgreater{}3600 & 4898557  &  & -370.00  & 138.34  & 579309   &  & -370.00  & 39.01   & 21589   &  & -370.00  & 39.12   & 21589   \\
miblp-20-20-50-0110-15-9  & -584.00  & 20.00              & 12193    &  & -584.00  & 1.96    & 4072     &  & -584.00  & 2.07    & 582     &  & -584.00  & 2.00    & 582     \\
miblp-20-20-50-0110-5-13  & -519.00  & 2582.26            & 10718411 &  & -519.00  & 2307.90 & 10196998 &  & -519.00  & 2516.55 & 7128138 &  & -519.00  & 2464.25 & 7128138 \\
miblp-20-20-50-0110-5-15  & -617.00  & 2170.42            & 8462491  &  & -617.00  & 1219.33 & 6630921  &  & -617.00  & 1297.16 & 4018058 &  & -617.00  & 1310.67 & 4018058 \\
miblp-20-20-50-0110-5-16  & -833.00  & 6.04               & 20129    &  & -833.00  & 4.86    & 19013    &  & -833.00  & 7.36    & 17680   &  & -833.00  & 6.60    & 17680   \\
miblp-20-20-50-0110-5-17  & -944.00  & 3.71               & 20049    &  & -944.00  & 3.99    & 21541    &  & -944.00  & 5.21    & 17679   &  & -944.00  & 4.78    & 17679   \\
miblp-20-20-50-0110-5-19  & -431.00  & 52.75              & 197157   &  & -431.00  & 27.52   & 147268   &  & -431.00  & 32.98   & 77256   &  & -431.00  & 26.56   & 77256   \\
miblp-20-20-50-0110-5-1   & -548.00  & 16.54              & 69127    &  & -548.00  & 12.98   & 64067    &  & -548.00  & 17.63   & 42364   &  & -548.00  & 14.50   & 42364   \\
miblp-20-20-50-0110-5-20  & -438.00  & 17.24              & 80526    &  & -438.00  & 15.81   & 76091    &  & -438.00  & 18.83   & 60944   &  & -438.00  & 17.30   & 60944   \\
miblp-20-20-50-0110-5-6   & -1061.00 & 66.85              & 301416   &  & -1061.00 & 58.74   & 284550   &  & -1061.00 & 63.82   & 222566  &  & -1061.00 & 63.70   & 222566  \\
lseu-0.100000             & 1120.00  & 248.50             & 1071409  &  & 1120.00  & 248.54  & 1071409  &  & 1120.00  & 262.55  & 1003976 &  & 1120.00  & 270.78  & 1003976 \\
lseu-0.900000             & 5838.00  & \textgreater{}3600 & 1631763  &  & 5838.00  & 1063.02 & 4718749  &  & 5838.00  & 14.12   & 1023    &  & 5838.00  & 14.10   & 1023    \\
p0033-0.500000            & 3095.00  & 6.08               & 33614    &  & 3095.00  & 6.24    & 33614    &  & 3095.00  & 10.55   & 20695   &  & 3095.00  & 10.48   & 20695   \\
p0033-0.900000            & 4679.00  & 14.66              & 23699    &  & 4679.00  & 0.65    & 3455     &  & 4679.00  & 0.06    & 27      &  & 4679.00  & 0.06    & 27      \\
p0201-0.900000            & 15025.00 & \textgreater{}3600 & 1037538  &  & 15025.00 & 20.58   & 18801    &  & 15025.00 & 6.66    & 2481    &  & 15025.00 & 6.62    & 2481    \\
stein27-0.500000          & 19.00    & 7.94               & 22537    &  & 19.00    & 7.36    & 21515    &  & 19.00    & 5.82    & 12115   &  & 19.00    & 5.88    & 12115   \\
stein27-0.900000          & 24.00    & 29.12              & 36927    &  & 24.00    & 1.25    & 4445     &  & 24.00    & 0.02    & 15      &  & 24.00    & 0.02    & 15      \\
stein45-0.100000          & 30.00    & 49.86              & 89035    &  & 30.00    & 50.47   & 89035    &  & 30.00    & 50.33   & 89035   &  & 30.00    & 50.19   & 89035   \\
stein45-0.500000          & 32.00    & 640.82             & 963098   &  & 32.00    & 635.22  & 952123   &  & 32.00    & 516.60  & 640308  &  & 32.00    & 519.59  & 640308  \\
stein45-0.900000          & 40.00    & \textgreater{}3600 & 2651427  &  & 40.00    & 85.92   & 103661   &  & 40.00    & 0.16    & 63      &  & 40.00    & 0.14    & 63     \\
\hline
\end{tabular}
}
\end{table}
%%%%%%%%%%%%%%%%%%
%table for figure 5
\begin{table}[h!]
\caption{Detailed results of Figure~\ref{fig:heuristicsParTimeMore5Sec}}
\label{tab:fig5}
\centering
\resizebox{\columnwidth}{!}{
%\begin{tabular}{lccccccccccccccc}
\begin{tabular}{llllllllllllllll}
\hline
\multicolumn{1}{c}{}&
\multicolumn{3}{c}{noHeuristics}&
\multicolumn{1}{c}{}&
\multicolumn{3}{c}{impObjectiveCut}&
\multicolumn{1}{c}{}&
\multicolumn{3}{c}{secondLevelPriority}&
\multicolumn{1}{c}{}&
\multicolumn{3}{c}{weightedSums} \\
\cline{2-4}
\cline{6-8}
\cline{10-12}
\cline{14-16}
Instance&   BestSol & Time(s) &Nodes  & &BestSol & Time(s) &Nodes & &BestSol & Time(s) &Nodes& &BestSol & Time(s) &Nodes\\
\hline
bmilplib-110-10           & -177.67  & 116.32  & 55177   &  & -177.67  & 117.32  & 55177   &  & -177.67  & 146.32  & 55177   &  & -177.67  & 1024.90            & 55177   \\
bmilplib-110-1            & -181.67  & 0.93    & 306     &  & -181.67  & 0.96    & 306     &  & -181.67  & 2.90    & 306     &  & -181.67  & 67.04              & 306     \\
bmilplib-110-2            & -110.67  & 1.16    & 303     &  & -110.67  & 1.26    & 303     &  & -110.67  & 1.64    & 303     &  & -110.67  & 25.89              & 303     \\
bmilplib-110-3            & -215.16  & 0.89    & 360     &  & -215.16  & 1.50    & 361     &  & -215.16  & 1.93    & 360     &  & -215.16  & 30.52              & 360     \\
bmilplib-110-4            & -197.29  & 1.34    & 148     &  & -197.29  & 1.54    & 148     &  & -197.29  & 2.79    & 148     &  & -197.29  & 68.13              & 148     \\
bmilplib-110-6            & -148.25  & 3.84    & 1448    &  & -148.25  & 4.24    & 1448    &  & -148.25  & 4.84    & 1446    &  & -148.25  & 59.02              & 1446    \\
bmilplib-110-7            & -160.86  & 0.91    & 205     &  & -160.86  & 2.19    & 205     &  & -160.86  & 1.98    & 205     &  & -160.86  & 52.88              & 207     \\
bmilplib-110-8            & -155.00  & 7.66    & 2274    &  & -155.00  & 7.68    & 2274    &  & -155.00  & 10.06   & 2274    &  & -155.00  & 73.62              & 2274    \\
bmilplib-110-9            & -192.92  & 0.39    & 146     &  & -192.92  & 1.33    & 146     &  & -192.92  & 1.15    & 146     &  & -192.92  & 16.48              & 154     \\
bmilplib-160-10           & -189.82  & 7.73    & 728     &  & -189.82  & 7.89    & 728     &  & -189.82  & 11.28   & 728     &  & -189.82  & 65.95              & 728     \\
bmilplib-160-1            & -165.00  & 4.76    & 881     &  & -165.00  & 4.76    & 881     &  & -165.00  & 5.82    & 893     &  & -165.00  & 51.48              & 893     \\
bmilplib-160-2            & -178.24  & 5.95    & 507     &  & -178.24  & 6.64    & 507     &  & -178.24  & 6.96    & 507     &  & -178.24  & 53.77              & 507     \\
bmilplib-160-3            & -174.94  & 13.71   & 1102    &  & -174.94  & 13.57   & 1102    &  & -174.94  & 15.78   & 1102    &  & -174.94  & 70.46              & 1102    \\
bmilplib-160-4            & -135.83  & 19.13   & 2447    &  & -135.83  & 18.91   & 2463    &  & -135.83  & 24.00   & 2447    &  & -135.83  & 116.48             & 2422    \\
bmilplib-160-5            & -140.78  & 7.72    & 668     &  & -140.78  & 8.94    & 668     &  & -140.78  & 10.34   & 668     &  & -140.78  & 56.74              & 668     \\
bmilplib-160-6            & -111.00  & 5.02    & 627     &  & -111.00  & 5.03    & 627     &  & -111.00  & 5.40    & 627     &  & -111.00  & 19.14              & 627     \\
bmilplib-160-7            & -96.00   & 18.72   & 2855    &  & -96.00   & 17.83   & 2855    &  & -96.00   & 19.22   & 2855    &  & -96.00   & 74.44              & 2855    \\
bmilplib-160-8            & -181.40  & 1.74    & 311     &  & -181.40  & 2.12    & 311     &  & -181.40  & 2.44    & 311     &  & -181.40  & 28.91              & 311     \\
bmilplib-160-9            & -207.50  & 2.02    & 268     &  & -207.50  & 2.85    & 268     &  & -207.50  & 2.98    & 268     &  & -207.50  & 40.33              & 268     \\
bmilplib-210-10           & -130.59  & 8.59    & 1100    &  & -130.59  & 8.66    & 1100    &  & -130.59  & 10.02   & 1100    &  & -130.59  & 73.44              & 1100    \\
bmilplib-210-1            & -136.80  & 6.67    & 550     &  & -136.80  & 6.68    & 550     &  & -136.80  & 7.59    & 550     &  & -136.80  & 80.52              & 550     \\
bmilplib-210-2            & -117.80  & 19.78   & 2306    &  & -117.80  & 20.49   & 2306    &  & -117.80  & 22.29   & 2298    &  & -117.80  & 149.96             & 2308    \\
bmilplib-210-3            & -130.80  & 12.93   & 1380    &  & -130.80  & 13.04   & 1386    &  & -130.80  & 16.16   & 1380    &  & -130.80  & 218.18             & 1380    \\
bmilplib-210-4            & -162.20  & 3.64    & 309     &  & -162.20  & 5.41    & 309     &  & -162.20  & 4.63    & 309     &  & -162.20  & 41.06              & 309     \\
bmilplib-210-5            & -134.00  & 20.03   & 2079    &  & -134.00  & 20.03   & 2079    &  & -134.00  & 22.31   & 2079    &  & -134.00  & 161.23             & 2079    \\
bmilplib-210-6            & -125.43  & 45.28   & 4875    &  & -125.43  & 46.46   & 4875    &  & -125.43  & 52.83   & 4878    &  & -125.43  & 230.84             & 4875    \\
bmilplib-210-7            & -169.73  & 11.46   & 1181    &  & -169.73  & 12.35   & 1153    &  & -169.73  & 14.19   & 1181    &  & -169.73  & 148.58             & 1152    \\
bmilplib-210-8            & -101.46  & 9.60    & 942     &  & -101.46  & 12.36   & 976     &  & -101.46  & 12.71   & 942     &  & -101.46  & 255.24             & 942     \\
bmilplib-210-9            & -184.00  & 240.14  & 9466    &  & -184.00  & 242.65  & 9466    &  & -184.00  & 267.90  & 9466    &  & -184.00  & 1639.04            & 9466    \\
bmilplib-260-10           & -151.73  & 73.10   & 4716    &  & -151.73  & 75.91   & 4716    &  & -151.73  & 119.40  & 4716    &  & -151.73  & 783.50             & 4716    \\
bmilplib-260-1            & -139.00  & 10.10   & 887     &  & -139.00  & 11.58   & 887     &  & -139.00  & 11.79   & 887     &  & -139.00  & 135.30             & 887     \\
bmilplib-260-2            & -82.62   & 18.07   & 1607    &  & -82.62   & 18.48   & 1607    &  & -82.62   & 29.56   & 1607    &  & -82.62   & 195.56             & 1607    \\
bmilplib-260-3            & -144.25  & 8.70    & 518     &  & -144.25  & 9.34    & 518     &  & -144.25  & 9.60    & 518     &  & -144.25  & 105.89             & 518     \\
bmilplib-260-4            & -117.33  & 66.89   & 4426    &  & -117.33  & 62.45   & 4426    &  & -117.33  & 80.66   & 4426    &  & -117.33  & 671.57             & 4426    \\
bmilplib-260-5            & -121.00  & 29.19   & 2165    &  & -121.00  & 27.30   & 2165    &  & -121.00  & 30.22   & 2165    &  & -121.00  & 212.90             & 2165    \\
bmilplib-260-6            & -124.00  & 40.59   & 2420    &  & -124.00  & 39.48   & 2420    &  & -124.00  & 42.70   & 2420    &  & -124.00  & 426.31             & 2420    \\
bmilplib-260-7            & -137.80  & 44.45   & 3200    &  & -137.80  & 45.80   & 3200    &  & -137.80  & 48.73   & 3200    &  & -137.80  & 267.10             & 3200    \\
bmilplib-260-8            & -119.89  & 10.92   & 1025    &  & -119.89  & 10.46   & 1025    &  & -119.89  & 11.77   & 1025    &  & -119.89  & 102.82             & 1025    \\
bmilplib-260-9            & -160.00  & 36.36   & 2526    &  & -160.00  & 38.08   & 2526    &  & -160.00  & 39.54   & 2526    &  & -160.00  & 231.78             & 2526    \\
bmilplib-310-10           & -141.86  & 5.51    & 397     &  & -141.86  & 6.24    & 397     &  & -141.86  & 6.86    & 397     &  & -141.86  & 95.24              & 397     \\
bmilplib-310-1            & -117.00  & 44.79   & 2624    &  & -117.00  & 41.07   & 2624    &  & -117.00  & 44.15   & 2624    &  & -117.00  & 381.97             & 2624    \\
bmilplib-310-2            & -105.00  & 98.39   & 5372    &  & -105.00  & 94.00   & 5372    &  & -105.00  & 97.58   & 5372    &  & -105.00  & 385.78             & 5372    \\
bmilplib-310-3            & -127.52  & 149.07  & 7067    &  & -127.52  & 141.25  & 7067    &  & -127.52  & 151.24  & 7067    &  & -127.52  & 694.04             & 7067    \\
bmilplib-310-4            & -147.78  & 70.22   & 4767    &  & -147.78  & 73.75   & 4908    &  & -147.78  & 74.72   & 4863    &  & -147.78  & 510.45             & 4892    \\
bmilplib-310-5            & -161.45  & 34.22   & 1993    &  & -161.45  & 34.42   & 1993    &  & -161.45  & 35.32   & 1993    &  & -161.45  & 205.57             & 1993    \\
bmilplib-310-6            & -141.18  & 169.96  & 8300    &  & -141.18  & 172.88  & 8300    &  & -141.18  & 214.66  & 8300    &  & -141.18  & \textgreater{}3600 & 4459    \\
bmilplib-310-7            & -142.00  & 139.15  & 7263    &  & -142.00  & 140.50  & 7263    &  & -142.00  & 164.46  & 7263    &  & -142.00  & 1416.36            & 7263    \\
bmilplib-310-8            & -115.34  & 19.23   & 921     &  & -115.34  & 19.53   & 952     &  & -115.34  & 23.06   & 921     &  & -115.34  & 425.10             & 921     \\
bmilplib-310-9            & -115.65  & 32.31   & 1590    &  & -115.65  & 32.39   & 1590    &  & -115.65  & 34.56   & 1590    &  & -115.65  & 151.96             & 1590    \\
bmilplib-360-10           & -108.59  & 25.75   & 1106    &  & -108.59  & 23.23   & 1106    &  & -108.59  & 26.41   & 1102    &  & -108.59  & 156.13             & 1102    \\
bmilplib-360-1            & -133.00  & 158.38  & 4923    &  & -133.00  & 148.68  & 4920    &  & -133.00  & 188.86  & 4923    &  & -120.00  & \textgreater{}3600 & 1611    \\
bmilplib-360-2            & -138.44  & 148.90  & 4493    &  & -138.44  & 138.96  & 4493    &  & -138.44  & 160.91  & 4493    &  & -138.44  & 1961.79            & 4493    \\
bmilplib-360-3            & -131.00  & 65.41   & 2654    &  & -131.00  & 56.59   & 2654    &  & -131.00  & 66.60   & 2654    &  & -131.00  & 609.13             & 2654    \\
bmilplib-360-4            & -119.00  & 42.95   & 1564    &  & -119.00  & 44.65   & 1564    &  & -119.00  & 61.20   & 1564    &  & -119.00  & 825.58             & 1564    \\
bmilplib-360-5            & -164.26  & 44.45   & 1713    &  & -164.26  & 45.39   & 1713    &  & -164.26  & 41.95   & 1713    &  & -164.26  & 221.40             & 1713    \\
bmilplib-360-6            & -110.12  & 142.81  & 5520    &  & -110.12  & 127.35  & 5520    &  & -110.12  & 148.17  & 5520    &  & -110.12  & 2890.94            & 5556    \\
bmilplib-360-7            & -105.00  & 89.22   & 3346    &  & -105.00  & 88.94   & 3346    &  & -105.00  & 86.08   & 3346    &  & -105.00  & 357.79             & 3346    \\
bmilplib-360-8            & -98.25   & 44.85   & 1686    &  & -98.25   & 48.77   & 1686    &  & -98.25   & 72.01   & 1686    &  & -98.25   & 2774.87            & 1686    \\
bmilplib-360-9            & -127.22  & 50.24   & 2642    &  & -127.22  & 52.33   & 2642    &  & -127.22  & 58.90   & 2642    &  & -127.22  & 1191.63            & 2642    \\
bmilplib-410-10           & -153.37  & 258.89  & 7428    &  & -153.37  & 254.18  & 7428    &  & -153.37  & 278.87  & 7428    &  & -153.37  & 2497.78            & 7428    \\
bmilplib-410-1            & -103.50  & 87.94   & 1944    &  & -103.50  & 76.78   & 1944    &  & -103.50  & 88.08   & 1944    &  & -103.50  & 2256.70            & 1944    \\
\hline
\end{tabular}
}
\end{table}
%%%%%%%%%%%%%%%%%%
%table for figure 5 (continued)
\addtocounter{table}{-1}
\begin{table}[h!]
\vskip -0.65in
\caption{Detailed results of Figure~\ref{fig:heuristicsParTimeMore5Sec} (continued)}
\label{tab:fig5Continue}
\centering
\resizebox{\columnwidth}{!}{
%\begin{tabular}{lccccccccccccccc}
\begin{tabular}{llllllllllllllll}
\hline
\multicolumn{1}{c}{}&
\multicolumn{3}{c}{noHeuristics}&
\multicolumn{1}{c}{}&
\multicolumn{3}{c}{impObjectiveCut}&
\multicolumn{1}{c}{}&
\multicolumn{3}{c}{secondLevelPriority}&
\multicolumn{1}{c}{}&
\multicolumn{3}{c}{weightedSums} \\
\cline{2-4}
\cline{6-8}
\cline{10-12}
\cline{14-16}
Instance&   BestSol & Time(s) &Nodes  & &BestSol & Time(s) &Nodes & &BestSol & Time(s) &Nodes& &BestSol & Time(s) &Nodes\\
\hline
bmilplib-410-2            & -108.59  & 103.57  & 2887    &  & -108.59  & 92.86   & 2887    &  & -108.59  & 97.35   & 2887    &  & -108.59  & 947.06             & 2887    \\
bmilplib-410-3            & -96.24   & 147.99  & 3781    &  & -96.24   & 157.30  & 3781    &  & -96.24   & 257.83  & 3781    &  & -79.49   & \textgreater{}3600 & 878     \\
bmilplib-410-4            & -119.50  & 100.21  & 2995    &  & -119.50  & 100.04  & 2995    &  & -119.50  & 100.85  & 2995    &  & -119.50  & 1042.52            & 2995    \\
bmilplib-410-5            & -119.22  & 71.23   & 1520    &  & -119.22  & 61.46   & 1527    &  & -119.22  & 107.73  & 1520    &  & -119.22  & 426.60             & 1527    \\
bmilplib-410-6            & -151.31  & 13.54   & 533     &  & -151.31  & 13.74   & 533     &  & -151.31  & 17.48   & 533     &  & -151.31  & 727.04             & 533     \\
bmilplib-410-7            & -123.00  & 35.87   & 1177    &  & -123.00  & 35.79   & 1175    &  & -123.00  & 41.41   & 1177    &  & -123.00  & 406.53             & 1177    \\
bmilplib-410-8            & -125.78  & 169.15  & 4480    &  & -125.78  & 171.56  & 4480    &  & -125.78  & 190.33  & 4480    &  & -125.78  & 1408.43            & 4488    \\
bmilplib-410-9            & -100.77  & 82.17   & 2071    &  & -100.77  & 79.34   & 2071    &  & -100.77  & 85.38   & 2071    &  & -100.77  & 478.13             & 2071    \\
bmilplib-460-10           & -102.51  & 228.12  & 4465    &  & -102.51  & 204.04  & 4465    &  & -102.51  & 238.28  & 4465    &  & -102.51  & 1557.08            & 4465    \\
bmilplib-460-1            & -97.59   & 569.22  & 10803   &  & -97.59   & 610.72  & 10803   &  & -97.59   & 666.13  & 10803   &  & -93.40   & \textgreater{}3600 & 1240    \\
bmilplib-460-2            & -139.00  & 43.65   & 964     &  & -139.00  & 46.56   & 964     &  & -139.00  & 45.27   & 964     &  & -139.00  & 557.54             & 964     \\
bmilplib-460-3            & -86.50   & 223.58  & 3882    &  & -86.50   & 219.35  & 3882    &  & -86.50   & 251.89  & 3882    &  & -82.83   & \textgreater{}3600 & 2966    \\
bmilplib-460-4            & -107.03  & 412.76  & 7856    &  & -107.03  & 425.15  & 7856    &  & -107.03  & 399.79  & 7794    &  & -102.61  & \textgreater{}3600 & 4784    \\
bmilplib-460-5            & -100.50  & 170.30  & 3025    &  & -100.50  & 189.63  & 3025    &  & -100.50  & 177.04  & 3025    &  & -100.50  & 1788.18            & 3025    \\
bmilplib-460-6            & -107.00  & 236.30  & 4143    &  & -107.00  & 205.48  & 4143    &  & -107.00  & 266.17  & 4143    &  & -107.00  & 2457.96            & 4143    \\
bmilplib-460-7            & -83.75   & 294.68  & 5252    &  & -83.75   & 282.83  & 5252    &  & -83.75   & 341.16  & 5252    &  & -83.75   & \textgreater{}3600 & 1538    \\
bmilplib-460-8            & -115.39  & 94.67   & 1903    &  & -115.39  & 92.26   & 1908    &  & -115.39  & 140.92  & 1903    &  & -103.50  & \textgreater{}3600 & 611     \\
bmilplib-460-9            & -128.70  & 327.68  & 6185    &  & -128.70  & 316.77  & 6227    &  & -128.70  & 318.04  & 6185    &  & -128.70  & 3166.24            & 6185    \\
bmilplib-60-10            & -186.21  & 4.29    & 2590    &  & -186.21  & 4.32    & 2590    &  & -186.21  & 4.68    & 2590    &  & -186.21  & 22.24              & 2590    \\
bmilplib-60-1             & -153.20  & 2.79    & 1094    &  & -153.20  & 3.01    & 1094    &  & -153.20  & 3.47    & 1094    &  & -153.20  & 18.65              & 1094    \\
bmilplib-60-5             & -116.40  & 10.08   & 9996    &  & -116.40  & 10.41   & 9996    &  & -116.40  & 12.63   & 9996    &  & -116.40  & 35.73              & 9996    \\
bmilplib-60-6             & -187.31  & 3.34    & 1383    &  & -187.31  & 3.68    & 1383    &  & -187.31  & 4.07    & 1383    &  & -187.31  & 22.35              & 1383    \\
bmilplib-60-8             & -232.12  & 1.15    & 572     &  & -232.12  & 1.22    & 572     &  & -232.12  & 1.61    & 572     &  & -232.12  & 9.46               & 572     \\
bmilplib-60-9             & -136.50  & 12.34   & 9888    &  & -136.50  & 12.45   & 9888    &  & -136.50  & 13.56   & 9888    &  & -136.50  & 47.50              & 9888    \\
miblp-20-15-50-0110-10-10 & -206.00  & 1.06    & 423     &  & -206.00  & 1.07    & 423     &  & -206.00  & 1.01    & 423     &  & -206.00  & 1.81               & 423     \\
miblp-20-15-50-0110-10-2  & -398.00  & 9.09    & 2450    &  & -398.00  & 9.43    & 2450    &  & -398.00  & 8.85    & 2450    &  & -398.00  & 15.52              & 2450    \\
miblp-20-15-50-0110-10-3  & -42.00   & 0.62    & 267     &  & -42.00   & 0.72    & 267     &  & -42.00   & 0.68    & 267     &  & -42.00   & 1.36               & 267     \\
miblp-20-15-50-0110-10-6  & -246.00  & 4.56    & 340     &  & -246.00  & 4.43    & 340     &  & -246.00  & 4.24    & 340     &  & -246.00  & 6.40               & 340     \\
miblp-20-15-50-0110-10-9  & -635.00  & 14.05   & 2380    &  & -635.00  & 14.97   & 2382    &  & -635.00  & 15.48   & 2380    &  & -635.00  & 19.90              & 2380    \\
miblp-20-20-50-0110-10-10 & -441.00  & 692.00  & 134585  &  & -441.00  & 758.37  & 134585  &  & -441.00  & 743.82  & 134585  &  & -441.00  & 1113.34            & 134585  \\
miblp-20-20-50-0110-10-1  & -359.00  & 228.67  & 96141   &  & -359.00  & 232.30  & 96141   &  & -359.00  & 242.52  & 96141   &  & -359.00  & 407.03             & 96141   \\
miblp-20-20-50-0110-10-2  & -659.00  & 3.69    & 3191    &  & -659.00  & 4.15    & 3191    &  & -659.00  & 3.94    & 3191    &  & -659.00  & 9.48               & 3191    \\
miblp-20-20-50-0110-10-3  & -618.00  & 13.18   & 20788   &  & -618.00  & 13.68   & 20788   &  & -618.00  & 14.77   & 20788   &  & -618.00  & 54.78              & 20715   \\
miblp-20-20-50-0110-10-4  & -604.00  & 3145.16 & 830808  &  & -604.00  & 3286.17 & 830808  &  & -604.00  & 3223.61 & 830808  &  & -604.00  & \textgreater{}3600 & 666302  \\
miblp-20-20-50-0110-10-7  & -683.00  & 1887.37 & 3003967 &  & -683.00  & 1960.64 & 3014456 &  & -683.00  & 2049.54 & 3004274 &  & -683.00  & \textgreater{}3600 & 2542388 \\
miblp-20-20-50-0110-10-8  & -667.00  & 80.45   & 12857   &  & -667.00  & 94.47   & 12857   &  & -667.00  & 87.72   & 12858   &  & -667.00  & 134.18             & 12859   \\
miblp-20-20-50-0110-10-9  & -256.00  & 20.99   & 31245   &  & -256.00  & 21.84   & 31245   &  & -256.00  & 22.38   & 31245   &  & -256.00  & 38.71              & 31245   \\
miblp-20-20-50-0110-15-1  & -450.00  & 59.38   & 3506    &  & -450.00  & 59.36   & 3506    &  & -450.00  & 62.04   & 3506    &  & -450.00  & 85.24              & 3506    \\
miblp-20-20-50-0110-15-2  & -645.00  & 50.14   & 17251   &  & -645.00  & 52.11   & 17251   &  & -645.00  & 61.00   & 17251   &  & -645.00  & 222.62             & 17251   \\
miblp-20-20-50-0110-15-3  & -593.00  & 71.15   & 3081    &  & -593.00  & 66.66   & 3081    &  & -593.00  & 67.67   & 3081    &  & -593.00  & 95.58              & 3083    \\
miblp-20-20-50-0110-15-4  & -441.00  & 43.00   & 1625    &  & -441.00  & 43.74   & 1625    &  & -441.00  & 39.84   & 1625    &  & -441.00  & 57.56              & 1625    \\
miblp-20-20-50-0110-15-5  & -379.00  & 651.28  & 16715   &  & -379.00  & 634.58  & 16715   &  & -379.00  & 639.80  & 16715   &  & -379.00  & 750.81             & 16715   \\
miblp-20-20-50-0110-15-6  & -596.00  & 17.31   & 1657    &  & -596.00  & 17.40   & 1657    &  & -596.00  & 18.50   & 1657    &  & -596.00  & 25.94              & 1657    \\
miblp-20-20-50-0110-15-7  & -471.00  & 111.02  & 13405   &  & -471.00  & 109.14  & 13405   &  & -471.00  & 113.83  & 13405   &  & -471.00  & 191.29             & 13405   \\
miblp-20-20-50-0110-15-8  & -370.00  & 39.12   & 21589   &  & -370.00  & 39.12   & 21589   &  & -370.00  & 44.93   & 21589   &  & -370.00  & 122.50             & 21589   \\
miblp-20-20-50-0110-15-9  & -584.00  & 2.00    & 582     &  & -584.00  & 2.24    & 582     &  & -584.00  & 1.93    & 582     &  & -584.00  & 3.44               & 584     \\
miblp-20-20-50-0110-5-13  & -519.00  & 2709.29 & 6515097 &  & -519.00  & 2756.40 & 6514987 &  & -519.00  & 2828.20 & 6521367 &  & -519.00  & 3465.54            & 6515355 \\
miblp-20-20-50-0110-5-15  & -617.00  & 1130.92 & 4312915 &  & -617.00  & 1123.27 & 4312915 &  & -617.00  & 1202.16 & 4312915 &  & -617.00  & 1826.00            & 4309751 \\
miblp-20-20-50-0110-5-16  & -833.00  & 1.80    & 5913    &  & -833.00  & 1.76    & 5913    &  & -833.00  & 2.22    & 5925    &  & -833.00  & 9.74               & 5925    \\
miblp-20-20-50-0110-5-17  & -944.00  & 1.53    & 4038    &  & -944.00  & 1.48    & 4037    &  & -944.00  & 1.84    & 4038    &  & -944.00  & 8.31               & 4212    \\
miblp-20-20-50-0110-5-19  & -431.00  & 25.50   & 116041  &  & -431.00  & 26.02   & 116041  &  & -431.00  & 27.33   & 116065  &  & -431.00  & 45.05              & 116285  \\
miblp-20-20-50-0110-5-1   & -548.00  & 8.45    & 31298   &  & -548.00  & 9.08    & 31298   &  & -548.00  & 9.13    & 31298   &  & -548.00  & 22.96              & 31298   \\
miblp-20-20-50-0110-5-20  & -438.00  & 10.82   & 33315   &  & -438.00  & 10.92   & 33315   &  & -438.00  & 11.25   & 33315   &  & -438.00  & 17.54              & 33130   \\
miblp-20-20-50-0110-5-6   & -1061.00 & 51.74   & 213928  &  & -1061.00 & 52.57   & 213928  &  & -1061.00 & 56.75   & 214449  &  & -1061.00 & 235.63             & 234647  \\
lseu-0.100000             & 1120.00  & 3.15    & 8603    &  & 1120.00  & 3.11    & 8603    &  & 1120.00  & 3.82    & 13352   &  & 1120.00  & 45.32              & 22411   \\
lseu-0.900000             & 5838.00  & 14.10   & 1023    &  & 5838.00  & 14.11   & 1023    &  & 5838.00  & 251.51  & 1023    &  & 5838.00  & 21.57              & 1023    \\
p0033-0.500000            & 3095.00  & 0.25    & 1467    &  & 3095.00  & 0.28    & 1467    &  & 3095.00  & 0.29    & 1471    &  & 3095.00  & 0.80               & 2239    \\
p0033-0.900000            & 4679.00  & 0.06    & 27      &  & 4679.00  & 0.06    & 27      &  & 4679.00  & 0.07    & 27      &  & 4679.00  & 0.12               & 27      \\
p0201-0.900000            & 15025.00 & 6.62    & 2481    &  & 15025.00 & 7.09    & 2481    &  & 15025.00 & 359.66  & 2481    &  & 15025.00 & 34.52              & 2481    \\
stein27-0.500000          & 19.00    & 6.51    & 17648   &  & 19.00    & 6.44    & 17648   &  & 19.00    & 6.76    & 16969   &  & 19.00    & 9.57               & 16969   \\
stein27-0.900000          & 24.00    & 0.02    & 15      &  & 24.00    & 0.02    & 15      &  & 24.00    & 0.02    & 15      &  & 24.00    & 1.62               & 15      \\
stein45-0.100000          & 30.00    & 64.27   & 90241   &  & 30.00    & 64.12   & 90241   &  & 30.00    & 83.80   & 90002   &  & 30.00    & 256.23             & 58729   \\
stein45-0.500000          & 32.00    & 471.57  & 753845  &  & 32.00    & 459.73  & 753845  &  & 32.00    & 511.05  & 862249  &  & 32.00    & 880.47             & 565867  \\
stein45-0.900000          & 40.00    & 0.14    & 63      &  & 40.00    & 0.16    & 63      &  & 40.00    & 0.16    & 63      &  & 40.00    & 72.23              & 63 \\  
\hline
\end{tabular}
}
\end{table}
%endCoralReportTable
}

\end{document}